\documentclass[10pt]{amsart}
\usepackage{wasysym}
\usepackage{amsmath}
\usepackage{amssymb}
\usepackage{amsthm}
\usepackage{kbordermatrix}
\usepackage{tikz}
\usetikzlibrary{matrix,arrows,backgrounds,shapes.misc,shapes.geometric,patterns,calc,positioning}
\usetikzlibrary{calc,shapes}
\usepackage{wrapfig}
\usepackage{epsfig}  
\usepackage{mathtools}
\usepackage[margin=1in]{geometry}
\usepackage{color}	 
\usepackage{xy,color}
\xyoption{poly}
\xyoption{2cell}
\xyoption{all}
\usepackage{bm}

\newtheorem{thm}{Theorem}[section]
\newtheorem{prop}[thm]{Proposition}
\newtheorem{lemma}[thm]{Lemma}
\newtheorem{corollary}[thm]{Corollary}
\newtheorem{conjecture}[thm]{Conjecture}
    
\newtheorem*{thm*}{Theorem}

\usepackage{hyperref}
\theoremstyle{definition}
\newtheorem{definition}[thm]{Definition}
\theoremstyle{remark}
\newtheorem{remark}[thm]{Remark}
\newtheorem{example}[thm]{Example}
\numberwithin{equation}{section}

\newcommand{\calc}{\mathcal{C}}
\newcommand{\cals}{\mathcal{S}}

\newcommand{\cald}{\mathcal{D}}

\newcommand{\ssi}{\Leftrightarrow}
\newcommand{\za}{\alpha}
\newcommand{\zb}{\beta}
\newcommand{\zd}{\delta}

\newcommand{\ze}{\epsilon}
\newcommand{\zg}{\gamma}
\newcommand{\zG}{\Gamma}

\newcommand{\zs}{\sigma}
\newcommand{\zS}{\Sigma}
\newcommand{\zL}{\Lambda}
\newcommand{\zO}{\Omega}
\newcommand{\kb}{\Bbbk}
\newcommand{\ot}{\leftarrow}

\newcommand{\Hom}{\textup{Hom}}
\newcommand{\add}{\textup{add}}

\newcommand{\rad}{\textup{rad}\,}

\newcommand{\id}{\textup{id}}
\newcommand{\pd}{\textup{pd}}

\newcommand{\Ext}{\textup{Ext}}
\newcommand{\End}{\textup{End}}
\newcommand{\cmp}{\textup{CMP}}
\newcommand{\scmp}{\underline{\cmp}}
\newcommand{\diag}{\textup{Diag}(\cals)}
\newcommand{\coker}{\textup{coker}}
\newcommand{\im}{\textup{im}}
\newcommand{\Aut}{\textup{Aut}}

\newcommand{\eins}{1}
\newcommand{\zwei}{2}
\newcommand{\drei}{3}
\newcommand{\vier}{4}
\newcommand{\fuenf}{5}
\newcommand{\sechs}{6}
\newcommand{\sieben}{7}
\newcommand{\acht}{8}
\newcommand{\neun}{9}
\newcommand{\zehn}{10}
\newcommand{\elf}{11}
\newcommand{\zwoelf}{12}
\newcommand{\dreizehn}{13}

\title{A geometric model for syzygies over  2-Calabi-Yau tilted algebras} 
\author{Ralf Schiffler}
\thanks{The first author was supported by the NSF grants  DMS-1800860,  DMS-205461, and by the University of Connecticut. The second author was supported in part by NSF grant DMS-2054255 and NSF Postdoctoral Fellowship MSPRF-1502881. This work was partially
supported by a grant from the Simons Foundation.  The authors would like to thank the Isaac Newton Institute for Mathematical Sciences for support and hospitality during the programme Cluster Algebras and Representation Theory when work on this paper was undertaken. This work was supported by: EPSRC Grant Number EP/R014604/1.}
\address{Department of Mathematics, University of Connecticut, Storrs, CT 06269-1009, USA}
\email{schiffler@math.uconn.edu}

\author{Khrystyna Serhiyenko}
\thanks{} 
\address{Department of Mathematics, University of Kentucky, Lexington, KY 40506-0027, USA }
\email{khrystyna.serhiyenko@uky.edu}

\begin{document}
\maketitle

\setcounter{tocdepth}{2}

\begin{abstract}
In this article, we consider the class of 2-Calabi-Yau tilted algebras that are defined by a quiver with potential whose dual graph is a tree. We call these algebras \emph{dimer tree algebras} because they can also be realized as quotients of dimer algebras on a disc. These algebras are wild in general.  For every such algebra $B$, we construct a polygon $\cals$ with a checkerboard pattern in its interior that gives rise to a category $\diag$.  The indecomposable objects of $\diag$ are the 2-diagonals in $\cals$, and its morphisms are given by certain pivoting moves between the 2-diagonals. We conjecture that the category $\diag$ is equivalent to the stable syzygy category over the algebra $B$, such that the rotation of the polygon corresponds to the shift functor on the syzygies.  In particular, the number of indecomposable syzygies is finite and the projective resolutions are periodic.  We prove the conjecture in the special case where every chordless cycle in the quiver is of length three.  

As a consequence, we obtain an explicit description of the projective resolutions.  Moreover, we show that the syzygy category is equivalent to the 2-cluster category of type $\mathbb{A}$, and we introduce a new derived invariant for the algebra $B$ that can be read off easily from the quiver. 

\end{abstract}

\tableofcontents

\section{Introduction}\label{sect 1}
\subsubsection*{Overview} In this paper, we study the syzygy categories of certain 2-Calabi-Yau tilted algebras. Syzygies are the building blocks for free resolutions in commutative algebra as well as for projective resolutions of modules in representation theory of algebras. Every module has a projective resolution which may be thought of  as an approximation of the module by projectives. Conversely, any morphism between projectives defines a module via its cokernel. For any finite dimensional algebra, there are only finitely many indecomposable projectives but, in general, there is no hope for a classification of all indecomposable modules. 

It is therefore natural to study the subcategory of all syzygy modules over the algebra. In the extreme cases, it is possible that every module is a syzygy (if the algebra is self-injective) or that only the projective modules are syzygies (if the algebra is hereditary), and in general the behavior can be anywhere in between these two extremes. A particularly nice situation is when the number of indecomposable syzygies is finite. It is an open problem to classify all syzygy-finite algebras.

We are interested in 2-Calabi-Yau tilted algebras, a class of finite dimensional non-commutative algebras over a field, introduced in \cite{R} in the context of the categorification of cluster algebras by \cite{BMRRT, A}. These algebras are far-reaching generalizations of cluster-tilted algebras introduced in \cite{CCS,BMR} and studied extensively by a large number of authors, see for example the lecture notes \cite{Assem, BuMa} and references therein. 
For example, every (finite dimensional) Jacobian algebra of a quiver with potential is 2-Calabi-Yau tilted \cite{A}. These algebras also arise naturally in mathematical physics, for example, in relation to Postnikov diagrams and dimer models \cite{HK, Po, JKS, BKM,Pr}. Another family of examples of 2-Calabi-Yau tilted algebras are the 3-preprojective algebras of an algebra of global dimension at most 2, see \cite{IO}.

Keller and Reiten showed in \cite{KR} that 2-Calabi-Yau tilted algebras are Iwanaga-Gorenstein and that their stable syzygy categories are 3-Calabi-Yau. The Iwanaga-Gorenstein property implies that the syzygy category is equivalent to the category of Cohen-Macaulay modules over the algebra, which in turn is equivalent to  the singularity category \cite{Bu,O}. 
In this setting, Cohen-Macaulay modules are also the same as Gorenstein-projective modules. 
 Cohen-Macaulay modules play a central role in commutative algebra and algebraic geometry, in particular in the McKay correspondence and resolutions of singularities, see for example the textbooks \cite{BD,LW,Y}, the survey \cite{I}, and more recently in matrix factorization \cite{HIMO, M, PV}. 

In this paper, we restrict to a certain class of 2-Calabi-Yau tilted algebras $B$, which we call \emph{dimer tree algebras}, because they can be realized as a quotient of a dimer algebra on a disc. 
For example, algebras arising from the coordinate rings of the Grassmannians $\textup{Gr}(3,n)$ are dimer tree algebras.

 We give a complete description of the category of all syzygies $\scmp\,B$, including a description of objects, morphisms, shift, exact structure, and the Auslander-Reiten quiver. Furthermore, we construct explicit projective resolutions for the syzygies. In particular, we show that these algebras, which are generally
of wild representation type, only admit a finite number of indecomposable syzygies.

\subsubsection*{Main results} Let $B$ be a dimer tree algebra, that is, $B$ is a Jacobian algebra given by a quiver with potential such that  the dual graph of the quiver is a tree, see Definition~\ref{def Q}. We construct a combinatorial model for the syzygy category inside a polygon $\cals$ with $2N$ vertices.  Our polygon is equipped with a checkerboard pattern defined by a set of radical lines $\rho(i)$ corresponding to the vertices $i$ of the quiver. Then each 2-diagonal $\zg$ in $\cals$ corresponds to an indecomposable syzygy $M_\zg$  and its intersection with the radical lines gives a projective presentation of $M_\zg$. 

 We consider the category $\diag$ of 2-diagonals. The morphisms in $\diag$ are defined combinatorially via certain pivoting moves between the 2-diagonals. Moreover, $\diag$ is a triangulated category with shift functor given by the rotation by $\pi/N$.
If we omit the checkerboard pattern, this category has already appeared in \cite{BM}, where it was used to give a combinatorial model for the 2-cluster category of type $\mathbb{A}$.
\footnote{Elsewhere in the literature e.g. \cite{KR2}, this category is referred to as the 3-cluster category. We recall the definition in Section~\ref{sec25}.}

Before stating our main result, we introduce the following notation.
 Let $\zg$ be a 2-diagonal in $\cals$. Then $\zg$ crosses several of the radical lines $\rho(i)$ of the checkerboard pattern of $\cals.$ Each of these crossings has  degree $0$ or $1$ according to Definition \ref{def degree}. We define two projective $B$-modules $P_0(\zg)$ and $P_1(\zg)$ as follows. 
Let $P_0(\zg)=\oplus_iP(i)$, where the sum is over all vertices $i$ of $ Q$ such that the radical line $\rho(i)$ crosses $\zg$ in degree 0.
Similarly, let $P_1(\zg)=\oplus_jP(j)$, where the sum is over all vertices $j$ of $Q$ such that the radical line $\rho(j)$ crosses $\zg$ in degree 1. Let $\zO$ denote the syzygy functor.
\medskip

We are now ready to state our main conjecture and our main result.

\begin{conjecture}\label{main conj}
Let $B$ be a dimer tree algebra. For each 2-diagonal $\zg$ in $\cals$ there exists a morphism $f_\zg\colon P_1(\zg)\to P_0(\zg)$ such that the mapping $\zg\mapsto \coker f_\zg$ induces an equivalence of categories 
\[F\colon\diag \to \scmp \, B.\]
 Under this equivalence, the radical line $\rho(i)$ corresponds to the radical of the indecomposable projective $P(i)$ for all $i\in Q_0$. The clockwise rotation $R$ of $\cals$ corresponds to the inverse shift $\zO$ in $\scmp\,B$ and $R^2$ corresponds to the  inverse Auslander-Reiten translation $\tau^{-1}=\zO^2$. 
Thus
\[ F(\rho(i))=\rad P(i)\]
\[ F(R(\zg))= \zO\,F(\zg)\]
\[ F(R^2(\zg))= \tau^{-1}\,F(\zg)\]
Furthermore, $F$ maps the 2-pivots in $\diag$ to the irreducible morphisms in $\scmp\, B$, and the  meshes in $\diag$ to the Auslander-Reiten triangles in $\scmp\,B$.
\end{conjecture}

In this paper, we prove the conjecture in the following special case. We shall prove the full conjecture in a forthcoming paper using the model developed here.  Recall that a cycle in a graph is said to be chordless if it is equal to the induced subgraph on its vertices.

\begin{thm}\label{thm1}(Theorem \ref{main thm})
 The conjecture holds in the special case where every chordless cycle in $Q$ is of length 3. 
\end{thm}

\subsubsection*{A small example}
Below we provide a small example; a bigger and more detailed one is given in section~\ref{sect 3.5}.
Let $B$ be the algebra given by the  quiver in Figure~\ref{fig:intro-ex} 
with potential the sum of the two chordless cycles. 
The Auslander-Reiten quiver of the syzygy category $\text{CMP}\,B$ is given below, and the Auslander-Reiten quiver of the stable syzygy category $\scmp\,B$ is obtained from this one by removing the projective modules.  

\[\xymatrix@R=10pt@C=10pt{&{\begin{smallmatrix}P(1)\end{smallmatrix}} \ar[dr] && {\begin{smallmatrix}P(4)\end{smallmatrix}}\ar[dr] && {\begin{smallmatrix}P(3)\end{smallmatrix}} \ar[dr] && {\begin{smallmatrix}P(2)\end{smallmatrix}}\ar[dr]\\
{\begin{smallmatrix}2\\3\end{smallmatrix}} \ar[ur]\ar[dr] && {\begin{smallmatrix}1\\2\end{smallmatrix}} \ar[ur]\ar[dr] && {\begin{smallmatrix}4\,\,\,\\ 1\,5\\2\end{smallmatrix}}\ar[ur]\ar[dr] && {\begin{smallmatrix}3\\4\end{smallmatrix}}\ar[ur]\ar[dr] && {\begin{smallmatrix}2\\3\end{smallmatrix}}\\
&{\begin{smallmatrix}2\end{smallmatrix}}\ar[ur]\ar[dr] && {\begin{smallmatrix}1\,5\\2\end{smallmatrix}}\ar[ur] && {\begin{smallmatrix}4\end{smallmatrix}}\ar[ur] && {\begin{smallmatrix}3\end{smallmatrix}}\ar[ur]\\
&&{\begin{smallmatrix}P(5)\end{smallmatrix}}\ar[ur]
}\]

The polygon $\mathcal{S}$ is shown in Figure~\ref{fig:intro-ex}. The arc $\zg$ as in the figure crosses radical lines 1,2, and 5.  This gives a morphism $f_{\zg}\colon P(2)\to P(1)\oplus P(5)$ between the projectives, whose cokernel is the corresponding syzygy $M_\zg=\begin{smallmatrix}1\,5\\2\end{smallmatrix}$.  Moreover, the projective resolution of $M_\zg$ can be constructed by applying the clockwise rotation $R$.   Thus, $R\,\zg$ is the radical line $\rho(1)$ and its associated syzygy $M_{R\zg}=\text{coker}\,f_{R\zg}\colon P(4)\to P(2)$ is the radical of the projective $P(1)$. Hence, $\Omega M_{\zg} = M_{R\zg}=\rad\, P(1)$ and the first few terms of the projective resolution of $M_\zg$ are as follows. 

\[\xymatrix{&\dots\ar[r]&P(4)\ar[r]^{f_{R\zg}}&P(2)\ar[r]^-{f_\zg}&P(1)\oplus P(5)\ar[r]&M_{\zg}\ar[r]&0}\]

\begin{figure}

 \begin{minipage}{.5\textwidth}
\[ \xymatrix@R40pt@C40pt{1\ar[r] & 2 \ar[d]\\4\ar[u]&3\ar[l]\ar[r] & 5\ar[ul]}\]  \end{minipage}%
    \begin{minipage}{0.5\textwidth}
\centerline{
 \scalebox{0.8}{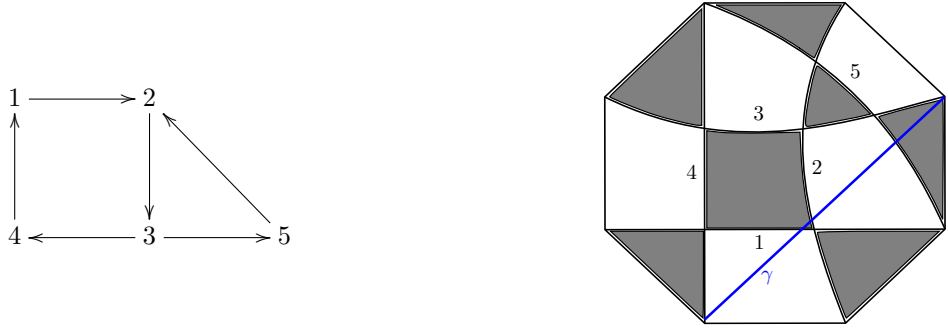}}
\end{minipage}
\caption{A quiver $Q$ together with its checkerboard polygon $\mathcal{S}$. The radical lines are labeled by the vertices of  $Q$. The crossing points of two radical lines represent arrows in the quiver, and the interior shaded regions represent chordless cycles in $Q$.}
\label{fig:intro-ex}
\end{figure}

\subsubsection*{Applications}

As an application, we obtain the following corollaries. The first one uses the results of \cite{BM}.

\begin{corollary}
 The category $\scmp B$ is equivalent to the 2-cluster category of type $\mathbb{A}_{N-2}.$ In particular, the number of indecomposable syzygies is $N(N-2)$.
\end{corollary}
Equivalently this category can also be described as the stable module category over the self-injective algebra  
$ \zL_n$
defined as the quotient of the path algebra of an oriented $n$-cycle by the $ \rad^{n-1} $.

The corollary shows that the size of the polygon $2N$ determines the syzygy category up to equivalence. In Corollary~\ref{cor size of S}, we provide a simple formula for $2N$ as a weighted sum over the boundary arrows of $Q$, where each arrow has weight 1 or 2. Since derived equivalent algebras have equivalent singularity categories, we obtain a new derived invariant that is easy to compute. 

\begin{corollary}
 The size of the polygon is a derived invariant for the algebra $B$.
\end{corollary}

 The checkerboard pattern on $\cals$ determines the algebra $B$ completely.
On the other hand, two algebras may have the same size polygon while having  non-equivalent derived categories. It is an interesting problem to determine when two checkerboard patterns are derived equivalent.

\begin{corollary}(Corollary \ref{cor:536})
 The projective resolution of any syzygy is periodic of period $N$ or $2N$. An indecomposable syzygy $M_\zg$ has period $N$ if and only if the corresponding 2-diagonal $\zg$ is a diameter in $\cals.$
\end{corollary}

Recall that a module is rigid if it has no nontrivial self-extensions.
\begin{corollary}
(Corollary \ref{cor rigid})
 The indecomposable syzygies over $B$ are rigid $B$-modules.
\end{corollary}
Moreover, we conjecture that the indecomposable syzygies are $\tau$-rigid in $\textup{mod}\,B$. This is of interest, since the $\tau$-rigid indecomposable modules correspond to cluster variables in the cluster algebra of $Q$ provided they are reachable by mutation.

\begin{corollary}
(Corollary \ref{cor ext})
 Let $M,N$ be indecomposable syzygies over $B$. Then the dimension of $\Ext^1_B(M,N)\oplus \Ext^1_B(N,M)$ is equal to the number of crossing points between the corresponding 2-diagonals. In particular, the dimension is either 1 or 0.
\end{corollary}

It was shown in \cite{AB} that the Auslander-Reiten translation $\tau$ in $\textup{mod}\,B$ induces an equivalence from the stable syzygy category $\scmp\,B$ to the stable cosyzygy category $\underline{\textup{CMI}}\,B$.  In our geometric model the same checkerboard polygon describes both categories. To switch to the cosyzygies, it suffices to apply the Nakayama functor $\nu$ which replaces the projectives with injectives. 

\begin{corollary}
 (Corollary \ref{cor cmpcmi}) 
The following diagram commutes. 
\[\xymatrix@C40pt{
\scmp\,B\ar[r]^{\tau}&\underline{\textup{CMI}}\,B\\
\diag\ar[u]^{\textup{cok}\,f_\zg} \ar[r]^{1}
&\diag \ar[u]_{\textup{ker}\,\nu\! f_\zg}  
}
\]
\end{corollary}

The proofs of the above corollaries do not depend on the additional assumption of Theorem~\ref{thm1} and will therefore generalize to the setting of Conjecture~\ref{main conj}.

We also classify all checkerboard patterns up to the decagon in Example~\ref{exbysize}. The decagon admits 17 different checkerboard patterns up to symmetry.

\subsubsection*{A few words about the proof}
The key ingredient in most of our arguments is an explicit definition of the morphism $f_\zg$ in terms of certain paths in the quiver, which we call \emph{valid paths} and whose definition involves the checkerboard pattern.  This provides a complete description of not just the modules but also the morphisms in the projective resolution of the syzygies.  It would also be useful to have a more conceptual definition of $f_\zg$, but we have not found one so far.  Let us point out that, in general, $f_\zg\colon P_1(\zg)\to P_0(\zg)$ is not simply a generic map, and the knowledge of $P_0(\zg)$ and $P_1(\zg)$ does not determine $f_\zg$ in this way.  

The first part of the proof is to show that our category of 2-diagonals $\diag$ gives a subcategory of the syzygy category.  The second step is then to show that we obtain every syzygy in this way, which may actually come as a surprise, since the algebra is wild in general.  For this we show that $\diag$ gives a finite component of the Auslander-Reiten quiver of $\scmp\,B$ by explicitly constructing the Auslander-Reiten triangles.  In addition, the argument requires a detailed analysis of the endomorphisms of the indecomposable syzygies $M_\zg$. 

\subsubsection*{Related work}
For the very special class of cluster-tilted algebras of finite representation type, the syzygy categories were studied before by  Chen, Geng and Lu in \cite{CGL}, where they gave a classification of the syzygy categories of these algebras. In particular, they show that the components of $\scmp \,B$ are equivalent to the stable categories of the self-injective algebras $\zL_n$. Their procedure involves a case by case analysis that uses a  classification of the derived  equivalence classes  of cluster-tilted algebras of Dynkin type in \cite{BHL,BHL2}. 
 Later Lu extended these results to simple polygon-tree algebras \cite{L}. One of the ingredients of the proof is successive mutation at vertices of the exterior cycles and reduction to a cluster-tilted algebra of Dynkin type $\mathbb{D}$. 
 These algebras are special cases of those included in Conjecture~\ref{main conj} and the results provide evidence for the conjecture. 
 The above results determine only the type of the syzygy category but do not describe the objects or the morphisms.

Garcia-Elsener and the first author have described the syzygy category of cluster-tilted algebras of type $\mathbb{D}$ in terms of arcs in a once-punctured polygon in \cite{GES}. 

For gentle algebras, the singularity categories have been described by Kalck in \cite{K} using $m$-cluster categories of type $\mathbb{A}_1$.  In our setting the algebra is gentle if and only if the quiver has a unique chordless cycle. This has been extended to skew-gentle algebras by Chen and Lu in \cite{CL}.
For further results on singularity categories of finite dimensional algebras see \cite{C,C2, CDZ, LZ, Sh}.

\subsubsection*{Future directions}
In a forthcoming paper \cite{SS4}, we prove  Conjecture~\ref{main conj} in full generality and in \cite{SS5}, we extend the construction to skew group algebras of dimer tree algebras. 
In a different paper, we will describe the connection to dimer algebras, where we will show how to embed our checkerboard polygon
in an alternating strand diagram of the dimer model, see Example~\ref{ex dimer}. 

Furthermore, it will be interesting to see if we can relax the conditions on the quiver such as allowing the dual graph to be disconnected or to contain cycles, see Remark~\ref{rem relax}.   Other future directions include the behavior of the checkerboard polygon under mutations, a description of the syzygies in terms of their composition factors, and the question of $\tau$-rigidity of the indecomposable syzygies. 

\subsection*{Acknowledgements:} We thank Alastair King and Matthew Pressland for interesting discussions on the connection to dimer algebras, as well as the referees for helpful suggestions for the improvement of the paper. 
\section{Preliminaries}\label{sect 2}
Let $\kb$ be an algebraically closed field. If $A$ is a finite-dimensional $\kb$-algebra, we denote by $\textup{mod}\,A$ the category of finitely generated right $A$-modules. Let $D$ denote the standard duality $D=\Hom(-,\kb)$. If $Q_A$ is the ordinary quiver of the algebra $A$, and $i$ is a vertex of $Q_A$,  we denote by $P(i),I(i),S(i)$ the corresponding indecomposable projective, injective, simple $A$-module, respectively. We let $\pd M$, $ \id M$ denote the projective dimension, respectively the injective dimension of the $A$-module $M$.

A \emph{loop} in a  quiver $Q$ is an arrow that starts and ends at the same vertex. Two arrows are called \emph{parallel} if they share the same starting point and the same endpoint.  An \emph{oriented cycle} in $Q$ is a path that starts and ends at the same vertex. Thus a loop is an oriented cycle of length 1. A \emph{2-cycle} is an oriented cycle of length 2.

For further information about representation theory and quivers we refer to \cite{ASS,S2}.

\medskip
\subsection{2-Calabi-Yau tilted algebras}
A $\kb$-linear triangulated category $\calc$ with split idempotent and finite-dimensional Hom spaces is said to be \emph{2-Calabi-Yau} 
if
$D\,\Ext_\calc^1(X,Y)\cong \Ext_\calc^1(Y,X)$, for all objects $X,Y\in \calc$. Let $\calc$ be a 2-CY category. A basic object $T$ in $\calc$ is called a \emph{cluster-tilting object} if $\add T=\{X\in \calc\mid \Ext^1_\calc(X,T)=0\}$. The endomorphism algebra $B=\End_\calc T$ of a cluster-tilting object $T$ is called a \emph{2-Calabi-Yau tilted} algebra. 
These algebras are a far reaching generalization of cluster-tilted algebra and have been studied extensively. For example, every finite dimensional Jacobian $\kb$-algebra $B$ in the sense of \cite{DWZ} is 2-Calabi-Yau tilted, because the associated generalized cluster category $\calc_B$ contains a cluster-tilting object $T$ whose endomorphism algebra is $B$, see \cite{A}.  On the other hand, not every 2-Calabi-Yau tilted algebra is a Jacobian algebra, see \cite{L}.  

Let us highlight the following result by Keller and Reiten.
Recall that a $\kb$-algebra $A$ is said to be \emph{Iwanaga-Gorenstein of dimension $d$} if $\pd DA=\id A=d<\infty$.

\begin{thm}
\label{thmKRGorenstein}
Every 2-Calabi-Yau tilted algebra is Iwanaga-Gorenstein of dimension at most 1.
\end{thm}

 For further results about 2-Calabi-Yau tilted algebras, we refer to the surveys \cite{Assem, R}.

\subsection{Cohen-Macaulay modules over 2-Calabi-Yau tilted algebras}\label{sect CM}
From now on, let $B$ be a 2-Calabi-Yau tilted algebra.
A $B$-module $M$ is said to be projectively Cohen-Macaulay if $\Ext^i_B(M,B)=0$ for all $i>0$. In other words, $M$ has no extensions with projective modules.
Dually, a $B$-module $N$ is said to be injectively Cohen-Macaulay if $\Ext^i_B(DB,N)=0$ for all $i>0$. 

We denote by $\cmp\,B$ and $\textup{CMI}\,B$ the full subcategories of $\textup{mod}\,B$ whose objects are the projectively Cohen-Macaulay modules
or the injectively Cohen-Macaulay modules, respectively. Both categories are Frobenius categories. The projective-injective objects in $\cmp\,B$ are  precisely the projective $B$-modules, whereas the projective-injective objects in $\textup{CMI}\,B$ are precisely the injective $B$-modules. The corresponding stable categories $\scmp\,B$ and $\underline{\textup{CMI}}\,B$ are triangulated categories.  The syzygy operator $\zO$ in $\textup{mod}\,B$ is the inverse shift in $\scmp\,B$, and the cosyzygy operator $\zO^{-1}$ in $\textup{mod}\,B$ is the shift in $\underline{\textup{CMI}}\,B$.

The Auslander-Reiten translations induce quasi-inverse triangle equivalences \cite[Chapter X]{BR}
\begin{equation}
\label{eq cmpcmi}
\tau\colon \scmp\,B \to \underline{\textup{CMI}}\,B
\qquad 
\tau^{-1}\colon
\underline{\textup{CMI}}\,B
\to \scmp\,B.
\end{equation}

Moreover, by Buchweitz's theorem \cite[Theorem 4.4.1]{Bu}, there exists a triangle equivalence between $\scmp\,B$ and the singularity category $\cald^b(B)/\cald^b_{perf} (B)$ of $B$.
Keller and Reiten showed in \cite{KR} that the category $\scmp \,B$ is 3-Calabi-Yau.
The following result was proved in \cite{GES}. 
\begin{thm}\label{gesthm}
 Let $M$ be an indecomposable module over a 2-Calabi-Yau tilted algebra $B$. Then the following are equivalent.
 \begin{itemize}
\item [(a)] $M$ is a non-projective syzygy;
\item [(b)] $M \in \textup{ind}\,\scmp\,B$; 
\item [(c)] $\zO^2_B \tau_B M \cong M$.
\end{itemize}
\end{thm}
We may therefore use the terminology ``syzygy'' and ``Cohen-Macaulay module'' interchangeably. 


For convenience of the reader, we give a proof of the following fact.
 
\begin{prop}\label{prop extclosed} The category
  $\cmp \,B$ is closed under extensions.
\end{prop}

\begin{proof}
  Let $\xymatrix{0\ar[r]&L\ar[r]&M\ar[r]&N\ar[r]&0}$ be a short exact sequence in $\textup{mod}\,B$ with $L,N\in\cmp\,B$. Applying the functor $\Hom_B(-,B)$ yields an exact sequence
\[\xymatrix{\cdots\ar[r]&
\Ext_B^{i}(N,B)\ar[r]&\Ext_B^{i}(M,B)\ar[r]&\Ext_B^{i}(L,B)\ar[r]&\ldots
}\] for every $i\ge 1$. 
Now $\Ext_B^{i}(N,B)=\Ext_B^{i}(L,B)=0$, because  $N$ and $L$ are Cohen-Macaulay modules.  Thus $\Ext_B^{i}(M,B)=0$, and hence $M\in\cmp\,B$.
\end{proof}

\begin{corollary}\label{cor 82}
 Let $M\in\cmp\,B$.  Then $M$ is rigid in $\cmp\,B$ if and only if $M$ is rigid in $\textup{mod}\,B$.\qed
\end{corollary}

\medskip
\subsection{Translation quivers and mesh categories} We review here the notions of translation quiver and mesh category from \cite{Ri, H}. These notions are often used in order to define a category from combinatorial data. Examples of such constructions are the combinatorial constructions of cluster categories of finite type in \cite{BM, BM2, CCS, S1}. 

A {\em  translation
 quiver} $(\zG,\tau)$ is a quiver $\zG=(\zG_0,\zG_1)$ without loops
 together with an injective map $\tau\colon \zG_0'\to\zG_0$  (the {\em translation}) from a subset $\zG_0'$ of $\zG_0$ to $\zG_0$ such that, for all vertices $x\in\zG_0'$, $y\in \zG_0$, 
the number of arrows from $y \to x$ is equal to the number of arrows
 from $\tau x\to y$. 
Given a  translation quiver $(\zG,\tau)$, a \emph{polarization of} $\zG$ is
 an injective map $\sigma:\zG_1'\to\zG_1$, where $\zG_1'$ is the set of all arrows $\za\colon y\to x$ 
 with $x \in \zG_0'$, such that 
$\sigma(\za)\colon \tau x\to y$  for every arrow $\za\colon y\to x\in \zG_1$.
From now on we assume that $\zG$ has no multiple arrows. In that case, there is a unique polarization of $\zG$.

The {\em path category } of a translation quiver $(\zG,\tau)$ is the category whose  objects are
the vertices $\zG_0$ of $\zG$, and, given $x,y\in\zG_0$, the $\kb$-vector space of
morphisms from $x$ to $y$ is given by the $\kb$-vector space with basis
the set of all paths from $x $ to $y$. The composition of morphisms is
induced from the usual composition of  paths.
The {\em mesh ideal} in the path category of $\zG$ is the ideal
generated by the {\em mesh relations}
\begin{equation}\nonumber
m_x =\sum_{\za:y\to x} \sigma(\za) \za 
\end{equation}  
for all $x \in \zG_0'$.

The {\em  mesh category } of the translation quiver $(\zG,\tau)$ is the
quotient of its path 
category by the mesh ideal.

\medskip
\subsection{The category of 2-diagonals of a polygon} \label{sect 23} Let $\cals $ be a regular polygon with  an even number of vertices, say $2N$. Let $R$ be the automorphism of $\cals$  given by a clockwise rotation about $180/N $ degrees. Thus $R^{2N}$ is the identity. 

Following Baur and Marsh, we define the category $\diag$ of 2-diagonals of $\cals$ as follows. 
The indecomposable objects of $\diag$ are the 2-diagonals in $\cals$. Recall that a 2-diagonal is a diagonal of $\cals$ connecting two vertices such that the two polygons obtained by cutting $\cals$ along the diagonal both have an even number of vertices and  at least 4. In particular, boundary segments are not 2-diagonals.

The irreducible morphisms of $\cals$ are given by 2-pivots. We recall the definition below. An illustration is given in Figure \ref{fig 2-pivots}. 
\begin{definition}\label{def 2pivot}
 Let $\zg$ be a 2-diagonal in the checkerboard polygon $\cals$ and denote its endpoints by $a$ and $x$. Denote by $b$ the clockwise neighbor of $a$, and by $c$ the clockwise neighbor of $b$ on the boundary of $\cals$. At the other end, denote by $y$ the clockwise neighbor of $x$, and by $z$ the clockwise neighbor of $y$ on the  boundary of $\cals$. 
 
Unless $a$ and $z$ are neighbors on the boundary, the 2-diagonal $\zg'$  connecting $a$ and $z$ is called the {\em 2-pivot of $\zg$ fixing the endpoint $a$}.

Unless $c$ and $x$ are neighbors on the boundary, the 2-diagonal $\zg''$  connecting $c$ and $x$  is called the {\em 2-pivot of $\zg$ fixing the endpoint $x$}.
\end{definition}
\begin{figure}
\begin{center}
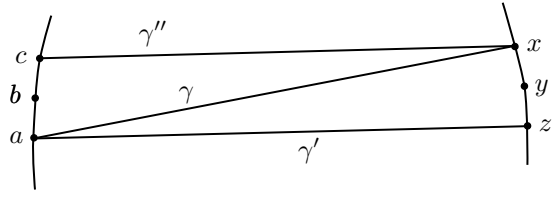
\caption{$\zg'$ is the 2-pivot of $\zg$ fixing the endpoint $a$ and $\zg''$ is the 2-pivot of $\zg$ fixing the endpoint $b$.}
\label{fig 2-pivots}
\end{center}
\end{figure}

Let $\zG$ be the quiver whose vertices is the set of 2-diagonals in $\cals$, and there is an arrow from the 2-diagonal $\zg$ to the 2-diagonal $\zg'$ precisely if $\zg'$ is obtained from $\zg$ by a 2-pivot. Then the  pair  $(\zG, R^{-2})$ is a translation quiver.

\begin{definition}\cite{BM}
 The category $\diag$    of 2-diagonals in the  polygon $\cals$ is the mesh category of the translation quiver $(\zG,R^{-2})$.
\end{definition}

\medskip
\subsection{The 2-cluster category of type $\mathbb{A}$} \label{sec25} In this subsection, let  $H$ be the path algebra of a Dynkin quiver of type $\mathbb{A}_r$. Let $\calc^2$ denote the 2-cluster category of type $\mathbb{A}_{r}$. This category is defined as the orbit category of the bounded derived category $\cald^b(\textup{mod}\,H)$ by the functor $\tau_{\cald}^{-1}[2]$. 
Here $\tau_{\cald}$ is the Auslander-Reiten translation and $[2]=[1]\circ [1]$ is the second power of the shift functor in the derived category.  Thus
\begin{equation*}
\calc^2=\cald^b(\textup{mod}\,H)/\tau_{\cald}^{-1}[2].\end{equation*}
This category was introduced in \cite{K,T}, and was studied in \cite{BRT,IY, KR2, Tor}.

\begin{thm}\label{thm BM}
 \cite{BM} Let $\cals $ be a polygon with $2N$ vertices. Then the category $\diag$ is equivalent to the 2-cluster category of type $\mathbb{A}_{N-2}$.
\end{thm}

Under this equivalence each 2-diagonal of $\cals$ corresponds to an indecomposable object in $\calc^2$. Moreover, there exists a nontrivial extension between two indecomposable objects in $\calc^2$ if and only if
the corresponding 2-diagonals cross. 
And the maximal sets of non-crossing 2-diagonals, or quadrangulations,   correspond to the cluster-tilting objects in $\calc^2$. 

In particular, since every indecomposable object in $\calc^2$ is rigid, we have the following.
\begin{corollary}\label{lem rigid}
 Every 2-diagonal is a rigid object in the category $\diag $.
\end{corollary}

\begin{remark}
 Baur and Marsh actually proved more generally that the category of $m$-diagonals is equivalent to the $m$-cluster category of type $\mathbb{A}$. 
\end{remark}

\section{Construction and properties of the checkerboard polygon }\label{sect 3} 
\medskip
\subsection{The construction}\label{sect 3.1}
In this section, we define  dimer tree algebras as a class of 2-Cababi-Yau tilted algebras by imposing restrictions on the quiver and specifying the potential. We start by  constructing a checkerboard polygon $\cals$ from the quiver $Q$ in three steps. First, we associate the dual graph $G$ to $Q$, then we construct the twisted completed dual graph $\widetilde G$ from $G$, and finally we obtain the checkerboard polygon $\cals$ as the medial graph of $\widetilde G$. The smallest example of the construction is given below. 
The definition of the algebra $B$ is given at the end of this subsection, because in order to define the signs of the terms of the potential, we first need to introduce a distance function on the dual graph.
\begin{example}
\label{ex A3} 
\ \\
\[
\begin{array}{cccccccc}
  \vcenter{\vbox{\xymatrix{&2\ar[rd]\\1\ar[ru]&&3\ar[ll]}}}
&  &
   \vcenter{\vbox{  \xymatrix{\cdot\ar@{-}[rd]&&\cdot\ar@{-}[ld]\\&\cdot\ar@{-}[d]\\&\cdot}  }}
  &&
  \vcenter{\vbox{   \xymatrix{&\cdot\ar@{-}[ld]\ar@{-}[rd]\\ 
   \cdot\ar@{-}[rd]&&\cdot\ar@{-}[ld]\\
   \cdot\ar@{-}[u]\ar@{-}[rd]&\cdot\ar@{-}[d]&\cdot\ar@{-}[ld]\ar@{-}[u]
   \\&\cdot}  }}
& &\vcenter{\hbox{  \scalebox{0.5}{\huge
\begingroup%
  \makeatletter%
  \providecommand\color[2][]{%
    \errmessage{(Inkscape) Color is used for the text in Inkscape, but the package 'color.sty' is not loaded}%
    \renewcommand\color[2][]{}%
  }%
  \providecommand\transparent[1]{%
    \errmessage{(Inkscape) Transparency is used (non-zero) for the text in Inkscape, but the package 'transparent.sty' is not loaded}%
    \renewcommand\transparent[1]{}%
  }%
  \providecommand\rotatebox[2]{#2}%
  \newcommand*\fsize{\dimexpr\f@size pt\relax}%
  \newcommand*\lineheight[1]{\fontsize{\fsize}{#1\fsize}\selectfont}%
  \ifx\svgwidth\undefined%
    \setlength{\unitlength}{177.08946191bp}%
    \ifx\svgscale\undefined%
      \relax%
    \else%
      \setlength{\unitlength}{\unitlength * \real{\svgscale}}%
    \fi%
  \else%
    \setlength{\unitlength}{\svgwidth}%
  \fi%
  \global\let\svgwidth\undefined%
  \global\let\svgscale\undefined%
  \makeatother%
  \begin{picture}(1,0.87086508)%
    \lineheight{1}%
    \setlength\tabcolsep{0pt}%
    \put(0,0){\includegraphics[width=\unitlength,page=1]{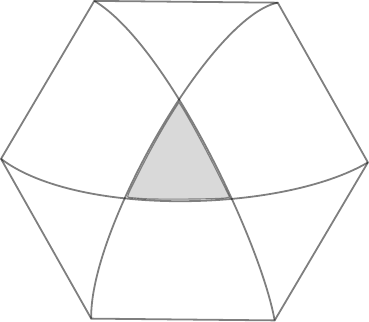}}%
    \put(0.58500132,0.48220724){\makebox(0,0)[lt]{\lineheight{1.25}\smash{\begin{tabular}[t]{l}1\end{tabular}}}}%
    \put(0.46221593,0.2411183){\makebox(0,0)[lt]{\lineheight{1.25}\smash{\begin{tabular}[t]{l}3\end{tabular}}}}%
    \put(0.34141681,0.48227617){\makebox(0,0)[lt]{\lineheight{1.25}\smash{\begin{tabular}[t]{l}2\end{tabular}}}}%
  \end{picture}%
\endgroup%
}}}
\\Q&&G&&\widetilde G&&\vcenter{\hbox{$\cals$}}
\end{array} 
\]
\end{example}

A larger example is given in subsection~\ref{sect 3.5}.

\subsubsection{The quiver $Q$}\label{sect 3.1.1} A \emph{chordless cycle} in a quiver $Q$ is a cyclic path $C=x_0\to x_1 \to\dots\to x_t\to x_0$ such that $x_i\ne x_j$ if $i\ne j$ and the full subquiver on vertices $x_0,x_1,\dots, x_t$ is equal to $C$. 
The arrows that lie in exactly one chordless cycle will be called {\em boundary arrows} and those that lie in two or more chordless cycles {\em interior arrows} of $Q$. 

\begin{definition}
  The \emph{dual graph} $G$ of $Q$ is defined as follows. The set of vertices $G_0$ is the union of 
the set of chordless cycles of $Q$ and the set of boundary arrows of $Q$. 
The set of edges $G_1$ is the union of two sets
called the set of \emph{trunk edges} and the set of \emph{leaf branches}.  A trunk edge $\xymatrix{C\ar@{-}[r]^\za&C'}$ is drawn between any pair of chordless cycles  $(C,C')$ that share an arrow $\za$. A leaf branch $\xymatrix{C\ar@{-}[r]^\za &\za}$ is drawn between any pair $(C,\za)$ where $C$ is a chordless cycle and $\za$ is a boundary arrow such that $\za $ is contained in $C$.
\end{definition}

\begin{definition}[The quiver]
\label{def Q}  Throughout the paper,  we let $Q$ be a finite connected quiver without loops and 2-cycles satisfying the following conditions. 
\begin{itemize}
\item [(Q1)] Every arrow of $Q$ lies in at least one chordless cycle.
\item[(Q2)]\label{tree} The dual graph of $Q$ is a tree. 
\item[(Q3)]\label{boundary} The boundary arrows of $Q$ form a simple (non-oriented) cycle.
\end{itemize}
\end{definition}

The following properties follow easily from the definition.
\begin{prop}\label{prop Q}  Let $Q$ be a quiver satisfying Definition~\ref{def Q}.
 \begin{enumerate}
\item $Q$ has no parallel arrows.
\item $Q$ is planar.
\item For all arrows $\za $ of $Q$,
 \begin{enumerate}
\item[(a)] either $\za$ lies in exactly one chordless cycle, 
\item[(b)] or $\za$ lies in exactly two chordless cycles.\end{enumerate}
\item Any two chordless cycles in $Q$ share at most one arrow.
\end{enumerate}
\end{prop}
\begin{proof}
 (1) Suppose $\za,\zb$ are parallel arrows in $Q$ and let $C$ be a chordless cycle containing $\za$. Then the full subquiver on the vertices of $C$ also contains $\zb$, so it does not equal $C$, a contradiction.
 
 (2) Let $G$ be the dual graph of $Q$. Since $G$ is a tree, there is a planar embedding of $G$. With respect to this planar embedding, we can reconstruct the quiver as follows.
 We get a subdivision of the plane into unbounded regions by extending the leaves of $G$ to infinity. Each region corresponds to a vertex in $Q$. Each edge in $G$ corresponds to an arrow connecting the vertices of the two adjacent regions. This gives a planar embedding of $Q$.
 
 (3) If an arrow $\za$ lies in three (or more) chordless cycles $C_1,C_2,C_3$ then $\za$ gives rise to the following three edges 
 $\xymatrix@C10pt{C_1 \ar@{-}[r]& C_2\ar@{-}[r] & C_3\ar@{-}[r] &C_1}$ in $G$, contradicting that $G$ is a tree.

(4) If two chordless cycles share two arrows, then there would be two edges between the corresponding vertices in $G$ contradicting (Q2).
\end{proof}

\begin{remark}\label{rem relax}
 We could relax Condition (Q2) and allow  $G= G'\coprod G''$ to be a disjoint union of trees that correspond to disjoint subquivers $Q',Q''$ each satisfying  Definition~\ref{def Q}, and there is  exactly one arrow from one to the other and this arrow is not contained in a relation. In that case the algebra is a triangular matrix algebra and using \cite[Theorem 3.3]{PZ}  together with  \cite[Lemma 3.7]{CGL} we get that $\cmp\,B=\cmp\,B'\coprod\cmp\,B''$.  In particular, we can relax conditions on the quiver by allowing  arrows $\za$ on which there are no relations such that $Q\setminus \{\za\}$ is disconnected.

The condition that the tree $G$ has no cycles seems more serious. If $G$ were allowed to have cycles our construction would not yield a polygon but a more complicated surface. It is an interesting question if the results and conjectures of this paper generalize to that setting.
\end{remark}

\subsubsection{The twisted completed dual graph $\widetilde G$}
\label{sect 3.1.3}
If $Q$ has exactly one chordless cycle the twisted  dual graph $\widetilde G$ is defined to be the dual graph $G$. Suppose now that $Q$ has at least two chordless cycles. 
Then there exists a chordless cycle $C_0$ that is connected to  exactly one other chordless cycle $C_1$.
  We are going to twist the graph $G$ at every trunk edge in the following way.
  
 We label the chordless cycles $C_0,C_1,C_2,\ldots, C_s$ in such a way that the length $d(i)$ of the unique path from $C_0$ to $C_i$ in $G$ satisfies the condition $ i<j$  if  $d(i)< d(j) $.

 To construct $\widetilde G$, we start with $G$ and twist at the edge $\xymatrix{C_0\ar@{-}[r]^\za&C_1}$, meaning that we keep the connected component of $G\setminus{\za}$ that contains $C_0$, flip the other component  and reconnect the two components via $\za$. Call the result $G_1$. An example is given below.
 
\[\xymatrix@R15pt@C15pt{ &\cdot\ar@{-}[d] & \cdot\ar@{-}[d] & \cdot\ar@{-}[dl] \\
\cdot\ar@{-}[r] &C_0\ar@{-}[r]^\za & C_1\ar@{-}[r] & \cdot\ar@{-}[r]\ar@{-}[d] &\cdot &&\ar[r]^{\textup{twist along $\za$}}&\ \\
&&&\cdot }
\qquad \qquad
\xymatrix@R15pt@C15pt{ &\cdot\ar@{-}[d] &&\cdot\ar@{-}[d] \\
\cdot\ar@{-}[r] &C_0\ar@{-}[r]^\za & C_1\ar@{-}[r] & \cdot\ar@{-}[r]&\cdot\\
&&\cdot\ar@{-}[u] & \cdot\ar@{-}[ul]}
\]  
Recursively, the graph $G_i$ is the twist of the graph $G_{i-1}$ along the last edge on the unique path from $C_0$ to $C_i$. After $s$ steps, we have twisted once along every trunk edge of $G$ and we obtain the graph $G_s$. 

To define the graph $\widetilde G$, we complete the graph $G_s$ as follows.  For every pair of neighboring leave vertices $\za,\zb$ in $G_s$
\begin{itemize}
\item we add an edge $\xymatrix{\za\ar@{-}[r] &\zb}$, if the resulting face has an even number of edges; or
\item we add a new vertex $x$ and two edges  $\xymatrix{\za\ar@{-}[r] &x}$,  $\xymatrix{x\ar@{-}[r] &\zb}$, otherwise.
\end{itemize}
An example is shown below.
\[
\begin{array}{ccccc}
 \xymatrix{ \za &&&\zb\\\bullet\ar@{-}[r]\ar@{-}[u] &
\bullet\ar@{-}[r] &\bullet\ar@{-}[r] &\bullet\ar@{-}[u] 
} 
& \begin{array}{c}\\\textup{\vspace{10pt} completes to}\\ \longrightarrow\end{array}& 
\xymatrix{ \za\ar@{-}[rrr] &&&\zb\\\bullet\ar@{-}[r]\ar@{-}[u] &
\bullet\ar@{-}[r] &\bullet\ar@{-}[r] &\bullet\ar@{-}[u] 
}
\\
 \xymatrix{ \za &&\zb\\\bullet\ar@{-}[r]\ar@{-}[u] &
\bullet\ar@{-}[r]  &\bullet\ar@{-}[u] 
} 
& \begin{array}{c}\\\textup{\vspace{10pt} completes to}\\ \longrightarrow\end{array}& 
\xymatrix{ \za\ar@{-}[r] &x\ar@{-}[r]&\zb\\\bullet\ar@{-}[r]\ar@{-}[u] &
\bullet\ar@{-}[r] &\bullet\ar@{-}[u] 
}
\end{array}
\]
Note that in both cases the vertices $\za,\zb$ of $\widetilde G$ lie in a face that has an even number of vertices.
Moreover
\[\widetilde G_0 = G_0 \cup\{ \textup{completion vertices}\}
\qquad \textup{and}\qquad \widetilde G_1 = G_1 \cup\{ \textup{completion edges}\}.
\]
We denote by $\widetilde G_2$ the set of bounded faces of $\widetilde G$.
By construction, every face in $\widetilde G_2$ has an even number 
 ($\ge 4$) of edges and at least one and at most two of them are completion edges.

\subsubsection{The checkerboard polygon $\cals$}
\label{sect 3.1.4}
Recall that the \emph{medial graph} $M(H)$ of a planar graph $H$ is defined as the graph that has one vertex for each edge of $H$, and two vertices are connected in $M(H)$ if the corresponding edges in $H$ are consecutive in one of the faces of $H$.
 
The checkerboard polygon $\cals$ is obtained from the medial graph of $\widetilde G$ by adding one edge for every leaf vertex of $G$. 
The faces of $\cals$ come in two types: faces that surround a vertex of $G$ and faces that sit inside a face of $\widetilde G$. 
We think of $\cals $ as a polygon with checkerboard pattern where the shaded regions are those coming from vertices in $G$, thus from the chordless cycles and the boundary arrows in $Q$.

\begin{remark}\label{rem orientation}
 If we started with the opposite quiver we would obtain the same checkerboard polygon. To fix orientations, we shall always assume that the first chordless cycle $C_0$ of the quiver $Q$ corresponds to running around the associated shaded region of the polygon $\cals$ in counterclockwise direction.  Then the other chordless cycles in the quiver also correspond to the counterclockwise direction in their respective shaded regions, since we use the twisted dual graph. Indeed, in the quiver, the orientation of the chordless cycles  alternates between clockwise and counterclockwise, and when twisting the dual graph all chordless cycles acquire the counterclockwise orientation.
\end{remark}

The interior vertices of $\cals$ correspond to  arrows in $Q$, and every interior vertex has degree four. An edge in $\cals$ may connect two interior vertices, or an interior vertex and a boundary vertex, or two boundary vertices. 

An edge 
in $\cals$ that connects two interior vertices $\za,\zb$ carries the label $i$, where $i$ is the unique vertex in $Q$  that is  shared by the two arrows $\za$ and $\zb$. 

Every shaded boundary region in $\cals$ contains precisely one boundary edge as well as  two interior edges that meet at an interior vertex $\za$. The two interior edges are labeled by the starting vertex $s(\za)$ and the terminal vertex $t(\za)$ of the arrow $\za $ in $Q$ such that, when running around the shaded region in counterclockwise order, we go along $s(\za)$ towards $\za$ and along $t(\za)$ away from $\za$. 
The boundary edges in $\cals $ are not labeled.

\subsubsection{The algebra $B$}
We are now ready to define the algebra we are going to study. 
\begin{definition}
Let  $Q$ be quiver satisfying the conditions of Definition~\ref{def Q} and let $W$ be the potential $W=\sum_{i=0}^t (-1)^{d(i)} C_i$. Then the Jacobian algebra over $\kb$ of $(Q,W)$ is called a \emph{dimer tree algebra}.
\end{definition}
We are going to show in subsection \ref{sect 3.7} that every cyclic path in $Q$ is zero in $B$ and that any two nonzero parallel paths in $Q$ are equal. In particular, this implies that dimer tree algebras are finite-dimensional and thus 2-Calabi-Yau tilted. 

\medskip
\subsection{Example}\label{sect 3.5} Let $Q$ be the quiver below. It has 5 chordless cycles and 8 boundary arrows. We have labeled the   arrows in red.
\[ \xymatrix@R40pt@C40pt{1\ar[r]^{{\color{red} 1}}&2\ar[r]^{{\color{red} 2}}&3&4\ar[l] _{{\color{red} 3}}\\
5\ar[r]_{{\color{red} 6}}\ar@{<-}[u]^{{\color{red} 4}}&6\ar[lu]_{{\color{red} \za}}\ar@{<-}[u]_{{\color{red} \zb}}&7\ar@{<-}[l]^{{\color{red} 7}}\ar[r]_{{\color{red} 8}}\ar[lu]_{{\color{red} \zd}}\ar@{<-}[u]_{{\color{red} \ze}}&8\ar[u]_{{\color{red} 5}}}\]

Its dual graph $G$ below has 5 trunk vertices labeled $C_0,\ldots,C_4$ which correspond to the 5 chordless cycles in $Q$. It also has 8 leaf vertices labeled in red which correspond to the boundary arrows in $Q$.
\[\xymatrix@R20pt@C20pt{
&&\color{red} 1\ar@{-}[d]&&\color{red} 2\ar@{-}[d]&\color{red} 3\ar@{-}[d]\\
\color{red} 4\ar@{-}[r] &C_0\ar@{-}[r]&C_1\ar@{-}[r]&C_2\ar@{-}[r]&C_3\ar@{-}[r]&C_4\ar@{-}[r]&\color{red} 5 \\
&\color{red} 6\ar@{-}[u]&&\color{red} 7\ar@{-}[u]&&\color{red} 8\ar@{-}[u]\\
}\]

The twisted completed dual graph $\widetilde G$ is shown below. 
Note that the vertices 1 and 2 are now at the bottom because they have been involved in an odd number twist operations, while the vertices 3,7, and 8 have not changed position, since they have been involved in an even number of twists. 
The completion step has introduced 4 completion vertices marked by a dot in the figure below and 12 completion edges which create 8 bounded faces.

\[\xymatrix@R20pt@C20pt{\cdot\ar@{-}[rrrrr]\ar@{-}[d]&&&&&\color{red} 3\ar@{-}[d]&\cdot\ar@{-}[d]\ar@{-}[l]
\\
\color{red} 4\ar@{-}[r] &C_0\ar@{-}[r]& C_1\ar@{-}[r]&C_2\ar@{-}[r]&C_3\ar@{-}[r]&C_4\ar@{-}[r]&\color{red} 5 
\\
\cdot\ar@{-}[r]\ar@{-}[u]&\color{red} 6\ar@{-}[u]\ar@{-}[r]&\color{red} 1\ar@{-}[u]\ar@{-}[r]&\color{red} 7\ar@{-}[u]\ar@{-}[r]&\color{red} 2\ar@{-}[u]\ar@{-}[r]&\color{red} 8\ar@{-}[u]\ar@{-}[r]&\cdot\ar@{-}[u]\\
}\]

The polygon $\cals$ is shown below. The medial graph of $\widetilde G$ is drawn in blue and the 8 extra edges that come from the 8 leaf vertices of $G$ are drawn in red. The faces that correspond to a vertex of $G$ are labeled by that vertex as follows: the faces $C_0,\ldots,C_4$ are labeled in black and the faces corresponding to the leaf vertices $1,2,\ldots 8$ are labeled in red. The interior vertices of $\cals$ correspond to arrows in $Q$. We have labeled only those vertices that correspond to the interior arrows $\za,\zb,\zd,\ze$, since the labels of the remaining interior vertices of $\cals$ are the same as the red labels of the adjacent boundary regions.
\[\color{blue}\xymatrix@R15pt@C15pt{&&&&\cdot\ar@{-}@[red][rrrrrrr]_(.7){\color{red} \textup{\normalsize 3}}\ar@{-}[lllld]
\ar@{-}[rrrrrrd]_{\color{blue} 4}&&&&&&&\cdot\ar@{-}[rd]
\\
\cdot\ar@{-}@[red][dd]^{\color{red} \textup{\normalsize 4}}\ar@{-}[rd]^{\color{blue} 5}&&&&&&&&&&\cdot\ar@{-}[ru]_{\color{blue} 3} &&\cdot\ar@{-}@[red][dd]
_{\color{red} \textup{\normalsize 5}}
\\
&\cdot\ar@{-}@/^{10pt}/[rr] ^{\color{blue} 1}&\color{black} C_0&\color{black}\za\ar@{-}@/^{10pt}/[rr]^{\color{blue} 6}&\color{black} C_1&\color{black}\zb\ar@{-}@/^{10pt}/[rr]^{\color{blue} 2}&\color{black} C_2&\color{black}\zd\ar@{-}@/^{10pt}/[rr]^{\color{blue} 7}&\color{black} C_3&\color{black}\ze\ar@{-}[ru]^{\color{blue} 3}&\color{black} C_4&\cdot\ar@{-}[lu]_{\color{blue} 4}\ar@{-}[ru]^{\color{blue} 8}
\\
\cdot\ar@{-}[ru]_{\color{blue} 1}\ar@{-}[rd]&&\cdot\ar@{-}[ru]_{\color{blue} 6}\ar@{-}[lu]^{\color{blue} 5}\ar@{-}[rd]^{\color{blue} 5}&&\cdot\ar@{-}[ru]_{\color{blue} 2}\ar@{-}[lu]^{\color{blue} 1}\ar@{-}[rd]^{\color{blue} 1}&&\cdot\ar@{-}[ru]_{\color{blue} 7}\ar@{-}[lu]^{\color{blue} 6}\ar@{-}[rd]^{\color{blue} 6}&&\cdot\ar@{-}[ru]_{\color{blue} 3}\ar@{-}[lu]^{\color{blue} 2}\ar@{-}[rd]^{\color{blue} 2}&&\cdot\ar@{-}[ru]_{\color{blue} 8}
\ar@{-}[lu]^{\color{blue} 7}\ar@{-}[rd]^{\color{blue} 7}
&&\cdot\ar@{-}[lu]^{\color{blue} 4}
\\
&\cdot\ar@{-}@[red][rr]^{\color{red} \textup{\normalsize 6}}\ar@{-}[ru]^{\color{blue} 6}&&\cdot\ar@{-}@[red][rr]^{\color{red} \textup{\normalsize 1}}\ar@{-}[ru]^{\color{blue} 2}&&\cdot\ar@{-}@[red][rr]^{\color{red} \textup{\normalsize 7}}\ar@{-}[ru]^{\color{blue} 7}&&\cdot\ar@{-}@[red][rr]^{\color{red} \textup{\normalsize 2}}\ar@{-}[ru]^{\color{blue} 3}&&\cdot\ar@{-}@[red][rr]^{\color{red} \textup{\normalsize 8}}\ar@{-}[ru]^{\color{blue} 8}&&\cdot\ar@{-}[ru]
}\]
Each interior edge of $\cals$  carries a blue label that corresponds to a vertex of $Q$.
 For example, the edge connecting the vertices $\za$ and $\zb$ in $\cals$  is labeled by  the vertex $6$ of $Q$, since 6 is the  common vertex of the arrows $\za $ and $\zb$.

Moreover, the vertex shared by the faces labeled $C_0$ and the red 6 corresponds to the arrow labeled 6 in the quiver. This arrow shares the vertex $6$ with the arrow $\za$, so the label on the southeast edge of the face $C_0$ is labeled by a blue 6.

The polygon $\cals $ with its checkerboard pattern is redrawn in Figure \ref{fig 3.1}. The shaded regions are those labeled by a vertex of $G$.

\begin{figure}[htbp]
\begin{center}
\small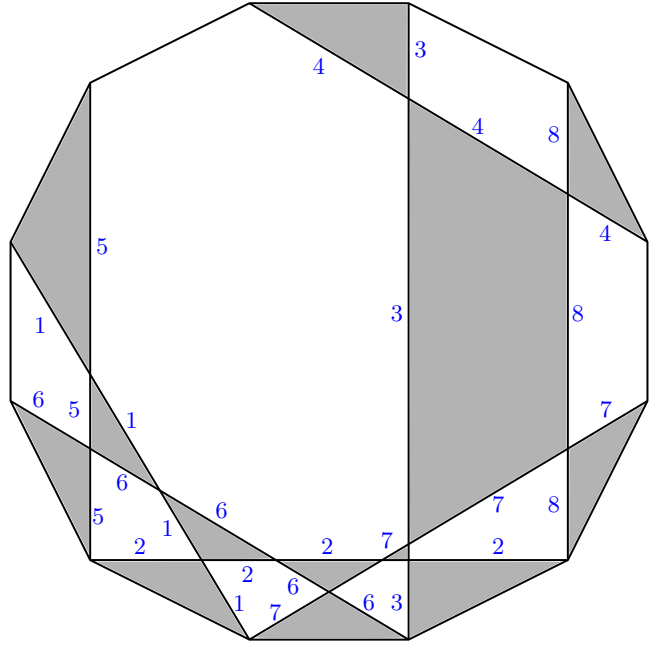
\caption{The checkerboard polygon in the example of Section \ref{sect 3.5}. The diagonals shown are called radical lines, they are labeled by the vertices of the quiver $Q$. The crossing points of two radical lines represent an arrow in the quiver and the interior shaded regions represent chordless cycles in $Q$. The boundary shaded regions correspond to the boundary arrows of $Q$.}
\label{fig 3.1}
\end{center}
\end{figure}

Conjecture \ref{main conj}  claims that the syzygy category $\cmp\,B$ is equivalent to the category of 2-diagonals.  To illustrate this, we compute the projective resolutions of the syzygies below. 

Let $\zg$ denote the radical line $\rho(4)$ in Figure \ref{fig 3.1}. Note that $\zg$ crosses the radical lines $\rho(3)$ and  $\rho(8)$. Then $\zg$ represents the projective presentation
 \[ \xymatrix{P(8) \ar[r]^{f_\zg} &P(3)\ar[r]& M_\zg \ar[r]&0}\]
and we see that  $M_\zg=\rad P(4)=\begin{smallmatrix}3\\7\end{smallmatrix}$.
Thus the radical line $\zg=\rho(4)$ represents the projective presentation of the radical of $P(4)$.
The map $f_\zg$ factors through the syzygy $\zO M_\zg=S(8)$ of $M_\zg$ whose projective presentation
 \[ \xymatrix{P(4) \ar[r]^{f_{R\zg}} &P(8)\ar[r]& \zO M_\zg \ar[r]&0}\]
is represented by the 2-diagonal $R \zg$ where $R$ is the clockwise rotation about the angle $2\pi/12$. Notice that $R\zg$ crosses the radical lines $\rho(8) $ and $\rho (4)$ corresponding to the projective modules in the projective presentation. 

We can continue this procedure and represent the projective presentation of the higher syzygies $\zO^i M_\zg$ by the rotations $R^i\zg$ of the diagonal $\zg$. In particular, after 12 steps we come back to where we started because the polygon has 12 vertices. The complete projective resolution is shown below over two lines. 
{\small\[
 \xymatrix@C0pt@R5pt{
&
P(4)\ar[rr] \ar[rd]&&
P(5)\ar[rr] \ar[rd]&&
P(1)\ar[rr] \ar[rd]&&
P(6)\ar[rr] \ar[rd]&&
P(2)\oplus P(5) \ar[rr] \ar[rd]&&
P(1)\oplus P(7)\ar@{.>}[rr] &&
\\
{\begin{smallmatrix}3\\7\end{smallmatrix}}\ar[ru]
 &&{\begin{smallmatrix}4\end{smallmatrix}}\ar[ru]
 &&{\begin{smallmatrix}5\\6\\7\\8\end{smallmatrix}}\ar[ru]
 &&{\begin{smallmatrix}1\\2\\3\end{smallmatrix}}\ar[ru]
 &&{\begin{smallmatrix}6\\7\\8\\4\end{smallmatrix}}\ar[ru]
 &&{\begin{smallmatrix} 2\ 5\\\ \ 3\ 6\\\ \  7\\\ \ 8 \end{smallmatrix}}\ar[ru]
 \\
\ar@{.>}[r]&P(3)\oplus P(6)\ar[rr] \ar[rd]&& 
P(2)\oplus P(8)\ar[rr] \ar[rd]&&
P(7)\ar[rr] \ar[rd]&&
P(4)\ar[rr] \ar[rd]&&
 P(8)\ar[rr]\ar[rd]
&& P(3)\ar[rd]&
 \\
{\begin{smallmatrix}1\ 7 \\ \ \ 2\ 8\\ \ \ \ \ 4\\\  \ 3\end{smallmatrix}}\ar[ru]
&&{\begin{smallmatrix}3\ 6\\7\\8\end{smallmatrix}}\ar[ru]
 &&{\begin{smallmatrix}8\ \ \\4\ 2\\3\end{smallmatrix}}\ar[ru]
 &&{\begin{smallmatrix}7\end{smallmatrix}}\ar[ru]
 &&{\begin{smallmatrix}4\\3\end{smallmatrix}}\ar[ru]
 &&8\ar[ru]
 &&{\begin{smallmatrix}3\\7\end{smallmatrix}}
}\]}

This projective resolution corresponds to the rotation orbit of the diagonal $\rho(4)$ in the polygon. This orbit contains all the 2-diagonals that cut the polygon into a quadrilateral and a decagon.  In the module category of $B$ this yields 12 indecomposable non-projective syzygies.

In the polygon, there is precisely one other orbit under the rotation and it contains all the 2-diagonals that cut the polygon into a hexagon and an octagon.  This orbit also yields 12 indecomposable syzygies.

The Auslander-Reiten quiver of the syzygy categories $\cmp B$ and $\scmp B$ are given below.
\[\scalebox{0.9}{\xymatrix@R10pt@C10pt{
&P(7)\ar[rd] &&&&P(5)\ar[rd] &&&&P(4)\ar[rd]
\\
{\begin{smallmatrix}8\ \ \\4\ 2\\3\end{smallmatrix}}
\ar[rd] \ar[ru] &&
{\begin{smallmatrix}1\ 7 \\ \ \ 2\ 8\\ \ \ \ \ 4\\\  \ 3\end{smallmatrix}}
\ar[rd] \ar[r]&
P(6)\ar[r]&
{\begin{smallmatrix}6\\7\\8\\4\end{smallmatrix}}
\ar[rd] \ar[ru] &&
{\begin{smallmatrix}5\\6\\7\\8\end{smallmatrix}}
\ar[rd]  &&
{\begin{smallmatrix}3\\7\end{smallmatrix}}
\ar[rd] \ar[ru]  &&
{\begin{smallmatrix}4\\3\end{smallmatrix}}
\ar[rd] \ar[r] &
P(8)\ar[r]& 
{\begin{smallmatrix}8\ \ \\4\ 2\\3\end{smallmatrix}}
\\
&{\begin{smallmatrix} 1\ \ 8\\ 2\ 4 \\3\end{smallmatrix}}
\ar[ru]\ar[rd] &&
{\begin{smallmatrix} 7\\8\\4\end{smallmatrix}}
\ar[ru]\ar[rd] &&
{\begin{smallmatrix} 6\\7\\8\end{smallmatrix}}
\ar[ru]\ar[rd] &&
{\begin{smallmatrix} 5\ \ \\6\ 3\\77\\8\end{smallmatrix}}
\ar[ru]\ar[rd] &&
{\begin{smallmatrix} 3\end{smallmatrix}}
\ar[ru]\ar[rd] &&
{\begin{smallmatrix} 2\ 4\\ 3\end{smallmatrix}}
\ar[ru]\ar[rd]
\\
{\begin{smallmatrix} 1\ \ \\ 2\ 4\\3 \end{smallmatrix}}
\ar[ru]\ar[rd] &&
{\begin{smallmatrix} 8\\4\end{smallmatrix}}
\ar[ru]\ar[rd] &&
{\begin{smallmatrix} 7\\8\end{smallmatrix}}
\ar[ru]\ar[rd]\ar[r]&
P(3)\ar[r]&
{\begin{smallmatrix} 3\ 6\\77\\8\end{smallmatrix}}
\ar[ru]\ar[rd] &&
{\begin{smallmatrix} 5\ \ \\6\ 3\\7\\8\end{smallmatrix}}
\ar[ru]\ar[rd] &&
{\begin{smallmatrix} 2\\3\end{smallmatrix}}
\ar[ru]\ar[rd]&&{\begin{smallmatrix} 1\ \ \\ 2\ 4\\3 \end{smallmatrix}}
\\
&{\begin{smallmatrix} 4 \end{smallmatrix}}
\ar[ru]&&
{\begin{smallmatrix} 8 \end{smallmatrix}}
\ar[ru]&&
{\begin{smallmatrix} 7 \end{smallmatrix}}
\ar[ru]&&
{\begin{smallmatrix} 3\ 6\\7\\8 \end{smallmatrix}}
\ar[ru]\ar[rd]&&
{\begin{smallmatrix} 2\ 5\\\ \ 3\ 6\\\ \  7\\\ \ 8 \end{smallmatrix}}
\ar[ru]\ar[r]&P(1)\ar[r]&
{\begin{smallmatrix} 1\\2\\3 \end{smallmatrix}}
\ar[ru]&&
\\
&&&&&&&&P(2)\ar[ru]
}}
\]

\[\scalebox{0.9}{\xymatrix@R10pt@C10pt{
\\
{\begin{smallmatrix}8\ \ \\4\ 2\\3\end{smallmatrix}}
\ar[rd]  &&
{\begin{smallmatrix}1\ 7 \\ \ \ 2\ 8\\ \ \ \ \ 4\\\  \ 3\end{smallmatrix}}
\ar[rd] &
&
{\begin{smallmatrix}6\\7\\8\\4\end{smallmatrix}}
\ar[rd]  &&
{\begin{smallmatrix}5\\6\\7\\8\end{smallmatrix}}
\ar[rd]  &&
{\begin{smallmatrix}3\\7\end{smallmatrix}}
\ar[rd]   &&
{\begin{smallmatrix}4\\3\end{smallmatrix}}
\ar[rd]  &
& 
{\begin{smallmatrix}8\ \ \\4\ 2\\3\end{smallmatrix}}
\\
&{\begin{smallmatrix} 1\ \ 8\\ 2\ 4 \\3\end{smallmatrix}}
\ar[ru]\ar[rd] &&
{\begin{smallmatrix} 7\\8\\4\end{smallmatrix}}
\ar[ru]\ar[rd] &&
{\begin{smallmatrix} 6\\7\\8\end{smallmatrix}}
\ar[ru]\ar[rd] &&
{\begin{smallmatrix} 5\ \ \\6\ 3\\77\\8\end{smallmatrix}}
\ar[ru]\ar[rd] &&
{\begin{smallmatrix} 3\end{smallmatrix}}
\ar[ru]\ar[rd] &&
{\begin{smallmatrix} 2\ 4\\ 3\end{smallmatrix}}
\ar[ru]\ar[rd]
\\
{\begin{smallmatrix} 1\ \ \\ 2\ 4\\3 \end{smallmatrix}}
\ar[ru]\ar[rd] &&
{\begin{smallmatrix} 8\\4\end{smallmatrix}}
\ar[ru]\ar[rd] &&
{\begin{smallmatrix} 7\\8\end{smallmatrix}}
\ar[ru]\ar[rd]&
&
{\begin{smallmatrix} 3\ 6\\77\\8\end{smallmatrix}}
\ar[ru]\ar[rd] &&
{\begin{smallmatrix} 5\ \ \\6\ 3\\7\\8\end{smallmatrix}}
\ar[ru]\ar[rd] &&
{\begin{smallmatrix} 2\\3\end{smallmatrix}}
\ar[ru]\ar[rd]&&{\begin{smallmatrix} 1\ \ \\ 2\ 4\\3 \end{smallmatrix}}
\\
&{\begin{smallmatrix} 4 \end{smallmatrix}}
\ar[ru]&&
{\begin{smallmatrix} 8 \end{smallmatrix}}
\ar[ru]&&
{\begin{smallmatrix} 7 \end{smallmatrix}}
\ar[ru]&&
{\begin{smallmatrix} 3\ 6\\7\\8 \end{smallmatrix}}
\ar[ru]&&
{\begin{smallmatrix} 2\ 5\\\ \ 3\ 6\\\ \  7\\\ \ 8 \end{smallmatrix}}
\ar[ru]&
&
{\begin{smallmatrix} 1\\2\\3 \end{smallmatrix}}
\ar[ru]&&
\\
}}
\]

Notice that the second syzygy $\zO^2$ corresponds to the inverse Auslander-Reiten translation $\tau^{-1}$ in ${\cmp\,B}$. For example, $\zO^2 {\begin{smallmatrix} 3\\7\end{smallmatrix}}={\begin{smallmatrix} 4\\3\end{smallmatrix}} = \tau^{-1} {\begin{smallmatrix} 3\\7\end{smallmatrix}}$.

\begin{example}
 \label{ex dimer}
 The checkerboard polygon $\cals$ from Figure~\ref{fig 3.1} can be embedded into the alternating strand diagram on a disc with the same number of boundary vertices shown in the left picture in Figure~\ref{fig dimer}.  The orientation of the strands is such that the shaded regions are oriented while the white regions are alternating. The corresponding dimer algebra on the disc is given by the quiver on the right in the same figure.
  Each vertex represents a white region in the alternating strand diagram and two regions are connected by an arrow if they share a crossing point.
 
 The full subquiver on the vertices 1, 2, \dots, 8 is equal to the twisted quiver $\widetilde Q$ of $Q$ in the sense of Bocklandt \cite{Bocklandt}. The vertices 9, 10, \dots 20 are frozen vertices. This 	quiver gives a seed for a cluster algebra associated to a certain positroid variety in the Grassmannian $\textup{Gr}(5,12)$.  The dual graph of $\widetilde Q$ is equal to the twist of the dual graph of  $Q$. 
 
 Thus the checkerboard polygon models the syzygy category of the Jacobian algebra of $Q$ and at the same time it models a seed in the cluster algebra of the twisted quiver $\widetilde Q$.
 
 \begin{figure}
\begin{minipage}{.5\textwidth}
\centering
 \scalebox{.55}{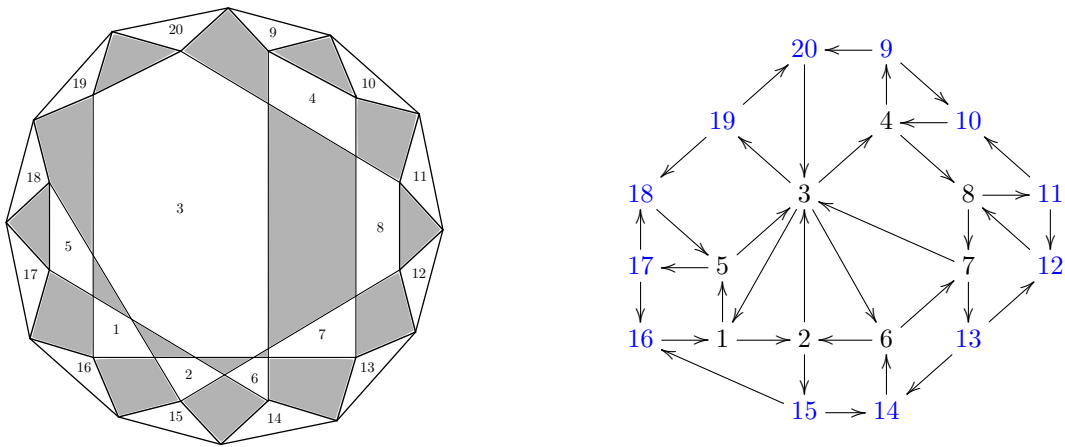}
\end{minipage}%
\begin{minipage} {.5\textwidth}
 \[\xymatrix@R15pt@C15pt{
&&\color{blue}20\ar[dd] & \color{blue}9\ar[l]\ar[rd]&& \\
&\color{blue}19\ar[ru]\ar[ld]&&4\ar[u]\ar[rd]&\color{blue}10\ar[l]&\\
\color{blue}18\ar[rd]&&3\ar[lu]\ar[ru]\ar@/_0pt/[rdd]\ar[ldd]&&8\ar[r]\ar[d]&\color{blue}11\ar[lu]\ar[d]\\
\color{blue}17\ar[u]\ar[d]&5\ar[l]\ar[ru]&&&7\ar[llu]\ar[d]&\color{blue}12\ar[lu]\\
\color{blue}16\ar[r]&1\ar[u]\ar[r]&2\ar@/_0pt/[uu]\ar[d]&6\ar[l]\ar[ru]&\color{blue}13\ar[ru]\ar[ld]&\\
&&\color{blue}15\ar[llu]\ar[r]&\color{blue}14\ar[u]}
 \]
\end{minipage}
\caption{An alternating strand diagram that contains the checkerboard polygon of Figure~\ref{fig 3.1} and its corresponding quiver.}
\label{fig dimer}
\end{figure}

\end{example}
\medskip
\subsection{Properties of $\cals$}\label{sect 3.3}

We have already observed that the interior vertices of the checkerboard polygon $\cals$ correspond to arrows in the quiver $Q$, and  the shaded regions in $\cals$ correspond to chordless cycles and boundary arrows in $Q$.

\begin{lemma}
 \label{lem 31}
 Let $\za$ be an interior vertex of $\cals$ and let $\zb_i,i=1,2,3,4$ be the four vertices connected to $\za$ in clockwise order. Denote by $x_i$ the label of the edge $\za$---$\zb_i$ in $\cals$.  
 Then $x_i=x_{i+2}$ and $x_i\ne x_{i+1}$, where the addition in the indices is modulo 4. Thus the opposite edges carry the same label. 
\end{lemma}
\begin{proof}
 The vertex $\za$ corresponds to an arrow in $Q$ and the four labels $x_i$ must be taken from the endpoints $s(\za), t(\za)$ of $\za$ in $Q$. Consider a pair of consecutive edges  $\za$---$\zb_i$,  $\za$---$\zb_{i+1}$ that are part of the same shaded region in $\cals$.  Assume first that this shaded region  has no side on the boundary of $\cals$. Then the region corresponds to a chordless cycle $C$ of $Q$ and the three arrows $\zb_i,\za,\zb_{i+1}$ form a subpath of $C$. In particular, the arrows $\zb_i,\za$ have the starting vertex $s(\za)$ of $\za$ in common, while the arrows $\za, \zb_{i+1}$ share the terminal vertex $t(\za)$ of $\za$. Thus $x_i=s(\za)$ and $x_{i+1}= t(\za)$.
 Now assume that the shaded region is labeled by a boundary vertex of $Q$. Then we also have $x_i=s(\za)$ and $x_{i+1}=t(\za)$, by definition of the labeling. 
 Because of the checkerboard pattern, the same argument shows
 that $x_{i+2}=s(\za) $ and $x_{i+3}=t(\za)$. 
 Moreover, since the quiver $Q$ does not contain any loops, we have $x_i\ne x_i+1$.
 \end{proof}
 
\begin{definition}
  The piecewise linear curve $\rho(i)$ given as the union
 of edges in $\cals$ that carry the label $i$ will be called the \emph{radical line} of the vertex $i\in Q_0$. 
The polygon $\cals$ together with the system of radical lines $\{\rho(i) \mid i\in Q_0\}$ will be called the {\em checkerboard polygon} of $Q$. \end{definition}

\begin{remark}
  We shall show  in Proposition \ref{prop:radicals correspond} that $\rho(i)$ corresponds to the radical of the projective $P(i)$ at vertex $i$. 
  \end{remark}

It follows from Lemma \ref{lem 31} that $\rho(i)$ is connected and starts and ends at the boundary of $\cals$ and by condition (Q3) of  Definition~\ref{def Q} the curve $\rho(i)$ is unique. Since $G$ is a tree, we also know that $\rho(i)$ does not intersect itself.
 It is always possible and often convenient to draw $\rho(i) $ as a smooth curve. Sometimes, we can draw $\rho(i)$ as a straight line segment, see for example Figure \ref{fig 3.1}.
 
 \subsubsection{White regions}

\begin{lemma}
 \label{lem 32}
 \begin{itemize}
\item[(a)] Every white region in $\cals$ has an even number of edges with distinct labels, and either exactly one boundary vertex or two boundary vertices connected by a boundary edge.
  \item[(b)] Every boundary vertex in $\cals$ is incident to exactly one white region and one or two shaded regions.
  \item[(c)] Every boundary vertex in $\cals$  is incident to at least one and at most two radical lines.  
\end{itemize}
\end{lemma}
\begin{proof}
(a) By construction, the faces in $\widetilde G$ have an even number of edges, hence the white regions have an even number of vertices, thus edges. The edge labels are distinct because of Lemma \ref{lem 31} and since $G$ is a tree. The fact that $G$ is a tree also implies that the white region has at least one vertex on the boundary. Moreover, it has exactly two boundary vertices connected by a boundary edge precisely if it comes from a face of $
\widetilde G$ which has a completion vertex.

(b) The shaded boundary regions correspond to boundary arrows in $Q$ and the two boundary vertices of a shaded boundary region correspond to the endpoint of that arrow. In particular, every shaded boundary region has exactly two boundary vertices. This implies that every boundary vertex $x$ is incident to exactly one white region $W$. Moreover, if $x$ is the only boundary vertex of $W$ then $x$ is incident to two shaded regions and if $W$ contains a second boundary vertex then $x$ is incident to only one shaded region.

Part (c) follows directly from (b).
\end{proof}

\begin{lemma}
 \label{lem 33} Let $W$ be a white region in $\cals$, with interior edges labeled $i_1, i_2, \ldots, i_t$ starting at the boundary and going around $W$ in the clockwise direction. 
  \begin{itemize}
\item[(a)] There is a unique path 
  $\mathfrak{c}(W)=\za_{i_1}\za_{i_2}\cdots\za_{i_{t-1}}$ in $Q$
   and a corresponding sequence of distinct chordless cycle $C_1,C_2,\ldots ,C_{t-2}$ in $Q$ 
  such that  $\za_{i_j}$ is an arrow $i_j\to i_{j+1}$ and
  $\za_{i_j}\za_{i_{j+1}}$ is a subpath of $C_j=\za_{i_j}\za_{i_{j+1}}C_j'$. 

  \item[(b)] There are two unique paths
  $\mathfrak{v}_1(W) =  C'_{t-2}C'_{t-4} \cdots C'_a$ and 
  $\mathfrak{v}_2(W) =  C'_{t-3}C'_{t-5} \cdots C'_b$
  with $a=1,b=2$ if $t$ is odd, and $a=2,b=1$ if $t$ is even.
\item[(c)] The full subquiver $Q(W)$ of $Q$ whose vertices are those visited by $\mathfrak{c}(W), \mathfrak{v}_1(W),$ or $ \mathfrak{v}_2(W)$ is the quiver given by the union of the cycles $C_1,...,C_{t-2}$, and this quiver is equal to the full subquiver generated by 
$\mathfrak{c}(W)$ if and only if all $C_i$ are 3-cycles. 
\end{itemize}
\end{lemma}

\begin{proof}
(a) Every edge $i_j$ in the sequence bounds $W$ on one side and a shaded region, hence a chordless cycle $C$, on the other side.  The vertices $\za_{i_2}, \ldots, \za_{i_{t-2}}$ of $W$ correspond to arrows in $Q$ that lie in exactly two chordless cycles in $Q$. Say $\za_{i_j}$ lies in the chordless cycles $C_{j-1} $ and $C_{j}$. Then 
the sequence $C_1,\za_{i_2},C_2,\za_{i_3},\ldots, \za_{i_{t-2}},C_{t-2}$ defines a path in the trunk of $G$. In particular, $C_j\ne C_\ell$, since $G$ is a tree. 

(b) We need to show that the paths are well-defined. Since $C_j=\za_{i_j}\za_{i_{j+1}}C_j'$ is a cycle starting at $i_j$, and the arrow $\za_{i_{j+1}}$ ends at $i_{j+2}$, we see that $C_j'$ is a path from $i_{j+2}$ to $i_{j}$. Therefore the terminal point $C_{j+2}'$ is the starting point of $C_j'$. 

(c) By definition the vertices of $Q(W)$ are precisely those in the cycles, and the  path $\mathfrak{c}(W)$ visits all of them if and only if every $C_j'$ is a single arrow, or, equivalently, if $C_j$ is a three cycle. Moreover, each arrow contained in the cycles also lies in one of the paths $\mathfrak{c}(W),\mathfrak{v}_1(W)$ or $\mathfrak{v}_2(W)$. There are no other arrows in $Q(W)$ because the dual graph of $Q$ is a tree.
\end{proof}

\begin{definition}\label{def:whitepath}
 Let $W$ be a white region. The path $\mathfrak{c}(W)$ is called the \emph{cycle path} of $W$ and the paths $\mathfrak{v}_i(W)$ are called the \emph{maximal valid paths} of $W$.
\end{definition}
\begin{remark}
 The cycle path $\mathfrak{c}(W)$ goes clockwise around the white region $W$ while the valid paths $\mathfrak{v}_i(W)$ go counterclockwise.
\end{remark}
\begin{example}\label{ex paths}
 In the example in Section \ref{sect 3.5}, there are 8 white regions. The corresponding paths are listed below, starting with the large white region and going clockwise around the polygon shown in Figure \ref{fig 3.1}.

 \[
\begin{array}{llllllll}
  \xymatrix@R20pt@C20pt{1&2&3&4\ar[l] _{{\color{red} 3}}\\
5\ar@{<-}[u]^{{\color{red} 4}}&6\ar[lu]_{{\color{red} \za}}\ar@{<-}[u]_{{\color{red} \zb}}&7\ar[lu]_{{\color{red} \zd}}\ar@{<-}[u]_{{\color{red} \ze}}}
& \quad&
 \xymatrix@R20pt@C20pt{3&4\ar[l] _{{\color{red} 3}}\\
&8\ar[u]_{{\color{red} 5}}}
& \quad&
 \xymatrix@R20pt@C20pt{&4\\
7\ar[r]_{{\color{red} 8}}&8\ar[u]_{{\color{red} 5}}}
&\quad &
\xymatrix@R20pt@C20pt{2\ar[r]^{{\color{red} 2}}&3\\
&7\ar[r]_{{\color{red} 8}}\ar@{<-}[u]_{{\color{red} \ze}}&8}
\\ \\
\xymatrix@R20pt@C20pt{
2\ar[r]^{{\color{red} 2}}&3\\
6&7\ar@{<-}[l]^{{\color{red} 7}}\ar[lu]_{{\color{red} \zd}}}
& &
 \xymatrix@R20pt@C20pt{1\ar[r]^{{\color{red} 1}}&2\\
&6\ar@{<-}[u]_{{\color{red} \zb}}&7\ar@{<-}[l]^{{\color{red} 7}}}
& &
 \xymatrix@R20pt@C20pt{1\ar[r]^{{\color{red} 1}}&2\\
5\ar[r]_{{\color{red} 6}}
&6\ar[lu]_{{\color{red} \za}}}
& &
 \xymatrix@R20pt@C20pt{1\ar[d]_{\color{red} 4}\\
5\ar[r]_{{\color{red} 6}}
&6}
\end{array}
\]
The corresponding maximal valid paths
are \[
\begin{array}{llllllll}
  \xymatrix@R20pt@C20pt{1\ar[r]^{{\color{red} 1}}&2\ar[r]^{{\color{red} 2}}&3&4\\
5\ar[r]^{{\color{red} 6}}&6\ar[r]^{{\color{red} 7}}&7\ar[r]^{{\color{red} 8}}&8\ar[u]_{{\color{red} 5}}}
& \quad&
 \xymatrix@R20pt@C20pt{3\ar[d]^{\color{red} \ze} &4\\
7\ar[r]^{{\color{red} 8}}&8}
& \quad&
 \xymatrix@R20pt@C20pt{3\ar[d]_{{\color{red} 5}}&4\ar[l]_{{\color{red} 3}}\\
7&8}
&\quad &
\xymatrix@R20pt@C20pt{2&3&4\ar[l]_{{\color{red} 3}}\\
&7\ar[lu]_{{\color{red} \za}}&8\ar[u]_{{\color{red} 5}}}
\\ \\
\xymatrix@R20pt@C20pt{
2\ar[d]_{{\color{red} \zb}}&3\ar[d]^{{\color{red} \ze}}\\
6&7}
& &
 \xymatrix@R20pt@C20pt{1&2\\
&6\ar[lu]_{{\color{red} \zd}}&7\ar[lu]_{{\color{red} \zd}}}
& &
 \xymatrix@R20pt@C20pt{1\ar[d]_{{\color{red} 4}}&2\ar[d]^{{\color{red} \zb}}\\
5
&6}
& &
 \xymatrix@R20pt@C20pt{1\\
5
&6\ar[ul]_{{\color{red} \za}}}
\end{array}
\]\end{example}

\begin{prop}
 \label{prop white region}
 There is a bijection 
 \[
\begin{array}{rcl}
\varphi\colon \{\textup{white regions in $\cals$}\}  & \longrightarrow    &  \{\textup{boundary arrows in $Q$}\}  \\
  W&\longmapsto   & \textup{first arrow in }\mathfrak{c}(W)   
\end{array}
\]
given by mapping the white region $W$ to the first arrow of its cycle path $\mathfrak{c}(W)$.
Moreover,  $W$ has exactly two vertices on the boundary of $\cals$, if $\mathfrak{c}(W)$ is of even length, and $W$ has exactly one vertex on the boundary, otherwise.
\end{prop}

\begin{proof}
 Notice first that the map $\varphi$ is well-defined by Lemma \ref{lem 33}.
 
   Every boundary arrow $\za$ of $Q$ determines a unique shaded region at the boundary of $\cals$ which in turn is adjacent to precisely two white regions $W_1$ and $W_2$.  For one of them, say $W_1$, the arrow $\za$ is the initial arrow in the path $\mathfrak{c}(W_1)$, thus $\za=\varphi (W_1)$ and $\phi $ is surjective. On the other hand, $\za $ is the terminal arrow in $\mathfrak{c}(W_2)$ which shows that $\varphi$ is injective.

 It remains to show the statement on the number of boundary vertices in $W$. If the path $\mathfrak{c}(W)$ has even length then it goes through an odd number of vertices, and thus $W$ has an odd number of edges in the interior of $\cals$. Similarly,  if  $\mathfrak{c}(W)$ has odd length then $W$ has an even number of edges in the interior of $\cals$. Now the statement follows from Lemma \ref{lem 32}.
\end{proof}

Thanks to Proposition \ref{prop white region} we can label the cycle paths by the boundary arrows of the quiver. Part (a) of the following definition is a reformulation adapted to this point of view.

\begin{definition}
 \label{def cycle path 2}
 Let $\za$ be a  boundary arrow in $Q$.  
 \begin{itemize}
\item[(a)] The 
 \emph{cycle path} of $\za$ is the unique path 
 $\mathfrak{c}(\za)=\za_1\za_2\cdots\za_{\ell(\za)}$ such that 
 \begin{itemize}
\item [(i)] $\za_1=\za$ and $\za_{\ell(\za)}$  are boundary arrows, and 
 $\za_2,\ldots,\za_{\ell(\za)-1}$ are interior arrows,
\item [(ii)] every subpath of length two $\za_i\za_{i+1}$, 
 is a subpath of a chordless cycle $C_i$, and $C_i\ne C_j$ if $i\ne j$. 
\end{itemize}

\item[(b)] The \emph{weight} $\text{wt}(\za)$ of $\za$ is defined as
\[\textup{wt}(\za) =\left\{ 
\begin{array}
 {ll} 1&\textup{if the length of  $\mathfrak{c}(\za)$ is odd;}\\
 2&\textup{if the length of  $\mathfrak{c}(\za)$ is even.}\\
\end{array} \right.\]

\end{itemize}
\end{definition}

The proposition yields the following important formula for the size of the polygon.
\begin{corollary}
 \label{cor size of S}
The number of boundary edges in $\cals$ is equal to 
\[\sum_\za \textup{wt}(\za),\]
where the sum is over all boundary arrows of $Q$.
\end{corollary}
\begin{proof}
 Each of the boundary edges in the shaded regions of $\cals$
 corresponds to a unique boundary arrow in $Q$. According to the proposition, each boundary in a white region corresponds to a unique boundary arrow $\za$ in $Q$ whose cycle path $\mathfrak{c}(\za)$ is of even length. Now the result follows from the definition of the weight $\textup{wt}(\za)$ of $\za$.
\end{proof}

\begin{example}
 In our running example, there are 8 boundary arrows and their cycle paths are listed in Example~\ref{ex paths}. We have $\textup{wt}({\za}) =2$ for the arrows $\za={\color{red} 3},
 {\color{red} 5},
 {\color{red} 8}, 
 {\color{red} 4},$ and 
 $\textup{wt}({\za}) =1$ for the arrows $\za={\color{red} 2},
 {\color{red} 7},
 {\color{red} 1}, 
 {\color{red} 6}.$ Thus Corollary \ref{cor size of S} implies that the polygon has 12 sides.    
 \end{example}

\begin{lemma}
 \label{lem 34} The number of boundary vertices of $\cals$ is even and at least 6.
\end{lemma}
\begin{proof}
Let $m$ be the number of boundary vertices of $\cals$. By construction, $m$ is the number of boundary edges in $\widetilde G$, and we use induction on the number of faces in $\widetilde G$. 
 Since $Q$ has at least one chordless cycle, which is at least of length 3, the smallest possible case is the one shown in Example \ref{ex A3}.   Thus $m\ge 6$ and $\widetilde G$ has at least 3 faces.

Choose a face of $\widetilde G$. Let $2a$ be the number of edges of the chosen face, and denote by $b$ the number of its edges that lie on the boundary of $\widetilde G$. When we remove that face from $\widetilde G$, the difference in the number of boundary edges between $\widetilde G$ and the resulting graph is $-b+(2a-b)=2(a-b)$ which is even, so the parity is preserved. This shows that $m$ is even.
 \end{proof}
\subsubsection{Radical lines and 2-diagonals}
\begin{lemma}
 \label{lem 35}
The radical lines $\rho(i)$ and $\rho(j)$ cross if and only if there is an arrow $i\to j$ or $j\to i $ in $Q$.
\end{lemma}
\begin{proof}
 If  $\rho(i)$ and $\rho(j)$ cross then there is an interior vertex $\za$ in $\cals$ that lies on both  $\rho(i)$ and $\rho(j)$. Thus $\za$ corresponds to an arrow in $Q$ that connects $i$ and $j$.
 Conversely, if $\xymatrix@C15pt{i\ar[r]^\za&j}$ is an arrow in $Q$ then there is an interior vertex $\za $ in $\cals$ which is of degree 4. By Lemma \ref{lem 31}, $\za$ lies on both  $\rho(i)$ and $\rho(j)$.
\end{proof}

\begin{corollary}
 \label{cor even}
 There is an even  number of boundary arrows $\za$ in $Q$ whose cycle path $\mathfrak{c}(\za)$ is of odd length.
\end{corollary}
\begin{proof}
 This follows from Corollary \ref{cor size of S} and Lemma \ref{lem 34}.
\end{proof}

Recall that a \emph{2-diagonal} in a polygon is (the homotopy class of) a diagonal such that 
the dissection of the polygon along this diagonal consists of two polygons each with an even number of sides. 
\begin{lemma}
 \label{lem 36}
Each radical line $\rho(x) $ is a 2-diagonal in $\cals$.
\end{lemma}
\begin{proof}
 We proceed by induction on $n$, the number of vertices in $Q$.
 If $n=3$, we are done by Example~\ref{ex A3}. Suppose that $n\ge 4$. We consider two cases. 
 
 (a) Suppose first there exists a boundary arrow $\xymatrix{i\ar[r]^\zb&j}$ and such that the unique chordless cycle containing $\zb$ is of length at least 4 and there are no interior arrows at $i$ or $j$. 
 In this situation, we define a new quiver $Q'$ by contracting the arrow $\zb$, meaning that we identify the vertices $i$ and $j$ and remove $\zb$, see below. 
 
 \[
Q \qquad \vcenter{\xymatrix@R10pt{ i\ar[dd]_\zb&h\ar[l]_\za \ar@{.}[r]&\  
 \\ \\
 j\ar[r]_\zg&k\ar@{.}[r]&\ } } \qquad \longrightarrow \qquad 
 Q'\qquad \vcenter{\xymatrix@R10pt{ &h\ar[ld]_\za\ar@{.}[r]&\ 
 \\ 
 ij \ar[rd]_\zg\\ 
 &k\ar@{.}[r]&\ } }
 \]
 
At the level of $\widetilde G$, the edge $\zb$ bounds two faces $F_1, F_2$. Since $\za$ and $\zg$ are boundary arrows in $Q$, both $F_i$ have 4 vertices one of which, say $v_i$, is a completion vertex, $i=1,2$. 
 The graph $\widetilde{G'}$ can be obtained from $\widetilde G$ by removing the edge $\zb$ together with the two boundary edges incident to it, and identifying the vertices $v_1$ and $v_2$, see below.
 \[\widetilde G\qquad \vcenter{ \xymatrix{v_1 \ar@{-}[r] \ar@{-}[d] &\bullet \ar@{-}[d]^\za\ar@{.}[r]&\ 
 \\
 \bullet\ar@{-}[r]^\zb\ar@{-}[d]&\bullet \ar@{-}[d] ^\zg\ar@{.}[r]&\ 
\\
v_2\ar@{-}[r]&\bullet \ar@{.}[r]&\ }}
 \qquad \longrightarrow \qquad 
 \widetilde{G'} \qquad \vcenter{
 \xymatrix{  &\bullet \ar@{-}[d]^\za\ar@{.}[r]&\ 
 \\
v_1v_2\ar@{-}[rd]\ar@{-}[ru]&\bullet \ar@{-}[d] ^\zg\ar@{.}[r]&\ 
\\
&\bullet \ar@{.}[r]&\ }
 }
 \]
 At the level of $\cals$, the vertex $\zb$ is the crossing point of the radical lines $\rho(i)$ and $\rho(j)$. The two shaded regions at this crossing are a triangular boundary region on one side and an interior region with at least 4 sides corresponding to the chordless cycle in $Q$ that contains the arrow $\zb$.
 The two white regions at the crossing are quadrilaterals coming from the faces $F_1,F_2$ of $\widetilde G$.  The vertex $\zb$ is connected to the two vertices $\za$ and $\zg$ because there are no interior arrows at $i$ and $j$. Moreover $\za$ and $\zg$ are each a vertex of a triangular boundary region, because they are boundary arrows in $Q$. 
 The polygon $\cals'$ can be obtained from $\cals$ by removing the shaded boundary region at $\zb$ together with the two adjacent boundary edges, and replacing the two edges $\zb$---$\za$, $\zb$---$\zg$  with one edge $\za$---$\zg$ and closing the white region with one edge on the boundary, see below. 
 
\begin{center}
\small 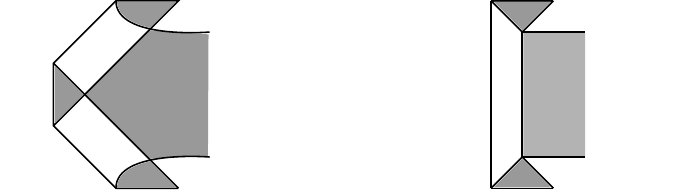  
\end{center}
Now consider the two polygons obtained by cutting $\cals$ along the radical line $\rho(x)$. If $i=x$ or $i=y$ then $\rho(x)$ is clearly a 2-diagonal and we are done.  If $x\ne i,j$, then one of the two polygons is also obtained by cutting $\cals'$ along $\rho(x)$.
By induction, this polygon has an even number of boundary edges, and by Lemma \ref{lem 34}, the other polygon as well. This completes the proof in case (a). 

(b) Suppose we are not in the situation of case (a). Then there exists a boundary arrow $\za\colon i\to j$ that lies in a unique 3-cycle 
$C=\xymatrix{i\ar[r]^\za&j\ar[r]^\zb&k\ar@/^10pt/[ll]^\zg}$ such that $\zb$ is also a boundary arrow, and $\zg$ is an interior arrow. 
We define a new quiver $Q'$ by removing the vertex $j$ and its adjacent arrows $\za,\zb$, see below.
\[Q\qquad \vcenter{\xymatrix@R10pt{&i\ar[ld]_\za\ar@{.}[r]&\ 
\\
j\ar[rd]_\zb\\
&k\ar[uu]_\zg\ar@{.}[r]&\ 
}}
\qquad \longrightarrow \qquad
Q'\qquad
\vcenter{\xymatrix@R10pt{i\ar@{.}[r]&\ 
\\
\\
k\ar[uu]_\zg\ar@{.}[r]&\ 
}}
\]
At the level of $\widetilde G$, the cycle $C$ is a vertex with 3 edges $\za,\zb,\zg$. The two edges $\za$ and $\zb$ connect $C$ to the boundary vertices $\za$ and $\zb$, forming a quadrilateral with vertices $C,\za,\zb$ and a completion vertex $w$. The face containing the edges $\za$ and $\zg$ has vertices $C,\za,v_1,v_2,\ldots$, where $v_1$ may be a completion vertex. Similarly, the face containing the edges $\zb$ and $\zg$ has vertices $C,\zb,u_1,u_2,\ldots$ where $u_1$ may be a completion vertex. In order to illustrate both cases, assume that $v_1$ is not a completion vertex and $u_1$ is a completion vertex. Label the two edges at $u_1$ by $e_1$ and $e_2$.  

$\widetilde {G'} $ is obtained from $\widetilde G$ by removing the vertices $w, \zb$ and their adjacent edges, as well as identifying the vertex $C$ with the vertex $u_1$, see below. Note that 
 in $\widetilde {G'}$, the vertex $\za$ is a completion vertex and the vertex $u_2$ is not. 
\[\widetilde G\qquad \vcenter{\xymatrix@R10pt@C10pt{ 
&&v_1\ar@{-}[r]\ar@{-}[ld] &v_2\ar@{.}[r]&\ 
\\
&\za\ar@{-}[ld]\ar@{-}[rd]^\za\\
w\ar@{-}[rd]&&C\ar@{-}[ld]^\zb\ar@{-}[r]^\zg &\bullet\ar@{.}[r]&\  \\
&\zb\ar@{-}[rd]_{e_1}\\
&&u_1\ar@{-}[r]_{e_2} &u_2\ar@{.}[r] &\ 
}}
\qquad \longrightarrow \qquad
\widetilde {G'}\qquad \vcenter{\xymatrix@R10pt@C10pt{ 
&&v_1\ar@{-}[r]\ar@{-}[ld] &v_2\ar@{.}[r]&\ 
\\
&\za\ar@{-}[rd]^\za\\
&&C\ \ar@{-}[r]^\zg &\bullet\ar@{.}[r]&\  \\
&\\
&& &u_2\ar@{-}[luu]^{e_2}\ar@{.}[r] &\ 
}}
\]
 
Now consider $\cals$. The cycle $C$ becomes a triangular shaded region with sides labeled $i,j,k$ and vertices $\za,\zb,\zg$. The adjacent white region towards the boundary is a quadrilateral with edge labels $i,j,k$ and a boundary edge that corresponds to the completion vertex $w$ in $\widetilde G$. The adjacent triangular shaded boundary regions in $\cals$ correspond to  the arrows $\za$ and $\zb$, and the edge labels are shown in the figure below. The vertices $v_1$ and $u_2$ are not  completion vertices in $\widetilde G$ and therefore  give rise to  shaded regions in $\cals$. On the other hand, the vertex $u_1$ is a completion vertex in $\widetilde G$ and it does not yield a shaded region in $\cals$.

The polygon $\cals'$ can be obtained from $\cals$ by removing the three regions to the left of the radical line $\rho(j)$ together with the two edges adjacent to the vertex $e_1$, as well as identifying the vertices $\zb$ and $e_2$, see below.
  
\begin{center}
\small 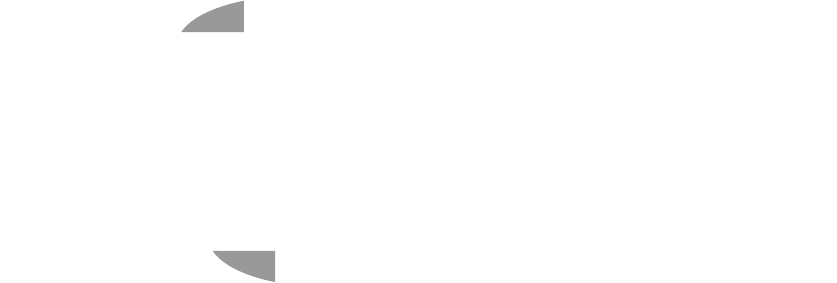  
\end{center}

Now consider the two polygons obtained by cutting $\cals$ along the radical line $\rho(x)$ with $x\ne j$. 
If $x\ne i, k$ then one of these two polygons is also obtained by cutting $\cals'$ along $\rho(x)$.  Using induction this completes the proof as in case (a).

Now suppose $x=i$ or $x=k$. Let $A$ and $B$ denote the two polygons in $\cals$ obtained by cutting along $\rho(x)$ such that the shaded triangle $C$ lies inside $B$. Similarly, let $A'$ and $B'$ denote the two polygons in $\cals'$ obtained by cutting along $\rho(x)$ such that the shaded triangle $C$ lies inside $B'$. Then, if $x=i$, the polygon $A'$ is obtained from $A$ by removing the triangular shaded region on the boundary of $\cals$ with labels $\za, i, j$. In particular, $A$ and $A'$ have the same number of boundary edges. 
On the other hand, if $x=k$ then $A'$ is obtained from $A$ by removing the triangular region at the boundary of $\cals$ with labels $\zb, k, j $ together with the  two edges incident to $e_1$ and identifying the vertices $\zb$ and $e_2$. In particular, $A$ has exactly two more boundary edges than $A'$, so both have the same parity.
By induction, the result follows as in case (a).
Finally, if $x=j$ the result is obvious.
\end{proof}

\subsubsection{Orientation of 2-diagonals} We now define an orientation on each 2-diagonal in the polygon $\cals$, which will allow us later to give a direction to the crossing of two 2-diagonals.
 
The polygon $\cals$ has an even number of boundary vertices, by Lemma \ref{lem 34}. We give a sign to each boundary vertex in such a way that the signs alternate along the boundary of $\cals$. There are exactly two ways of doing this, and it does not matter which one we choose. 

Then a 2-diagonal in $\cals$ is the homotopy class of a line segment that connects two boundary vertices of opposite sign.  We orient each 2-diagonal in the direction from $-$ to $+$. In particular, this defines an orientation on the radical lines $\rho(i)$ in the checkerboard pattern of $\cals$. 

\begin{lemma}
 \label{lem orientation}
 Each  region in the checkerboard pattern in $\cals$ is bounded by a sequence of segments of oriented radical lines, and possibly one boundary segment. 
  For a shaded region these oriented segments form an oriented path bounding the region. For a white region, the orientation of the segments is alternating around the region.
\end{lemma}
\begin{proof}
 Every crossing point between two radical lines is incident to four regions and, locally, the opposite regions have the same type of orientation (same direction or alternating) on their boundary segments. Opposite regions also have the same color. Thus it suffices to show that shaded regions at the boundary of $\cals$ satisfy the statement. This however, is easy to see, since such a region has exactly three vertices, one that is interior and two on the boundary of $\cals$, and the two boundary vertices have opposite signs. This implies that the two segments of the radical lines that bound the region form an oriented path. 
\end{proof}
\begin{remark}
 The orientation of the radical lines in $\cals$ is not related to and should not be confused with the orientation of the arrows in the quiver. Every arrow in $Q$ corresponds to an interior vertex in the checkerboard pattern of $\cals$ which is incident to exactly two shaded regions. The segments of the radical lines go clockwise around  one of these regions and counterclockwise around the other. 
 Similarly, the segments around a white region $W$ have alternating orientation, while the path $\mathfrak{c}(W)$ associated to $W$ in Lemma  \ref{lem 33} is an oriented path in $Q$.
\end{remark}

The orientation of the 2-diagonals allows us to define the degree of a crossing between a 2-diagonal and a radical line as follows. 
\begin{definition}
 \label{def degree}
 Let $\rho(i)$ be a radical line in the checkerboard polygon $\cals$,  and let $\zg$ be  an arbitrary 2-diagonal  that has a crossing with $\rho(i)$. We define the \emph{degree} of the crossing to be
 \[ 
\begin{array}
 {ll}
 0&\textup{if $\rho(i)$ crosses $\zg$ from right to left};\\
 1&\textup{if $\rho(i)$ crosses $\zg$ from left to right}.
\end{array}\]
\end{definition}

\medskip
\subsection{Properties of dimer tree algebras}\label{sect 3.7}
Let $B$ be a dimer tree algebra.  In this subsection, we give some basic properties of the algebra. Throughout we use the notation $w\sim w'$ to indicate that two parallel paths $w, w'$ in $Q$  are equal in $B$.

We start with a property of the quiver. Recall that $Q$ comes with a fixed embedding in the plane. A vertex of $Q$ is called a \emph{boundary vertex} if it is incident to the unique unbounded face of $Q$. 

\begin{lemma}
 \label{lem 36a}
All vertices of $Q$ are boundary vertices.
\end{lemma}
\begin{proof}
 Suppose $x\in Q_0$ is a vertex that is not a boundary vertex. Denote by $C_1,C_2,\ldots,C_s$ the faces of $Q$ that are incident to $x$ in clockwise order. Then every $C_i$ is a bounded face, hence a chordless cycle in $Q$. Moreover every two consecutive faces $C_i,C_{i+1}$ share an arrow $\za_i$ that is incident to $x$. This configuration gives rise to a cycle $\xymatrix@C10pt{C_1\ar@{-}[r]^{\za_1} &C_2\ar@{-}[r]^{\za_2} & \cdots \ar@{-}[r]^{\za_{s-1}}&C_s\ar@{-}[r]^{\za_s} &C_1}$  in the dual graph $G$, which is a contradiction to condition (Q2).
\end{proof}

\begin{lemma}
 \label{lem 37} 
(a) If $\za_1$ is a boundary arrow and $\za_1\za_2\cdots\za_s$ is the unique chordless cycle containing $\za_1$ then the path $\za_2\cdots\za_s$ is zero in $B$.

(b) If $\za_1$ is an interior arrow and $\za_1\za_2\cdots\za_s$ and $\za_1\za'_2\cdots\za'_t$ are the two chordless cycles containing $\za_1$ then the paths $\za_2\cdots\za_s$ and $\za'_2\cdots\za'_t$ are equal in $B$.
\end{lemma}
\begin{proof}
 This follows directly from the definition of the Jacobian ideal.  Indeed, in case (a), the arrow $\za_1$ lies in a unique chordless cycle, hence in a unique summand of the potential $W$, and thus the partial derivative 
 $\partial_{\za_1} W$ is exactly equal to the path $\za_2\cdots\za_s$.
 In case (b),  the arrow $\za_1$ lies in exactly two summands of $W$, and thus 
 $\partial_{\za_1} W$ is the difference of the two paths $\za_2\cdots\za_s$ and $ \za'_2\cdots\za'_t$.
\end{proof}

\begin{prop} 
 \label{prop 38}
Every cyclic path in $Q$ is zero in $B$.
\end{prop}

\begin{proof} 
 Let $w$ be a cyclic path in $Q$. We proceed by induction on the length $\ell(w)$ of $w$. Since $Q$ is without loops and 2-cycles, the base case is $\ell(w)=3$. Thus suppose $w=\za\zb\zg$ is a 3-cycle at vertex $x$ in $Q$. If $\za $ is a boundary arrow then Lemma \ref{lem 37} yields $\zb\zg=0$ and thus $w=0$. Otherwise, $\za$ is contained in  a unique other 3-cycle $w'=\za\zb'\zg'$. Again using Lemma \ref{lem 37}, we have $\zb\zg\sim\zb'\zg'$ and thus $w\sim w'$. 
Moreover, $\zg\ne\zg'$, because $x$ is a boundary vertex, by Lemma~\ref{lem 36a}. 
Similarly, if $\zg'$ is a boundary arrow, then $\za\zb'=0$ and hence $w\sim w'=0$ in $B$. Otherwise, $\zg'$ is contained in a unique other 3-cycle $w''=\za'\zb''\zg'$ and Lemma \ref{lem 37} implies $w''\sim w'\sim w$. 
Moreover $\za'\ne \za$ because $x$ is a boundary vertex of $Q$.

Continuing this way, we either obtain a 3-cycle $w^{(t)}\sim w$ such that $w^{(t)}$  contains a boundary arrow, and in this case $w=0$, or we obtain an unbounded number of arrows incident to the vertex $x$, which is impossible, since $Q$ is finite. This completes the proof for 3-cycles.

Now let $w$ be a cycle of length $\ell(w)$ greater than 3 and denote its start and terminal vertex by $x$. We may assume that $w$ does not contain a proper subcycle, that is,  $w$ does not visit the same vertex twice except for $x$. In other words, the set of all vertices $V(w)$ that lie on $w$ has cardinality $\ell(w)$. 

Let $G$ be the dual graph of $Q$, and $G^\circ$ be the graph obtained from $G$ by removing the leaf vertices (and their edges). Thus $G^\circ=(G_0^\circ,G_1^\circ)$ with $G_0^\circ = \{\textup{chordless cycles in $Q$}\}$ and $G_1^\circ=\{ \xymatrix@C15pt{C\ar@{-}[r]^\za&C'} \mid \textup{the chordless cycles $C,C'$ share the arrow $\za$}\}$. 

Let $R$ be the full subquiver of $Q$ whose vertex set is $V(w)$. In particular $w$ is a cyclic path consisting of boundary arrows in $R$. 
Note however, since we do not assume that $w$ is chordless, the quiver $R$ may have arrows that do not lie on $w$.
Let $H^\circ$ be the  subgraph of $G^\circ$ associated to $R$. Since $G$ is a tree, $H^\circ$ is a tree as well, and we may choose a leaf vertex  $C$ in $H^\circ$. Thus $C=\za_1\za_2\cdots\za_t$ is a chordless cycle in $R$ which has at most one interior arrow, say $\za_1$, in $R$. We may choose $C$ such that the starting point $x$ of $w$ does not lie on the path $\za_3\cdots\za_{t-1}$.  
Then the path $\za_2\cdots\za_t$ must be a subpath of $w$. Since $\za_1 $ is an interior arrow in $R$, there exists a unique other chordless cycle  $C'=\za_1\za_2'\cdots\za'_s$ in $R$ containing $\za_1$, and we have 
$\za_2\cdots\za_t \sim \za_2'\cdots\za'_s$. 
Let us point out here that the two paths have the same starting vertex and the same terminal vertex, but they do not have any other vertex in common. 

Now define $w'$ to be the path obtained from $w$ by replacing $\za_2\cdots\za_t $ with $\za_2'\cdots\za'_s$, and let $V(w')$ be the set of all vertices on $w'$. 
Let $y=t(\za_2')\in V(w')$. 
Since $w'$ is a path in $R$, and $R$ is the subquiver of $Q$ whose vertices are $V(w)$, we must have $V(w')\subset V(w)$. In particular, the vertex $y$ is also a vertex on the path $w$. Therefore, the path $w'$ must visit the vertex $y$ at least twice, and thus $w'$ contains a proper subcycle $w''$ which must be zero by induction. Hence $w\sim w'=0$, and the proof is complete. 
\end{proof}

\begin{prop}
 \label{prop 39} 
 Any two nonzero parallel paths in $Q$ are equal in $B$. 
 
 \end{prop}
\begin{proof}
Let $x,y\in Q_0$ and let $v=\za_1\za_2\cdots\za_s$, $w=\zb_1\zb_2\cdots\zb_t$ be two nonzero paths from $x$ to $y$. 
Without loss of generality, we may assume that $v$ and $w$ do not cross each other, because otherwise we would work with the shorter parallel  subpaths instead. We proceed by induction on $\max\{\ell(v),\ell(w)\}=\max\{s,t\}$. Since there are no parallel arrows in $Q$, the base case for the induction is when $s=t=2$. Then there exists a path $u$ from $y$ to $x$ such that $uv$ is a chordless cycle. Observe that $Q$ has no interior vertices, by Lemma~\ref{lem 36a}, thus using condition (Q3) of Definition~\ref{def Q} the path $u$ must be a single arrow. Similarly, there exists an arrow  $u'$ from $y$ to $x$ such that $uw$ is a chordless cycle, and since $Q$ has no parallel arrows we must have $u=u'$. Therefore $v=w$ by Lemma \ref{lem 37}.

Now suppose that one of $s,t$ is at least 3. Let $R$ be the full subquiver of $Q$ whose vertex set is the set of vertices visited by $v$ or $w$. From our conditions of $Q$, we conclude that $R$ is a planar graph, and each of its faces is a chordless cycle. Moreover, since the paths $v,w$ do not share a vertex other than $x$ and $y$, the arrows in the two paths $v$ or $w$ are boundary arrows in $R$ and all other arrows in $R$ are interior arrows. Moreover, all vertices of $R$ are boundary vertices.
In particular, the vertex $x$ is the starting point of two boundary arrows $\za_1$ and $\zb_1$. 

We are going to show first that the number of chordless cycles in $R$ that contain the vertex $x$ is exactly two. 
Indeed, since $\za_1$ and $\zb_1 $ both start at $x$, there must be at least two such chordless cycles. Suppose there are more than two. Then there must be at least four and we denote the first three of them by
\[C_1=\zb_1 C_1' \zg_1,\qquad C_2=\ze_1 C_2' \zg_1 \qquad C_3=\ze_1 C_3' \zg_2 
\]
see Figure \ref{fig 32}.
Let $z_1=s(\zg_1)$, $z_2=t(\ze_1)$ and $z_3=s(\zg_2)$.
Then all three vertices $z_1,z_2,z_3 $ are vertices of the subquiver $R$ and thus must lie on one of the two paths $v,w$. There are two possibilities; either all three vertices lie on one of the paths or two vertices lie on one path and the third lies on the other. 
To reach a contradiction, it suffices to show that neither $v$ nor $w$ can contain two of the vertices $z_1,z_2,z_3$, and, by symmetry and because the paths $v,w$ do not cross each other, it suffices to show that $w$ cannot contain $z_1$ and $z_2$. 

Suppose the contrary. Then $w$ must visit $z_1$ before $z_2$, because otherwise $w$ would cross itself. Let $w_1$ be the subpath of $w$ from $t(\zb_1)$ to $z_1$, let $w_2$ be the subpath from $z_1$ to $z_2$ 
and $w_3$ the subpath from $z_2$ to $y$, see the left picture in Figure \ref{fig 32}. 
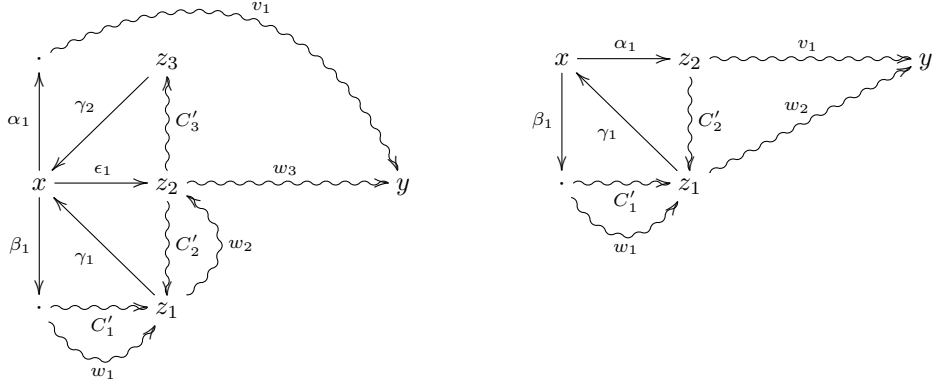
\begin{figure}[htbp]
\begin{center}

\[\xymatrix@R35pt@C35pt{
\cdot 
\ar@/^40pt/@{~>}[rrrd] ^{v_1}
& z_3 \ar[dl]_{\zg_2} 
\\
x \ar[u] ^{\za_1} \ar[r]^{\ze_1} \ar[d]_{\zb_1}
&z_2
\ar@{~>}[u] _{C_3'}
\ar@{~>}[d] ^{C_2'}
\ar@/^20pt/@{<~}[d] ^{w_2}
\ar@{~>}[rr]^{w_3} && y
\\
\cdot 
\ar@{~>}[r] _{C_1'}
\ar@/_20pt/@{~>}[r] _{w_1}
&z_1\ar[ul]^{\zg_1}
}
\qquad 
\qquad 
\xymatrix@R35pt@C35pt{
x \ar[r] ^{\za_1}  \ar[d]_{\zb_1}
&z_2 \ar@{~>}[rr]^{v_1}
\ar@{~>}[d] ^{C_2'} &&y
\\
\cdot 
\ar@{~>}[r] _{C_1'}
\ar@/_20pt/@{~>}[r] _{w_1}
&z_1\ar[ul]^{\zg_1}
\ar@{~>}[rru] ^{w_2}
}
\]
\caption{Proof of Proposition \ref{prop 39}}
\label{fig 32}
\end{center}
\end{figure}

Then the path $C_2'$ must be an arrow, because otherwise it would run through an interior vertex, which is impossible. In particular, the path $\ze_1C_2'$ is shorter than $w$, and, by induction, we conclude that $\ze_1C_2'\sim \zb_1 w_1 $, since both paths are from $x$ to $z_1$. Therefore we have $w\sim  \ze_1C_2'w_2w_3 = 0$, by Proposition~\ref{prop 38}, since $C_2'w_2$ is a cyclic path. This contradiction shows that  the number of chordless cycles that contain the vertex $x$ is exactly two. 

We now denote these two chordless cycles by
\[C_1=\xymatrix{x\ar[r]^{\zb_1} &\cdot\ar@{~>}[r] ^{C_1'}
&z_1 \ar[r]^{\zg_1}& x} 
\qquad\textup{and} \qquad C_2=
\xymatrix{x\ar[r]^{\za_1} &\cdot\ar@{~>}[r] ^{C_2'}
&z_1 \ar[r]^{\zg_1}& x}, \] 
see the right hand side of Figure \ref{fig 32}.
The vertex $z_1$ must lie on one of the paths $v$ or $w$. Without loss of generality, we assume that $z_1$ lies on $w$, and we write 
\[ w=
{\xymatrix{x\ar[r]^{\zb_1} &\cdot\ar@{~>}[r] ^{w_1}
&z_1 \ar@{~>}[r]^{w_2}& y} } 
\qquad\textup{and} \qquad 
v=
{\xymatrix{x\ar[r]^{\za_1} &\cdot\ar@{~>}[r] ^{v_1}
& y} } .
\]
Then the path $C_2'$ must be an arrow, because otherwise it would go through an interior vertex, which is impossible. In particular, we have the following inequalities on the lengths of subpaths
$\ell(\za_1C_2')<\ell(w)$, $ \ell(C_2'w_2)<\ell(w)$ , $\ell(v_1)<\ell(v)$, 
and, by induction, we conclude 
$\za_1C_2'\sim\zb_1w_1$ and $C_2'w_2\sim v_1.$
Consequently,
$v=\za_1v_1\sim\za_1C_2'w_2\sim \zb_1w_1w_2=w$.
\end{proof}

An algebra is called \emph{schurian} if $\dim\Hom(P(i),P(j))\le 1$ for all vertices $i,j$ of $Q$. With this terminology, we have the following reformulation of Proposition~\ref{prop 39}.

\begin{corollary}
 Every dimer tree algebra is schurian.
\end{corollary}

\medskip
\subsection{Some results toward the main conjecture}

In this section we show a few results that hold in the setting of Conjecture \ref{main conj}. We start by showing that two radical lines cross if and only if there is an extension between the corresponding radicals.

\begin{prop}
 \label{prop rad}
 Let $\rho(i)$ and $\rho(j)$ be two radical lines in $\cals$. Then 
 the following are equivalent.
\begin{enumerate}
\item [(a)]  $\rho(i)$ and $\rho(j)$ cross;
\item [(b)]  there is an arrow $i\to j$ or and arrow $j\to i $ in the quiver $Q$;
\item [(c)] $\Ext^1(\rad P(i),\rad P(j))\oplus \Ext^1(\rad P(j),\rad P(i)) \ne 0.$
\end{enumerate} 
\end{prop}
\begin{proof} The equivalence of (a) and (b) was proved in Lemma \ref{lem 35}. 
Suppose now (b) holds and let us assume without loss of generality the arrow is  $j\to i$. 
 Then we have the following commutative diagram with exact rows and columns,
 \[\xymatrix{
 &0\ar[d]&0\ar[d]&0\ar[d]&\\
 0\ar[r]&\rad P(i)\ar[r]\ar[d]&X\ar[r]\ar[d]&\rad P(j)\ar[r]\ar[d]&0\\
 0\ar[r]&P(i)\ar[r]\ar[d]^f&P(i)\oplus P(j)\ar[r]\ar[d]^{\left[\begin{smallmatrix} g_i&g_j\end{smallmatrix}\right]}& P(j)\ar[r]\ar[d]^h&0\\
 0\ar[r]&S(i)\ar[r]\ar[d]&
{\begin{smallmatrix}
 j\\i
\end{smallmatrix}}
 \ar[r]\ar[d]&
 S(j)\ar[r]\ar[d]&
 0\\
 &0&0&0 }
 \]
 where $f,g_j$ and $h$ are projective covers, the morphism $\left[\begin{smallmatrix} g_i&g_j\end{smallmatrix}\right]$ is given by the horseshoe lemma  and $X=\ker{\left[\begin{smallmatrix} g_i&g_j\end{smallmatrix}\right]}$. Since $g_j$ is a projective cover, we have $X\cong P(i)\oplus \ker g_j$. Therefore the top row of the diagram is a non-split short exact sequence, and thus  
 $\Ext^1(\rad P(j),\rad P(i)) $ is nonzero. Hence (c) holds.

Conversely, suppose that (c) holds and assume without loss of generality that the nonzero space is $\Ext^1(\rad P(j),\rad P(i))$. Applying the functor $\Hom(\rad P(j), -)$ to the short exact sequence
\[ \xymatrix{0\ar[r]&\rad P(i)\ar[r]& P(i)\ar[r]&S(i)\ar[r]&0}\] 
yields an exact sequence
\[\xymatrix{
\Hom(\rad P(j), S(i))\ar[r]&
\Ext^1(\rad P(j), \rad P(i)) \ar[r] &
\Ext^1(\rad P(j),  P(i))=0
}\]
where the last term is zero, since $\rad P(j)$ is a Cohen-Macaulay module. By assumption, the middle term is nonzero, and we can conclude the existence of a nonzero morphism $f\in \Hom(\rad P(j), S(i))$. 
It follows that $P(j)$ is supported at vertex $i$, thus there is a path  $j=x_0\to x_1\to \ldots\to x_s=i$, and this path is unique in $B$, by Proposition \ref{prop 39}.
 If $s>1$ then the morphism $f$ would produce the  following commutative diagram
 \[\xymatrix{k=\rad P(j)_{x_{s-1}}\ar[r]\ar[d]_{f_{x_{s-1}}}& \rad P(j)_i=k\ar[d]^{f_i\ne 0}\\
 0=S(i)_{x_{s-1}}\ar[r]&S(i)_i=k
 }\]
 which is impossible. Thus $s=1$ and our path is an arrow $j\to i$ and (b) holds.
\end{proof}
\begin{prop}\label{lem:cross}
Let $\zg, R(\zg)$ be 2-diagonals in $\cals$.  Then 
\begin{itemize}
\item[(a)] $\textup{add}\,P_0(\zg)\cap \textup{add}\,P_1(\zg) = \{0\}$ and $\textup{add}\,P_0(R(\zg))\cap \textup{add}\,P_1(R(\zg)) = \{0\}$.
\item[(b)] $P_1(\zg) \cong P_0(R(\zg))$.
\end{itemize}
\end{prop}

\begin{proof}
Part (a) follows from the assumption that $\zg, R(\zg)$ are 2-diagonals in $\cals$, so they have a minimal number of crossings with radical lines of $\cals$.  Then each of them crosses a radical line of $\cals$ at most once, and the direction of the crossing then determines whether this radical line gives rise to a summand of $P_0$ or $P_1$.  

To show part (b), suppose $P(j)\in \text{add}\, P_1(\zg)$.  Then the radical line $\rho(j)$ crosses $\zg$ from left to right, see Figure~\ref{fig:cross}.  Because the endpoints of $\zg$ and $R(\zg)$ are one apart, it follows that the arcs have opposite directions.   Then if $\rho(j)$ also crosses $R(\zg)$ then it crosses $R(\zg)$ from right to left and $P(j)\in \text{add}\, P_0(R(\zg))$.   Now, suppose $\rho(j)$ does not cross $R(\zg)$.  Then $\rho(j)$ has a common endpoint with $R(\zg)$, which we call $x$.   Without loss of generality suppose the direction of $R(\zg)$ and $\rho(j)$ is such that they both end in $x$.   Then $\rho(j)$ crosses $\zg$ from right to left, so $P(j)\in\text{add}\,P_0(\zg)$.  Then $P(j)\in \textup{add}\,P_0(\zg)\cap \textup{add}\,P_1(\zg)$, which contradicts part (a) of the proposition.  This shows that $P(j)\in \text{add}\, P_0(R(\zg))$, and more generally that $\text{add}\,P_1(\zg) \subset \text{add}\, P_0(R(\zg))$.  The reverse inclusion follows similarly.  
\end{proof}

\begin{remark}
In particular, the composition $P_1(R (\zg)) \xrightarrow{f_{R(\zg)} } P_1(\zg)\xrightarrow{f_{\zg}} P_0(\zg)$ is defined.  We will show in Proposition~\ref{big-lemma} that up to an automorphism of $P_1(\zg)$ this sequence is exact in the case where every chordless cycle in $Q$ has length three. 
\end{remark}

\begin{figure}
\begin{center}
{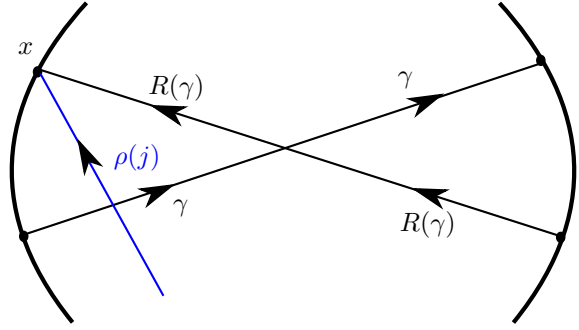}
\caption{Proof of Proposition~\ref{lem:cross}.}
\label{fig:cross}
\end{center}
\end{figure}
 
 \medskip 
 
\section{Objects of the syzygy category} \label{sect 5}
For the remainder of the paper, we 
let $Q$ be a quiver satisfying Definition~\ref{def Q}  and such that every chordless cycle in $Q$ has length three, and let $B$ denote the corresponding dimer tree algebra.  In this section we define and study the map $f_\zg: P_1(\zg)\to P_0(\zg)$ for each 2-diagonal $\zg$ in the associated checkerboard polygon $\cals$ of size $2N$.    The main result of this section says that the cokernel of $f_{\zg}$ is an indecomposable syzygy in $\text{mod}\,B$, see Corollary~\ref{cor:536}.

Below we provide an overview of this section.  

In Section~\ref{sect 5.1} we introduce the necessary terminology and then present the map $f_{\zg}$ in Definition~\ref{def:map}.   Then we  compare the two maps $f_{\zg}$ and $f_{R(\zg)}$, where $R$ denotes the clockwise rotation of the arc $\zg$ by $180/N$ degrees.  The main result of Section~\ref{sect 5.3} is Proposition~\ref{big-lemma} which shows that the sequence 
\[P_1(R (\zg)) \xrightarrow{\bar{f}_{R(\zg)} } P_1(\zg)\xrightarrow{f_{\zg}} P_0(\zg)\]
 is exact, where $\bar{f}_{R(\zg)}$ is obtained from $f_{R(\zg)}$ by applying certain isomorphisms.  Afterwards we study the properties of the cokernel of $f_{\zg}$ which we denote by $M_{\zg}$.  By definition, the map $f_{\zg}$ depends on the choice of a representative in the homotopy class of the arc $\zg$, and in Proposition~\ref{homotopy} of Section~\ref{sect 5.4} we show that the associated cokernel $M_{\zg}$ is independent of the choice of $\zg$.  This gives a well-defined map from $\diag$ to $\scmp\,B$.   The goal of the next subsection is to prove Proposition~\ref{prop:ind} that $M_{\zg}$ is an indecomposable module.  Finally, in the last subsection Section~\ref{sect 5.6} we combine all these results to conclude that $M_{\zg}$, the cokernel of $f_{\zg}$, is an indecomposable object in $\scmp\,B$ whose syzygy $\Omega M_{\zg}$ is isomorphic to $M_{R(\zg)}$, see Corollary~\ref{cor:536} and Theorem~\ref{thm:omega}.

\medskip
\subsection{Definition of $f_\zg$}\label{sect 5.1}

Let $\zg$ be a fixed representation of a 2-diagonal in $\diag$ such that it has a minimal number of crossings with radical lines in $\cals$ and except for the endpoints $\zg$ crosses one radical line at a time.  Then as we move along $\zg$, it crosses a subset of the radical lines in $\cals$ in a certain order.   Between any two consecutive crossings of $\zg$ with radical lines $\rho(i), \rho(j)$ the arc $\zg$ traverses either a shaded region or a white region in $\cals$ bounded by   $\rho(i), \rho(j)$.
A pair  $(i,j)$ is called a \emph{crossing pair} for  $\gamma$ if $\gamma$ crosses a shaded region and it crosses two radical lies $\rho(i), \rho(j)$ that bound this region consecutively.   Moreover, we always use the letter $i$ for the degree 0 crossing and letter $j$ for the degree 1 crossing. Thus, $P(i)\in \text{add}\, P_0(\zg)$ and $P(j)\in \text{add}\, P_1(\zg)$.  

Note that every shaded region in $\cals$ is a triangle, since by assumption every chordless cycles in $Q$ has length three.  This implies that $\rho(i), \rho(j)$ always cross $\zg$ in opposite directions so $P(i)$ and $P(j)$ cannot both belong to $P_0(\gamma)$ or $P_1(\gamma)$.  In addition, the vertices $i,j$ in a crossing pair $(i,j)$ belong to the same 3-cycle in $Q$, so they are connected by an arrow $i\to j$ or $j\to i$ which we call $\alpha$.  In this case, we will sometimes use $\xymatrix{(i \ar@{-}[r] ^{\alpha} & j)}$ to denote the crossing pair $(i,j)$.   

Choosing a representative of $\zg$ determines the corresponding order on the crossing pairs.  At the endpoints $\zg$ may cross only one side of a shaded region which lies on the boundary of $\cals$.   In what follows, this crossing will simply be denoted by $(i)$ or $(j)$ depending on whether it gives a summand of $P_0(\zg)$ or $P_1(\zg)$ respectively.

Given $\zg$ we define a \emph{crossing sequence} for $\zg$ 
\[ (i_0)\, \text{ or }\, (j_0), (i_1, j_1), (i_2, j_2), \dots, (i_n, j_n),  (i_{n+1})\, \text{ or }\, (j_{n+1})\]
to be a complete ordered set of crossing pairs for $\zg$ for some $n\geq 0$ and where $i_k$ denotes the degree 0 crossing and $j_k$ denotes the degree 1 crossing of the pair $(i_k, j_k)$.  Note that depending on the end behavior of $\zg$ the elements $i_0, j_0, i_{n+1}, j_{n+1}$ may or may not appear in the crossing sequence for $\zg$.  Here, we depict the most general case.   

\begin{remark}
The crossing sequence for a fixed $\zg$ is unique up to reversing the order.  However, different representatives $\zg, \zg'$ of a 2-diagonal in $\diag$ can give different crossing sequences. 
\end{remark}

Next we introduce two important notions needed to define the map $f_{\zg}$.

\begin{definition}
We say that a path $w$ in the quiver $Q$ is {\it valid} if no two arrows in $w$ lie in the same 3-cycle.  
\end{definition}

\begin{remark}
A valid path in $Q$ is nonzero in the algebra $B$ because such a path does not contain any relations.  Therefore, there is at most one valid path between any two vertices of $Q$.   
\end{remark}

We will often use $x\leadsto y$ to denote a path in the quiver starting at vertex $x$ and ending in vertex $y$. 

\begin{definition}\label{def:forward}
Let the crossing sequence for $\zg$ be as follows 
\[ (i_0)\, \text{ or }\, (j_0), (i_1, j_1), (i_2, j_2), \dots, (i_n, j_n),  (i_{n+1})\, \text{ or }\, (j_{n+1})\]
and let $s, t \in \{0, \dots, n+1\}$.
For $| s-t| = 1$, we say the step from $s$ to $t$ in the crossing sequence is \emph{forward} if there exists a valid path $i_s \leadsto j_t$ or a valid path $j_s \leadsto i_t$.  For general $s, t \in \{0, \dots, n+1\}$ and $s\not=t$, the subsequence of the crossing sequence from $s$ to $t$ is \emph{forward} if each of its steps is forward.   
\end{definition}

\begin{definition}\label{def:map}
Let $\zg$ be a fixed representative of a 2-diagonal in $\diag$ with crossing sequence 
\[ (i_0)\, \text{ or }\, (j_0), (i_1, j_1), (i_2, j_2), \dots, (i_n, j_n),  (i_{n+1})\, \text{ or }\, (j_{n+1}).\]
Let $P_1(\zg)= \bigoplus_l P(j_l)$ and $P_0(\zg)= \bigoplus_l P(i_l)$ and define
\[f_\zg\colon P_1(\zg) \longrightarrow P_0(\zg)\] 
by setting
\[  f_{{\gamma}_{s,t}}\colon P(j_t) \longrightarrow P(i_s)\] 
the multiplication by the path $w: i_s \leadsto j_t$ in $Q$ unless $s\not=t$ and 
\begin{itemize}
\item[(i)] the subsequence of the crossing sequence from $s$ to $t$ is not forward, or
\item[(ii)] the path $w$ is not valid
\end{itemize}
and in these cases we define $f_{\zg_{i_s, j_t}}=0$. 
\end{definition}

\begin{remark}
In the definition of the map $f_{\zg}$ condition (i) does not imply condition (ii).  For example, suppose that there is a 2-diagonal $\zg$ such that  $(i_s, j_s), \dots, (i_{s+3}, j_{s+3})$ is a subsequence of its crossing sequence and that these vertices form a full subquiver of $Q$ given below. 

\[\xymatrix@R=10pt@C=10pt{&j_{s+3}\ar[r] & i_{s+3}\ar[dl]\\
j_s\ar[d] & \ar[l] i_{s+1} \ar[d]\ar[r] \ar[u]& j_{s+2}\ar[d] \ar[u] \\
i_s\ar[ur] & j_{s+1}\ar[l]\ar[r] & i_{s+2}\ar[ul]}
\]

Then the step from $s$ to $s+1$ is not forward, so the map $f_{\zg}$ will not contain the path $i_s\to i_{s+1}\to  j_{s+3}$ even though this path is valid.  
\end{remark}

Note that, for $s\not=t$, if $f_{\zg_{i_s, j_t}}\not=0$ then the path $w$ is valid and all steps from $s$ to $t$ are forward.  Also, if $s=t$ then $i_s, j_s$ are connected by an arrow and lie in a common 3-cycle, so there is a path $w: i_s\leadsto  j_s$ in $Q$. If there is an arrow $\alpha: i_s\to j_s$ in $Q$ then $w=\alpha$, and if there is an arrow $\alpha: j_s \to i_s$ then the path is given the composition of two other arrows $i_s \to \bullet \to j_s$ in this 3-cycle.  

\begin{remark}
By Proposition~\ref{prop 39} if there is a path from $i_s$ to $j_t$ in $Q$ then it is unique in the algebra $B$, so the map $f_\zg$ is well-defined.   It is not a generic map in general, as there may be nonzero paths in $Q$ from $i_s$ to $j_t$ that contain two arrows in the same 3-cycle.  We also note that $f_\zg$ depends on the choice of the representative $\zg$ of the 2-diagonal, however we will show later in Proposition~\ref{homotopy} that the cokernel of $f_\zg$ is independent of the representative and gives an indecomposable object in   $\scmp\, B$.
\end{remark}
%

\medskip
\subsection{Exactness}\label{sect 5.3}

In this section,  given two 2-diagonals in $\cals$ related by a clockwise rotation $R$, we define particular representatives $\zg, R(\zg)$ of these 2-diagonals that satisfy certain compatibility criteria.  Then we show that the corresponding maps $f_{\zg}, f_{R(\zg)}$ are related as follows $\text{ker}\, f_{\zg}= \text{Im}\, \bar{f}_{R(\zg)}$, where $\bar{f}_{R(\zg)}$ is obtained from ${f}_{R(\zg)}$ by introducing some negative signs.  As a corollary we conclude that the syzygy functor $\Omega$ in $\scmp\,B$ is given by the clockwise rotation $R$ in $\cals$.

Let $\zg$ be a representative of a 2-diagonal in $\cals$, and let $f_{\zg}: P_1(\zg)\to P_0(\zg)$ be the corresponding map given in Definition~\ref{def:map}.
For a positive inter $k$, let $J_k = \begin{bsmallmatrix}1 &  && &\\ & -1 &&   \\ && 1 && \text{\Large 0} \\
\text{\Large 0}&&& \ddots \\ &&&& (-1)^{k+1} \end{bsmallmatrix}$ denote a $k\times k$ diagonal matrix with diagonal entries alternating between $1$ and $-1$. 
Given a 2-diagonal represented by an arc $\zg$ we define a new map $\bar{f}_{\zg}: P_1(\zg)\to P_0(\zg)$ by $\bar{f}_{\zg}:=J_{|P_0(\zg)|} \, f_{\zg} \, J_{|P_1(\zg)|}$.

\begin{remark} \label{rem:bar}

\begin{itemize}
\item[(a)] The signs in $\bar{f}_{\zg}$ are $\begin{bsmallmatrix}+ & - & + & \cdots \\ - & + & - & \cdots \\ + & - & + & \cdots \\  \vdots & \vdots & \vdots & \ddots  \end{bsmallmatrix}$.
\item[(b)] The cokernels of $f_{\zg}$ and  $\bar{f}_{\zg}$ are isomorphic.  Indeed, this follows from the 5-Lemma applied to the following commutative diagram with exact rows.
\[ \xymatrix@R=20pt@C=30pt{P_1(\zg)\ar[r]^{f_{\zg}} \ar[d]^{J_{|P_1(\zg)|}}_{\cong}  & P_0(\zg)\ar[r] \ar[d]^{J_{|P_0(\zg)|}}_{\cong}  & \text{coker}\, {f}_{\zg} \ar[d]_{\cong} \ar[r] & 0 \ar[d]_{\cong} \ar[r] &0  \ar[d]_{\cong} \\
P_1(\zg)\ar[r]^{\bar{f}_{\zg}}   & P_0(\zg)\ar[r]   & \text{coker}\, \bar{f}_{\zg} \ar[r] & 0  \ar[r] &0} \] 
\end{itemize}
\end{remark}


\begin{definition}\label{def:compatible} Let $\zg, \zg'$ be representatives of 2-diagonals. 
We say that $\zg$ and $\zg'$ are \emph{compatible} if
\begin{itemize}
\item[(a)] whenever $(i,j)$ is a crossing pair for $\zg$, and $\zg'$ crosses both $\rho(i)$ and $\rho(j)$ then either $(i,j)$ or $(j,i)$ is a crossing pair for $\zg'$, or
\item[(b)] $\zg'=R\zg$ is a representative of the rotation of $\zg$ and both $\zg, \zg'$ are radical lines. 
 
\end{itemize}
\end{definition}

\begin{remark}
 \begin{itemize}
\item[(a)] In part (a) of the definition the order of vertices in a pair changes whenever the degree of the crossings of $\zg$ and $\rho(i),\rho(j)$ is opposite to the degree of the crossings between $\zg'$ and $\rho(i),\rho(j)$. 
\item[(b)] A radical line has essentially two representatives, one running parallel on the left of the radical line and the other running parallel on the right. These two representatives will have different crossing sequences. However, as we shall see in Proposition~\ref{homotopy},   both representatives will induce the same morphism on projectives.
\end{itemize}
\end{remark}

In this section, we will only be interested in the compatibility of the representatives $\zg, R(\zg)$.  The notion of compatibility means that that the two arcs $\zg, R(\zg)$ follow each other closely and have the same crossing sequences except possibly at the ends.

\begin{remark}
For every pair of 2-diagonals in $\cals$ related by the rotation $R$ there exist representatives $\zg, R(\zg)$ that are compatible.  
\end{remark}


Now we take a pair of compatible 2-diagonals $\zg, R(\zg)$ and consider the sequence of morphisms 
\[\xymatrix@R=10pt@C=30pt{P_1(R(\zg))\ar[r]^-{\bar{f}_{R(\zg)}} & P_0(R(\zg)) = P_1(\zg)\ar[r]^-{f_{\zg}} & P_0(\zg)  }\]
where the equality in the middle follows from Proposition~\ref{lem:cross} (b).  

\begin{prop}\label{big-lemma}
Let $\zg, R(\zg)$ be a pair of compatible arcs in $\cals$.  Then $\textup{ker}\,f_{\zg} = \textup{Im}\,\bar{f}_{R(\zg)}$. 
\end{prop}

\begin{proof}
See section \ref{Aexactness}.\end{proof}

\medskip
\subsection{Homotopy}\label{sect 5.4}

Recall that the map $f_{\zg}$ depends on the crossing sequence of a particular representative $\zg$ of a given 2-diagonal.  In this subsection, we show  that the cokernel of $f_\zg$ is independent of the choice of the representative. 

Let $\zg$ be a representative of a 2-diagonal in $\cals$. We shall say that a homotopy of $\zg$ is \emph{trivial} if it fixes the crossing sequence of $\zg$, and it is \emph{elementary} if it moves $\zg$ across the meeting point $p$ of two shaded triangles of which $\zg$ crosses all four sides and is trivial elsewhere. For example, the two arcs $\zg_1$ and $\zg_2$ in Figure~\ref{fig:47} differ by an elementary homotopy whose point $p$ is the crossing points of the radical lines labeled 1 and 3.

We define the \emph{degree} of an elementary homotopy to be the degree of the crossing between $\zg$ and any of the two radical lines that meet at the point $p$. We remark that the degree is well-defined because both radical lines at $p$ cross $\zg$ in the same direction. Indeed, this follows immediately from the fact that the directions of the radical lines alternate along white regions.

An elementary homotopy of $\zg$ acts on the crossing sequence of $\zg$ by replacing two consecutive pairs 
\[\left\{\begin{array}{ll} (x,w),(y,z) \textup{ by }(y,w),(x,z), &\textup{if the degree is zero;}\\
(w,x),(z,y) \textup{ by }(w,y),(z,x), &\textup{if the degree is one.}\end{array}\right.
\]
\smallskip

We are now ready for the main result of this subsection.

\begin{prop}\label{homotopy}
Let $\zg_1, \zg_2$ be two representatives of the same 2-diagonal in $\cals$.  Then $\textup{coker}\,f_{\zg_1}\cong \textup{coker}\,f_{\zg_2}$. 
Moreover, if the 2-diagonal is a radical line then $f_{\zg_1}=f_{\zg_2}.$
\end{prop}

\begin{proof}
Suppose first that the arc $\zg_1$ is not a radical line.  It suffices to show that the cokernel does not change under an elementary homotopy. 
 So suppose $\zg_1, \zg_2$ are as in Figure~\ref{fig:47}.   

\begin{figure}
\centerline{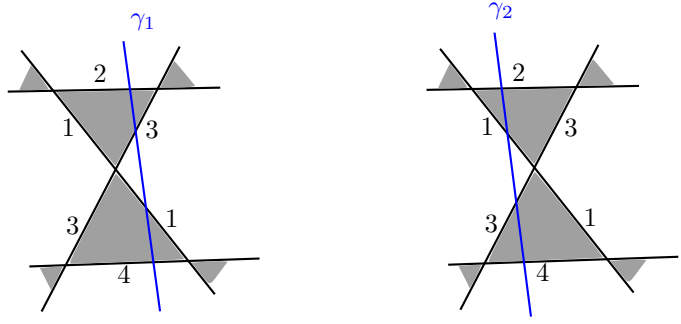}
\caption{Proof of Proposition~\ref{homotopy}. The arc $\zg_2$ is obtained from the arc $\zg_1$ by an elementary homotopy.}
\label{fig:47}
\end{figure}

The crossing sequences for $\zg_1, \zg_2$ only differ by two consecutive crossing pairs $\{2,3\}, \{1,4\}$ for $\zg_1$ and $\{1,2\}, \{4,3\}$ for $\zg_2$. 
Note that the step between these pairs is forward for $\zg_1$ but it is not forward for $\zg_2$, because by Remark~\ref{rem:forward} forward steps correspond to moving counterclockwise around the boundary of a white region.  Moreover, 1 and 3 cross $\zg_1, \zg_2$ in the same degree because the direction of the boundary edges around a white region is alternating.  The degree of the crossing of $2$ and $4$ is opposite to that of $1$ and $3$ because the orientation is directed along the boundary edges of a shaded region.  We suppose that the direction of $\zg_1, \zg_2$ is such that $1,3$ are the degree 0 crossings and $2,4$ are the degree 1 crossings.  The case when $1,3$ are of degree 1 and $2,4$ are of degree 0 follows in the same way.
Then let 
\[ (i_0)\, \text{ or }\, (j_0), (i_1, j_1), \dots, (i_t, j_t), (3,2), (1,4), (i_{t+3}, j_{t+3}), \dots, (i_n, j_n),  (i_{n+1})\, \text{ or }\, (j_{n+1})\]
\[ (i_0)\, \text{ or }\, (j_0), (i_1, j_1), \dots, (i_t, j_t), (1,2), (3,4), (i_{t+3}, j_{t+3}), \dots, (i_n, j_n),  (i_{n+1})\, \text{ or }\, (j_{n+1})\]
be the crossing sequences for $\zg_1, \zg_2$ respectively.  

Define 
\[P_a=\bigoplus_{h=0}^t P(j_h) \;\;\;\;\; P_a'=\bigoplus_{h=0}^t P(i_h)\;\;\;\;\; P_b=\bigoplus_{h=t+3}^{n+1} P(j_h)\;\;\;\;\; P_b' = \bigoplus_{h=t+3}^{n+1} P(i_h)\]
Then 
\[P_0(\zg_1)=P_0(\zg_2)=P_a'\oplus P(1)\oplus P(3)\oplus P_b' \;\;\;\;\;\;P_1(\zg_1)=P_1(\zg_2)=P_a\oplus P(2)\oplus P(4)\oplus P_b\]

We claim that there exists an isomorphism $\varphi$ such that the following diagram commutes. 

\begin{equation}
\label{diag:573}
\xymatrix{P_a\oplus P(2)\oplus P(4)\oplus P_b \ar[r]^{f_{\zg_1}} \ar[d]^{1} & P_a'\oplus P(3)\oplus P(1)\oplus P_b'\ar[d]^{\cong}_{\varphi}\\
P_a\oplus P(2)\oplus P(4)\oplus P_b\ar[r]^{f_{\zg_2}} & P_a'\oplus P(1)\oplus P(3)\oplus P_b'
}
\end{equation}

To prove the claim there are four cases to consider depending on the direction of the step from $t$ to $t+1$ and the step $t+2$ to $t+3$.  We show the claim in two cases only, since the other two are dual. 

Suppose first that both steps are forward. Then the quiver for $Q(\zg_1)=Q(\zg_2)$ is as follows, where the quivers appear in Definition~\ref{def:5.8}. 


\[\xymatrix@C=10pt@R=10pt{i_{t+3} & \ar[l] \cdots & 4\ar[l] \ar[r] & 1 \ar[dl] & \cdots\ar[l] & j_t\ar[l]\\
j_{t+3}& \ar[l] \cdots & 3\ar[u]\ar[l] \ar[r] & 2\ar[u]  & \cdots\ar[l] & i_t\ar[l]
}
\]

The maps $f_{\zg_1}, f_{\zg_2}$ are given in matrix form as in the equation below, which also shows the isomorphism $\varphi$.

\[
\varphi f_{\zg_1}=\begin{bsmallmatrix}1_{P_a'} & 0 & 0 & 0 \\
0 & 1\to 3 & -1_{ P(1)} & 0 \\
0 & 1_{P(3)} & 0 & 0 \\
0 &  0 & 0 & 1_{P_b'}
\end{bsmallmatrix} \begin{bsmallmatrix}f_{a,a} & f_{a,2} & 0 & 0 \\
0 & 3\to 2 & 3\to 4 & f_{3,b} \\
0 & 0 & 1\leadsto 4 & f_{1,b} \\
0 &  0 & 0 & f_{b,b}
\end{bsmallmatrix} = 
\begin{bsmallmatrix}f_{a,a} & f_{a,2} & 0 & 0 \\
0 & 1\leadsto 2 & 0 & 0 \\
0 & 3\to 2 & 3\to 4 & f_{3,b} \\
0 &  0 & 0 & f_{b,b}
\end{bsmallmatrix} 
=f_{\zg_2}
\]

Indeed, for $f_{\zg_1}$ we have an upper triangular matrix, because all steps from $t$ to $t+3$ are forward.  The position $(1,3)$ is zero, because any path from $i_h$ to $4$ with $h\leq t$ factors through the path $1\to 3\to 4$ and therefore is not valid. By Lemma~\ref{rect-trap-lemma1}(3), position $(1,3)$ being zero implies that position $(1,4)$ is zero as well.   For $f_{\zg_2}$ the entries $(2,1), (3,1), (4,1)$  are zero because the step from $t$ to $t+1$ is forward, the entries $(1,3), (2,3)$ are zero because the step from $t+1$ to $t+2$ is not forward, and the entry $(4,3)$ is zero because the step from $t+2$ to $t+3$ is forward.  Finally the entries $(1,4), (2,4)$ are zero because entries $(1,3), (2,3)$ are zero, and entry $(4,2)$ is zero because $(4,3)$ is also zero.  The map $\varphi$ is clearly an isomorphism, and the equality above holds.  This shows the claim that the diagram commutes in this case.

Suppose now that the step from $t$ to $t+1$ is forward and the step from $t+2$ to $t+3$ is not.  The quiver in this case is as follows. 

\[\xymatrix@C=10pt@R=10pt{
3\ar[r]\ar[d] & 4 \ar[d] & \cdots\ar[l] & i_{t+3}\ar[l]\\
2\ar[r] & 1\ar[ul] & \cdots \ar[l] & j_{t+1}\ar[l]\\
&\ddots\ar[ul]&\ddots\ar[ul]\\
&&i_t\ar[ul]&j_t\ar[ul]
}
\]

The maps $f_{\zg_1}, f_{\zg_2}$ are given in matrix form as in the equation below, which also shows the same isomorphism $\varphi$ as before.   

\[
\varphi f_{\zg_1}=\begin{bsmallmatrix}1_{P_a'} & 0 & 0 & 0 \\
0 & 1\to 3 & -1_{ P(1)} & 0 \\
0 & 1_{P(3)} & 0 & 0 \\
0 &  0 & 0 & 1_{P_b'}
\end{bsmallmatrix} \begin{bsmallmatrix}f_{a,a} & f_{a,2} & 0 & 0 \\
0 & 3\to 2 & 3\to 4 & 0 \\
0 & 0 & 1\leadsto 4 & 0 \\
0 &  0 & f_{b,4} & f_{b,b}
\end{bsmallmatrix} = 
\begin{bsmallmatrix}f_{a,a} & f_{2,a} & 0 & 0 \\
0 & 1\leadsto 2 & 0 & 0 \\
0 & 3\to 2 & 3\to 4 & 0 \\
0 &  0 & f_{b,4} & f_{b,b}
\end{bsmallmatrix} 
=f_{\zg_2}
\]

Indeed, for $f_{\zg_1}$ we have five zeros below the diagonal because steps from $t$ to $t+2$ are forward.  The entries $(1,4), (2,4), (3,4)$ are zero because the step from $t+2$ to $t+3$ is not forward.  Then entry $(1,3)$ is again zero by the same reasoning as in the previous case.  For $f_{\zg_2}$ the entries $(2,1), (3,1), (4,1)$ are zero because the step from $t$ to $t+1$ is forward, and we have the five zeros above the diagonal because the two steps from $t+1$ to $t+3$ are not forward.  The entry $(4,2)$ is zero because every path from $i_h$ to $2$ with $h\geq t+3$ factors through the path $1\to 3\to 2$ and therefore is not valid.  Again we see that $\varphi$ is an isomorphism and the equation above holds.  This proves the claim that the diagram (\ref{diag:573}) of 2-term complexes commutes.  The 5-Lemma implies that $f_{\zg_1}, f_{\zg_2}$ have isomorphic cokernels and thus the lemma holds in the case when $\zg_1, \zg_2$ are not homotopic to a radical line in $\cals$. 

Now, suppose that $\zg_1, \zg_2$ are homotopic to a radical line $\rho(k)$.  Then $\zg_1, \zg_2$ cross the same set of radical lines as $\rho(k)$.  Note that $\zg_1, \zg_2$ do not cross $\rho(k)$.  By Lemma~\ref{lem 35} we have $\rho(k)$ crosses $\rho(h)$ if and only if there is an arrow between $k$ and $h$ in the quiver $Q$.  Moreover, $\zg_1, \zg_2$ have the same orientation as $\rho(k)$, which means that if there is an arrow $k\to i$ in $Q$ then $P(i) \in \text{add}\,P_0(\zg_1)$ and if there is an arrow $j\to k$ then $P(j)\in  \text{add}\,P_1(\zg_1)$.

Label the two boundary arrows in $Q$ adjacent to $k$ by $\partial_1, \partial_2$.  The subquiver of $Q$ determined by $k$ and the vertices adjacent to $k$ is as follows, when one of the boundary arrows starts in $k$ and another ends in $k$.  The cases when both boundary arrows start in $k$ or both boundary arrows end in $k$ follow in the exact same way.  Thus, for simplicity we may assume that we have the following situation. 

\[
\xymatrix@C=10pt@R=10pt{
&j_t\ar[dr]^{\partial_2}\\
i_t\ar[ur]&&k\ar[rr]^{\partial_1} \ar[ll]\ar[d]&& i_1\ar[dl]\\
&j_2\ar[ur]\ar@{.}[ul]&i_2\ar[l]\ar[r]&j_1\ar[ul]\\
}
\]

Because the only arrows between these vertices are $i_{s+1}\to j_{s}\leftarrow i_{s}$  for $s\in \{1, \dots, t-1\}$, it follows that in the crossing sequence for $\zg_1, \zg_2$ the vertex $j_s$ is in a pair with $i_s$ or $i_{s+1}$.  Moreover, one crossing pair determines the pairing for the remaining vertices in the crossing sequence uniquely.  This means that there are two possible crossing sequences 
\[ (i_1, j_1), (i_2, j_2), \dots, (i_t, j_t) \hspace{.2cm} \text{or}  \hspace{.2cm} (i_1), (i_2, j_1), \dots, (i_t, j_{t-1}), (j_t)\]
for any diagonal homotopic to $\rho(k)$.  Let the crossing sequences for $\zg_1, \zg_2$ be as above.  For $\zg_1$ no step in its crossing sequence is forward and all steps are trapezoidal with the path from $i_{l+3}$ to $j_l$ for all $l=1, \dots, t-3$ being invalid.  Hence $f_{\zg}$ is zero except on the main diagonal and the diagonal below the main diagonal.  For $\zg_2$ all steps in its crossing sequence are forward, and all of them except for the first and the last are trapezoidal with the path from $i_t$ to $j_1$ being invalid.  Then the two maps $f_{\zg_1}, f_{\zg_2}$ remain the same. 

\[f_{\zg_1}= f_{\zg_2}=\begin{bsmallmatrix}  
i_1\to j_1 & 0 & \cdots & \cdots & \cdots \\
i_2\to j_1 & i_2\to j_2 & 0 & \cdots & \cdots  \\
0 & i_3\to j_2 & i_3 \to j_3 & 0 & \dots \\
\vdots & 0 & \ddots & \ddots & \vdots
\end{bsmallmatrix}
\]

This shows that the lemma holds in the case when $\zg_1, \zg_2$ are homotopic to a radical line, and the proof is complete. 
\end{proof}

The following corollary, which will be important in the next sections, is a direct consequence of the above proof. 

\begin{corollary}
 \label{cor:comp}
 Let $\zg_1,\zg_2$ be two representatives of the same 2-diagonal $\zg$ and ${f_{\zg_1} },{f_{\zg_2} }$ the corresponding morphisms. Then there is a commutative diagram
 \[ \xymatrix{P_1(\zg)\ar[d]_{\varphi_1} \ar[r]^{f_{\zg_1} }& P_0(\zg) \ar[d]^{\varphi_0}\\
P_1(\zg) \ar[r]^{f_{\zg_2}} & P_0(\zg) 
 }
 \]
with $\varphi_0,\varphi_1$ the isomorphisms induced by the homotopy between $\zg_1$ and $\zg_2$. \qed
\end{corollary}
\begin{definition}
 \label{def:automorphism of homotopy}
 We call the pair $(\varphi_1,\varphi_0)\in \Aut(P_1(\zg))\oplus\Aut(P_0(\zg))$ the \emph{automorphism of the homotopy} that transforms $\zg_1$ into $\zg_2$, and we say that 
 \[ f_{\zg_2}=\varphi_0\,f_{\zg_1}\,\varphi_1^{-1}
 \]
 is obtained from $f_{\zg_1}$ by \emph{conjugation with the homotopy automorphism.}
\end{definition}

\medskip
\subsection{Indecomposibility}\label{sect 5.5}

In this section we establish that the cokernel of $f_{\zg}$ is indecomposable.  First we need a few preparatory lemmas. 
\begin{lemma}\label{lem:ab}
Let $\zg$ be a representative of a 2-diagonal in $\cals$, and let $f_{\zg}: \bigoplus_{h=1}^{n} P(j_h) \to  \bigoplus_{h=1}^m P(i_h)$ be the associated morphism.
Suppose we have a commutative diagram 
\[\xymatrix{
 \bigoplus_{h=1}^n P(j_h) \ar[r]^{f_{\zg}} \ar[d]^{g_j}&  \bigoplus_{h=1}^m  P(i_h) \ar[d]^{g_i} \\
 \bigoplus_{h=1} ^n P(j_h) \ar[r]^{f_{\zg}} &  \bigoplus_{h=1}^ m P(i_h)
}\]
for some matrices $g_j =(a_{h,h'})$ and $g_i = (b_{h,h'})$. Then $f_{\zg_{s,t}}=0$ or $b_{s,s}=a_{t,t}$ for all $s,t$.
\end{lemma}
\begin{proof}
 This is proved in section \ref{Asect 5.5}.\end{proof}

\begin{lemma}\label{lem:arrow}
Let $\zg$ be a representative of a 2-diagonal in $\cals$.  Suppose that the two steps from $s-1$ to $s+1$ are forward in the crossing sequence for $\zg$.  Then there is a valid path $i_{s-1}\leadsto j_{s+1}$ if and only if the arrow $\alpha_s$ is oriented $\alpha_s: j_s \to i_s$.  In particular, at least one of $f_{{\zg}_{s,s}}, f_{{\zg}_{s-1, s+1}}$ is nonzero.  
\end{lemma}

\begin{proof}
This follows from the two pictures below, where we depict the two cases depending on the orientation of the arrow $\alpha_s$.  

\[
\xymatrix@R=10pt@C=10pt{
&&i_{s+1}&j_{s+1} &&& j_{s-1}\ar[r] & \cdots \ar[r]  & i_s\ar[d] \\
&&\raisebox{0pt}[0.9\height][0.3\height]{ $\vdots$ } \ar[u] &\raisebox{0pt}[0.9\height][0.3\height]{ $\vdots$ } \ar[u] &&& i_{s-1}\ar[r] & \cdots \ar[r] & \bullet \ar[r]\ar[d]&j_s\ar[ul]_{\alpha_s} \ar[d]\\
j_{s-1}\ar[r] & \cdots \ar[r] & \bullet \ar[r] \ar[u]& i_s\ar[d]^{\alpha_s} \ar[u]&&& &&\raisebox{0pt}[0.9\height][0.3\height]{ $\vdots$ } \ar[d] &\raisebox{0pt}[0.9\height][0.3\height]{ $\vdots$ } \ar[d] \\
i_{s-1}\ar[r] & \cdots & \cdots\ar[r]&j_s\ar[ul] &&& &&j_{s+1}&i_{s+1}
}
\]

Moreover, either there is an arrow $j_s\to i_s$ and $f_{{\zg}_{s-1, s+1}}$ is a nonzero valid path $i_{s-1}\leadsto j_{s+1}$ or there is an arrow $i_s\to j_s$ which gives the nonzero entry $f_{{\zg}_{s,s}}$.
\end{proof}

\begin{prop}\label{prop:ind}
Let $\zg$ be a representative of a 2-diagonal in $\cals$, then $\textup{coker}\, f_{\zg}$ is indecomposable. 
\end{prop}

\begin{proof}
Let $f=f_{\zg}: \bigoplus_{h=1}^{n} P(j_h) \to  \bigoplus_{h=1}^m P(i_h)$.   Suppose we have a commutative diagram 

\[\xymatrix{
 \bigoplus_{h=1}^n P(j_h) \ar[r]^{f_{\zg}} \ar[d]^{g_j}&  \bigoplus_{h=1}^m  P(i_h) \ar[d]^{g_i} \\
 \bigoplus_{h=1} ^n P(j_h) \ar[r]^{f_{\zg}} &  \bigoplus_{h=1}^ m P(i_h)
}\]

for some matrices $g_j =(a_{h,h'})$ and $g_i = (b_{h,h'})$. By Lemma~\ref{lem:ab} we have $f_{s,t}=0$ or $b_{s,s}=a_{t,t}$ for all $s,t$.

Consider the entry $f_{s,s+1}$ of the matrix $f$. Recall that it corresponds to a path from $i_s$ to $j_{s+1}$.  Suppose now without loss of generality that the step from $s$ to $s+1$ is forward. Then, by Definition~\ref{def:forward}, there is a valid path $i_s$ to $j_{s+1}$ or a valid path $j_s$ to $i_{s+1}$.  Note that if the latter path exists then Lemma~\ref{lem:one-step} implies that both paths are valid.    Thus $f_{s,s+1}\not=0$ and the above equation with $t=s+1$ yields

\begin{equation}\label{eq5.1}
b_{s,s}=a_{s+1,s+1}.
\end{equation}

If the step from $s-1$ to $s$ is forward too, then $f_{s-1, s}\not=0$ and hence 

\begin{equation}\label{eq5.2}
b_{s-1,s-1}=a_{s,s}.
\end{equation}

By Lemma~\ref{lem:arrow}, we have $f_{s,s}\not=0$ or $f_{s-1,s+1}\not=0$.  Suppose first $f_{s,s}\not=0$.  Then $b_{s,s}=a_{s,s}$ and thus 

\begin{equation}\label{eq5.3}
b_{s-1,s-1}=a_{s,s}=b_{s,s}=a_{s+1,s+1}
\end{equation}

where the first equality follows from (\ref{eq5.2}) and the last equality follows from (\ref{eq5.1}). 

Suppose now $f_{s-1, s+1}\not=0$.  Then $b_{s-1, s-1}=a_{s+1,s+1}$ and thus 

\[a_{s,s}=b_{s-1,s-1}=a_{s+1,s+1}=b_{s,s}\]

where again the first equality follows from (\ref{eq5.2}) and the last equality follows from (\ref{eq5.1}).  This is the same equation as (\ref{eq5.3}).  

If the step from $s-1$ to $s$ is not forward but the step from $s$ to $s+1$ is still forward, then $f_{s,s-1}\not=0$, and thus 

\begin{equation}\label{eq5.4}
b_{s,s}=a_{s-1,s-1}.
\end{equation}

Moreover, since the step from $s-1$ to $s$ and the step from $s$ to $s+1$ are in opposite direction, the arrow $\alpha_s: \xymatrix{i_s \ar@{-}[r]&j_s}$ is not a boundary arrow and therefore $f_{s,s}\not=0$.   Thus $b_{s,s}=a_{s,s}$ and thus equations (\ref{eq5.1}) and (\ref{eq5.4}) imply 

\[
b_{s,s}=a_{s-1,s-1}=a_{s,s}=a_{s+1,s+1}.
\]

Similar, if both steps from $s-1$ to $s+1$ are not forward we get 

\[
a_{s-1,s-1}=b_{s,s}=a_{s,s}=b_{s+1,s+1}
\]

and if the first step is forward but the second is not we get 

\[
a_{s,s}=b_{s-1,s-1}=b_{s,s}=b_{s+1,s+1}.
\]


This implies that in all possible orientations of the steps we obtain $a_{s,s}=a_{t,t}=b_{s,s}=b_{t,t}$ for all $s$ and $t$.  In other words, all diagonal entries in the matrices $g_j,  g_i$ are equal.  

Thus, $g_j = \lambda I + g_j'$ and $g_i = \lambda I + g_i'$ where $g_j'$ and $g_i'$ are nilpotent.  Thus, every endomorphism of the two term complex $f_{\zg}$ is of the form $\lambda\cdot \text{Id} + N'$ with $\lambda\in \kb$ and $N'$ nilpotent.  In particular, every endomorphism of $\text{coker}\,f_{\zg}$ is of the form $\lambda\cdot \text{Id}_{\text{coker}\,f_{\zg}} + N$, where $N$ is nilpotent, and thus the endomorphism ring of $\text{coker}\,f_{\zg}$ is local.   Therefore, $\text{coker}\,f_{\zg}$ is indecomposable. 
\end{proof}



\medskip
\subsection{Syzygy $M_{\zg}$}\label{sect 5.6}

In light of the results we obtained in the previous sections, we can finally make the following definition.

\begin{definition}
Given a 2-diagonal $\zg\in\diag$ define $M_{\zg}:= \text{coker}\,f_{\zg}$. 
\end{definition}

Note that the definition of the map $f_{\zg}$ depends on the representative of the 2-diagonal, however by Proposition~\ref{homotopy} its cokernel does not.  This implies that $M_{\zg}$ is well-defined.  

Next, we obtain a number of important results that summarize the key properties of $M_{\zg}$.  

\begin{thm}\label{thm:omega}
If $\zg\in\diag$ then $\Omega \,M_{\zg}\cong M_{R(\zg)}$.  
\end{thm}

\begin{proof}
By definition of $M_{\zg}$ and Proposition~\ref{big-lemma} we have an exact sequence in $\text{mod}\,B$ as follows.

\[\xymatrix{
P_1(R(\zg))\ar[r]^{\bar{f}_{R(\zg)}} & P_1(\zg)\ar[r]^{f_{\zg}} & P_0(\zg) \ar[r] & M_{\zg}\ar[r] & 0
}\]

By Remark~\ref{rem:bar}(b) and the definition of $M_{R(\zg)}$ we have $\text{coker}\, \bar{f}_{R(\zg)} \cong \text{coker}\, {f}_{R(\zg)} \cong M_{R(\zg)}$.   Since modules $P_1(\zg), P_0(\zg)$ and $P_1(\zg), P_1(R(\zg))$ are nonzero modules that have no summands in common by Proposition~\ref{lem:cross}(a), we conclude that the sequence above is the beginning of a minimal projective resolution of $M_{\zg}$.  This implies that $\Omega \,M_{\zg}\cong M_{R(\zg)}$.  
\end{proof}

\begin{corollary}\label{cor:536}
Let $\cals$ be a checkerboard polygon of size $2N$, then $M_{\zg}$ is a nonzero indecomposable non-projective syzygy in $\textup{mod}\,B$ and 

\[M_{\zg}\cong 
\begin{cases}
\Omega^{N} M_{\zg} & \text{if }\,\zg \,\text{ is a diameter}\\
\Omega^{2N} M_{\zg} & \text{otherwise.}
\end{cases}\]
\end{corollary}

\begin{proof}
By definition of $R$, we have that $\zg$ equals $R^{2N}(\zg)$ or $R^{N}(\zg)$ if $\zg$ is a diameter of $\cals$.   Successive applications of Theorem~\ref{thm:omega} imply that 
\[\Omega^k M_{\zg}=\Omega^{k-1}\Omega\, M_{\zg}\cong \Omega^{k-1}M_{R(\zg)} \cong \dots \cong M_{R^k(\zg)}\]
for all integers $k\geq 0$.  Hence, $M_{\zg}$ is a non-projective syzygy which is also indecomposable by Proposition~\ref{prop:ind}.  Finally, $M_{\zg}$ is nonzero because $P_0(\zg), P_1(\zg)$ are nonzero projectives that have no summands in common.  
\end{proof}

The next proposition shows that radical lines in $\cals$ correspond to radicals of projective $B$-modules. 

\begin{prop}\label{prop:radicals correspond}
For a radical line $\rho(k)$ in $\cals$ the syzygy $M_{\rho(k)}$ is isomorphic to $\textup{rad}\,P(k)$. 
\end{prop}

\begin{proof}
By construction $P_0(\rho(k)), P_1(\rho(k))$ consist of all projective summands $P(i), P(j)$ such that in the quiver $Q$ there is an arrow $k\to i, j\to k$ respectively.  Then $Q(\rho(i))$ consists of $k$ and all vertices connected to $k$ by an arrow.  As in the proof of Proposition~\ref{homotopy} there are two possible crossing sequences for $\rho(i)$, and we have the following configuration of $i$'s and $j$'s.

\[\xymatrix@R=10pt@C=10pt{
&&i_{s-1}\ar[r] \ar@{.}[l]&j_{s-1}\ar[dll]\\
&k\ar[rr] \ar[ur] \ar[dr]  \ar@{.}[u] \ar@{.}[dl]&& i_s\ar[u]\ar[d]\\
&j_{s+1}\ar[u] \ar@{.}[l]&i_{s+1} \ar[r]\ar[l]& j_s \ar[ull]\\
}
\]

Here none of the steps are forward and all paths in 
\[f_{\rho(k)}= \begin{bsmallmatrix} \ddots \\  \cdots & 0 & i_{s-1}\to j_{s-2} & i_{s-1}\to j_{s-1} & 0 & \cdots & \cdots & \cdots \\
\cdots & \cdots & 0 & i_{s}\to j_{s-1} & i_s\to j_s & 0 & \cdots  & \cdots \\
\cdots & \cdots & \cdots & 0 & i_{s+1}\to j_s & i_{s+1}\to j_{s+1} & 0 & \cdots \\
&&&&& \ddots
\end{bsmallmatrix}\] 
are given by arrows. 

Now, we construct a minimal projective presentation of $\text{rad}\,P(k)$ and show that it is the same as $f_{\rho(k)}$.  First, we observe that the top of $\text{rad}\,P(k)$ consists of all simple modules $S(i)$ such that there is an arrow $k\to i$.  Hence, the projective cover of $\text{rad}\,P(k)$ is $P_0(\rho(k))$, and let $\pi: P_0(\rho(k)) \to \text{rad}\,P(k)$ denote the canonical surjection, where the top of $P(i_s)$, a summand of $P_0(\rho(k))$, maps to the top of $\text{rad}\,P(k)$ via the identity map.  

Next, we compute the projective cover of $\text{ker}\,\pi$.  First, suppose that the arrow $j_s\to k$ lies in two 3-cycles.  From the quiver, we see that the simple module $S(j_s)$ that lies in the top of $\text{rad}^2\,P(k)$ is covered twice by $\pi$ as exactly two summands $P({i_s}), P(i_{s+1})$ of $P_0(\rho(k))$ map to it.   If the arrow $j_s\to k$ lies in a single 3-cycle, then the path $k\leadsto j_s$ factors through exactly one of $i_s, i_{s+1}$.  In this case $\text{rad}\,P(k)$ is not supported at $j_s$, but the respective summand $P(i_s), P(i_{s+1})$ contains $S(j_s)$ in the top of its radical that maps to zero under $\pi$.  Therefore, in all cases we obtain that $P_1(\rho(k))$ is a summand of the projective cover of $\text{ker}\,\pi$.   Moreover, the kernel of $\pi$ is generated by all nonzero paths $i_s\leadsto x$ ending in some vertex $x$ such that  $k\to i_s\leadsto x=0$ or $k\to i_s\leadsto x = k\to i_t \leadsto x$ for some path $i_t\leadsto x$ where $P(i_t)\in\text{add}\,P_0(\rho(k))$.  If $k\to i_s\leadsto x=0$ then up to commutativity the first arrow in the path $i_s\leadsto x$ lies in the same 3-cycle as $k$, so $i_s\leadsto x$ factors through one of $j_s, j_{s-1}$.  If $k\to i_s\leadsto x = k\to i_t \leadsto x$ for some path $i_t\leadsto x$ then since $Q$ has no interior vertices, we must have $|s-t|=1$ and $i_s\leadsto x$ factors through one of $j_s, j_{s-1}$.  This shows that the top of $\text{ker}\,\pi$ equals $\bigoplus_s S(j_s)$.   Thus, the projective cover of $\text{ker}\,\pi$ is isomorphic to $P_1(\rho(k))$.   Moreover, we can take the associated composition $P_1(\rho(k))\to \text{ker}\,\pi \to P_0(\rho(k))$ to be given by all arrows $i_s\to j_s, i_s\to j_{s-1}$.  This gives precisely the same map as $f_{\rho(k)}$ and proves  that $M_{\rho(k)}\cong \textup{rad}\,P(k)$.
\end{proof}

\section{Morphisms of the syzygy category} \label{sect 6}
In this section, we study the morphisms in the stable syzygy category $\scmp\,B$ in terms of the morphisms in the category of 2-diagonals $\diag$. The main results of this section are the following. We define the morphisms on the degree 0 and 1 terms of the projective resolutions of the syzygies and prove in 
Theorem~\ref{lem 63}
that this construction induces a morphism on the syzygies. 

The irreducible morphisms in $\diag$ are given by the 2-pivots introduced in Definition \ref{def 2pivot}.
Our first goal is to associate a morphism in $\scmp \,B$ to each 2-pivot. We need a preparatory lemma.
\begin{lemma}\label{lem 62}
 Let $\zg,\zg'$ be 2-diagonals in $\cals$ such that $\zg'$ is obtained from $\zg$ by a 2-pivot. Suppose there exist vertices $i,i' $ in $Q$ such that 
 \begin{enumerate}
\item[(a)] $\zg$ crosses the radical line $\rho(i)$ and $\zg'$ crosses the radical line $\rho(i')$ and these crossings have the same degree,
\item[(b)] $\zg$ does not cross $\rho(i')$ and $\zg' $ does not cross $\rho(i)$.
\end{enumerate}
Then there is an arrow $i'\to i $ in $Q$.
\end{lemma}
\begin{proof}
We use the notation in Figure \ref{fig 2-pivots}. In particular the common endpoint of $\zg$ and $\zg'$ is the vertex $a$ and the other endpoints are $x$ and $z$. 
Without loss of generality, we may assume that $\zg$ is oriented from $a$ towards $x$. Then $\zg'$ is oriented from $a$ towards $z$. 

Since $\zg$ crosses $\rho(i)$, but $\zg'$ doesn't, it follows that one endpoint of $\rho(i)$ is $y$ or $z$ and the other lies to the left of $\zg$. Similarly, one endpoint of $\rho(i')$ is $x$ or $y$ and the other lies to the right of $\zg'$.

Suppose first that $\rho(i)$ and $\rho(i')$ both have endpoint $y$. Then both $\rho(i),\rho(i')$ are oriented from $y$ towards their other endpoint. Thus $\rho (i)$ crosses $\zg$ in degree 0, but $\rho(i')$ crosses $\zg'$ in degree 1. This implies that the crossings have different degrees, a contradiction to our assumption.

If $\rho(i)$ has endpoint $z$ and $\rho(i')$ has endpoint $x$ we also obtain this contradiction.

Now suppose $\rho(i)$ starts at vertex $y$ and $\rho(i')$ ends at vertex $x$, see Figure \ref{figlem 62}. Then $\rho(i)$ and $\rho(i')$ cross at an interior point $\za_0$ of the checkerboard pattern of $\cals$. This interior vertex corresponds to an arrow $\za$ that connects $i$ and $i'$ in $Q$.  Moreover, the 4  regions adjacent to $\za_0$
are either shaded or not, and the shading is determined by the orientation of the radical lines $\rho(i)$ and $\rho(i')$ according to Lemma \ref{lem orientation}. Namely, the segments of the radical lines form an oriented path around the shaded regions and they are alternating around the white regions. Thus the shading at $\za_0$ must be as in Figure \ref{figlem 62}.
Now, the direction of the arrow $\za$ is determined 
by going around the shaded regions in counterclockwise direction.
When we do this in either of the two shaded region, we encounter $\rho(i')$ first and then $\rho(i)$. Therefore the arrow $\za$ is oriented accordingly $\za\colon i'\to i$, and the proof is complete in this case. 

The only remaining case has $\rho(i)$ ending at $z$ and $\rho(i')$ ending at $y$, and its proof is  similar.
\end{proof}
\begin{figure}
\begin{center}
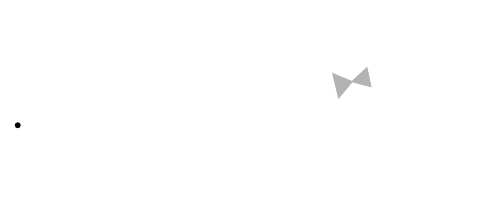
\caption{Proof or Lemma \ref{lem 62}}
\label{figlem 62}
\end{center}
\end{figure}

Next, we define maps $g_0, g_1$ from the summands of a projective resolution of a 2-diagonal to its 2-pivot.   

\begin{definition}\label{def:pivot-maps}
Let $\zg,\zg'$ be 2-diagonals in $\cals$ such that $\zg'$ is obtained from $\zg$ by a 2-pivot. Let $f_\zg\colon P_1\to P_0$ and $f_{\zg'}\colon P'_1\to P'_0$ be the morphisms associated to $\zg$ and $\zg'$ in Definition~\ref{def:map}.  We define morphisms $g_1\colon P_1\to P_1'$ and $g_0\colon P_0\to P_0'$ on the indecomposable summands of $P_1$ and $P_0$ as follows. 

If $P(i)$ is an indecomposable summand in both $P_0$ and $P_0'$ we let the component $P(i)\to P(i)$  of $g_0$ be the identity map and the components  $P(i)\to P(j) $  and  $P(j)\to P(i) $ be 
the zero map, for all $j\ne i$. 
If $P(i)$ is an indecomposable summand of $P_0$ but not of $P_0'$ and $P(i')$ is an indecomposable summand of $P_0'$ but not of $P_0$, we let the component $P(i)\to P(i')$ of $g_0$  be the map given by composing with the arrow $i'\to i$ given by Lemma \ref{lem 62}. 
We define the components $P(j)\to P(j') $ to be the zero map, for any other $j,j'$. 

The morphism $g_1\colon P_1\to P_1'$ is defined in the same way. 
\end{definition}

Given a representative of a 2-diagonal $\zg$ and its 2-pivot $\zg'$, recall Definition~\ref{def:compatible}(a) of compatibility of $\zg, \zg'$.  In particular, this means that the two crossing sequences for a compatible pair of arcs agree except at the very end.   

The next lemma gives a precise description of the matrices for the maps $g_0, g_1$. 

\begin{lemma}\label{lem:M}
Let $\zg, \zg'$ be compatible 2-diagonals in $\cals$ such that $\zg'$ is obtained from $\zg$ by a 2-pivot, and let $g_0, g_1$ be the maps given in Definition~\ref{def:pivot-maps}.  Then, up to reversing the direction of the crossing sequences, the matrices for $g_0, g_1$ are of the form $I_n$ or $\begin{bsmallmatrix} I_n & 0 \end{bsmallmatrix}$ or $\begin{bsmallmatrix} I_n \\ 0 \end{bsmallmatrix}$ or $\begin{bsmallmatrix} I_n & 0 \\ 0 & M\end{bsmallmatrix}$, where $I_n$ is the identity matrix with $n\geq 0$ and $M$ equals one of the following $\begin{bsmallmatrix} i\to h \end{bsmallmatrix}$, $\begin{bsmallmatrix} i\to h& i\to h' \end{bsmallmatrix}$, $\begin{bsmallmatrix} i\to h'\\ i'\to h \end{bsmallmatrix}$.  Moreover, the number of arrows in $g_0$ and $g_1$ combined is at most three. 
\end{lemma}

\begin{proof}
Construct the crossing sequences for $\zg, \zg'$ by starting at the common endpoint of the arcs and then moving towards the other endpoints.  Then the two crossing sequences for $\zg, \zg'$ start with a common subsequence of degree zero crossings $i_1, \dots, i_n$ in that order.  Moreover, we assume that this subsequence is maximal, meaning that $\zg,\zg'$ have different crossings in degree zero after $i_n$.  By definition of $g_0$ we conclude that its matrix contains a block $I_n$ in its top left corner.  If there are no other arrows in $g_0$, then its matrix is one of $I_n$, $\begin{bsmallmatrix} I_n & 0 \end{bsmallmatrix}$, or $\begin{bsmallmatrix} I_n \\ 0 \end{bsmallmatrix}$ and the lemma follows. 

For the remainder of this proof, we suppose that there is at least one arrow in $g_0$.  By definition of $g_0$ its matrix is of the form $\begin{bsmallmatrix} I_n & 0 \\ 0 & M\end{bsmallmatrix}$ for some matrix $M$, because any row and column with entry $1_{P(i_k)}$ for some $k$ has all other entries being zero.   

First, we claim that $g_0$ contains at most two arrows.   By definition, an arrow $i\to h$ in $g_0$ comes from a pair of crossing radical lines $\rho(i), \rho(h)$ in $\cals$ such that $\rho(h)$ starts at $y$ and crosses $\zg$ while $\rho(i)$ crosses $\zg'$ and ends in $x$, see Figure~\ref{fig:pivot-map} on the left.  There are at most two radical lines at every boundary vertex of $\cals$, by Lemma~\ref{lem 32}(c).  If there are also $\rho(h'), \rho(i')$ with endpoints $y, x$ and crossing $\zg, \zg'$ respectively, then we obtain at least four crossings, where $\rho(i), \rho(i')$ each cross both of $\rho(h),\rho(h')$.  Then $Q$ admits a subquiver that is given below on the left. 

\[\xymatrix@R=15pt@C=15pt{i\ar[r] \ar[d]& h &&&& k\ar[r] \ar[d] & p\ar[d]\ar[r] & i'\ar[d]\\
 h' & i'\ar[u]\ar[l] &&&& p'\ar[r] \ar@/^45pt/[urr]& i\ar[r] & h}
\]

This yields a contradiction, because there should be a sequence of 3-cycles in $Q$ between the arrows $i\to h$ and $i\to h'$, which means that $Q$ contains interior vertices.
This shows that either $\rho(i')$ or $\rho(h')$ or both are not in $\cals$, so $g_0$ contains an arrow $i\to h$ and at most one of $i\to h', i'\to h$.   This shows the claim that $g_0$ contains at most two arrows. 

Now we show that the total number of arrows in $g_0$ and $g_1$ is at most three.
Suppose on the contrary that both $g_0, g_1$ contain two arrows each.  An arrow $k\to p$ in $g_1$ comes from a pair of crossing radical lines $\rho(k), \rho(p)$ in $\cals$ such that $\rho(k)$ starts at $y$ and crosses $\zg'$ while $\rho(p)$ crosses $\zg$ and ends in $z$.   An arrow $i\to h$ in $g_0$ yields another pair of radical lines with endpoints $x, y$.  In particular, there are two radical lines $\rho(k), \rho(h)$ at vertex $y$, so there cannot be any additional radical lines at this vertex.  This implies that if $g_0, g_1$ contain two arrows each then there are two radical lines $i, i'$ ending in $x$ and crossing $\zg'$, and two radical lines $p, p'$ crossing $\zg$ and ending in $z$, see Figure~\ref{fig:pivot-map} on the left.  In this case we  obtain the subquiver of $Q$ given above on the right, which contains interior vertices.  This yields a contradiction, and shows that the total number of arrows in $g_0$ and $g_1$ is at most three.

\begin{figure}  
\centerline{\scalebox{.9}{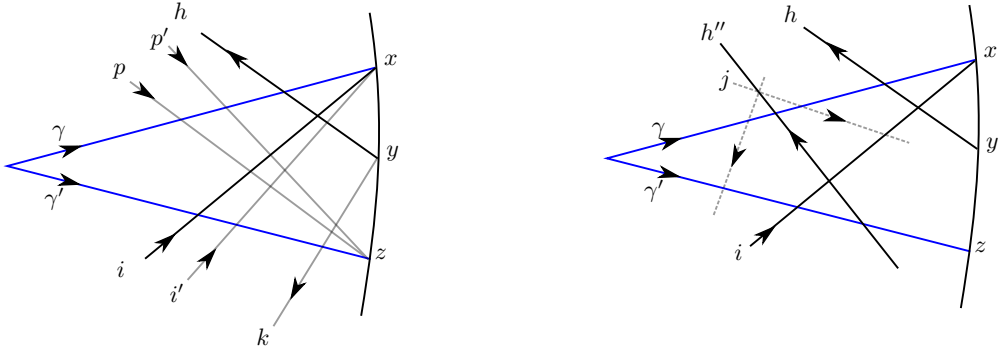}}
\caption{The proof of Lemma~\ref{lem:M}. Here the two dashed lines in the right figure illustrate the two possibilities for the arc $j$.}
\label{fig:pivot-map}
\end{figure}



Now we study the structure of $M$.  In the case when $M$ contains two arrows we label the radical lines so that $\zg$ crosses $h$ before $h'$ and similarly $\zg'$ crosses $i$ before $i'$. 

Now suppose that $M$ contains a row of zeros, which correspond to entries in $M$ that are paths starting at some fixed vertex $k$.  Then there is a 2-diagonal $k$ that crosses $\zg'$ in degree zero but it does not cross $\zg$.  Then $k$ crosses $\zg'$ and ends in $x$, so in particular it crosses $h$, as we assume $g_0$ contains at least one arrow $i\to h$.  But then, by definition of $g_0$, the matrix $M$ would contain an entry $k\to h$ and not a row of zeros as we assumed above.  This gives a contradiction.  We obtain a similar contradiction to $M$ containing a column of zeros, which shows that $M$ cannot contain a row or a column of zeros.   


Now we show that the top left entry of $M$ is an arrow $i\to h$.   By definition every entry of $M$ is either zero, $1_{P(h'')}$, or an arrow.   If the top left entry of $M$ is $1_{P(h'')}$, then we obtain a contradiction to $i_1, \dots, i_n$ being the maximal initial subsequence of the degree zero crossings that $\zg, \zg'$ have in common.  Hence, suppose that the top left entry of $M$ is zero.  By the above, $M$ does not contain any row or  column of zeros, so suppose that the first row of $M$ contains an arrow $i\to h$ and then the first column of $M$ must contain $1_{P(h'')}$ for some vertex $h''$.   Then $h''$ is a common crossing of $\zg, \zg'$ in degree zero such that $\zg'$ first crosses $i$ and then $h''$, while $\zg$ first crosses $h''$ and then $h$.  Thus, we are in the situation of Figure~\ref{fig:pivot-map} on the right.  Since $h''$ is not the last crossing of $\zg$ in degree zero, it follows that $(h'', j)$ is a crossing pair for $\zg$ for some vertex $j$.  If in addition $j$ crosses $\zg'$, then by definition of compatibility between $\zg, \zg'$ we conclude that $(h'', j)$ is also a crossing pair for $\zg'$.  Then $\zg'$ crosses $h'', j$ after crossing $i$.  In particular, $j$ and $i$ must cross, and we obtain arrows $j\to i\to h''$ and $j\to h''$, which is not a possible subquiver of $Q$.  Thus, we obtain a contradiction in the case when $j$ also crosses $\zg'$.  If $j$ does not cross $\zg'$, then it would end at vertex $z$.  But then we still obtain that $j$ must cross $i$ and we arrive at the same contradiction.  This shows that it is not possible for the first row of $M$ to contain an arrow $i\to h$ and for the first column to contain   $1_{P(h'')}$.  If $M$ contains an arrow $i\to h$ in the first column and $1_{P(h'')}$ in the first row then we also obtain a contradiction in a similar way.  This shows that the top left entry of $M$ is an arrow $i\to h$.  

Now suppose that $M$ contains an entry $1_{P(h'')}$.  Since $i\to h$ lies in the top left corner of $M$, it follows that $1_{P(h'')}$ cannot be in the first row or column of $M$.  Then $\zg, \zg'$ cross $h''$ after their crossings with $h, i$ respectively.  Then it follows that $h''$ crosses both $h$ and $i$, and in particular we obtain arrows $i\to h$ and $i\to h''\to h$ in the quiver.  This is not a possible configuration, so we obtain a contradiction.  This shows that the matrix $M$ cannot contain any entry of the form $1_{P(h'')}$.

Thus, we conclude that $M$ cannot contain a row or a column of zeros, or any entry of the form $1_{P(h'')}$, moreover we know that $M$ contains at most two arrows.  This implies that for the map $g_0$ the matrix $M$ must be of the form given in the statement of the lemma. 

The proof in the case of $g_1$ follows similarly.  
\end{proof}

The above lemma implies that in a given row or column of $g_i$ with $i=0,1$, there are at most two nonzero entries.  If $g_i$ contains two arrows in the same row, then let $g_i^r$ denote the map obtained from $g_i$ by changing the sign of the arrow in the last column of the matrix, and otherwise let $g_i^r=g_i$.  Similarly, if $g_i$ contains two arrows in the same column, then let $g_i^c$ denote the map obtained from $g_i$ by changing the sign of the arrow in the last row, and otherwise let $g_i^c=g_i$.

\begin{definition}\label{def pivot}
 The morphism $(g^r_0,g^c_1)$ defined above is called the \emph{pivot morphism} associated to the 2-pivot $\zg\mapsto \zg'$.
\end{definition}
\begin{thm}
 \label{lem 63} Let $\zg,\zg'$ be compatible 2-diagonals in $\cals$ such that $\zg'$ is obtained from $\zg$ by a 2-pivot and let $(g_0^r,g_1^c)$ be the morphism defined above. Then we have the following commutative diagram with exact rows
 \begin{equation}\label{eq 62}\xymatrix@C50pt{ P_1\ar[r]^{f_\zg} \ar[d]_{g_1^c}&P_0\ar[d]^{g_0^r}
\ar[r]^{\pi_\zg}&M_\zg\ar[r]\ar[d]^g&0 \\
P_1'\ar[r]^{f_{\zg'}}&P_0'\ar[r]^{\pi_{\zg'}}&M_{\zg'}\ar[r]&0 \\
}
 \end{equation}
where $g$ is the induced morphism on the cokernels.
%
 In particular, $g$ is a morphism of syzygies in $\cmp\,B$.
\end{thm}
\begin{proof} This is proved in section \ref{A2pivots} of the appendix. \end{proof}
\begin{definition}
 \label{def pivot morphism}
  \label{prop 65} Let $\zg,\zg'$ be 2-diagonals in $\cals$  and let $M_\zg,M_{\zg'}$ be the corresponding syzygies in $\scmp\,B$. 
If $\zg'$ is obtained from $\zg$ by a 2-pivot then the corresponding morphism $g\colon M_\zg\to M_{\zg'}$ in $\scmp\,B$ is called  {\em pivot morphism}.
\end{definition}

\section{Auslander-Reiten triangles of the syzygy category}\label{sect 7}
In this section, we give a combinatorial description of the Auslander-Reiten triangles in $\scmp\,B$ in terms of the 2-pivots in the category of 2-diagonals $\diag$. In subsection \ref{sect 7.2}, we show that the indecomposable objects in $\scmp\,B$ do not admit any nonzero nilpotent endomorphisms. As a consequence we see that the dimension of $\Ext^1$ between an indecomposable $M$ and its AR-translate is equal to one. In subsection~\ref{sect 6.2}, we show that the Auslander-Reiten triangles of  $\diag$ give rise to commutative diagrams in $\scmp \,B$. We use these results in subsection 
\ref{sect 7.4}, where we construct the short exact sequences in $\textup{mod}\,B$ that induce the AR-triangles  $\scmp\,B$. 
\medskip
\subsection{Nilpotent endomorphisms of $M_\zg$}\label{sect 7.2} 
 In this subsection, we show that the syzygy $M_\zg$ admits no nonzero nilpotent endomorphisms in $\scmp\,B$.

Let $\zg$ be a 2-diagonal in $\cals$ and $f_\zg\colon P_1(\zg)\to P_0(\zg)$ be the morphism defined in Definition~\ref{def:map}. 
Let $g\colon M_\zg\to M_\zg $ be a nilpotent endomorphism, and let $g_0\colon P_0(\zg)\to P_0(\zg)$, 
$g_1\colon P_1(\zg)\to P_1(\zg)$ be the induced endomorphisms of the projectives. Thus we have the following commutative diagram with exact rows in $\textup{mod}\,B$.
\begin{equation}
\label{diagram 72}
 \xymatrix{ P_1(\zg)\ar[r]^{f_\zg}\ar[d]_{g_1} &P_0(\zg)\ar[d]^{g_0}\ar[r]^\pi &M_\zg\ar[d]^g\ar[r]&0 \\ P_1(\zg)\ar[r]^{f_\zg}&P_0(\zg)\ar[r]^\pi&M_\zg\ar[r]&0}
\end{equation}

The main result of this subsection is Theorem~\ref{thm nilpotent endo} which states  that $g$ is zero in the stable category $\scmp\,B$. The proof of this theorem requires a detailed analysis of the entries of the matrix of $g_0$. We partition the matrix twice into column blocks and row blocks and then reduce $g_0$ successively by removing one block at a time. 

As a direct consequence, we show in Corollary~\ref{cor:ext} that in $\scmp\, B$ the dimension of the first extension group between $\zO^2\,M_\zg$ and $M_\zg$ is equal to 1.
The following is proved in the appendix as Theorem~\ref{Athm nilpotent endo}.
\begin{thm} 
 \label{thm nilpotent endo}
 Let $\zg$ be a 2-diagonal and $M_\zg$ the associated indecomposable syzygy over $B$. Then $M_\zg$ does not admit any nonzero nilpotent endomorphisms in $\scmp\,B$.
\end{thm}

Recall that $\scmp\,B$ is a triangulated category with inverse shift given by $\Omega$.  It has almost-split triangles where $\tau^{-1}$ is given by $\Omega^2$, so in particular   $\textup{dim Ext}^1_{\scmp\,B}(\Omega^2 M, M) \geq 1$.  Now we show that the dimension of this space is actually equal to 1. 

\begin{corollary}\label{cor:ext}
Let $\zg$ be a 2-diagonal in $\mathcal{S}$ and $M=M_{\zg}$ be the associated indecomposable syzygy in $\scmp\,B$. Then 
\[\textup{dim Ext}^1_{\scmp\,B}(\Omega^2 M, M) = 1.\]
\end{corollary}

\begin{proof}
Consider the following sequence of isomorphisms 
\[ \textup{Hom}_{\scmp\,B} (M,M) \cong \textup{Ext}^1_{\scmp\,B}(M, \Omega\,M) \cong D \textup{Ext}^2_{\scmp \,B}(\Omega\,M,M)\cong D \textup{Ext}^1_{\scmp\, B}(\Omega^2M,M) \]
where the first and the last step follow because $\Omega$ is the inverse shift in $\scmp\,B$ and the second step follows because $\scmp\,B$ is a 3-Calabi-Yau.  
By Theorem~\ref{thm nilpotent endo} there are no nilpotent endomorphisms of $M$ in $\scmp\,B$, and since $M$ is indecomposable, by Proposition~\ref{prop:ind}, \cite[Corollary 4.20]{S2} implies 
\[\text{dim}\, \textup{Hom}_{\scmp\,B} (M,M) = 1.\]
Then the result follows. 
\end{proof}

\medskip
\subsection{Mesh relations}
\label{sect 6.2} The main result 
of this subsection, Proposition~\ref{prop mesh}, shows that the Auslander-Reiten triangles of the category of 2-diagonals $\diag$ give rise to commutative diagrams in the stable category $\scmp \,B$. In subsection \ref{sect 7.4}, we will show that these commutative diagrams actually are Auslander-Reiten triangles in $\scmp \,B$.

Let $\zg$ be a 2-diagonal in $\cals$. We say that $\zg$ is \emph{short} if one of the two polygons obtained by cutting $\cals$ along $\zg$ is a quadrilateral.  In other words, the short 2-diagonals are those that are the start of a unique 2-pivot, and the non-short 2-diagonals are those that are the start of exactly two 2-pivots.

If $\zg$ is not short, let $\zg',\zg''$ denote the 2-diagonals that are obtained from $\zg$ by a 2-pivot and let $R^2\zg$ be the 2-diagonal obtained from $\zg$ by applying the rotation $R$ twice, see Figure~\ref{fig mesh}. Thus $R^2\zg$ can be obtained from $\zg$ in two ways by a sequence of two 2-pivots either passing through $\zg'$ or $\zg''$. In the category $\diag$ this corresponds to the AR triangle
\[ \zg\to\zg'\oplus\zg''\to R^2 \zg\to R\zg\] and a mesh in the AR quiver.
\begin{figure}
\begin{center}
\scalebox{0.8}{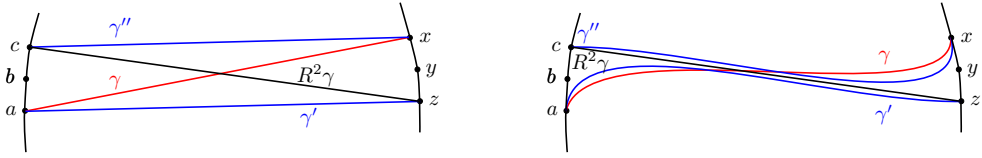}
\caption{The left picture shoes the 2-diagonals corresponding to the mesh $\zg\to\zg'\oplus\zg''\to R^2 \zg$ in the case where the 2-diagonal is not short. The right picture shows deformations of these 2-diagonals that are pairwise compatible, where we assume that none of $\zg,R\zg,R^2\zg$ is a radical line. }
\label{fig mesh}
\end{center}
\end{figure}

In the case where $\zg$ is short, only one of the two diagonals $\zg',\zg''$, say $\zg'$, exists. In this case the AR triangle is of the form 
\[ \zg\to\zg'\to R^2 \zg\to R\zg.\] 

We  denote the corresponding pivot morphisms $g',g'',h',h''$ in $\scmp\,B$ as illustrated in the following diagrams

\begin{equation}
 \label{eqmesh}
 \xymatrix{&M_{\zg'}\ar[rd]^{h'}\\M_\zg\ar[ru]^{g'}\ar[rd]_{g''}&&M_{R^2\zg}\\
&M_{\zg''}\ar[ru]_{h''}}
\qquad\qquad
\xymatrix{&M_{\zg'}\ar[rd]^{h'}\\M_\zg\ar[ru]^{g'}&&M_{R^2\zg}}
\end{equation}
where the diagram on the left corresponds to the case where the 2-diagonal $\zg$ is not short and the diagram on the right to the case where $\zg$ is short.

\begin{prop}
 \label{prop mesh}  Let $\zg$ be a 2-diagonal in $\cals$  such that $R\zg$ is not a radical line. 
With the notation of diagram \textup{(\ref{eqmesh})}, we have the following identities in $\scmp\,B$.
\begin{itemize}
\item [(a)]  If $\zg$ is not short  then $h'g'=h''g''$. 
\item[(b)] if $\zg $ is short then $h'g'=0$. 
\end{itemize}
\end{prop}
\begin{proof}
 This result is proved in section \ref{Asect 6.2} of the appendix.
\end{proof}

\medskip
\subsection{Auslander-Reiten triangles for $M_{\zg}$}\label{sect 7.4}
The main result of this subsection, Theorem~\ref{thm:ARtriangles} shows that the mesh relations given by 2-pivots and rotation $R^2$ in $\diag$ correspond to AR-triangles in $\scmp\,B$. 
To prove this result, we first show that in each rotation orbit in $\diag$ there is a mesh relation that induces a short exact sequence in $\cmp B$ between $M_\zg$ and $\zO^2 M_\zg$. Then using that $\dim \Ext^1_{\scmp\,B} (\zO^2 M_\zg,M_\zg)=1$, proved in Corollary~\ref{cor:ext}, we see that this sequence induces an AR-triangle in $\scmp\,B$. Finally, the proof of the main result then follows from the correspondence $R\sim \zO$.

Before stating the first result, recall the definition of $\zg', \zg''$ for a 2-diagonal $\zg$ which give rise to the corresponding sequence pivot morphisms $M_{\zg}\to M_{\zg'}\oplus M_{\zg''} \to M_{R^2\zg}$, see diagram~\eqref{eqmesh}.    Let $\{R^i \zg\}$ denote the orbit of a 2-diagonal $\zg$ under the action of $R$.

\begin{prop}\label{prop:AR}
Suppose $\zg$ is a 2-diagonal in $\mathcal{S}$ such that $R\zg$ is not a radical line.  
\begin{itemize}
\item[(a)] If $\zg$ 
is a radical line
$\rho(i)$ is not short then there exists a short exact sequence in $\textup{mod}\,B$ 
\[0\to M_{\zg}\to M_{\zg'}\oplus M_{\zg''} \oplus P(i) \to M_{R^2\zg}\to 0.\]
\item[(b)] If $\zg$
is a radical line 
$\rho(i)$ is short then there exists a short exact sequence in $\textup{mod}\,B$ 
\[0\to M_{\zg}\to M_{\zg'} \oplus P(i) \to M_{R^2\zg}\to 0.\]
\item[(c)] If $\{R^i \zg\}$ does not contain any radical lines then $\zg$ is not short and there exists a short exact sequence in $\textup{mod}\,B$ 
\[0\to M_{\zg}\to M_{\zg'}\oplus M_{\zg''} \to M_{R^2\zg}\to 0.\]
\end{itemize}
\end{prop}

\begin{proof}
 For a proof see section \ref{Asect 7.4} of the appendix.
\end{proof}

  The next lemma says that except for the case of the hexagon $\{R^i \zg\}$ cannot consist entirely of radical lines.

\begin{lemma}\label{lem:orbit}
Let $\mathcal{S}$ be a checkerboard polygon of size at least 8, and let $\zg$ be a 2-diagonal in $\cals$ that is not short.   For every orbit $\{R^i \zg\}$ there exists some $\zg'\in \{R^i \zg\}$ such that $R(\zg')$ is not a radical line in $\mathcal{S}$. 
\end{lemma}

\begin{proof}
Suppose on the contrary that every 2-diagonal in $\{R^i \zg\}$ is a radical line in $\mathcal{S}$.  Then in particular, $\zg, R\zg, R^2\zg, R^3\zg$ are four distinct radical lines in $\mathcal{S}$ which we can label $\rho(1), \rho(2), \rho(3), \rho(4)$ respectively.  Each pair of these arcs intersects, because $\zg$ is not short, and we obtain the following subquiver. 
\[\xymatrix@C=15pt@R=15pt{1\ar[r] \ar@/^8pt/[rr]& 2 \ar[d] & 4\ar[l]\\ & 3\ar[ur]\ar[ul]}
\]
This contradicts the assumptions on the quiver that every vertex is boundary, see Lemma~\ref{lem 36}.
\end{proof}

The main theorem of this subsection says that the exact sequences in $\text{mod}\,B$ discussed above give rise to Auslander-Reiten triangles in $\scmp\,B$. 

\begin{thm}\label{thm:ARtriangles}
Let $\zg$ be a 2-diagonal in $\mathcal{S}$.  
\begin{itemize}
\item[(a)] If $\zg$ is not short then  
\[M_{\zg}\to M_{\zg'}\oplus M_{\zg''} \to M_{R^2\zg}\to M_{\zg}[1]\]
is an Auslander-Reiten triangle in $\scmp\,B$.
\item[(b)] If $\zg$ is short then 
\[M_{\zg}\to M_{\zg''} \to M_{R^2\zg}\to M_{\zg}[1]\]
is an Auslander-Reiten triangle in $\scmp\,B$.
\end{itemize}
\end{thm}

\begin{proof}
First suppose that $\zg$ is not short.  Moreover, suppose that $\zg=\rho(i)$ such that $R(\zg)$ is not a radical line or $\{R^i \zg\}$ does not contain any radical lines, then we obtain a short exact sequence in $\text{CMP}\,B$ as in Proposition~\ref{prop:AR} (a) or (c) respectively. This sequence reduces to a triangle 
\begin{equation}\label{eq1} M_{\zg}\to M_{\zg'}\oplus M_{\zg''} \to M_{R^2\zg}\to M_{\zg}[1]
\end{equation}
by passing to $\scmp \,B$.  If the connecting map $M_{R^2\zg}\to M_{\zg}[1]$ where to be zero then taking its cone and completing it to a triangle we obtain the following sequence.
 \[M_{\zg}\to M_{\zg'}\oplus M_{\zg''} \to M_{R^2\zg}\xrightarrow{0} M_{\zg}[1] \to M_{R^2\zg}[1]\oplus M_{\zg}[1]\] 
Then $( M_{\zg'}\oplus M_{\zg''})[1] \cong M_{R^2\zg}[1]\oplus M_{\zg}[1]$ which is a contradiction, since all of these modules are distinct.  Therefore the connecting map $M_{R^2\zg}\to M_{\zg}[1]$ is nonzero, which means that the triangle in (\ref{eq1}) gives a nonzero element of $\text{Ext}^1_{\scmp \,B}(M_{R^2\zg}, M_{\zg})$.   By Theorem~\ref{thm:omega}, we have $\Omega^2 M_{\zg}\cong M_{R^2\zg}$, so Corollary~\ref{cor:ext} implies that the triangle in (\ref{eq1}) generates $\text{Ext}^1_{\scmp \,B}(M_{R^2\zg}, M_{\zg})$. Since $\text{CMP}\,B$ is a triangulated category with inverse Auslander-Reiten translation given by $\Omega^2$ it follows that this triangle is an Auslander-Reiten triangle in $\scmp \,B$.  Hence part (a) holds whenever $\zg=\rho(i)$ such that $R \zg$ is not a radical line or whenever $\{R^i \zg\}$ does not contain any radical lines.   

Recall that the rotation $R$ corresponds to $\Omega$, which is the same as the inverse shift in $\scmp \,B$.   Then an Auslander-Reiten triangle starting in $M_{R^i\zg}$ is obtained from the Auslander-Reiten triangle starting with $M_{\zg}$ given in (\ref{eq1}) by applying $\Omega^i$.  This completes the proof of part (a), since every 2-diagonal $\zg'$ in $\mathcal{S}$ either lies in the same orbit under $R$ as some radical line $\rho(i)$, and by Lemma~\ref{lem:orbit} we can always choose $\rho(i)$ whose rotation is not a radical line, or its orbit $\{R^i \zg'\}$ does not contain any radical lines.  

The proof of part (b) follows from in the same way as above starting from a short exact sequence in Proposition~\ref{prop:AR}(b).   
\end{proof}

\section{Main Result}\label{sect main}
We are now ready for the main result of the paper. It proves Conjecture \ref{main conj} in a special case.

\begin{thm}
 \label{main thm}
 Let $B$ be a dimer tree algebra given by a quiver $Q$ with potential that satisfies Definition~\ref{def Q} and such that every chordless cycle has length three. Let $\cals$ be the polygon constructed in subsection \ref{sect 3.1} and, for every 2-diagonal $\zg$ in $\cals$, denote by  $f_\zg$ the morphism defined in subsection \ref{sect 5.1}. Then the mapping $\zg\mapsto \coker f_\zg$ induces an equivalence of categories
 \[F\colon \diag \to \scmp \,B\]
such that
\begin{itemize}
\item [(a)]
the radical line $\rho(i)$ corresponds to the radical of the indecomposable projective $P(i)$, for all $i\in Q_0$,
\[ F(\rho(i))=\rad P(i).\]
\item [(b)] The clockwise rotation $R$ of $\cals$ corresponds to the inverse shift $\zO$  and the counterclockwise rotation $R^{-1}$ to the shift $\zO^{-1}$ in $\scmp \,B$,
\[ F(R(\zg))= \zO\,F(\zg) \qquad \textup{and}\qquad  F(R^{-1}(\zg))= \zO^{-1}\,F(\zg).\]

\item [(c)] The square of the rotations  corresponds to the   Auslander-Reiten translations in $\scmp \,B$, 
\[ F(R^2(\zg))= \tau^{-1}\,F(\zg) \qquad \textup{and}\qquad 
F(R^{-2}(\zg))= \tau^{}\,F(\zg).\]

\item [(d)] $F$ induces a bijection between the 2-pivots in $\diag$ and  the irreducible morphisms between indecomposables in $\scmp\, B$.

\item [(e)]  $F$ induces an isomorphism of Auslander-Reiten quivers
\[F\colon \zG_\diag \to\zG_{\scmp\,B}.\]

\end{itemize}
\end{thm}

\begin{proof}

 On indecomposable objects, the functor $F$ is defined by $F(\zg)=\coker f_\zg$. It is well-defined on objects because the cokernel of $f_\zg$ is a syzygy by Proposition \ref{big-lemma} and it is indecomposable by Proposition \ref{prop:ind}. 
The morphisms in $\diag$ are given by compositions of 2-pivots modulo mesh relations. The functor $F$ is defined on a 2-pivot by mapping it to the pivot morphism introduced in Definition \ref{def pivot}. On compositions of 2-pivots, $F$ is defined as the composition of the images. This is well-defined, because, as shown in subsection \ref{sect 6.2}, $F$ respects the mesh relations in $\diag$. Moreover, the dimension of $\Hom_{\diag}(\zg,\zd)$ is either 0 or 1 for any pair of 2-diagonals $\zg,\zd$, and this shows that $F$ is faithful.

The case where the polygon $\cals$ has 6 sides is illustrated in Example \ref{ex A3}. In this case, the theorem follows by inspection. Indeed, the category $\diag$ has three indecomposable objects, the three 2-diagonals in $\cals$, and the category $\scmp\,B$ also has three indecomposables, the three simple modules.  The Auslander-Reiten quiver of each category consists of three vertices and no arrows. 

Suppose now that $\cals$ has at least 8 sides.

Part (a) of the theorem is proved in Proposition \ref{prop:radicals correspond} and part (b) in Theorem \ref{thm:omega}. Part~(c) follows from (b) and the fact that $\Omega^2=\tau^{-1}$ in $\scmp\,B$. 

The Auslander-Reiten quiver of $\diag$ consists of a single connected component which is finite. 
Theorem~\ref{thm:ARtriangles} shows that the functor $F$ maps the Auslander-Reiten triangles in $\diag $ to Auslander-Reiten triangles in $\scmp \,B$, and 2-pivots in $\diag $ to irreducible morphisms in $\scmp\,B$. Therefore, the image of the Auslander-Reiten quiver of $\diag$ is a connected component $\zG$ of the Auslander-Reiten quiver of $\scmp\,B$. 
We will show that this is the only component of the Auslander-Reiten quiver of $\scmp\,B$. Suppose there is another component $\zG'$. Suppose first that there exists a nonzero morphism $f\colon M\to M'$ with $M\in \zG, M'\in \zG'$ indecomposable. Since $f$ is not irreducible it must factor through one of the irreducible morphisms (hence an arrow in $\zG$) $g_1\colon M\to M_1$ starting at $M$, thus $f=f_1 g_1$  for some morphism $f_1\colon M_1\to M'$. Repeating the argument with $f_1$ we get a factorization $f=f_2g_2g_1$, where $g_2g_1$ is the composition of two arrows in $\zG$. Continuing this way we obtain arbitrary long factorizations
$f=f_tg_t\cdots g_2g_1$, where $g_t\cdots g_2g_1$ is a path in $\zG$. However, $\zG$ is a finite component and  the dimension of $\Hom(M,N)$ between two indecomposable objects $M,N$ in $\zG$ is at most 1. Thus any composition of a large enough number of arrows in $\zG$ is zero. 
This implies that $f=0$.

Therefore there is no nonzero morphism from $\zG$ to $\zG'$.  Let $M'$ be an indecomposable object in $\zG'$ and let $f\colon P\to M'$ be its projective cover. Let $P(i)$ be an indecomposable summand of $P$. 
Then $\rad P(i)$ is the image of the radical line $\rho(i)$ under our functor $F$. In particular, $\rad P(i)$ lies in $\zG$.  
Now, since $P(i)$ is projective-injective in $\cmp\,B$,  the morphism $f$ factors through $\tau^{-1}\rad P(i)$. Thus there exists a nonzero morphism $\tau^{-1}\rad P(i)\to M'$ from and object in $\zG$ to an object in $\zG'$, a contradiction.

This shows that the Auslander-Reiten quiver of $\scmp \,B$ consists of a single finite connected component that, under  the functor $F$, is isomorphic to the Auslander-Reiten quiver of $\diag$. This shows part (e) of the theorem and proves that the functor $F$ is an equivalence.
 \end{proof}

\section{Consequences of the main result}\label{sect consequences}
In this section, we give some immediate corollaries of Theorem \ref{main thm}. 
Throughout the section, we let $B$ be a dimer tree algebra for which all chordless cycles are of length three.

\medskip
\subsection{Rigidity and $\tau$-rigidity of syzygies}
Recall that a $B$-module $M$ is said to be \emph{rigid} if $\Ext_B^1(M,M)=0$.  
\begin{corollary}
 \label{cor rigid}
 The  indecomposable syzygies over $B$ are rigid $B$-modules.
\end{corollary}
\begin{proof}
Let $M$ be an indecomposable syzygy over $B$. Under the equivalence of categories in Theorem~\ref{main thm}, $M$ corresponds to a 2-diagonal in $\diag$ which, because of Lemma \ref{lem rigid}, is a rigid object in $\diag$. Thus $M$ is rigid in $\scmp\,B$. Now the result follows from Corollary~\ref{cor 82}.
\end{proof}

\begin{corollary}
 \label{cor ext}
 Let $M,N$ be indecomposable syzygies over $B$. Then the dimension of $\Ext^1_B(M,N)\oplus \Ext^1_B(N,M)$ is equal to the number of crossing points between the corresponding 2-diagonals. In particular, the dimension is either 1 or 0.
\end{corollary}
\begin{proof}
 This follows  from the equivalence of categories of Baur and Marsh in Theorem~\ref{thm BM} together with Theorem~\ref{main thm} and Proposition~\ref{prop extclosed}.
\end{proof}

The question of $\tau$-rigidity is much more subtle. This is because the Auslander-Reiten translation $\tau_{\scmp\,B}$ in the syzygy category is different from the Auslander-Reiten translation $\tau_B$ in the module category $\textup{mod}\,B$.

 Recall that a $B$-module is said to be $\tau$-rigid if $\Hom_B(M,\tau_B M)=0$, see \cite{AIR}. Since $B$ is a 2CY-tilted algebra, there is a bijection between indecomposable $\tau$-rigid $B$-modules and indecomposable rigid objects in the cluster category of $B$, and hence with the cluster variables in the corresponding cluster algebra.
 
  In general, the indecomposable syzygies over a 2CY-tilted algebra are not $\tau$-rigid, see the example below. For the 2CY-tilted algebras $B$ considered here however, we conjecture the following.

\begin{conjecture} Let $B$ be a 2-Calabi-Yau tilted algebra whose quiver satisfies Definition~\ref{def Q}. Then
 the indecomposable syzygies are $\tau$-rigid in $\textup{mod}\,B$. 
\end{conjecture}
If the conjecture holds then the 2-triangulations of the checkerboard polygon $\cals$ correspond to (non-maximal) sets of compatible cluster variables in the cluster algebra of $Q$. We think it would be interesting to study the completions of these partial clusters in the cluster algebra as well as in $\textup{mod}\,B$.

We give an example of a 2CY-tilted algebra that does not satisfy Definition~\ref{def Q} for which the conjecture fails. 
\begin{example} 
 Let $B$ be the 2-Calabi-Yau tilted algebra of type $\widetilde{A}_{2,1}$ with quiver
\[\xymatrix@R10pt@C40pt{&2\ar[rd]^\zb\\1\ar[ru] ^\za&&3\ar@<2pt>[ll]^\zs\ar@<-2pt>[ll]_\zg
}\]
bound by the relations $\za\zb=\zb\zg=\zg\za=0$. Then  $\rad P(2)=\begin{smallmatrix}3\\1\\2\end{smallmatrix}$ is an indecomposable syzygy which has a nonzero morphism into $\tau\, \rad P(2) =
\begin{smallmatrix}
 2\\3\\1
\end{smallmatrix}$.
Note that the algebra $B$ does not satisfy the condition (Q1) because the quiver has no chordless cycles. On the other hand, the syzygy category of $B$ is finite.
\end{example}

\medskip
\subsection{Syzygies and cosyzygies}
Recall from Section \ref{sect CM} that the Auslander-Reiten translations induce equivalences of categories
\begin{equation}
\tau\colon \scmp\,B \to \underline{\textup{CMI}}\,B
\qquad 
\tau^{-1}\colon
\underline{\textup{CMI}}\,B
\to \scmp\,B.
\end{equation}

 Thus it follows immediately from Theorem \ref{main thm} that $\underline{\textup{CMI}}\,B$ is also equivalent to the category of 2-diagonals $\diag$, 
 Thus every 2-diagonal $\zg$ defines a morphism $f^*_\zg\colon I_0(\zg)\to I_1(\zg)$ between injective modules such that $\ker f^*_\zg\in\underline{\textup{CMI}}\,B$ is a cosyzygy. 
In particular, the indecomposable projective $P(j)$ is a direct summand of $P_1(\zg)$ if and only if the indecomposable injective $I(j)$ is a direct summand of $I_0(\zg)$, while 
the indecomposable projective $P(i)$ is a direct summand of $P_0(\zg)$ if and only if the indecomposable injective $I(i)$ is a direct summand of $I_1(\zg)$. Therefore the Nakayama functor $\nu=D\Hom_B(-,B)$ gives the following commutative diagram with exact rows
\[\xymatrix{&&P_1(\zg)\ar[r]^{f_\zg}\ar[d]^\nu&P_0(\zg)\ar[r]\ar[d]^\nu&M_\zg\ar[r]&0\\
0\ar[r]&\tau_B M_\zg\ar[r]&I_0(\zg)\ar[r]^{f_\zg^*} & I_1(\zg)
}
\]
It follows that the syzygy $M_\zg$ in $\scmp\,B$ and its Auslander-Reiten translation $\tau_B M_\zg$ in $\underline{\textup{CMI}}\,B$ are both represented by the same 2-diagonal $\zg$ in $\diag$. We have proved the following result.
\begin{corollary}
 \label{cor cmpcmi} 
The following diagram commutes. 
\[\xymatrix@C40pt{
\scmp\,B\ar[r]^{\tau}&\underline{\textup{CMI}}\,B\\
\diag\ar[u]^{\textup{cok}\,f_\zg} \ar[r]^{1}
&\diag \ar[u]_{\textup{ker}\,f_\zg^*}  
}
\]
\end{corollary}

\medskip
\section{Examples}

\begin{figure}[h]
\begin{center}
\scalebox{0.6}{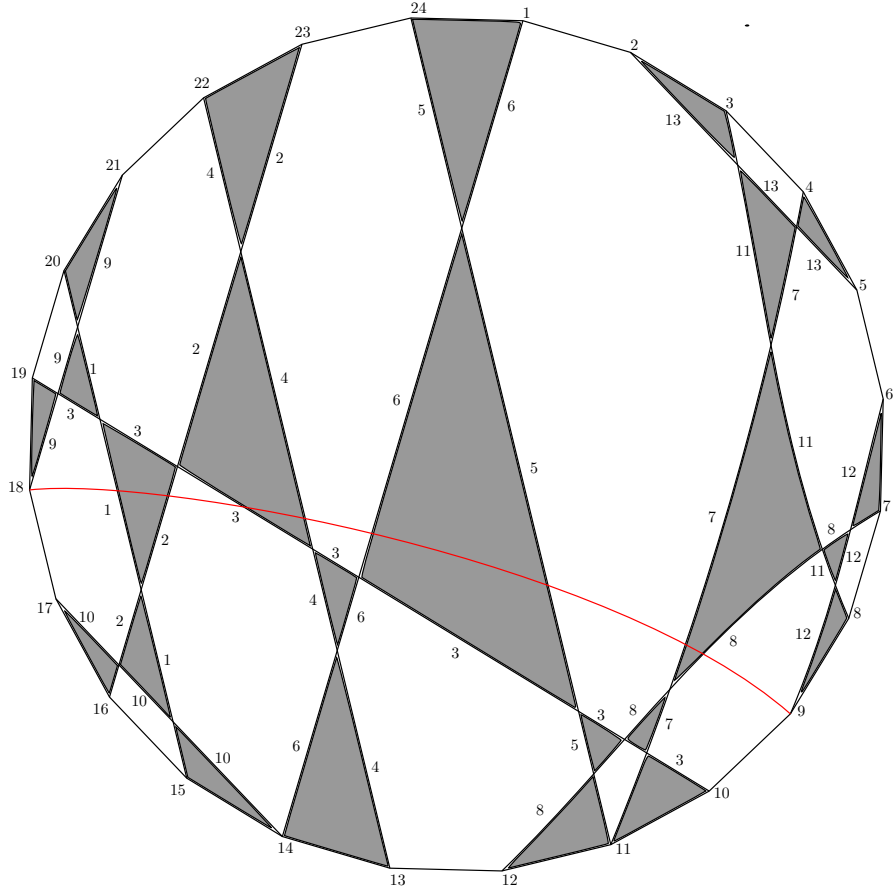}
\caption{Checkerboard polygon for the quiver in Example \ref{ex 11.1}.}
\label{fig ex 11.1}
\end{center}
\end{figure}
In this section we give several examples. 
\begin{example}
 \label{ex 11.1}
Let $Q$ be the quiver
\[\xymatrix{
&&\dreizehn\ar[r]&\elf\ar[r]\ar[ld]&\zwoelf\ar[ld]
\\
&\neun\ar[rd]&\sieben\ar[u]\ar[r]&\acht\ar[u]\ar[ld]
\\
&\eins\ar[u]\ar[d]&\drei \ar[l]\ar[u]\ar[r]\ar[d]&\fuenf\ar[u]\ar[d]
\\
10\ar[ru]&2\ar[l]\ar[ru]&4\ar[l]\ar[r]&6\ar[lu]
}
\]
The corresponding checkerboard polygon has 24 vertices and is depicted in Figure \ref{fig ex 11.1}.
The 2-diagonal $\zg$ that runs from the boundary vertex 18 to the boundary vertex 9 is drawn in red. Its crossing sequence is
$(\zwei,\eins),(\drei,\vier),(\sechs,\fuenf),(\acht,\sieben)$. The corresponding syzygy $M_\zg$ is the cokernel of the associated morphism 
\[f_\zg\colon P(\eins)\oplus P(\vier)\oplus P(\fuenf)\oplus P(\sieben) \longrightarrow P(\zwei)\oplus P(\drei)\oplus P(\sechs)\oplus P(\acht).\]
The matrix of $f_\zg$ is 
\[ 
\begin{bmatrix}
 \zwei\leadsto \eins & 0&0&0
 \\
 \drei\to\eins&\drei\to \vier &\drei\to \fuenf& \drei\to \sieben
 \\
 0&0&\sechs\leadsto\fuenf&\sechs\leadsto\sieben
 \\
 0&0&0&\acht\leadsto\sieben
\end{bmatrix}
\]
The rotation $R(\zg)$ is the 2-diagonal from boundary vertex 10 to boundary vertex 19. This is the radical line $\rho(3)$. The corresponding syzygy is the radical of $P(3)$ which is also given by the cokernel of the morphism
\[f_{R(\zg)}\colon P(\neun)\oplus P(\zwei)\oplus P(\sechs)\oplus P(\acht) \longrightarrow 
 P(\eins)\oplus P(\vier)\oplus P(\fuenf)\oplus P(\sieben)
\]
whose matrix is equal to 
\[
\begin{bmatrix}
 \eins\to \neun &\eins\to \zwei&0&0
 \\
 0&\vier\to \zwei&\vier\to\sechs&0
 \\
 0&0&\fuenf\to\sechs&\fuenf\to\acht
 \\
 0&0&0&\sieben\to \acht
\end{bmatrix} 
\]
We have the following projective resolution.
\[\xymatrix{\cdots\ar[r]&P_2(\zg)\ar[r]^{\overline{f}_{R(\zg)}}&P_1(\zg)\ar[r]^{f_\zg}&P_0(\zg)\ar[r]&M_\zg\ar[r]&0}\]
\end{example}


\begin{example}\label{exbysize} In this example we classify the quivers whose checkerboard polygon has 6,8 or 10 vertices. The quivers are listed below and the polygons are given in the same order in Figure \ref{fig ex size}. 
\bigskip

\underline{2 Hexagons} \\
$\xymatrix{1\ar[d]\\2\ar[r]&3\ar[lu]} \qquad
\xymatrix{1\ar[r] \ar[d]&3\ar[d]\\2\ar[r]&4\ar[lu]}$
\\
\bigskip

\hrule

\bigskip 

\underline{5 Octagons}
\\
$\xymatrix{1\ar[r]&2\ar[d]\\4\ar[u]&3\ar[l]}\qquad
\xymatrix{1\ar[r]&2\ar[d]\\4\ar[u]&3\ar[l]\ar[r]&5\ar[lu]}\qquad
\xymatrix{1\ar[r] \ar[d]&2\ar[d]\\3\ar[r]&4\ar[lu]\ar[r]&5\ar[lu]}
$\\
\\
$\xymatrix{&1\ar[ld]\ar[r]&2\ar[d]\\
3\ar[r]&4\ar[u]&5\ar[l]\ar[r]&6\ar[lu]}
\qquad
\xymatrix{1\ar[r]&2\ar[d]&3\ar@{<-}[l]\\
4\ar@{<-}[u]&5\ar@{<-}[l]\ar[lu]\ar[r]&6\ar@{<-}[u]\ar[ul]}
$
\\

\bigskip

\hrule

\bigskip

\underline{17 Decagons}
\\
\scalebox{1}{\begin{tabular}
{cccc} 
$ \xymatrix@R15pt@C15pt{1\ar[r]&2\ar[rd]\\3\ar[u]&4\ar[l]&5\ar[l]}
$
&
$ \xymatrix@R15pt@C15pt{1\ar[r]&2\ar@{<-}[r]\ar[rd] &3\\4\ar[u]&5\ar[l]\ar[r]&6\ar[u]}$
&
$
\xymatrix@R15pt@C15pt{1\ar[r]&2\ar[d]&3\ar[l]\\
4\ar[u]&5\ar[l]\ar[r]&6\ar[u]}
$
&$
\xymatrix@R15pt@C15pt{1\ar[r]&2\ar[d]&3\ar[l]\\
4\ar@{<-}[u]&5\ar@{<-}[l]\ar[lu]\ar[r]&6\ar[u]}
$
\\
\\
$
\xymatrix@R15pt@C15pt{1\ar[d]\ar[r]&2\ar[r]&3\ar[ld]\\
4\ar[r]&5\ar[r]\ar[lu]&6\ar[u]}
$
&
$
\xymatrix@R15pt@C15pt{1\ar[r]&2\ar[d]&3\ar[l]\\
4\ar@{<-}[u]&5\ar@{<-}[l]\ar[lu]\ar[ru]\ar@{<-}[r]&6\ar@{<-}[u]}
$
&$
\xymatrix@R15pt@C12pt{&1\ar[r]&2\ar[rd]&3\ar[l]\\
4\ar[r]\ar@{<-}[ru]&5\ar[u]&6\ar[l]\ar@{<-}[r]&7\ar[u]&
}
$
&
$
\xymatrix@R15pt@C12pt{1\ar[r]&2\ar[d]&3\ar[dr]\ar[l]\\
4\ar[u]&5\ar[l]\ar[r]&6\ar[u]&7\ar[l]}
$

\\
\\
$
\xymatrix@R15pt@C12pt{1\ar[r]&2\ar[r]\ar[d]&3\ar[d]\\
4\ar[u]&5\ar[l]\ar[r]&6\ar[lu]\ar[r]&7\ar[lu]
}
$&$
\xymatrix@R15pt@C12pt{&1\ar[r]\ar[ld]&2\ar[r]\ar[d]&3\ar[d]\\
4\ar[r]&5\ar[u]&6\ar[l]\ar[r]&7\ar[lu]
}
$
&
$
\xymatrix@R15pt@C12pt{1\ar[d]&2\ar[l]\ar@{<-}[r]\ar[dr]&3\ar@{<-}[d]\\
4\ar[ru]&5\ar[l]&6\ar[l]\ar@{<-}[r]&7\ar@{<-}[lu]} 
$
&
$\xymatrix@R15pt@C12pt{1\ar[r]&2\ar[d]&3\ar[l]\\
4\ar@{<-}[u]&5\ar@{<-}[l]\ar[lu]\ar[ru]\ar@{<-}[r]&6\ar@{<-}[u]\ar[r]&7\ar[lu]}
$
\\
\\
$
\xymatrix@R15pt@C12pt{1\ar[r]&2\ar[d]&3\ar@{<-}[l]\\
4\ar@{<-}[u]&5\ar@{<-}[l]\ar[lu]\ar[r]&6\ar[lu]\ar@{<-}[u]\ar[r]&7\ar[lu]}
$&$
\xymatrix@R15pt@C12pt{&1\ar[ld]\ar[r]&2\ar[d]&3\ar[dr]\ar[l]\\
4\ar[r]&5\ar[u]&6\ar[l]\ar[r]&7\ar[u]&8\ar[l]}
$&
$
\xymatrix@R15pt@C12pt{&1\ar[ld]\ar[r]&2\ar[d]&3\ar@{<-}[dr]\ar[l]\\
4\ar[r]&5\ar[u]&6\ar[l]\ar[r]&7\ar[lu]\ar@{<-}[u]&8\ar[l]}
$
&
$
\xymatrix@R15pt@C12pt{
&1\ar[r]\ar[d]&2\ar[d]&3\ar[l]\ar[d]\\
4\ar[ru]&5\ar[l]\ar[r]&6\ar[lu]\ar[ru]&7\ar[l]\ar[r]&8\ar[lu]}
$
\\
\\
$
\xymatrix@R15pt@C12pt{1\ar[r]&2\ar[d]&3\ar@{<-}[l]\ar[r]
&4\ar[d]\\
5\ar@{<-}[u]&6\ar@{<-}[l]\ar[lu]\ar[r]&7\ar[lu]\ar@{<-}[u]\ar[r]&8\ar[lu]}
$&
\end{tabular}
}

\begin{figure}
\begin{center}
\scalebox{0.6}{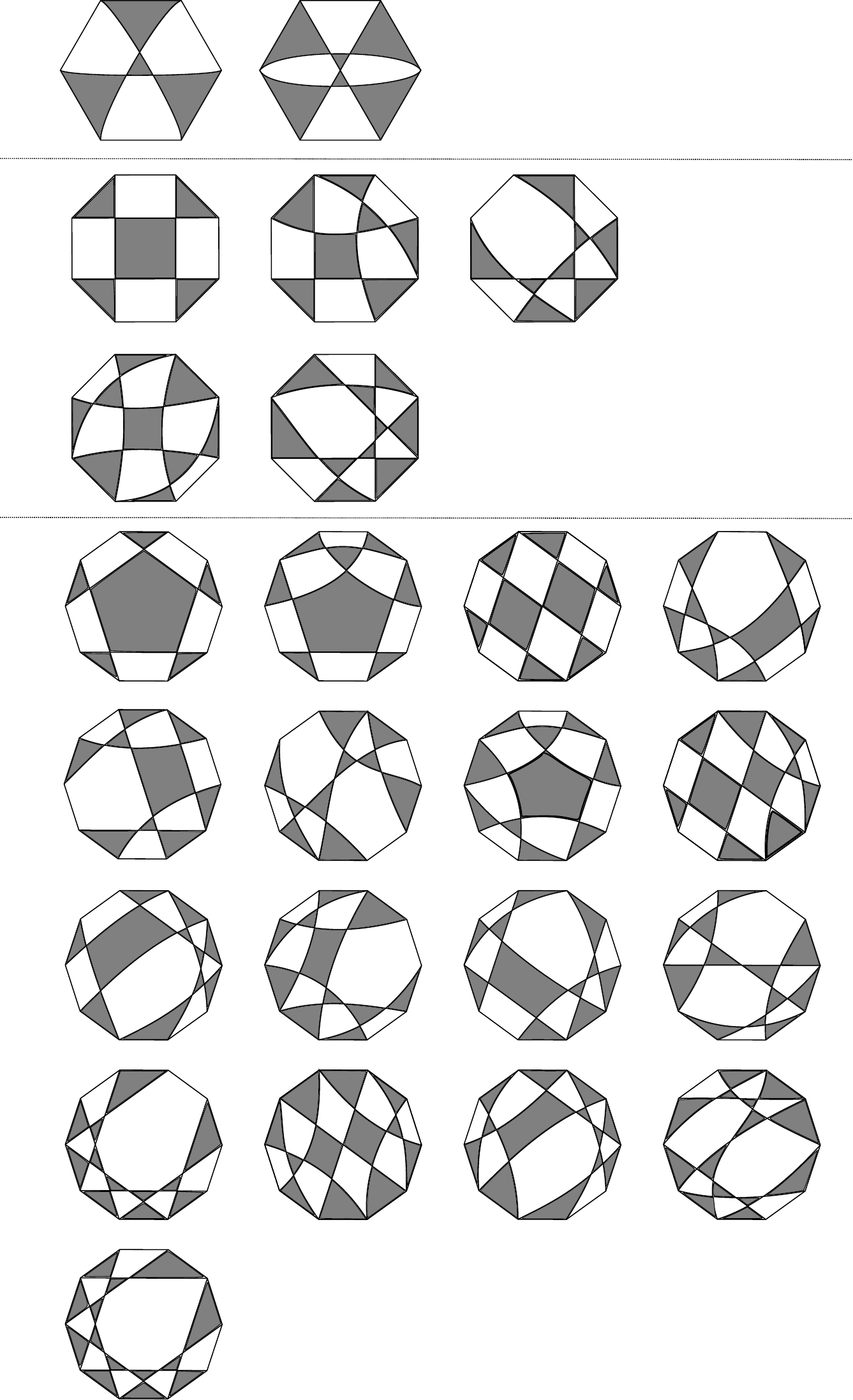}
\caption{All checkerboard polygons up to 10 vertices.}
\label{fig ex size}
\end{center}
\end{figure}

\end{example}

%
%
%
\pagebreak
\begin{example}
  Figure \ref{fig example2} shows three checkerboard polygons $\cals$ together with their respective quiver $Q$. The three quivers have 8 vertices, while the three polygons respectively have 10, 12 and 14 vertices. The two quivers on the left are of Dynkin type $\mathbb{E}_8$ or Grassmannian type $\textup{Gr}(3,8)$. The quiver on the right is of affine type $\widetilde{\mathbb{E}}_7$.
\end{example}

\begin{figure}
\begin{center}
\scalebox{0.5}{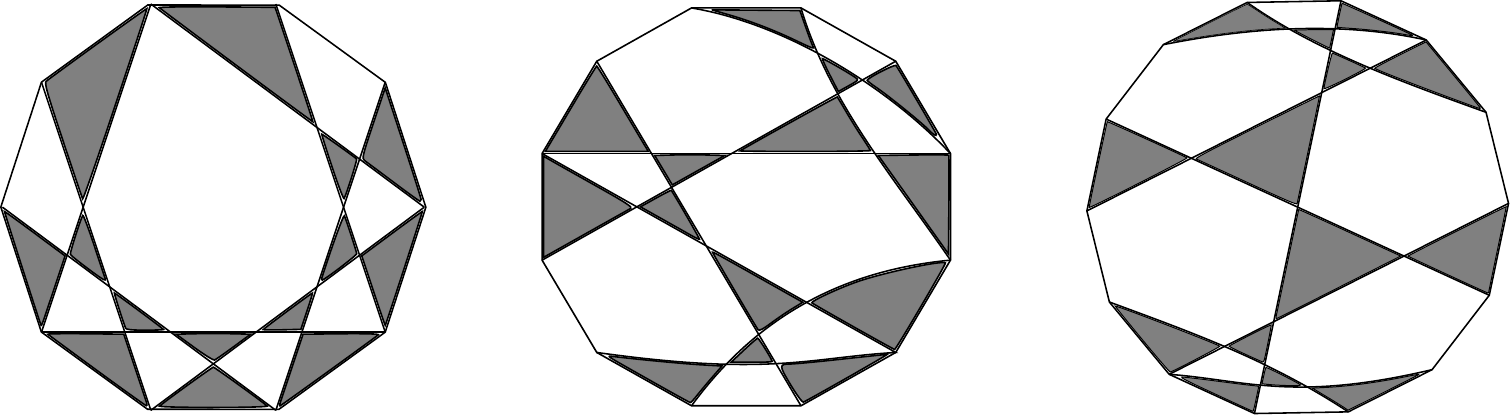}
\[\xymatrix{
\cdot\ar[r] &\cdot\ar[r]\ar[ld] &\cdot\ar[r]\ar[ld]  &\cdot\ar[ld]
\\
\cdot\ar[r]\ar[u] &\cdot\ar[r]\ar[u] &\cdot\ar[r] \ar[u]&\cdot\ar[u]
}\qquad 
\xymatrix{
\cdot\ar[r] &\cdot\ar[ld]\ar[rd] &\cdot\ar[r]\ar[l] &\cdot\ar[ld]
\\
\cdot\ar[r]\ar[u] &\cdot\ar[u] &\cdot\ar[l]\ar[u]\ar[r]&\cdot\ar[u]
}
\qquad \xymatrix@C20pt{\cdot\ar[rr]&&\cdot \ar[lld]\ar[d]\ar[rrd] &&\cdot\ar[ll]\\
\cdot\ar[r]\ar[u]&\cdot\ar[ru]&\cdot\ar[l]\ar[r]&\cdot\ar[lu]&\cdot\ar[l]\ar[u]
}
\]
\caption{Three checkerboard polygons together with their quiver.}
\label{fig example2}
\end{center}
\end{figure}

 \pagebreak
 \appendix
 
 {\setcounter{tocdepth}{1}}
 \section{Proofs of section \ref{sect 5}}
This is the first of three appendices which provide detailed proofs for sections \ref{sect 5}, \ref{sect 6}, and \ref{sect 7}.
\subsection{Properties of $f_\zg$}\label{sect 5.2}
In this section we establish certain properties of the map $f_\zg$, and in particular of valid paths in the quiver $Q$.   These results will be important for many proofs.

\subsubsection{Construction of $Q(\zg)$}\label{sec:quiver}
Here we define a full subquiver $Q(\zg)$ of $Q$ determined by the crossing sequence for $\zg$.  This quiver will be instrumental during the remainder of the paper.  We begin with considering the portion of the subquiver coming from two consecutive crossing pairs. 

Let $W$ be a white region in $\cals$ with bounding radical lines labeled $w_1, w_2, \dots, w_m$ in the clockwise orientation such that $w_1$ and $w_m$  are attached at the boundary of $\cals$.  Recall that by Lemma~\ref{lem 32}(a) a white region has either an edge with two vertices or a single vertex on the boundary of $\cals$.  Suppose that an arc $\zg$ enters $W$ by crossing two of its boundary edges $w_x, w_{x+1}$ either in the order $w_x, w_{x+1}$ or $w_{x+1}, w_{x}$.  Similarly, suppose $\zg$ exits $W$ by crossing two of its boundary edges $w_y, w_{y+1}$ either in this order or the opposite order.  For an example see Figure~\ref{fig:W-region}.  Let $Q(W)$ be the full subquiver of $Q$ on vertices $w_1, \dots, w_m$.   Note that three consecutive vertices $w_l, w_{l+1}, w_{l+2}$ for $l=1, \dots, m-2$ form a 3-cycle in $Q$ with arrows $w_l\to w_{l+1}\to w_{l+2}\to w_{l}$.  Also, let $Q(W, \zg)$ be the full subquiver of $Q(W)$ with vertices $w_x, w_{x+1}, \dots, w_y, w_{y+1}$ if $x<y$ or with vertices $w_{x+1}, w_{x}, \dots, w_{y+1}, w_{y}$ if $x>y$.  Then there are four possibilities for $Q(W, \zg)$ depending on the relationship between $x$ and $y$, and whether there is an even or odd number of vertices in $Q(W, \zg)$. See Figure~\ref{fig:quiver-pairs}.

\begin{figure}
    \centering
    \begin{minipage}{.5\textwidth}
       {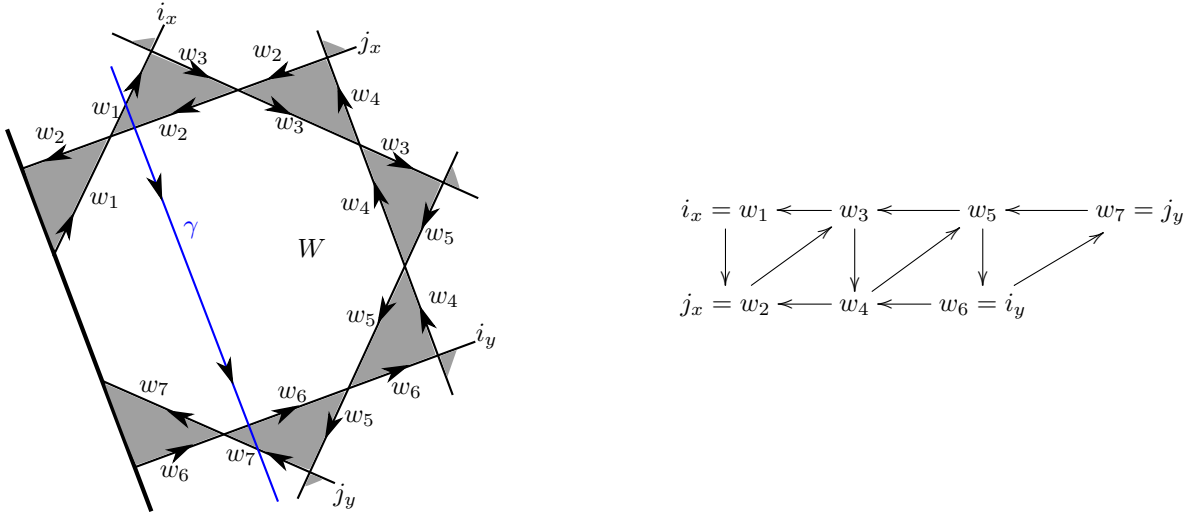}
    \end{minipage}%
    \begin{minipage}{0.5\textwidth}
       
        \hspace{.5cm} $\xymatrix@R=20pt@C=20pt{i_x=w_1\ar[d] & w_3\ar[l]\ar[d] & w_5 \ar[d]\ar[l] & w_7=j_y\ar[l]\\
j_x=w_2\ar[ur] & w_4\ar[l]\ar[ur] & w_6=i_y\ar[l]\ar[ur]}$
    \end{minipage}
    \caption{An arc $\zg$ traverses a white region $W$ and contains two pairs $(i_x, j_x), (i_y, j_y)$ in its crossing sequence.  In this case $Q(W)=Q(W,\zg)$. }
    \label{fig:W-region}
\end{figure}

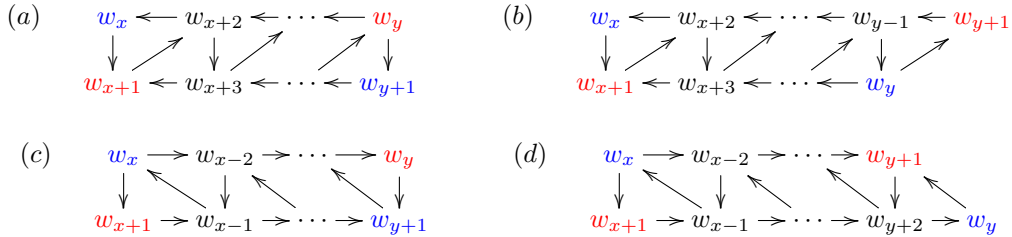
\begin{figure}[h]
\[\xymatrix@R=10pt@C=10pt{(a)&{\color{blue} w_x} \ar[d] & \ar[l] w_{x+2} \ar[d] & \ar[l] \cdots & \ar[l] {\color{red} w_y}\ar[d]  &&(b)&  {\color{blue} w_x} \ar[d] & \ar[l] w_{x+2} \ar[d] & \ar[l] \cdots & \ar[l] w_{y-1}\ar[d] & {\color{red} w_{y+1}}\ar[l]\\
&{\color{red} w_{x+1}}  \ar[ur]& \ar[l] w_{x+3}\ar[ur]  & \ar[l] \cdots \ar[ur]& \ar[l] {\color{blue} w_{y+1}} &&& {\color{red} w_{x+1}  }\ar[ur]& \ar[l] w_{x+3}\ar[ur]  & \ar[l] \cdots \ar[ur]& \ar[l] {\color{blue} w_{y}} \ar[ur]  } \]

\[\xymatrix@R=10pt@C=10pt{(c)&{\color{blue} w_x} \ar[d] \ar[r] &  w_{x-2} \ar[d] \ar[r] & \ar[r] \cdots  &  {\color{red} w_y}\ar[d]  &&(d)&  {\color{blue} w_x} \ar[d] \ar[r] &  w_{x-2} \ar[d] \ar[r] & \ar[r] \cdots  &  {\color{red} w_{y+1}\ar[d]}\\
&{\color{red} w_{x+1}}   \ar[r]& \ar[r] w_{x-1}\ar[ul]  & \ar[r] \ar[ul]\cdots &  {\color{blue} w_{y+1}} \ar[ul]&&& {\color{red} w_{x+1}}   \ar[r]& \ar[r] w_{x-1}\ar[ul]  & \ar[r] \ar[ul]\cdots &  w_{y+2} \ar[ul]  \ar[r] & {\color{blue} w_{y}}\ar[ul]} \]
\caption{The four possibilities for $Q(W, \zg)$. Cases (a) and (b) occur when $x<y$, and cases (c) and (d) occur when $x>y$. The step from $x$ to $y$ is rectangular in cases (a) and (c) and it is trapezoidal in cases (b) and (d). Vertices of the same color represent the same degree of crossings between the corresponding edges and $\zg$.}
\label{fig:quiver-pairs}
\end{figure}

We say that a step from $x$ to $y$ in the crossing sequence for $\zg$ is {\it rectangular} if $Q(W, \zg)$ has an even number of vertices and the step from $x$ to $y$ is {\it trapezoidal} if $Q(W, \zg)$ has an odd number of vertices.

Now, we analyze the subquiver of $Q$ determined by three consecutive crossing pairs in the crossing sequence for $\zg$.   By the above we obtain a subquiver determined by two consecutive pairs and we want to analyze how to piece these subquivers together.  For this construction it will not be important whether each of the steps is rectangular or trapezoidal, but the key distinction will be whether these steps are forward or not.   Let $W, W'$ be two white regions that are traversed consecutively by $\zg$.  Let $w_1, \dots, w_m$ and $w_1', \dots, w_{m'}'$ be the bounding radical lines of $W$ and $W'$ respectively, labeled clockwise around the white region such that $w_1, w_m, w_1', w_{m'}'$ meet the  boundary of $\cals$ in $W, W'$.  Then $\zg$ crosses $w_x, w_{x+1}, w_y, w_{y+1}$ in $W$ and $\zg$ crosses $w'_{x'}, w_{{x'+1}}', w_{y'}', w_{{y'+1}}'$ in $W'$, where  $w_{x'}'=w_y$ and $w_{y+1}=w_{x'+1}'$ is the unique pair of edges in $W\cap W'$ that is crossed by $\zg$.  Moreover, there is an arrow between $w_y$ and $w_{y+1}$ which we call the {\it connecting arrow}, see Figure~\ref{fig:WW}. The subquivers $Q(W, \zg)$ and $Q(W', \zg)$ share this connecting arrow. There are four cases coming from ($x<y$ or $x>y$) and ($x'<y'$ or $x'>y'$), see Figure~\ref{fig:crossing-pairs}.  The connecting arrow is labeled $\alpha$.   These quivers are fundamentally different in the following sense.

\begin{itemize}
\item $Q(W, \zg)$ and $Q(W', \zg)$ share no other vertex besides the two endpoints of the connecting arrow if ($x<y$ and $x'>y'$) or ($x>y$ and $x'<y'$).  In this case the connecting arrow is an interior arrow of the  quiver on vertices $Q(W, \zg)_0\cup Q(W', \zg)_0$.   The subsequences of the crossing sequence for $\zg$  from $x$ to $y'$ and from $y'$ to $x$ are not forward. 

 \item $Q(W, \zg)$ and $Q(W', \zg)$ share a 3-cycle containing the connecting arrow if ($x<y$ and $x'<y'$) or ($x>y$ and $x'>y'$). In this case the connecting arrow is a boundary arrow of the quiver on vertices $Q(W, \zg)_0\cup Q(W', \zg)_0$.  The subsequence of the crossing sequence for $\zg$ from $x$ to $y'$ is forward if ($x>y$ and $x'>y'$) and from $y'$ to $x$ is forward if ($x<y$ and $x'<y'$). 
 \end{itemize}
 
 \begin{figure}
\begin{center}
{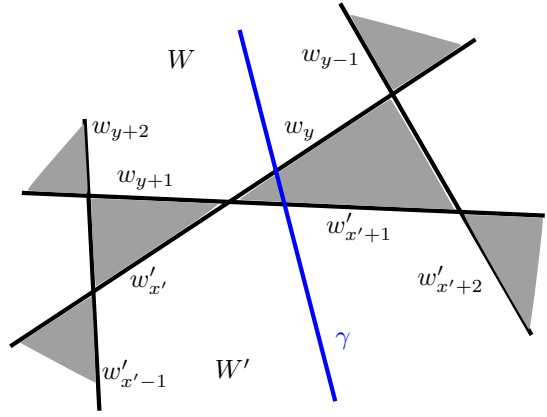}
\caption{An arc $\zg$ traverses a pair of consecutive white regions $W, W'$.}
\label{fig:WW}
\end{center}
\end{figure}

\begin{figure}
\[ \scalebox{.8}{\xymatrix@R=10pt@C=10pt{ {\begin{matrix}x<y\\x'>y'\end{matrix}}  &{\color{blue} w_x} \ar[d] & \ar[l] w_{x+2}  & \ar[l] \cdots & \ar[l] {\color{red} w_y=w'_{x'}} \ar[d]^{\alpha}  \ar[r] &\cdots\ar[r] &  \ar[r] w'_{y'+2}& \ar[d] {\color{blue} w'_{y'}} \\
&{\color{red} w_{x+1}}  \ar[ur]& \ar[l] \cdots  & \ar[l]  w_{y-1}\ar[ur]& \ar[l] {\color{blue} w_{y+1}=w'_{x'+1}}  \ar[r] & w'_{x'-1}\ar[ul] \ar[r] & \cdots \ar[r] & \ar[ul] {\color{red} w'_{y'+1}}}} \]

\[ \scalebox{.8}
{\xymatrix@R=10pt@C=10pt{{\begin{matrix}x>y\\x'<y'\end{matrix}} &{\color{blue} w_x} \ar[d] \ar[r] &  w_{x-2} \ar[d] \ar[r] & \ar[r] \cdots  &  {\color{red} w_y=w'_{x'}}\ar[d]^{\alpha} &\ar[l] w'_{x'+2} & \cdots \ar[l] & \ar[l] \ar[d]{\color{blue} w'_{y'}}\\
&{\color{red} w_{x+1}}   \ar[r]& \ar[r] w_{x-1}\ar[ul]  & \ar[r] \ar[ul]\cdots &  {\color{blue} w_{y+1}=w'_{x'+1}}\ar[ur] \ar[ul] &\cdots \ar[l] & \ar[l]\ar[ur] w'_{y'-1}& \ar[l] {\color{red} w'_{y'+1}}}}
\]

\[ \scalebox{.8}{\xymatrix@R=10pt@C=10pt{ {\color{blue} w_x} \ar[d] & \ar[l] w_{x+2}  & \ar[l] \cdots & \ar[l] {\color{red} w_y = w'_{x'}} \ar@<-3ex>[d]^{\alpha} &&  &\hspace{6ex}{\color{red} w'_{y'}} \ar[r] & {\color{blue} w'_{y'+1}} \ar[dl] \\
{\color{red} w_{x+1}}  \ar[ur]& \ar[l] \cdots & w_{y-1} =w'_{x'+2} \ar[dr]\ar[l]\ar[ur]&   \ar@{}[l]^(.25){}="c"^(.6){}="d" \ar "c";"d" \hspace{3ex}{ \color{blue} w_{y+1}=w'_{x'+1}} & &{\begin{matrix}x>y\\x'>y'\end{matrix}}&  \hspace{6ex}\vdots \ar@<-3ex>[u] & \hspace{-6ex}\vdots \ar@<3ex>[u]\\
&& \hspace{8ex}w'_{x'+4} \ar@<-4ex>[u]& w'_{x'+3} \ar@<3ex>[u]\ar[l]  &  && \hspace{6ex}w'_{x'-3} \ar[r] \ar@<-3ex>[u]& w'_{x'-2} \ar[dl] \ar@<3ex>[u]\\
&{\begin{matrix}x<y\\x'<y'\end{matrix}} & \hspace{8ex}\vdots \ar@<-4ex>[u]  \ar@{}[dr]^(.25){}="a"^(.85){}="b" \ar "a";"b"   & \ar@<3ex>[u] \hspace{-6ex}\vdots   &  {\color{blue} w_x} \ar[d] \ar[r] &  \cdots  \ar[r] & \ar[r]   w_{y+2} =w'_{x'-1} \ar@<-3ex>[u]&  {\color{red} w_y=w'_{x'}} \ar@<-3ex>[d]^{\alpha} \ar@<3ex>[u]\\
&&  \hspace{8ex}{\color{blue} w'_{y'+1}} \ar@<-4ex>[u] &\ar[l]\ar@<3ex>[u] {\color{red} w'_{y'}} \,\,\,\,&  {\color{red} w_{x+1}}   \ar[r]& \ar[r] w_{x-1}\ar[ul]  & \ar[r] \cdots &  {\color{blue} w_{y+1} =w'_{x'+1} \ar[ul]}}} \]
\caption{The subquivers of $Q$ on vertices $Q(W, \zg)_0\cup Q(W', \zg)_0$.  Vertices of the same color represent the same degree of crossings between the corresponding edges and $\zg$.}
\label{fig:crossing-pairs}
\end{figure}
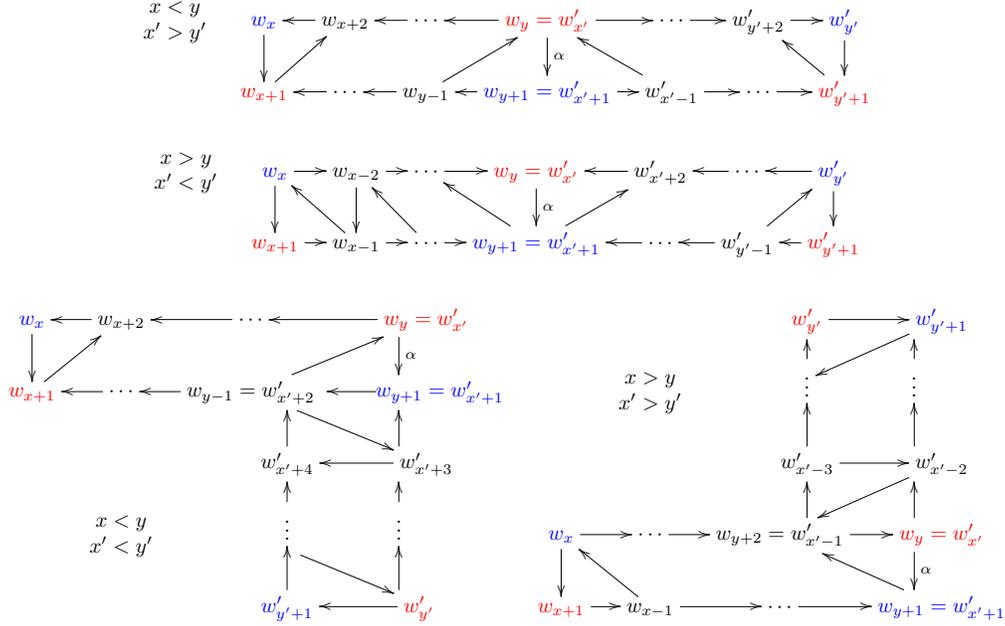

Given any subsequence of the crossing sequence for $\zg$ made up of consecutive crossing pairs we can use the above construction recursively to construct a subquiver $Q(\zg)$ of $Q$.   In order to define the quiver $Q(\zg)$ for the entire crossing sequence we first consider the case when $\zg$ crosses only one side of a shaded region $S$.  This occurs only at the endpoints of $\zg$ when $S$ lies on the boundary of $\cals$.  Then $\zg$ starts at the boundary of $\cals$ traverses $S$ and crosses a radical line $w_x$ as it exits the shaded region and enters a white region $W$.  As $\zg$ traverses $W$ and exits the white region it crosses another set of radical lines $w_y, w_{y+1}$ or a single line $w_y$.   Label the boundary of $W$ by $w_1, \dots, w_m$ in the clockwise direction as before.  Since $S$ lies on the boundary of $\cals$ it follows that $w_x$ is either the first or the last edge in $W$, so $x=1$ or $x=m$.  If $\zg$ crosses only a single radical line $w_y$ as it exits $W$ then $y=m$ or $y=1$ respectively.  Moreover, the crossing sequence for $\zg$ is $(w_x), (w_y)$, and it does not contain any crossing pairs. In this case we define $Q(\zg)=Q(W)$ to be the full subquiver of $Q$ on the vertices $w_1, \dots, w_m$.   If $\zg$ crosses two radical lines $w_y, w_{y+1}$ as it exits $W$ then the crossing sequence for $\zg$ starts with $(w_x), (w_y, w_{y+1})$ if $w_y$ is of degree 0 or $(w_x), (w_{y+1}, w_{y})$ if $w_{y+1}$ is of degree 0.. We define $Q(W, \zg)$ to be the full subquiver of $Q(W)$ on vertices $w_1, \dots, w_y, w_{y+1}$ if $x=1$ or the full subquiver of $Q(W)$ on vertices $w_y, w_{y+1}, \dots, w_m$ if $x=m$.   We define $Q(W,\zg)$ similarly if $\zg$ crosses a single edge of a shaded region as it exits $W$.  

\begin{definition}\label{def:5.8}
Let $\zg$ be a representative of a 2-diagonal.  The crossing sequence for $\zg$ determines a collection of subquivers $Q(W_1, \zg), \dots, Q(W_t, \zg)$ for some $t\geq 1$, and define $Q(\zg)$ to be the full subquiver of $Q$ on vertices $Q(W_1, \zg)_0 \cup \dots \cup Q(W_t, \zg)_0$. 
\end{definition}

In particular, $Q(\zg)$ is obtained by piecing together $Q(W_i, \zg)$ according to the construction in Figure~\ref{fig:crossing-pairs}.

\subsubsection{Valid paths}
Now we study properties of valid paths appearing in the definition of the map $f_{\zg}$.

Suppose $\zg$ crosses pairs of radical lines $w_x, w_{x+1}$ and then $w_y, w_{y+1}$.  Then there are two consecutive crossing pairs $(i_x, j_x), (i_y, j_y)$ in the crossing sequence for $\zg$ where the following sets are equal  $\{w_x, w_{x+1}\}= \{i_x, j_x\}$ and $\{w_y, w_{y+1}\}= \{i_y, j_y\}$.  
The next lemma is formulated in terms of valid paths, but it says that $i_x$ and $i_y$ (respectively $j_x$ and $j_y$) appear diagonally opposite from each other in $Q(W, \zg)$.  First, we make the following observation. 

\begin{remark}\label{rem:parity}
Consider the quiver determined by two consecutive crossing pairs $w_x, w_{x+1}$ and $w_y, w_{y+1}$ given in Figure~\ref{fig:quiver-pairs}.  We observe that if $a>b$ and $a$ and $b$ have the same parity then there is a valid path from $w_a$ to $w_b$ in $Q$ and there is no valid path from $w_b$ to $w_a$. 
\end{remark}

\begin{lemma}\label{lem:one-step}
If $(i_x, j_x), (i_y, j_y)$ are two consecutive crossing pairs for a 2-diagonal $\zg$ then exactly one of the following is true: 
\begin{itemize}
\item[(i)] there is a valid path from $i_x$ to $j_y$ and a valid path from $j_x$ to $i_y$ in $Q$; 
\item[(ii)] there is a valid path from $i_y$ to $j_x$ and a valid path from $j_y$ to $i_x$ in $Q$.
\end{itemize}
\end{lemma}

\begin{proof}
By Remark~\ref{rem:parity}, in order to prove the lemma it suffices to show that if $i_x=w_a$ for $a\in \{x, x+1\}$ then $j_y=w_b$ for $b\in \{y, y+1\}$ such that $a$ and $b$ have the same parity.  Now, we consider the orientation of the radical lines in the surface $\cals$ and the orientation of the arc $\zg$.  By Lemma~\ref{lem orientation} the radical lines along a white region alternate in direction.  Hence, if $\zg$ crosses $w_a$ and $w_b$, such that $a$ and $b$ have the same parity, then $\zg$ crosses the two radical lines in opposite directions.  That is, if $w_a$ crosses $\zg$ from left to right then $w_b$ crosses $\zg$ from right to left and vice versa.  This completes the proof. 
\end{proof}

\begin{remark}\label{rem:forward-steps}
The above lemma implies that in the crossing sequence for $\zg$ either the step from $x$ to $y$ is forward or the step from $y$ to $x$ is forward but not both, where $x$ and $y$ are consecutive.   For example, in Figure~\ref{fig:quiver-pairs} the step from $x$ to $y$ is forward in cases (c) and (d), and the step $y$ to $x$ is forward in cases (a) and (b).  
\end{remark}

\begin{remark}\label{rem:forward}
Let $(i_x, j_x), (i_y, j_y)$ be two consecutive crossing pairs for $\zg$ contained in a white region $W$ traversed by $\zg$. Then valid paths from $i_x\leadsto j_y$ or $j_x\leadsto i_y$ are subpaths of maximal valid paths of $W$, see Definition~\ref{def:whitepath}.  In particular, these paths correspond to moving counterclockwise around the boundary edges of $W$.  
\end{remark}

This completely describes valid paths coming from two consecutive crossing pairs.  Next, we consider valid paths determined by three or more consecutive crossing pairs.

\begin{lemma}\label{valid-paths-lem}
Let $(i_s, j_s), (i_{s+1}, j_{s+1}), \dots, (i_t, j_t)$ be a forward subsequence of the crossing sequence for $\zg\in \diag$.  If $t-s\geq 2$ then at most one of the paths $i_s\leadsto j_t$, $j_s\leadsto i_t$ is valid, and if $t-s=2$ then exactly one of these paths is valid. 
\end{lemma}

\begin{proof}
We consider the portion of the quiver $Q(\zg)$ coming from the first three crossing pairs $(i_s, j_s)$, $(i_{s+1}, j_{s+1})$, $(i_{s+2}, j_{s+2})$.  Since the steps from $s$ to $s+2$ are forward, $Q(\zg)$ will look like the last quiver in Figure~\ref{fig:crossing-pairs}. There are four possibilities depending on the direction of the arrow between $i_s$ and $j_s$ and whether the step from $s$ to $s+1$ is rectangular or trapezoidal,  see Figure~\ref{fig:valid-paths}.  By Lemma~\ref{lem:one-step} there are valid paths $i_s\leadsto j_{s+1}$ and $j_{s+1}\leadsto i_{s+2}$, which says that in the figure $j_{s+1}$ is directly horizontally from $i_s$ and $i_{s+2}$ is directly above or below $j_{s+1}$.  This justifies the particular labeling of $i,j$ vertices in the figure with subscripts $s+1$ and $s+2$.  Moreover, we label the remaining vertex of the 3-cycle in $Q(\zg)$ containing $i_l, j_l$ by $k_l$ for $l=s, s+1, \dots, t$.   

In cases (ii) and (iii) of Figure~\ref{fig:crossing-pairs} any path starting in $i_s$ and ending in $j_t$ with $t-s\geq 2$ passes through vertex $k_{s+1}$.  In particular, such a path is not valid because the path $i_s\leadsto k_{s+1}$ is not valid.  Similarly, in cases (i) and (iv) any path starting in $j_s$ and ending in $i_t$ with $t-s\geq 2$ passes through $k_{s+1}$ and it is not valid.   Moreover, if $t-s=2$ we see that in cases (ii) and (iii) there is a valid path $j_s\leadsto i_{s+2}$ and in cases (i) and (iv) there is a valid path $i_s\leadsto j_{s+2}$.  This shows the lemma.
\end{proof}

\begin{figure}
\[ \xymatrix@R=10pt@C=10pt{&(i)& j_{s+2}& i_{s+2} &&(ii)& i_s\ar[d] \ar[r] & \cdots & \cdots \ar[r] & j_{s+1} \ar[d]\ar[d]\\
&& \raisebox{0pt}[0.9\height][0.3\height]{ $\vdots$ } \ar[u]& \raisebox{0pt}[0.9\height][0.3\height]{ $\vdots$ } \ar[u] &&& j_s\ar[r]& k_s \ar[ul] \ar[r] & \cdots \ar[r] & k_{s+1}\ar[d]\ar[r] & i_{s+1} \ar[ul]\ar[d]\\
i_s\ar[d] \ar[r] & \cdots \ar[r] & k_{s+1} \ar[r]\ar[u]& j_{s+1} \ar[d]\ar[u] &&& &&&  \raisebox{0pt}[0.9\height][0.3\height]{ $\vdots$ } \ar[d]& \raisebox{0pt}[0.9\height][0.3\height]{ $\vdots$ } \ar[d]\\
j_s\ar[r]& k_s \ar[ul] \ar[r] & \cdots \ar[r] & i_{s+1}\ar[ul] &&& &&& i_{s+2}& j_{s+2} } \]

\[ \xymatrix@R=10pt@C=10pt{ &(iii) & i_{s+2}& j_{s+2} &&(iv) & j_s\ar[d] \ar[r] & \cdots & \cdots \ar[r] & i_{s+1} \ar[d]\ar[d]\\
&& \raisebox{0pt}[0.9\height][0.3\height]{ $\vdots$ } \ar[u]& \raisebox{0pt}[0.9\height][0.3\height]{ $\vdots$ } \ar[u] &&& i_s\ar[r]& k_s \ar[ul] \ar[r] & \cdots \ar[r] & k_{s+1}\ar[d]\ar[r] & j_{s+1} \ar[ul]\ar[d]\\
j_s\ar[d] \ar[r] & \cdots \ar[r] & k_{s+1} \ar[r]\ar[u]& i_{s+1} \ar[d]\ar[u] &&& &&&  \raisebox{0pt}[0.9\height][0.3\height]{ $\vdots$ } \ar[d]& \raisebox{0pt}[0.9\height][0.3\height]{ $\vdots$ } \ar[d]\\
i_s\ar[r]& k_s \ar[ul] \ar[r] & \cdots \ar[r] & j_{s+1}\ar[ul] &&& &&& j_{s+2}& i_{s+2} } \] 
\caption{The quiver $Q(\zg)$ in the proof of Lemma \ref{valid-paths-lem}.}
\label{fig:valid-paths}
\end{figure}
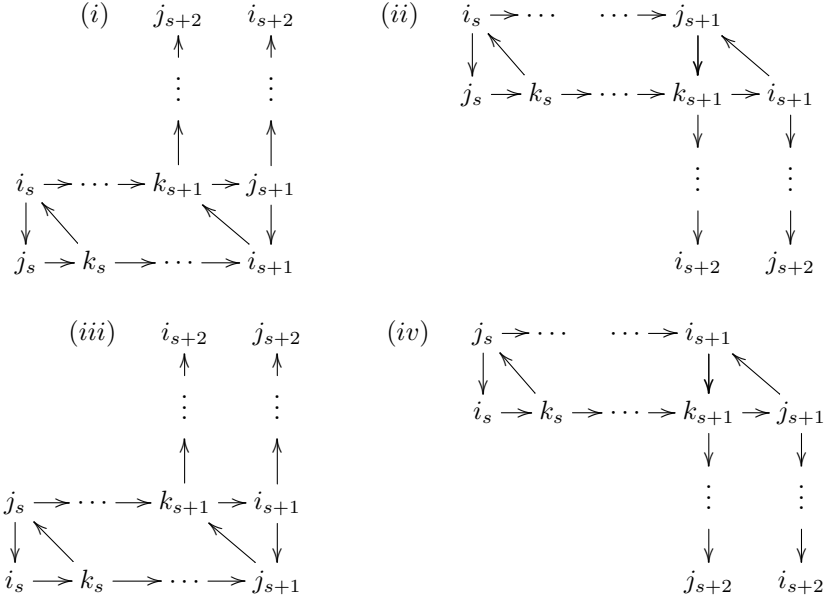

The following lemma describes the relationship between rectangular and trapezoidal steps and valid paths.  

\begin{lemma}(Rectangle-Trapezoid Lemma) \label{rect-trap-lemma1}
Let $(i_s, j_s), (i_{s+1}, j_{s+1}), \dots, (i_t, j_t)$ be a forward subsequence of the crossing sequence for some $\zg$.  Then the following properties hold.  
\begin{itemize}

\item[(1a)] If the step from $l+1$ to $l+2$ is rectangular for $l = s, \dots, t-3$ then either there exists a pair of valid paths $w_{l}:i_{l} \leadsto j_{l+2}, w_{l+1}: j_{l+1}\leadsto i_{l+3}$ or there exists a pair of valid paths $w_l:j_l\leadsto i_{l+2}, w_{l+1}: i_{l+1}\leadsto j_{l+3}$, and there is no valid path from $i_l$ to $j_{l+3}$ or from $j_l$ to $i_{l+3}$.

\item[(1b)] If the step from $l+1$ to $l+2$ is trapezoidal for $l = s, \dots, t-3$ then there is exactly one valid path either $i_l\leadsto j_{l+3}$ or $j_l\leadsto i_{l+3}$. 

\item[(2)] There is a valid path $i_s\leadsto j_t$ or a valid path $j_s\leadsto i_t$ if and only if all the steps from $s+1$ to $t-1$ are trapezoidal.  

\item[(3)] If there is a valid path $i_s\leadsto j_t$ then there are valid paths $i_{l} \leadsto j_{l'}$ for all $l, l' = s, s+1, \dots, t$ with $l<l'$. 

\item[(4)] If there exist valid paths $i_s\leadsto j_u$ and $i_t\leadsto j_w$ with $s<t<u<w$ then there exists a valid path $i_s\leadsto j_w$.  

\end{itemize}
\end{lemma}

\begin{proof}

Parts (1a) and (1b) of the lemma follow from the structure of quivers below.  Here we only depict the case when there is an arrow $j_{l+1}\to i_{l+1}$, and the quivers in the case when there is an arrow $i_{l+1}\to j_{l+1}$ are obtained from these by interchanging the labels of all $i$'s and $j$'s. 

\[\xymatrix@R=10pt@C=10pt{ && i_{l+3} & j_{l+3}\\
\text{(1a)}&& \raisebox{0pt}[0.9\height][0.3\height]{ $\vdots$ }  \ar[u]& \raisebox{0pt}[0.9\height][0.3\height]{ $\vdots$ } \ar[u] &&\text{(1b)}& j_{l+1}\ar[r] \ar[d] &\cdots & \cdots \ar[r] & i_{l+2} \ar[d]\\
j_{l+1} \ar[r]\ar[d] & \cdots\ar[r] & k_{l+2} \ar[r] \ar[u] & i_{l+2} \ar[d] \ar[u] &&& i_{l+1} \ar[r]  & k_{l+1} \ar[ul]\ar[r]  & \cdots \ar[r] & k_{l+2} \ar[r] \ar[d]& j_{l+2} \ar[ul]\ar[d]\\
i_{l+1}\ar[r] & k_{l+1}\ar[ul] \ar[r] & \cdots \ar[r] & j_{l+2}\ar[ul] &&& \raisebox{0pt}[0.9\height][0.3\height]{ $\vdots$ } \ar[u]& \raisebox{0pt}[0.9\height][0.3\height]{ $\vdots$ } \ar[u]  && \raisebox{0pt}[0.9\height][0.3\height]{ $\vdots$ } \ar[d]& \raisebox{0pt}[0.9\height][0.3\height]{ $\vdots$ }\ar[d]\\
\raisebox{0pt}[0.9\height][0.3\height]{ $\vdots$ } \ar[u] & \raisebox{0pt}[0.9\height][0.3\height]{ $\vdots$ }  \ar[u] &&&  && j_l\ar[u] & i_l\ar[u] && j_{l+3}& i_{l+3}\\
j_l\ar[u] & i_l\ar[u]}
\]

To show part (2) first suppose that there exists some $l$ with $s \leq l\leq t-3$ such that the step from $l+1$ to $l+2$ is rectangular.  Then $Q(\zg)$ contains as a full subquiver the quiver given (1a) above, where the labels of all $i$'s and $j$'s may be interchanged.  Then a path $w$ from $i_s$ to $j_t$ passes through some vertices in the rectangular portion of the quiver determined by $i_{l+1}, j_{l+1}, i_{l+2}, j_{l+2}$.  In particular, $w$ enters this rectangle at vertex $k_{l+1}$ or $i_{l+1}$ and exit the rectangle at vertex $k_{l+2}$ or $i_{l+2}$.  Then $w$ contains two arrows in the same 3-cycle, which shows that it is not a valid path.   Now, suppose that all steps from $s+1$ to $t-1$ are trapezoidal.  Then we have the following subquiver of $Q(\zg)$ if there is an arrow $j_{s+1}\to i_{s+1}$, and the labels of all $i$'s and $j$'s interchange if instead there is an arrow $i_{s+1}\to j_{s+1}$.

\[\xymatrix@R=10pt@C=10pt{(2)&i_{s+3}\ar[dr] &\bullet \ar[l]&\cdots\ar[l]& \cdots &\cdots&\ar[l] j_{s+2}\ar[r] & i_{s+2} \ar[dl] \\
&j_{s+3}\ar[u] \ar[d]& k_{s+3}\ar[l] \ar[d]\ar[u]&\cdots\ar[l]&\cdots&\cdots & \ar[l] k_{s+2} \ar[r] \ar[u] & \bullet \ar[u]\\
&\raisebox{0pt}[0.9\height][0.3\height]{ $\vdots$ } &\raisebox{0pt}[0.9\height][0.3\height]{ $\vdots$ } &&&& \raisebox{0pt}[0.9\height][0.3\height]{ $\vdots$ }  \ar[u]&\raisebox{0pt}[0.9\height][0.3\height]{ $\vdots$ }  \ar[u]\\
&\vdots&\vdots&&i_s \ar[r] & \cdots \ar[r] & k_{s+1} \ar[r] \ar[u]& j_{s+1}\ar[u] \ar[dl] & j_t&i_t\\
&\vdots\ar[d]&\vdots \ar[d]&&j_s\ar[r] & \cdots \ar[r] & i_{s+1}\ar[u] &&\vdots \ar[u] & \vdots\ar[u]\\
&\bullet\ar[d] & k_{s+4}\ar[d] \ar[l]\ar[r]&\cdots & \cdots &\cdots  & \cdots &\cdots \ar[r] & k_{t-1}\ar[r] \ar[u]\ar[d]& j_{t-1}\ar[d]\ar[u]\\
&i_{s+4}\ar[ur] & j_{s+4}\ar[l] \ar[r]&\cdots & \cdots &\cdots & \cdots &\cdots\ar[r] &\bullet \ar[r]& i_{t-1}\ar[ul]
}\]

We remark that the number of vertices in the quiver between two consecutive crossing pairs does not necessarily increase as the figure might suggest, this is only a convenient way to draw the quiver, which looks like it is spiraling out from the center as we move from $s$ to $t$.  
Here we see that there is a valid path from $i_s$ to $j_t$ that passes through vertices $k_{s+1}, k_{s+2}, \dots, k_{t-1}$ in this order.  Moreover, there is no valid path from $j_s$ to $i_t$ in this case.  This shows part (2) of the lemma. 

To show part (3), suppose that there is a valid path $i_s\leadsto  j_t$.  If $t=s+1$, then the statement holds trivially, and if $t=s+2$ then the statement follows from Lemma~\ref{lem:one-step}.  Finally, if $t\geq s+3$, then part (2) of this lemma implies that all steps from $s+1$ to $t-1$ are trapezoidal.  Moreover, since $i_s\leadsto j_t$ is valid, rather than $j_s\leadsto i_t$ being valid, then there is an arrow $j_l \to i_l$ for all $l=s+1, \dots, t-1$ as in the quiver above.  Because all the steps from $s+1$ to $t-1$ are trapezoidal, then again part (2) of this lemma implies that that there are valid paths $i_l \leadsto j_{l'}$ for all $l, l' = s, \dots, t$ and $l<l'$.  This shows part (3) of the lemma. 

Now we show the last part of the lemma.  If the sequence $s, t, u , w$ is consecutive, then the step from $t$ to $u$ must be trapezoidal since we have a pair of valid paths $i_s\leadsto j_u$, $i_t\leadsto j_w$, see parts (1a) and (1b) of the lemma.  Thus part (1b) of the lemma implies that there is a valid path $i_s\leadsto j_w$ and completes the proof in this case. 

If the sequence $s, t, u$ is consecutive and $w>u+1$ then all the step from $u$ to $w-1$ are trapezoidal by part (2) of the lemma.  Moreover, if the step form $t$ to $u$ is rectangular then (1a) together with the existence of a valid path $i_s\leadsto j_u$ implies that there is a valid path $j_t\leadsto i_{u+1}$.   Then Lemma~\ref{valid-paths-lem} implies that there is no valid path $i_t\leadsto j_{u+1}$.  Thus by part (3) of the lemma there is no valid path $i_t\leadsto j_w$, which is a contradiction.   Thus the step from $t$ to $u$ is trapezoidal.  In particular, all steps from $t$ to $w-1$ are trapezoidal and (2) yields a valid path $i_s\leadsto j_w$.  

The case when $t, u, w$ is consecutive and $s<t-1$ is similar.  It remains to consider the case when $s, t, u$ is not consecutive and $t, u, w$ is not consecutive.  In this situation all steps from $s+1$ to $w-1$ are trapezoidal by (2), and again by (2) there is a valid path $i_s\leadsto j_w$.  This completes the proof of (4). 
\end{proof}

\begin{lemma}\label{lem:ii-paths}
Let $\zg$ be a representative of a 2-diagonal in $\cals$.  If in the crossing sequence for $\zg$ all steps from $s$ to $t$ are forward then there are no valid paths in $Q$ from $i_s$ to $i_t$ or from $j_s$ to $j_t$.    
\end{lemma}

\begin{proof}
If $|s-t|=1$ then the statement follows from the structure of the quiver $Q(\zg)$ depicted in Figure~\ref{fig:quiver-pairs}.  Indeed, $i_s, i_t$ are vertices of a rectangle or a trapezoid that are diagonally opposite from each other by Lemma~\ref{lem:one-step}, so any path $i_s\leadsto i_t$ uses two arrows in the same 3-cycle and hence it is not valid.  Similar statement holds for $j_s, j_t$.  

If $|s-t|=2$, then lemma follows from the quivers in Figure~\ref{fig:valid-paths}.  Now, suppose that $|s-t|>2$ and without loss of generality assume $t>s$. If there is a rectangular step from $l+1$ to $l+2$ for some $l\in \{s, \dots, t-3\}$, then the quiver $Q(\zg)$, up to interchanging all $i$'s and $j$'s, is given in picture (1a) in the proof of Lemma~\ref{rect-trap-lemma1}.  We see that any path $i_s\leadsto i_t$ or $j_s\leadsto j_t$ passes through the rectangle determined by $i_{l+1}, i_{l+2}, j_{l+1}, j_{l+2}$. In particular, it cannot be valid.  If all the steps from $s+1$ to $t-1$ are trapezoidal, then the quiver $Q(\zg)$, up to interchanging all $i$'s and $j$'s,  is given in picture (2) in the proof of the same lemma.  Again we observe that there are no valid paths from $i_s$ to $i_t$ or from $j_s$ to $j_t$.    
\end{proof}

\subsubsection{Relations in $B$}
In this section we prove two lemmas describing the structure of certain paths in $Q$ that compose to zero or commute in the algebra $B$.  

Recall from Section~\ref{sect 3.1.1}  that the trunk of the dual graph of $Q$ consists of chordless cycles connected by interior arrows. 

\begin{lemma}\label{zero-relations-lemma}
Let $\eta$ be a nonzero path in the algebra $B$ such that the composition $\alpha\eta=0$ for some arrow $\alpha$.  Then up to commutativity $\eta$ starts with $\eta'$ where
\begin{itemize}
\item[(a)] $\eta'=\beta_1\delta_1$ such that $\alpha, \beta_1, \delta_1$ are arrows in the same 3-cycle, or   
\item[(b)] $\eta'=\beta_1\dots\beta_k$  such that $\alpha, \beta_1$ are arrows in the same 3-cycle,  $\beta_1, \dots, \beta_k$ are arrows that lie in distinct 3-cycles $C_1, \dots, C_k$ such that the path $\xymatrix@C=5pt{C_1\ar@{-}[r] & \bullet \ar@{-}[r] & C_2 \ar@{-}[r] & \bullet \ar@{-}[r] & \dots \ar@{-}[r] & \bullet \ar@{-}[r] & C_k}$ is in trunk of the dual graph of $Q$ where $C_k$ is a leaf.  In particular, there is the following subquiver of $Q$. 
\end{itemize}

\[
\xymatrix@R=15pt@C=20pt{
\bullet\ar[r]^{\beta_1}& \bullet \ar[r]^{\beta_2} \ar[dl]_{C_1}& \bullet \ar[r] \ar[dl]_{C_2}& \cdots \ar[r] & \bullet \ar[r]^{\beta_{k-1}} & \bullet\ar[r]^{\beta_k} \ar[dl]& \bullet \ar[dl]^{\delta_k}_{C_k}\\
\bullet\ar[r]\ar[u]^{\alpha}& \bullet \ar[r]\ar[u]& \bullet \ar[r] \ar[u]& \cdots \ar[r] & \bullet \ar[r] \ar[u]& \bullet \ar[u]
}
\]
\end{lemma}

\begin{proof}
Let $\eta'$ be a minimal nonzero subpath of $\eta$ appearing at the start of $\eta$ such that $\alpha\eta'=0$.  Since $\alpha\eta'=0$ we may choose (up to commutativity) the first arrow in $\eta'$ to be $\beta_1$ such that $\alpha, \beta_1$ lie in a common 3-cycle $C_1$.  Let $\delta_1$ be the third arrow in $C_1$.  Moreover, $\alpha\beta_1$ must lie in a relation that propagates further via commutativity and zero relations in $B$ that results in $\alpha\eta'$ being zero.   

If $\eta'=\beta_1$ then $\delta_1$ is a boundary arrow, and we obtain case (a) of the lemma.  

Suppose $\eta'$ consists of at least two arrows $\beta_1\beta_2$.   By minimality of $\eta'$ the composition $\alpha\beta_1\not=0$.  Then the arrow $\delta_1$ is not boundary, and there exist arrows $\alpha', \beta'_1$ such that $\alpha \beta_1=\alpha' \beta_1'$, see the quiver below on the left. 
\[
\xymatrix@R=20pt@C=20pt{
&&  &&&   \bullet \ar[r]^{\epsilon_2} & \bullet\ar[dl]_{\delta_2}\\  
\bullet\ar[r]^{\beta_1}& \bullet \ar[r]^{\beta_2} \ar[dl]_{\delta_1}& \bullet  &&&   \bullet\ar[r]^{\beta_1}\ar[u]^{\epsilon_1}& \bullet \ar[dl]_{\delta_1}  \ar[u]_{\beta_2}\\
\bullet\ar[r]_{\alpha'} \ar[u]^{\alpha}& \bullet \ar[u]_{\beta_1'}& &&& \bullet\ar[r]_{\alpha'} \ar[u]^{\alpha}& \bullet \ar[u]_{\beta_1'}
}
\]  

If $\beta_2=\delta_1$, then $\alpha_1\beta_1\delta_1=0$ and $\eta'=\beta_1\delta_1$ by the minimality of $\eta'$.  This gives case (a) of the lemma.   If the second arrow in $\eta'$ is $\beta_2\not=\delta_1$,  then $\alpha\eta'$ starts with $\alpha\beta_1\beta_2=\alpha'\beta_1'\beta_2$.  Because $\eta'\not=0$ but $\alpha\eta'=0$ then $\beta_1\beta_2$ or $\beta_1'\beta_2$ lie in a relation.   Thus, one of the paths $\beta_1\beta_2$ or $\beta_1'\beta_2$ lies in a 3-cycle.  We claim that the former case is not possible. 

Suppose on the contrary that $\beta_1\beta_2$ lies in a 3-cycle, and let $\delta_2$ be the third arrow in this 3-cycle.  Because $\beta_1\beta_2\not=0$ it follows that $\delta_2$ is not a boundary arrow.  Then there exists another 3-cycle containing $\delta_2$ and two other arrows $\epsilon_1, \epsilon_2$, see the quiver above on the right.   Note 
\[\alpha\beta_1\beta_2=\alpha \epsilon_1\epsilon_2=\alpha'\beta_1'\beta_2\not=0.\]   
The arrows $\alpha, \epsilon_1$ cannot belong to a common 3-cycle, because $t(\alpha)$ would be an interior vertex of $Q$, which is a contradiction.  Thus there is no relation on the path $\alpha\epsilon_1$.  Then $0=\alpha\eta'=\alpha\epsilon_1\epsilon_2\eta''$ implies $0=\epsilon_1\epsilon_2\eta''=\eta'$, a contradiction. This shows the claim that $\beta_1\beta_2$ cannot lie in a 3-cycle.

Then $\beta_2, \beta_1'$ lie a common 3-cycle $C_2$.  Let $\delta_2$ be the third arrow in $C_2$.  If $\delta_2$ is a boundary arrow, then $\alpha\beta_1\beta_2=0$ and $\eta'=\beta_1\beta_2$, and case (b) of the lemma is satisfied. If $\delta_2$ is not a boundary arrow then $\alpha\eta'$ is equivalent to a path starting with $\alpha'\beta_2''\beta_2'$, where $\delta_2, \beta_2'', \beta_2'$ form another 3-cycle containing $\delta_2$.   Note that $\alpha', \beta_2''$ cannot lie in a common 3-cycle, because $t(\alpha')$ would be an interior vertex of $Q$.  Thus, there are no more commutativity moves possible that involve only the first three arrows in $\alpha\eta'$.  

\[
\xymatrix@R=20pt@C=20pt{
\bullet\ar[r]^{\beta_1}& \bullet \ar[r]^{\beta_2} \ar[dl]^{\delta_1}& \bullet  \ar[dl]^{\delta_2}\\
\bullet\ar[r]_{\alpha'} \ar[u]^{\alpha}& \bullet \ar[u]_{\beta_1'} \ar[r]_{\beta_2''} & \bullet \ar[u]_{\beta_2'}
}
\]  


Next we can continue in the same way considering the next arrow in $\eta'$.  Eventually, we obtain $\eta'=\beta_1\dots\beta_k$ for some $k$ as in the quiver in the statement of the lemma where the arrow $\delta_k$ is boundary.  
\end{proof}

\begin{lemma}\label{lem:comm}
Let $\alpha$ be an arrow and $\eta = \beta_1\dots \beta_k$ be a path in the quiver $Q$ that lie in the full subquiver of $Q$ shown below.  Suppose $\alpha\eta' =  \eta \nu$ is a nonzero path in $B$ for some paths $\eta', \nu$.  Then up to commutativity $\nu$ starts with the arrow $\alpha'$.  

\[
\xymatrix@R=10pt@C=10pt
{\bullet \ar[r] & \bullet \ar[r] \ar[dl]& \cdots \ar[r] & \bullet \ar[r] & \ar[dl]\bullet\\
\bullet\ar[u]^{\alpha}\ar[r]_{\beta_1} & \bullet \ar[r]_{\beta_2} \ar[u]& \cdots\ar[r] & \bullet \ar[r]_{\beta_k} \ar[u]& \bullet\ar[u]_{\alpha'} }
\]

\end{lemma}

\begin{proof}
By the assumptions of the lemma, we have the full subquiver of $Q$ as in Figure~\ref{fig:lem-comm}, where we label the vertices $x_1, \dots, x_{k+1}$ and $y_1, \dots, y_{k+1}$.  Moreover, since $\alpha\eta',   \eta \nu$ are two equal nonzero path in $B$, then they end in some common vertex $c$.  Since there are no cycles in $B$ by Proposition~\ref{prop 38}, then the two paths do not self-intersect.  

We claim that the path $\eta'$ starts with the arrow $\alpha_1$.  Suppose not.  Then $\eta'$ is a path starting in $x_1$ and ending in $c$ that does not pass through $x_2$. We show the two possibilities for $\eta'$ in Figure~\ref{fig:lem-comm}.   If $\eta'$ is the upper path then $x_2$ is an internal vertex, and if $\eta'$ is the lower path then $y_1$ is an internal vertex.  This shows the claim that $\eta'$ starts with the arrow $\alpha_1$.  

\begin{figure}
\[
\hspace{2cm}*{\xymatrix@R=15pt@C=15pt
{x_1 \ar[r]^{\alpha_1} \ar@/^60pt/@{~>}@[red][rrrrrrdd]^{\eta'}   \ar@/_20pt/@{~>}@[blue][dd]& x_2 \ar[r]^{\alpha_2} \ar[dl]& \cdots \ar[r] & x_{k} \ar[r]^{
\alpha_{k}} & \ar[dl] x_{k+1} &&\\
y_1\ar[u]^{\alpha}\ar[r]_{\beta_1} & y_2 \ar[r]_{\beta_2} \ar[u]_{\epsilon_2}& \cdots\ar[r] & y_{k} \ar[r]_{\beta_k} \ar[u]^{\epsilon_k}& y_{k+1}\ar[u]_{\alpha'}\ar@{~>}[rrd]_{\nu}&&\\
\ar@/_20pt/@{~>}@[blue][rrrrrr]^{\eta'}&&&&&&c&&&&
}}
\]
\caption{Proof of Lemma~\ref{lem:comm}.}
\label{fig:lem-comm}
\end{figure}

Now we have $\eta'=\alpha_1\eta''$ for some path $\eta''$ and we obtain the following equalities of nonzero paths.
\[\alpha\eta'=\alpha\alpha_1\eta''=\beta_1\epsilon_2\eta''=\beta_1\beta_2\dots\beta_k\nu = \eta \nu \]
In particular, $\epsilon_2\eta'' = \beta_2, \dots, \beta_k \nu$ are nonzero paths in $B$.   By the same reasoning as before we can show that that first arrow in $\eta''$ is $\alpha_2$.  Continuing in this way we obtain the equality 
\[\epsilon_k\alpha_k\eta^{(k+1)}=\beta_k\alpha'\eta^{(k+1)}=\beta_k\nu\]
for some path  $\eta^{(k+1)}$ starting in $x_{k+1}$ and ending in $c$. Therefore, $\alpha'\eta^{(k+1)}=\nu$, and we obtain that up to commutativity $\nu$ starts with the arrow $\alpha'$.  This completes the proof of the lemma. 
\end{proof}

\subsection{Exactness}\label{Aexactness}
Here we prove  Proposition~\ref{big-lemma}, see  Proposition~\ref{Abig-lemma}
at the end of this subsection.
\begin{lemma}\label{lem:kernel}
Let $\zg, R(\zg)$ be a pair of compatible arcs in $\cals$.  Then $f_{\zg} \circ \bar{f}_{R(\zg)} = 0$. 
\end{lemma}

\begin{proof}
Let $A=f_\zg, B=\bar{f}_{R(\zg)}$ be two matrices with entries $A=(a_{h,l}), B=(b_{h,l})$.  To prove the lemma we need to show that the product $AB$ is zero.  The $(s,t)$-entry of the matrix $AB$ is the following sum of products of paths 
\[\sum_l  f_{{\zg}_{i_s, j_l}} \bar{f}_{{R(\zg)}_{j_l, i_t}}\]
from $i_s$ in the crossing sequence for $\zg$ to $i_t$ in the crossing sequence for $R(\zg)$. 
This is trivially zero if for all $l$ in the crossing sequence for $\zg$ the step from $s$ to $l$ is not forward or if in the crossing sequence for $R(\zg)$ the step from $l$ to $t$ is not forward.  Now, we fix $t$ such that there exist $l$ and $s$ such that all steps from $s$ to $l$ are forward and all steps from $l$ to $t$ are forward.  We may index the crossing pairs in the crossing sequences so that $s\leq l\leq t$.  Let $B_t$ be the $t$-th column of $B$ and let $l\leq t-1$ be least integer such that $b_{l,t}$ is a valid path.  Then either $l=1$ or $b_{l-1,l+1}$ is not a valid path in $B$, because by Lemma~\ref{valid-paths-lem}(4) it would yield a valid path $b_{l-1, t}$ contradicting minimality of $l$.    

Next, we study the structure of valid paths in the matrix $A$ given below, where valid paths are denoted by $\star$.  Let $\Delta(l,t)$ denote the triangular region in $A$ or $B$ whose vertices are at positions $(l,t), (l,l+2), (t-2, t)$.  Since $b_{l,t}$ up to sign is a valid path $j_l\leadsto i_t$ then every path in $\Delta(l,t)$ in $B$ is also valid by Lemma~\ref{rect-trap-lemma1}(3).  By Lemma~\ref{valid-paths-lem} each entry of $A$ in $\Delta(l,t)$ is zero, and thus again by Lemma~\ref{rect-trap-lemma1}(3) each entry of $A$ in the rectangular region above $\Delta(l,t)$ is zero.  The entry $a_{l-1, l+1}$ in $A$ is a valid path, because otherwise the entry $b_{l-1,l+1}$ in $B$ would be a valid path by Lemma~\ref{valid-paths-lem} and we have seen above that $b_{l-1,l+1}$ is not valid.

\[{\begin{smallmatrix}
 & & k &&&&&&l & l+1 & l+2 &&& t\\
 &\ddots &  &&&&&& &  &  &&& \\
k& & \star & \star & 0 & 0 & \cdots & \cdots &\cdots   & 0 & 0 & \cdots & \cdots & 0 \\
k+1&& 0 & \star & \star & \star & \star & \cdots &\cdots   & \star & 0& \cdots & \cdots & 0 \\
&&&0 &\star & \star & \star & \cdots & \cdots &  \star & 0& \cdots & \cdots & 0 \\
&&& & 0 & \ddots & \ddots &&&& \vdots &&& \vdots \\
&&&&&0& \star & \star & \star & \star & 0 & \cdots & \cdots & 0\\
l-1&&&&&& 0 & \star & \star & \star & 0 & \cdots & \cdots& 0\\
{\color{blue} l}&&&&&&& {\color{blue} 0} & {\color{blue} \star} & {\color{blue} \star} &{\color{blue} 0} & {\color{blue} \cdots} & {\color{blue} \cdots}&  {\color{blue} 0}\\
l+1&&&&&&&& 0 & \star & \star &0 & \cdots &   0\\
&&&&&&&&& 0 & \ddots & \ddots & \ddots & \vdots\\
t-2&&&&&&&&&& 0 & \star & \star & 0 \\
t-1&&&&&&&&&&& 0 & \star & \star \\
t&&&&&&&&&&&& 0 & \star \\
&&&&&&&&&&&& & & \ddots \end{smallmatrix} }\]

Let $k$ be the least positive integer such that $a_{k+1, l+1}$ is valid, or, if no such $k$ exists, we let $k=1$.  Then every entry in $A$ in the triangle $\Delta(k+1, l+1)$ is a valid path by Lemma~\ref{rect-trap-lemma1}(3), and $a_{k, l+1} =0$ by minimality of $l$.  Therefore, 
\[a_{k, k+2} = a_{k, k+3} = \dots = a_{k, l+1}=0\]
by Lemma~\ref{rect-trap-lemma1}(4).  Then Lemma~\ref{rect-trap-lemma1}(3) implies that every entry $a_{h, k+2},\dots, a_{h, l+1}$ with $h\leq k$ is zero. 

If $l>1$, then multiplying $AB_t$ yields a vector whose first $k$ entries are zero.  The $s$-th entry of $AB_t$ for $k+1<s < t-1$ is a difference of two paths $i_s\leadsto j_l\leadsto i_t, i_s\leadsto j_{l+1} \leadsto i_t$, which is zero, because all parallel paths in $Q$ are equal in $B$ by Proposition~\ref{prop 39}.  The $t$-th entry of $AB_t$ is a single path $i_t\leadsto j_t \leadsto i_t$, which is zero, because there are no cycles in $B$ by Proposition~\ref{prop 38}.   

If $l=1$ then $k=1$.  In addition, if the first row of $A$ contains two nonzero entries then the same argument as above implies that $AB_t$ yields a vector whose first $t$ entries are zero. Now suppose $l=1$, but the first row of $A$ contains only a single entry.  Then the crossing sequence for $\zg$ starts with $(i_1), (i_2, j_2)$ and the first row of $A$ contains a single path $i_1\leadsto j_2$.  In particular, there is no valid path from $i_1$ to $j_3$.    It suffices to show that the first entry in $AB_t$ is zero. Since the crossing sequence for $\zg$ starts with $(i_1), (i_2, j_2)$, then we are in the situation of Figure~\ref{fig:34} on the left, where the arrow $j_1\to i_1$ is a boundary arrow.  Moreover, since there is no valid path $i_1\to j_3$, then we have the corresponding quiver as in Figure~\ref{fig:34} on the right.

\begin{figure}

 \begin{minipage}{.5\textwidth}
      \centerline{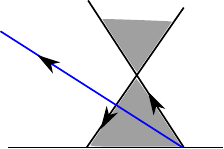}
    \end{minipage}%
    \begin{minipage}{0.5\textwidth}
       \[
\xymatrix@C=10pt@R=10pt{
i_1\ar[r] & \bullet\ar[dl] \ar[r] & \cdots \ar[r] & j_2\ar[dl]\\
j_1\ar[u] \ar[r] & \cdots \ar[r] & k\ar[r] \ar[d]& i_2\ar[u]\ar[d]\\
&&\vdots \ar[d]& \vdots\ar[d]\\
&&i_3 & j_3
}
\]
\end{minipage}

\caption{Proof of Lemma~\ref{lem:kernel}.}
\label{fig:34}
\end{figure}

The first entry in $AB_t$ is a composition $i_1\leadsto j_2 \leadsto i_t$, for some $t\geq 2$.   From the quiver we see that this path is zero, because it would start with the path $i_1\leadsto k$, which is zero, because $j_1\to i_1$ is a boundary arrow.   This shows that in all cases $AB_t$ is zero in the first $t$ entries.

Finally, if the step from $t$ to $t+1$ is forward, then remaining entries $t+1, t+2, \dots $ are zero  since $b_{h,t}=0$ for all $h\geq t+1$.  If the step from $t$ to $t+1$ is not forward, then we let $l'\geq t+1$ be the largest integer such that $b_{l',t}\not=0$, and we let $k'$ be the largest positive integer such that $a_{k'-1, l'-1}$ is valid.  Then we consider the structure of $A$ in rows $t, t+1, \dots, k'$, which is analogous to the one we have above but instead $A$ is lower triangular in that region instead of being upper triangular. Then a similar argument as before shows that $AB_t$ is zero in positions greater than $t$.   
Because $t$ was arbitrary this shows that $AB=0$. 
\end{proof}

The above lemma says that in particular $\text{Im}\,\bar{f}_{R(\zg)} \subset \text{ker} f_{\zg}$.  The rest of this section is dedicated to showing the reverse inclusion.   We shall need three preparatory lemmas. In the first, we consider special cases which we treat separately.

\begin{lemma}\label{lem:base}
Let $\zg$ be an arc in $\cals$ such that the crossing sequence for $\zg$ consists of $n-1$ steps that go forward from $1$ to $n$.  Let $x = \begin{bsmallmatrix}x_1 & 0 & \cdots & 0\end{bsmallmatrix}^T \in \textup{ker}\, f_{\zg}$, where $x_1$ is a nonzero path in $Q$.  Then $x\in \textup{Im}\,\bar{f}_{R(\zg)}$.  
\end{lemma}

\begin{proof}
We need to consider the following cases separately, see Figure~\ref{fig:cases12}.  These cases represent all possible starting configurations for $\zg$ and $R(\zg)$.  The starting point $a$ of $\zg$ is a boundary point of the checkerboard polygon $\cals$, and the boundary edge incident to it in clockwise direction is either in a white region $W_2$ (case 1) or in a shaded region (case 2).  The terminal point $b$ of $R(\zg)$ is the other endpoint of this edge.  

\begin{figure}
\centerline{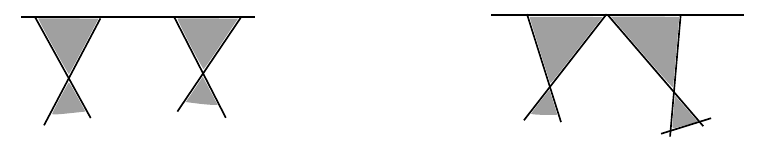}
\caption{Proof of Lemma~\ref{lem:base}.}
\label{fig:cases12}
\end{figure}

Recall that an edge between $i,j$ in a crossing pair represents an arrow $i\to j$ or $j\to i$.  The crossing sequence for $\zg$ will start in one of the following ways. 

\begin{itemize}
\item[] $(\xymatrix@C=10pt{i_1\ar@{-}[r] &j_1}), (\xymatrix@C=10pt{i_2\ar@{-}[r]& j_2}), \dots$ or  
\item[]  $(i_0), (\xymatrix@C=10pt{i_1\ar@{-}[r] &j_1}), (\xymatrix@C=10pt{i_2\ar@{-}[r] &j_2}), \dots$ or 
\item[]  $(j_0), (\xymatrix@C=10pt{i_1\ar@{-}[r] & j_1}), (\xymatrix@C=10pt{i_2\ar@{-}[r] & j_2}), \dots$
\end{itemize}

The path $x_1$ is an element of the first indecomposable summand $P(j')$ of $P_1(\zg)$.  Thus $s(x_1)=j'=j_1$ in the first two cases and $s(x_1)=j'=j_0$ in the third case. 

Now we distinguish eight difference cases in the crossing sequences for $\zg$ and $R(\zg)$, where the first three come from case 1 and the remaining 
five come from case 2 as in Figure~\ref{fig:cases12}.   Here recall that, because the direction of $\zg$ and $R(\zg)$ changes, it follows that if $(i,j)$ is a crossing pair for $\zg$, and $R(\zg)$ also crosses $i,j$, then $(j,i)$ is a crossing pair of $R(\zg)$. 

\begin{itemize}
\item[{\bf Cases}]  The crossing sequences start with 
\item[(1.1)] $(i'\to j')$ for $\zg$ and $(j')$ for $R(\zg)$, where $i'\to j'$ is a boundary arrow.
\item[(1.2)] $(i)$ for $\zg$ and $(i\to j)$ for $R(\zg)$, where $i\to j$ is a boundary arrow.
\item[(1.3)] $(\xymatrix@C=10pt{i_1\ar@{-}[r] & j_1})$ for $\zg$ and $(\xymatrix@C=10pt{j_1\ar@{-}[r] & i_1})$ for $R(\zg)$, where $\xymatrix@C=10pt{i_1\ar@{-}[r] & j_1}$ is an interior arrow coming from $W_2$.
\item[(2.1)] $(i)$ for $\zg$ and $(j'), (i\to j)$ for $R(\zg)$, where $i\to j$ is a boundary arrow.
\item[(2.2)] $(\xymatrix@C=10pt{i_1\ar@{-}[r] & j_1})$ for $\zg$ and $(j'), (\xymatrix@C=10pt{j_1\ar@{-}[r] & i_1})$ for $R(\zg)$, where $\xymatrix@C=10pt{i_1\ar@{-}[r] & j_1}$ is an interior arrow coming from $W_2$. 
\item[(2.3)] $(i'\to k')$ for $\zg$ and $(k' \to j')$ for $R(\zg)$, where there is a 3-cycle $\xymatrix@C=10pt{i'\ar[r] & k'\ar[r] & j' \ar@/^1pc/[ll]}$ with $j'\to i'$ a boundary arrow. 
\item[(2.4)] $(i'), (\xymatrix@C=10pt{i_1\ar@{-}[r] & j_1})$ for $\zg$ and $(\xymatrix@C=10pt{j_1\ar@{-}[r] & i_1})$ for $R(\zg)$, where $\xymatrix@C=10pt{i_1\ar@{-}[r] & j_1}$ is either an interior arrow coming from $W_3$ or the other boundary arrow coming from $W_3$, provided $W_3$ has a boundary edge. 
\item[(2.5)] $(i'), (\xymatrix@C=10pt{i_1\ar@{-}[r] & j_1})$ for $\zg$ and $(j_1)$ for $R(\zg)$, where $\xymatrix@C=10pt{i_1\ar@{-}[r] & j_1}$ is the other boundary arrow from $W_3$, provided $W_3$ has no boundary edge. 
\end{itemize}

We only provide the details of the proof in the cases (1.1), (2.3), and (2.4), because the other cases are proved in a similar way. 

\begin{figure}
\centerline{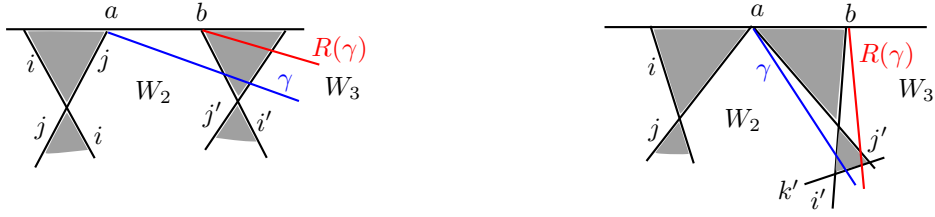}
\caption{Cases (1.1) and (2.3) in the proof of Lemma~\ref{lem:base}.}
\label{fig:12cases}
\end{figure}

{\bf Case (1.1)}  In this case the arc $\zg$ starts in vertex $a$, enters the white region $W_2$, then crosses $i', j'$ and enters the white region $W_3$.  Then the corresponding $R(\zg)$ starts in vertex $b$, crosses only $j'$ and enters region $W_3$, see Figure~\ref{fig:12cases}.   

By assumption $0=f_{\zg}(x)$, which means that the composition of the valid path $i'\to j'$  with the path $x_1$ is zero in $B$.  Hence $x_1$ is a path $j'\leadsto c$ starting in $j'$ and ending in some vertex $c$.  Since $x_1\not=0$, Lemma~\ref{zero-relations-lemma}(b) shows that the path $j'\leadsto c$ starts with a path $j'\leadsto  i_2\leadsto z$, and there is a boundary arrow $\partial_2$ starting at $z$ as in the quiver below.  The black portion of the quiver lies in $Q(\zg)$ and the red portion does not.  Moreover, $\partial_1$ is a boundary arrow. 

\[\xymatrix@C=10pt@R=10pt{
i'\ar[r] \ar[d]_{\partial_1}& \bullet \ar[r] \ar[d]& \cdots \ar[r] & j_2 \ar@[red][r] & \cdots \ar@[red][r] & \bullet\ar@[red][d]\\
j'\ar[r] & \bullet \ar[r] \ar[ul]& \cdots \ar[r] & i_2 \ar@[red][r] & \cdots \ar@[red][r] & \bullet\ar@[red][r] & z \ar@[red][r] \ar@[red][ul]_{\partial_2}& \cdots \ar@[red][r] & c
}
\]

 The arc $R(\zg)$ does not cross $i'$, so the map $\bar{f}_{R(\zg)} = \begin{bsmallmatrix}j'\leadsto i_2 & \cdots \\ -(j_2\leadsto i_2) & \cdots \\ 0 & \cdots \\ \vdots & \vdots\end{bsmallmatrix}$ has only two valid paths in the first column as all steps in $\zg$ are assumed to be forward.   Thus 

\[ \bar{f}_{R(\zg)} \begin{bsmallmatrix}i_2\leadsto z \leadsto c\\0\\ \vdots \\ 0\end{bsmallmatrix}=\begin{bsmallmatrix}j'\leadsto i_2\leadsto z \leadsto c\\ -(j_2 \leadsto i_2\leadsto z\ \leadsto c)\\0\\ \vdots \\ 0\end{bsmallmatrix}=x,\]
because the path $j_2\leadsto i_2\leadsto z$ is zero as $\partial_2$ is a boundary arrow.  This shows that $x\in \textup{Im}\,\bar{f}_{R(\zg)}$ and proves the lemma in this case. 

{\bf Case (2.3)} This case is special, since it is the only one that involves the diagonal $k'$, see Figure~\ref{fig:12cases}. 
By assumption $0=f_{\zg}(x)$, which means that the composition of the valid path $i'\to k'$ with $x_1$ is zero in $B$.  Hence $x_1$ is a path $k'\leadsto c$ starting in $k'$ and ending in some vertex $c$.  Consider the following subquiver of $Q$, where $\partial_1, \partial_2$ are boundary arrows. The black portion lies in $Q(\zg)$ and the red portion does not. Then by Lemma~\ref{zero-relations-lemma}, either $k'\leadsto c$ factors through $i_2$ and $z$, and then we proceed as in case (1.1), or $k'\leadsto c$ factors through $j'$.   

\[\xymatrix@C=10pt@R=10pt{
j'\ar[r]^{\partial_1}&  i'\ar[r]  \ar[d]& \bullet \ar[r] \ar[d]& \cdots \ar[r] & j_2 \ar@[red][r] & \cdots \ar@[red][r] & \bullet\ar@[red][d]\\
&k'\ar[r] \ar[ul]& \bullet \ar[r] \ar[ul]& \cdots \ar[r] & i_2 \ar@[red][r] & \cdots \ar@[red][r] & \bullet\ar@[red][r] & z \ar@[red][r] \ar@[red][ul]_{\partial_2}& \cdots \ar@[red][r] & c
}
\]

  The arc $R(\zg)$ crosses $j', k'$ but not $i'$, so the map $\bar{f}_{R(\zg)} = \begin{bsmallmatrix}k'\leadsto j' & \cdots \\ 0 & \cdots \\ \vdots & \vdots\end{bsmallmatrix}$ has only one valid path in the first column.  Then $\bar{f}_{R(\zg)} \begin{bsmallmatrix}j' \leadsto c\\0\\ \vdots \\ 0\end{bsmallmatrix}= x$, which shows that $x\in \textup{Im}\,\bar{f}_{R(\zg)}$ and proves the lemma in this case. 

\begin{figure}
\centerline{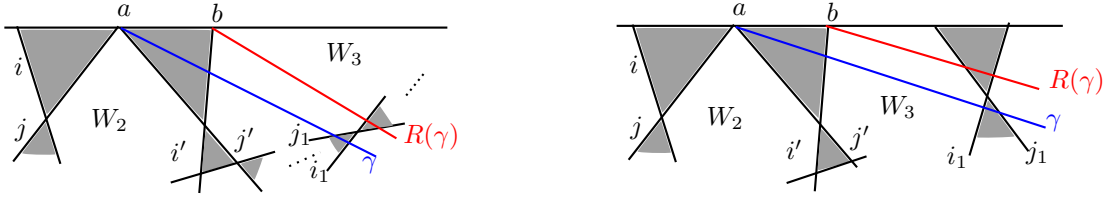}
\caption{Case (2.4) in the proof of Lemma~\ref{lem:base}.}
\label{fig:3case}
\end{figure}

{\bf Case (2.4)}  In this case, the arc $\zg$ starts in vertex $a$ crosses $i'$ and moves into region $W_3$, where it crosses a pair $i_1, j_1$.  Then $R(\zg)$ starts in vertex $b$ and then crosses $i_1, j_1$ in $W_3$, see Figure~\ref{fig:3case}.  On the left, we show the case when the arrow between $i_1, j_1$ is an interior arrow of $W_3$, and on the right, we show the case when the arrow between $i_1,j_1$ is a boundary arrow of $W_3$.   In the latter case, we must have $j_1\to i_1$, because white regions have an even number of sides, by Lemma~\ref{lem orientation}, and $i', j_1$ must have the same parity along the edges of $W_3$, by Remark~\ref{rem:parity}. In the former case, the arrow between $i_1$ and $j_1$ may be in either direction.  We have the following subquiver of $Q$, where $\partial_1$ is a boundary arrow.

\[\xymatrix@C=10pt@R=10pt{
j'\ar[r]\ar[d]_{\partial_1} & \bullet\ar[r] \ar[d]& \cdots \ar[r] & i_1\\
i'\ar[r] & \bullet\ar[ul] \ar[r] & \cdots \ar[r] & j_1
}
\]

Then $x_1$ is a path $j_1 \leadsto c$ and by assumption $0=f_{\zg}(x) = \begin{bsmallmatrix} i'\leadsto j_1 \leadsto c\\ i_1 \leadsto j_1 \leadsto c \\ 0 \\ \vdots \\ 0\end{bsmallmatrix}$.  We only prove the case when there is an arrow $i_1\xrightarrow{\alpha} j_1$, since the other case when there is an arrow $j_1\to i_1$ follows similarly.   Then $i_1\to j_1$ is not a boundary arrow, and by Lemma~\ref{zero-relations-lemma} the vanishing of $i_1\to j_1\leadsto c$ implies that $j_1\leadsto c$ factors through $i_1$ or $z$, where $z$ is either above or to the right of $j_1$ as in the quiver below.   However, the latter case is not possible because $f_{\zg}(x)=0$ implies that the valid path $i'\leadsto j_1$ composed with $x_1$ is also zero.  

\[\xymatrix@C=10pt@R=10pt{
&&& c \\
&&& \raisebox{0pt}[0.9\height][0.3\height]{ $\vdots$ } \ar@[red][u]\\
&&& z \ar@[red][r]^{\partial_2} \ar@[red][u] & \bullet\ar@[red][dl]\\
&&& \raisebox{0pt}[0.9\height][0.3\height]{ $\vdots$ }   \ar@[red][u]&\raisebox{0pt}[0.9\height][0.3\height]{ $\vdots$ }  \ar@[red][u]\\
&&&   i_2 \ar@[red][u]&  j_2 \ar@[red][u]\\
&&& \raisebox{0pt}[0.9\height][0.3\height]{ $\vdots$ }   \ar[u]& \raisebox{0pt}[0.9\height][0.3\height]{ $\vdots$ } \ar[dl] \ar[u]\\
j'\ar[r]\ar[d]_{\partial_1} & \bullet\ar[r] \ar[d]& \cdots \ar[r] &\bullet\ar[r] \ar[u] &  i_1\ar[d] \ar[u] \ar@[blue][r] & \cdots \ar@[blue][r] & \bullet \ar@[blue][d]\\
i'\ar[r] & \bullet\ar[ul] \ar[r] & \cdots \ar[r] & \bullet\ar[r]& j_1\ar[ul] \ar@[blue][r] & \cdots \ar@[blue][r] & \bullet  \ar@[blue][r] & z  \ar@[blue][r]   \ar@[blue][ul]_{\partial_2} & \cdots  \ar@[blue][r] & c
}
\]

We have $\bar{f}_{R(\zg)}= \begin{bsmallmatrix} j_1\leadsto i_1 & -(j_1\leadsto i_2) & \cdots \\ 0 & j_2 \leadsto i_2 & \cdots \\
\vdots & 0 & \cdots \\ 
\vdots & \vdots & \ddots \end{bsmallmatrix}$.
If $j_1\leadsto c$ factors through $i_1$ then $\bar{f}_{R(\zg)} \begin{bsmallmatrix} i_1 \leadsto c \\ 0 \\ \vdots \\0 \end{bsmallmatrix}=x$.  If  $j_1\leadsto c$ factors through $i_2$ then $\bar{f}_{R(\zg)} \begin{bsmallmatrix} 0 \\ i_2 \leadsto c \\ 0 \\ \vdots \\0 \end{bsmallmatrix}=x$, because $j_2\leadsto i_2 \leadsto c$ is zero as $\partial_2$ is a boundary arrow.   This completes the proof in case (2.4). 
\end{proof}

\begin{lemma}\label{lem:factor}
Let $\zg$ be an arc in $\cals$ such that the crossing sequence for $\zg$ consists of $n-1$ steps that go forward from $1$ to $n$.  Let $x = \begin{bsmallmatrix}x_1& \cdots & x_k &  0 & \cdots & 0\end{bsmallmatrix}^T \in \textup{ker}\, f_{\zg}$, where $x_k$ is a nonzero path in $Q$ starting in $j_k$, for $k\geq 2$.  Then $x_k$ factors through $i_k$ or $i_{k+1}$, if $k\not=n$. 
\end{lemma}

\begin{proof}
Let $f_{\zg}$ be given by a matrix $A = (a_{h,l})$.  Let $A_l$ denote the $l$-th row of $A$. Because all the steps in the crossing sequence for $\zg$ are forward $a_{h,l}=0$ for $h>l$.   In particular, in $A_k$ the first $k-1$ entries  $a_{k, 1}, a_{k, 2}, \dots, a_{k, k-1}$ are zero.  Then $A_k x= a_{k,k} x_k$, where $a_{k,k}$ is a path $i_k\leadsto j_k$.   We have $a_{k,k} x_k=0$, because $x\in  \text{ker}\, f_{\zg}$ by assumption.   

The path $x_k: j_k \leadsto c$ is ending in some vertex $c$ in $Q$ and $a_{k,k} x_k$ is a path $i_k\leadsto j_k \leadsto c$ in $Q$ that is zero in the algebra $B$. We consider two cases depending on whether $a_{k,k}$ is given by a single arrow or a composition of two arrows. 

(i) Let $a_{k,k} = i_k \xrightarrow{\alpha} j_k$ be an arrow.  Because $k\not=1,n$, and the steps in the crossing sequence for $\zg$ are all forward, we obtain the following full subquiver of $Q(\zg)$ represented by black arrows.  Next, we explain the remaining colored part of this quiver.  

\[\xymatrix@R=10pt@C=10pt{
&&i_{k+1}& j_{k+1}\\
&&\raisebox{0pt}[0.9\height][0.3\height]{ $\vdots$ }  \ar[u]& \raisebox{0pt}[0.9\height][0.3\height]{ $\vdots$ }  \ar[u]\\
&&\bullet\ar[r] \ar[u]& \bullet\ar[dl] \ar[u]\\
j_{k-1} \ar[r] & \cdots \ar[r]& \bullet \ar[d] \ar[r] \ar[u]& i_k \ar[u]\ar[d]_{\alpha} \ar@[red][r] & \bullet \ar@[red][r] \ar@[red][d]& \cdots \ar@[red][r] & \bullet \ar@[red][d]_{\beta_1}\\
i_{k-1} \ar[r] & \cdots \ar[r]& \bullet \ar[r] & j_k\ar[ul]^{}\ar@[red][r]_{\delta} & \bullet\ar@[red][r]\ar@[red][ul] & \cdots \ar@[red][r] & \bullet\ar@[red][r]_{\beta_2} & z\ar@[red][ul]_{\partial} \ar@[red][r] & \cdots \ar@[red][r] & c}\] 

By assumption the path $\alpha x_k= i_k \xrightarrow{\alpha}j_k \leadsto c$ is zero in $B$, where $x_k$ is nonzero. By Lemma~\ref{zero-relations-lemma}(a), $x_k$ factors through $i_k$ or, by Lemma~\ref{zero-relations-lemma}(b), $x_k$ factors through $i_{k+1}$ or through some vertex $z$ outside of $Q(\zg)$, where $\partial$ is a boundary arrow in $Q$ as in the figure above.   We must show that the latter is not possible.  Suppose on the contrary that 
\[x_k: j_k\xrightarrow{\delta} \bullet \to \cdots \to \bullet \xrightarrow{\beta_2} z \to \cdots \to c\]
as in the quiver above.

Now consider the product $A_{k-1}x=a_{k-1,k}x_k + a_{k-1, k-1} x_{k-1}=0$, a sum of paths starting at $i_{k-1}$.  Note that $x_{k-1}$ is a linear combination of paths starting in $j_{k-1}$, and  the path  $a_{k-1, k} x_k: i_{k-1}\leadsto j_k \leadsto c$ is nonzero, see the figure.  This means that $x_{k-1}$ contains $-x'$ as a summand, where $x'$ is a nonzero path $x': j_{k-1} \leadsto c$, such that in the sum $A_{k-1}x$ the two paths below add to zero.

\begin{equation}\label{eq:star}
\xymatrix{i_{k-1} \ar@{~>}[r]^-{a_{k-1, k}} & j_k\ar@{~>}[r]^-{x_k} & c \,\,\,- \,\,\, i_{k-1} \ar@{~>}[r]^-{a_{k-1, k-1}} & j_{k-1}\ar@{~>}[r]^-{x'} & c \,\, =\,\,0}  
\end{equation}

We will show that no such path $x'$ can exist.  Indeed, it is clear from the quiver that $x'$ is equivalent to a path that contains $\beta_1\beta_2$ as a subpath.   Because $\partial$ is a boundary arrow, we have $\beta_1\beta_2=0$, hence $x'=0$.  This is a contradiction, and proves the lemma in case (i). 

(ii)  Suppose $a_{k,k}=i_k \to \bullet \to j_k$ is a path of length two that forms a 3-cycle together with the arrow $j_k \xrightarrow{\alpha} i_k$.  Because the crossing sequence for $\zg$ is forward we have the following quiver.  

\[\xymatrix@C=10pt@R=10pt{
j_{k-1} \ar[r] & \cdots \ar[r]& i_k \ar[d] \ar@[red][r] & \bullet \ar@[red][d] \ar@[red][r] & \bullet \ar@[red][r] \ar@[red][d]& \cdots \ar@[red][r] & \bullet \ar@[red][d]_{\beta_1}\\
i_{k-1} \ar[r] & \cdots \ar[r]& \bullet \ar[r]^{\rho_1} \ar[d]& j_k \ar[d]^{\rho_2 }\ar[ul]_{\alpha}\ar@[red][r]_{\delta} & \bullet\ar@[red][r]\ar@[red][ul] & \cdots \ar@[red][r] & \bullet\ar@[red][r]_{\beta_2} & z\ar@[red][ul]_{\partial} \ar@[red][r] & \cdots \ar@[red][r] & c\\
&&\bullet\ar[d]\ar[r]&\bullet\ar[d]\ar[ul]\\
&& \raisebox{0pt}[0.9\height][0.3\height]{ $\vdots$ }   \ar[d] &\raisebox{0pt}[0.9\height][0.3\height]{ $\vdots$ }  \ar[d] \\
&&j_{k+1}&i_{k+1}
}\]

If $\alpha$ is not a boundary arrow, then $a_{k,k}\not=0$, but $a_{k,k}x_k=0$, and we proceed in the same way as in case (i). 
If $\alpha$ is a boundary arrow, then $a_{k,k}=0$.  Thus $A_k x = a_{k,k}x_k=0$, so the product of the $k$-th row of $A$ and $x$ does not provide any information, because it is zero for any $x_k$.  However, the $(k-1)$-st row of $A$ still gives equation~(\ref{eq:star}) for some summand $x'$ of $x_{k-1}$, provided that $a_{k-1,k}x_k$ is nonzero.    
Then we have the quiver as above without the part to the right of the arrow $\alpha$.  The arrow $\rho_1$ is the last arrow in $a_{k-1,k}$. If $x_k$ factors through $i_k$ or $i_{k+1}$ then the lemma follows, so suppose not.

Suppose first that $a_{k-1,k}x_k=0$.
Then the last arrow $\rho_1$ in $a_{k-1,k}$ and the first arrow in $x_k$ lie in a relation, and thus the first arrow in $x_k$ is $\alpha$ or $\rho_2$.  If the first arrow in $x_k$ is $\alpha$ then $x_k$ factors through $i_k$ contrary to our assumption above. Then the first arrow in $x_k$ is $\rho_2$ as in the picture. 
 By Lemma~\ref{zero-relations-lemma}, for $a_{k-1,k}x_k$ to be zero, the path $x_k$ must go through a 3-cycle that contains a boundary arrow that is not on $x_k$.   Thus, the quiver implies that $x_k$ must factor through $i_{k+1}$, a contradiction.   

Thus $a_{k-1,k}x_k\not=0$ and $A_{k-1}x$ is given by equation~(\ref{eq:star}).  By Lemma~\ref{lem:comm}, the path $x_k$ starts with the arrow $\alpha$, which contradicts the assumption that $x_k$ does not factor through $i_k$.  This proves case (ii). 
\end{proof}

\begin{lemma}\label{white-regions}
Let $W$ be a white region.  Label its edges $w_1, \dots, w_{2m}$ in clockwise order such that $w_1$ and $w_2$ meet in the interior of the polygon $\cals$ and $w_{2m}, w_1$ meet at the boundary vertex.  In particular, if $W$ has a boundary edge then it is the edge $w_{2m}$.  Let $\zg$ be an arc with endpoint in $W$.  
\begin{itemize}
\item[(1a)] If the endpoint of $\zg$ is incident to the edge $w_1$, then any crossing of $\zg$ with an even-indexed edge is of degree 0, and any crossing with an odd-indexed edge is of degree 1.
\item[(1b)]  If the endpoint of $\zg$ is not incident to the edge $w_1$, then $W$ has a boundary edge $w_{2m}$, any crossing of $\zg$ with an even-indexed edge is of degree 1, and any crossing with an odd-indexed edge is of degree 0. 
\item[(2)] The full subquiver $Q(W)$ of $Q$ is given below, where $\partial_1, \partial_2$ denote boundary arrows of $Q$ and depending on whether $W$ has a boundary edge (right) or not (left). 

\[\xymatrix@C=10pt@R=10pt{
w_1\ar[d]_{\partial_1} & w_3\ar[l]\ar[d] & \cdots \ar[l] & w_{2m-1}\ar[l]\ar[d]^{\partial_2} &&& w_1\ar[d]_{\partial_1} & w_3\ar[l]\ar[d] & \cdots \ar[l] & w_{2m-3}\ar[l]\ar[d] & \ar[l]w_{2m-1} \\
w_2\ar[ur] & w_4\ar[l] & \cdots \ar[l]\ar[ur]& w_{2m} \ar[l] &&& w_2\ar[ur] & w_4\ar[l] & \cdots \ar[l]\ar[ur]& w_{2m-2} \ar[l] \ar[ur]_{\partial_2}
}\] 
\end{itemize}
\end{lemma}

\begin{proof}
Parts (1a) and (1b) follow from the fact that the direction of the edges alternates around the white region, by Lemma~\ref{lem orientation}.  Part (2) follows from the construction, see for example Figure~\ref{fig:W-region}.
\end{proof}

We are ready to prove the complementary statement to Lemma~\ref{lem:kernel}.

\begin{lemma}\label{lem:ker-im}
Let $\zg, R(\zg)$ be a pair of compatible arcs in $\cals$.  Then $\textup{ker}\,f_{\zg} \subset \textup{Im}\,\bar{f}_{R(\zg)}$. 
\end{lemma}

\begin{proof}
Let $x\in \text{ker}\, f_{\zg}$.   First we prove the lemma in the case when all steps in the crossing sequence for $\zg$ are forward.  Let $k$ be the largest integer such that the $k$-th position in $x$ is nonzero.  We proceed by induction on $k$, and we outline the main steps of the proof below. 

{\flushleft Induction on $k$:
\begin{itemize}
\item [$k=1$] Apply Lemma~\ref{lem:base} to conclude the base case. 
\item[$k>1$]
{\begin{itemize}
\item[(a)] Write $x=x'+x_1$ with 
\begin{itemize}
\item[\small $\bullet$]$x', x_1\in \text{ker}\,f_{\zg}$, and 
\item[\small $\bullet$] $x_1$ has one less term in position $k$ than $x$ and is zero in positions greater then $k$, and
\item[\small $\bullet$] $x'\in \text{Im}\, \bar{f}_{R(\zg)}$.
\end{itemize}
\item[(b)] Repeat with $x_1$, where $x_1=x_1'+x_2$, and continue until we obtain $x_p\in \text{ker}\, f_{\zg}$ such that $x_p$ is zero in positions greater or equal to $k$.  By induction $x_p \in   \text{Im}\, \bar{f}_{R(\zg)}$, so $x=x_p+x_1' + x_2' + \dots + x_{p-1}' \in \text{Im}\, \bar{f}_{R(\zg)}$.
\end{itemize}}
\end{itemize}
}

Therefore, we need to show the base case $k=1$ and part (a) of the inductive step.  Then applying the argument in part (b) completes the proof.  

If $k=1$ then $x$ is nonzero only in the first position.  The first entry in $x$ is $\sum_{c\in Q_0} \lambda_c (j \leadsto c)$, a linear combination of paths starting in vertex $j$, where $j$ is the first vertex of degree 1 in the crossing sequence for $\zg$.  Moreover, we may assume that the sum runs over distinct vertices $c\in Q_0$, because parallel paths are equal in $B$, by Proposition~\ref{prop 39}.  Since $x\in \text{ker}\,f_{\zg}$, we obtain that the vector $\begin{bsmallmatrix}\lambda_c (j\leadsto c) & 0 & \cdots & 0\end{bsmallmatrix}^T \in \text{ker}\,f_{\zg}$, for all $c$.  By Lemma~\ref{lem:base}, every such vector is contained in  $\text{Im}\, \bar{f}_{R(\zg)}$, so by linearity we obtain that $x\in \text{Im}\, \bar{f}_{R(\zg)}$.  This completes the proof in the case $k=1$.  

Part (a).  Suppose $k>1$.  Similarly to the base case, the $k$-th entry of $x$ is a linear combination of paths starting in $j_k$, because the $k$-th column of $f_{\zg}$ consists of paths ending in $j_k$.   Moreover, we may assume again that the sum runs over paths ending in distinct vertices of $Q$.  Let $w: j_k \leadsto c$ be one of the nonzero paths in the position $k$ of $x$. By Lemma~\ref{lem:factor}, the path $w$ factors through $i_k$ or $i_{k+1}$, if $k\not=n$.   We will prove the case when $k=n$ and $w$ does not factor through $i_n$ later (call this case (c)), and now we assume that $w$ factors through $i_k$.  The case when $w$ factors through $i_{k+1}$ is similar. 

Then $w=w'u: \xymatrix{j_k \ar@{~>}[r]^{w'} & i_k \ar@{~>}[r]^{u} & c}$.  Let $s<k$ be the largest integer such that there is a valid path from $i_s$ to $j_{s+2}$ and, if no such $s$ exists, we let $s=1$.  This choice of $s$ can be seen in the matrix of $f_{\zg}$ as the largest row index for which there exists a non-zero entry on the third diagonal.  Let $r\leq s$ be the largest row index in $f_{\zg}$ such that there exists a valid path $i_r\leadsto j_{s+2}$.  Assuming $\zg$ crosses $j_1$, the matrix of $f_{\zg}$ is as follows, where the entry in position $(p,q)$ is labeled $\star$ if there is a valid path from $i_p$ to $j_q$ in that position and $p\not=q$.  For $p=q$ the symbol $\star$ denotes a path that is either a single arrow or a composition of two arrows in the same 3-cycle.  If $\zg$ does not cross $j_1$ then one needs to remove the first column of this matrix;  we call this case (d) and deal with it later.  

\[
\begin{smallmatrix}
&1 &&& r && s &  & s+2 && k-2 &  & k\\
1&\star & \star & \cdots & \cdots & 0 &\cdots & 0 & 0 & \cdots & 0 & 0 & 0\\
&0&\ddots & \ddots & \cdots & \vdots &&& \vdots & \vdots  & \vdots& \vdots& \vdots\\
r-1&&0 & \star & \star & 0 & 0 &\cdots & 0 & 0 & \vdots& \vdots& \vdots\\
{\color{blue} r} &&&{\color{blue} 0} &{\color{blue} \star} & {\color{blue} \star} & {\color{blue}\star} &{\color{blue} \cdots} & {\color{blue} \star} & {\color{blue} 0} & {\color{blue} \vdots} & {\color{blue}\vdots} & {\color{blue} \vdots}\\
&&&&\ddots & \ddots &\ddots & \vdots & \vdots & \vdots & \vdots& \vdots& \vdots\\
{\color{blue} s}&&&&&{\color{blue} 0}& {\color{blue} \star} & {\color{blue} \star} & {\color{blue} \star} & {\color{blue} 0} & {\color{blue} \vdots}& {\color{blue} \vdots}& {\color{blue}\vdots}\\ 
s+1&&&&&&0 & \star & \star & 0 & \vdots& \vdots & \vdots\\
s+2&&&&&&&0 & \star & \star & 0 & \vdots  & \vdots\\
&&&&&&&& \ddots & \ddots & \ddots & \ddots & \vdots\\
k-2&&&&&&&&&0 & \star & \star & 0 \\
k-1&&&&&&&&&&0 & \star & \star \\
{\color{blue} k}&&&&&&&&&&&\color{blue} 0& \color{blue} \star  \\  
&&&&&&&&&&&&&\ddots
\end{smallmatrix}
\]

Moreover, for each $h \in \{s+1, \dots, k-3\}$ the step in the crossing sequence for $\zg$ from $h+1$ to $h+2$ cannot be rectangular, because in that case, Lemma~\ref{rect-trap-lemma1}(1a) would imply that there is a valid path $i_h \leadsto j_{h+2}$ or $i_{h+1}\leadsto j_{h+3}$.  Then all these steps from $s+2$ to $k-1$ must be trapezoidal, and Lemma~\ref{rect-trap-lemma1}(2) implies that there is a valid path $j_{s+1}\leadsto i_k$, so part (3) of the same lemma yields valid paths $w_h': j_h \leadsto i_k$ for all $h\in \{ s+1, \dots, k-1\}$. 

Define 
\[x' = \begin{bsmallmatrix}0\\ \vdots \\ 0 \\ (-1)^{k-s-1} w_{s+1} \\ \vdots \\w_{k-2}\\ -w_{k-1}\\ w_{k} \\ 0 \\ \vdots \\ 0 \end{bsmallmatrix}\]
where $w_k$ is our path $w: \xymatrix{j_k \ar@{~>}[r]^{w'} & i_k \ar@{~>}[r]^{u} & c}$ and $w_h:= w_h' u = \xymatrix{j_h \ar@{~>}[r]^{w'_h} & i_k \ar@{~>}[r]^{u} & c}$ for $h\in \{s+1, \dots, k\}$.  In particular, all these paths end in the same subpath $u$.

Let the entries in the matrix $f_{\zg}$ be denoted by $a_{i,j}$.  Then by construction 
\[f_{\zg}(x')=\begin{bsmallmatrix} 0\\ \vdots \\0\\ (-1)^{k-s-1}(a_{r,s+1}w_{s+1}-a_{r,s+2}w_{s+2})\\
\vdots\\
(-1)^{k-s-1}(a_{s+1,s+1}w_{s+1}-a_{s+1,s+2}w_{s+2})\\
(-1)^{k-s-2}(a_{s+2, s+2}w_{s+2}-  a_{s+2,s+3}w_{s+3})\\
\vdots\\
-(a_{k-1,k-1}w_{k-1}-a_{k-1,k}w_k)\\
a_{k,k}w_k\\
0\\
\vdots\\
0 \end{bsmallmatrix}\]
where $a_{k,k}w_k: i_k\leadsto j_k \leadsto i_k \leadsto c$ is a path that includes an oriented cycle from $i_k$ to $i_k$ and hence it is zero by Proposition~\ref{prop 38}.   All other entries are differences of parallel paths and hence zero by Proposition~\ref{prop 39}.  Thus $x'\in \text{ker}\,f_{\zg}$. 

Let $x_1 = x-x'$, then $x_1\in \text{ker}\,f_{\zg}$, since both $x,x'\in \text{ker}\,f_{\zg}$.  Moreover, $x_1$ has one less term in position $k$ than $x$ and is zero in positions greater than $k$. We will show that $x'\in \text{Im}\, \bar{f}_{R(\zg)}$.  Let $y' = \begin{bsmallmatrix}0 & \cdots & 0 & u & 0 & \cdots & 0\end{bsmallmatrix}^T$, where $u$ is our path $u: i_k \leadsto c$ and it sits in the $l$-th row in $y'$ where $l$ is such that the $l$-th column of $\bar{f}_{R(\zg)}$ consists of valid paths ending in $i_k$.   Thus, the $l$-th column of $\bar{f}_{R(\zg)}$ is 

\[\begin{bsmallmatrix} 0 \\ \vdots \\ 0\\ 
\pm w'_{s+1}\\
\vdots\\
\pm w'_{k-1}\\
\mp w'_k\\
0\\
\vdots\\
0
\end{bsmallmatrix}\]
where the first $s$ entries are zero, because from the definition of $s$ we know there is a valid path $i_s\leadsto j_{s+2}$, thus there is no valid path $j_s\leadsto i_{s+2}$, by Lemma~\ref{valid-paths-lem}, and thus there is no valid path $j_s\leadsto i_k$, by Lemma~\ref{rect-trap-lemma1}(3).  Moreover, the signs in the column alternate, by definition of $\bar{f}_{R(\zg)}$.  Therefore, $\bar{f}_{R(\zg)}(y') = \pm x'$, so $x'\in \text{Im}\, \bar{f}_{R(\zg)}$.  This shows that $x=x'+x_1$ where $x', x_1$ satisfy the three conditions of part (a) in the outline of the proof.  This completes the proof of part (a) excluding the case when $k=n$ and $w$ does not factor through $i_n$ (case (c)) and excluding the case when $\zg$ does not cross $j_1$ (case (d)).  

Case (d).  Suppose that $\zg$ does not cross $j_1$. The map $f_{\zg}$ has the same structure as before, except we remove the first column.  If $s>1$, or $s=1$ and there is a valid path $i_1\leadsto j_3$ then the same argument as above goes through, so suppose $s=1$ and there is no valid path from $i_1$ to $j_3$.  Then the crossing sequence for $\zg$ starts with $(i_1), (i_2, j_2), \dots$, so $\zg$ starts in a shaded boundary region given by the diagonals $i_1, j_1$, then crosses the diagonal $i_1$, thereby entering a white region $W$, and then $\zg$ crosses $i_2, j_2$ exiting $W$, see Figure~\ref{fig:42}.  Note that if instead we drew $\zg$ starting at the endpoint of $i_1$, then moving to the right and crossing $j_1$, then $j_1$ would cross $\zg$ from left to right.  In particular, $P(j_1)$ would be a summand of $P_1(\zg)$, contrary to the assumption that the crossing sequence of $\zg$ starts with a vertex corresponding to a summand of $P_0(\zg)$.  Hence, Figure~\ref{fig:42} illustrates the only possibility for $\zg$ under the given assumptions. 

\begin{figure}
\centerline{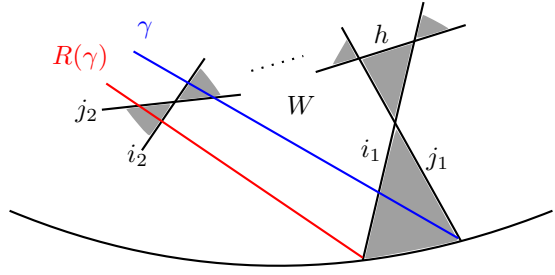}
\caption{Case (d) in the proof of Lemma~\ref{lem:ker-im}.}
\label{fig:42}
\end{figure}

The quiver $Q(W, \zg)$ is the full subquiver of $Q$ whose vertices correspond to the edges of the white region $W$ moving from $i_1$ to $i_2, j_2$ in the counterclockwise direction.  Since $i_1$ is the first edge in $W$ it has precisely two neighbors $j_1$ and $h$ in $Q(W, \zg)$, whereas $j_1$ is the second edge in $W$ and therefore it has precisely three neighbors in $Q(W, \zg)$.   Since $s=1$ and there is no valid path from $i_1$ to $j_3$, we have the following subquiver of $Q(\zg)$, where $\partial$ is a boundary arrow.

\[
\xymatrix@R=10pt@C=10pt{
&&& i_3 & j_3\\
&&&\raisebox{0pt}[0.9\height][0.3\height]{ $\vdots$ } \ar[r]\ar[u]& \raisebox{0pt}[0.9\height][0.3\height]{ $\vdots$ } \ar[u]\ar[dl]\\
j_1\ar[d]_{\partial} \ar[r] & \bullet\ar[r]\ar[d]&\cdots \ar[r] & p \ar[r] \ar[d]\ar[u]& i_2 \ar[d]\ar[u]\\
i_1\ar[r] & h \ar[r] \ar[ul]& \cdots  \ar[r] & \bullet \ar[r] & j_2\ar[ul]
}
\]

Define $x' = \begin{bsmallmatrix}\pm w_2 \\ \vdots \\ -w_{k-1}\\ w_k\\ 0\\ \vdots \\ 0 \end{bsmallmatrix}$, where $w_k$ are as before.  Then $f_{\zg}(x')= \begin{bsmallmatrix}\pm (i_1\leadsto j_2) w_2 \\ \pm(i_2\leadsto j_2)w_2 \mp (i_2\leadsto j_3)w_3 \\\vdots \\ -(i_{k-1}\leadsto j_{k-1})w_{k-1} + (i_{k-1}\leadsto j_k) w_k\\ (i_k\leadsto j_k) w_k\\ 0\\ \vdots \\ 0 \end{bsmallmatrix}$.   

The first entry is zero, because the path $w_2=j_2\to p \to i_2\leadsto c$, and the path from $i_1$ to $p$ is zero, since $\partial$ is a boundary arrow.  The proof that the other entries are zero is similar to the previous case.  Moreover, $x'=f_{R(\zg)}(y')$ with $y'$ as before.  This completes the proof in case (d).  

Case (c).  It remains to consider the case where $k=n$ and $w$ does not factor though $i_n$.  There are two possibilities depending on whether $\zg$ crosses $i_n$ or not. 

First, suppose that $\zg$ crosses $i_n$.   By the same argument as in the proof of Lemma~\ref{lem:factor}, we are in one of the following situations, because $w:\xymatrix{j_n \ar@{~>}[r]^{\bar{w}} & z \ar@{~>}[r]^{\bar{u}} & c}$ does not factor through $i_n$.   The two situations below depend on the direction of the arrow $\alpha_n$ between $i_n$ and $j_n$.  Here $\partial$ denotes a boundary arrow.

\[\xymatrix@C=10pt@R=10pt{
&& c &   &&&&  j_{n-1} \ar[r] & \cdots \ar[r] & i_n \ar[d]\\
&& \raisebox{0pt}[0.9\height][0.3\height]{ $\vdots$ } \ar@[red][u]   &&&&& i_{n-1}\ar[r] & \cdots \ar[r] & p\ar@[red][d] \ar[r] & j_n \ar[ul]_{\alpha_n}\ar@[red][d]\\
&& z \ar@[red][r]^{\partial} \ar@[red][u] & \bullet\ar@[red][dl]  && &&&&  \raisebox{0pt}[0.9\height][0.3\height]{ $\vdots$ } \ar@[red][d]&  \raisebox{0pt}[0.9\height][0.3\height]{ $\vdots$ }\ar@[red][d]  \\
&& \raisebox{0pt}[0.9\height][0.3\height]{ $\vdots$ }   \ar@[red][u]&\raisebox{0pt}[0.9\height][0.3\height]{ $\vdots$ }  \ar@[red][u]   && &&&& \bullet \ar@[red][r]&\bullet\ar@[red][ul]\ar@[red][d]\\
j_{n-1}\ar[r] & \cdots \ar[r] &p \ar[r] \ar@[red][u] &  i_n\ar[d]^{\alpha_n} \ar@[red][u]   && &&&& & z\ar@[red][d]\ar@[red][ul]^{\partial} \\
i_{n-1} \ar[r] & \cdots \ar[r] & \bullet\ar[r]& j_n\ar[ul] && &&&&&  \raisebox{0pt}[0.9\height][0.3\height]{ $\vdots$ }\ar@[red][d]\\
&&&&&&&&&& c
}\]

It suffices to show that $R(\zg)$ crosses $z$, because then $\bar{f}_{R(\zg)} \begin{bsmallmatrix} 0\\ \vdots \\0\\ \bar{u} \end{bsmallmatrix} = x'$.    Suppose not.    Then $\zg$ crosses $i_n, j_n$, enters a white region $W$ and ends on the boundary of $W$, because $(i_n, j_n)$ is the last pair in the crossing sequence for $\zg$.  Label the edges of $W$ as in the statement of Lemma~\ref{white-regions}.  Because $R(\zg)$ does not cross $z$, then $w_{2m}$ is a boundary segment, the endpoint of $\zg$ is incident to the edges $w_{2m}$ and $w_{2m-1}$, and $R(\zg)$ has endpoint incident to the edges $w_{2m}$ and $w_1$.    By Lemma~\ref{white-regions}(1b), $j_n$ is an edge of $W$ with even index, say $j_n=w_{2h}$, and $i_n$ is an edges  with odd index.  If $\alpha_n: i_n\to j_n$ then $i_n =w_{2h-1}$ and if $\alpha_n: j_n\to i_n$ then $i_n=w_{2h+1}$, by Lemma~\ref{white-regions}(2).   Moreover, since $z$ is a source of a boundary arrow, we have $z=w_1$, and thus $z$ lies on the same side as $i_n$ in $Q(W)$, meaning there is a valid path $i_n  = w_{2h+1}\to w_{2h-1}\to \dots \to w_3\to w_1 = z$.  This contradicts the two quivers above and completes the proof of part (c) when $\zg$ crosses $i_n$. 

To show the remaining part of (c), suppose $\zg$ does not cross $i_n$ and we still have that $k=n$.   Then $\zg$ ends in a shaded boundary region formed by the diagonals $i_n, j_n$.  Since $\zg$ and $R(\zg)$ both cross $j_n$, by Proposition~\ref{lem:cross}(b), we have the left picture and not the middle one in Figure~\ref{fig:44}.  In particular, $R(\zg)$ crosses $i_n$.

\begin{figure}
\centerline{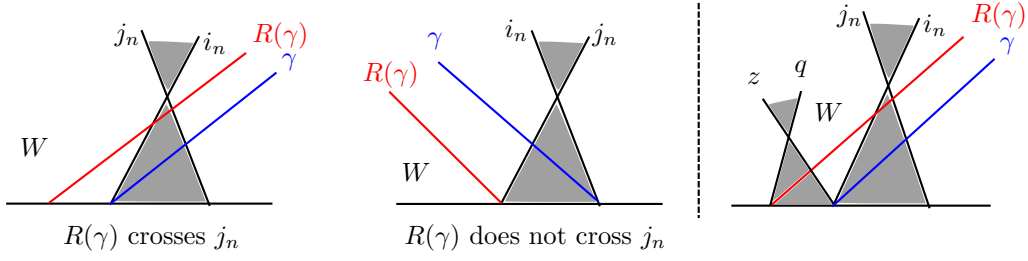}
\caption{Case (c) in the proof of Lemma~\ref{lem:ker-im}.}
\label{fig:44}
\end{figure}

Then the arrow $\alpha_n$ is oriented $\alpha_n: j_n\to i_n$, and the quiver is shown above on the right.   Then 

\[ f_{\zg} = \begin{bsmallmatrix} \ddots & \vdots & \vdots & \vdots \\
0 & i_{n-2}\leadsto j_{n-2} & i_{n-2}\leadsto j_{n-1} & \star\\
0 & 0 & i_{n-1}\leadsto j_{n-1} & i_{n-1} \leadsto j_n
\end{bsmallmatrix}\] 
where $\star$ is either 0 or a valid path $i_{n-2}\leadsto j_n$ if such a path exists.  The last row of $x$ contains our  nonzero path $w: j_n \leadsto c$.   So looking at the last row of $f_{\zg}$, either 

\begin{equation}\label{eq:1}
(i_{n-1}\leadsto j_n)w=0
\end{equation}
or there exists a path $w_{n-1}: j_{n-1}\leadsto c$ in the $(n-1)$-st row of $x$ such that 

\begin{equation}\label{eq:2}
(i_{n-1}\leadsto j_{n-1})w_{n-1}=(i_{n-1}\leadsto j_n)w.
\end{equation}

If $w$ factors through $i_n$, then the same argument as in the general case applies, since $R(\zg)$ crosses $i_n$.  If equation~(\ref{eq:2}) is satisfied then, by Lemma~\ref{lem:comm}, the path $w$ factors through $i_n$ and we are done.  Otherwise, equation~(\ref{eq:1}) is satisfied, and $w$ does not factor through $i_n$. 
Then $w$ factors through $z$, $w = \xymatrix{j_n \ar@{~>}[r]^{\bar{w}} & z \ar@{~>}[r]^{\bar{u}} & c}$ as in the previous quiver on the right, where $\partial$ is a boundary arrow.

Let $x'=\begin{bsmallmatrix}0\\\vdots\\0\\w\end{bsmallmatrix}$, then $f_{\zg}(x')=\begin{bsmallmatrix} \vdots \\ (i_{n-2}\leadsto j_n)w\\ (i_{n-1}\leadsto j_n)w\end{bsmallmatrix}=0$, 
because all paths $i_h\leadsto j_n$ factor through $p$ and $(p\to j_n)w=0$.   We will show that $R(\zg)$ crosses $z$.  Indeed, $z$ is the starting point of a boundary arrow $\partial$, and the number of 3-cycles between the two boundary arrows $\partial$ and $\alpha_n$ is even.  Therefore, the corresponding white region $W$ has no boundary edge, by Lemma~\ref{white-regions}(2).  Thus, we can complete our previous picture as in Figure~\ref{fig:44} on the right.  


In particular, $R(\zg)$ crosses $z=i_{n+1}$.  Then $\pm x'=\bar{f}_{R(\zg)} \begin{bsmallmatrix}0\\ \vdots \\ 0 \\ \bar{u}\end{bsmallmatrix} \in \text{Im}\,\bar{f}_{R(\zg)}$, as the last column of $\bar{f}_{R(\zg)}$ contains only one nonzero entry $\pm (j_n\to i_{n+1})$ in the last position.  This completes the proof in case (c).  

This finishes the argument of the inductive step and completes the proof of the lemma in the case when the crossing sequence for $\zg$ consists of all forward steps.  On the other hand, if none of the steps are forward then the proof is very similar using a dual argument working with the first nonzero entry of $x\in \text{ker}\, f_{\zg}$ instead of the last nonzero entry. 

Finally, if some of the steps are forward and some are not, we can start at a position  $t$ where the step from $t-1$ to $t$ is forward followed by a non-forward step from $t$ to $t+1$.   The matrix of $f_{\zg}$ at row $t$ contains a single valid path $i_t\leadsto j_t$ and otherwise has the following shape, where all entries not marked by stars are zero. 

\[\begin{bsmallmatrix}
\ddots \\
& \star & \star & \cdots & \star \\
&  & \star & \cdots & \star  & \\
&&  \text{\large 0}& \ddots & \vdots & \text{\large 0}\\
{\color{blue} } && & & {\color{blue} \star} \\
&&&\text{\large 0}&\vdots & \ddots & \text{\large 0}\\
&&&& \star & \cdots & \star &\\
&&&& \star & \cdots & \star & \star\\
&&&&&&&&\ddots
\end{bsmallmatrix} = 
\begin{bsmallmatrix}\ddots \\
&A_t^- &\vdots & \text{\large 0}\\
&\cdots &{\color{blue} \star }& \cdots \\
& \text{\large 0}&\vdots& A_t^+\\
&&&&\ddots 
\end{bsmallmatrix}\]

It is upper triangular in the block corresponding to the forward steps and lower triangular in the block corresponding to the non-forward steps. Then we can work with the $t$-th row of $x\in \text{ker}\, f_{\zg}$, where $t$ is the last position for the forward step $t-1$ to $t$ and the first position for the non-forward step $t$ to $t+1$, working in both directions at the same time.  In the quiver, the pair $(i_t, j_t)$ is a sink for the crossing sequence 

\[\xymatrix@R=10pt@C=10pt{
j_{t-1} \ar[r] & \cdots \ar[r] & i_t \ar[d] & \cdots \ar[l] & j_{t+1}\ar[l]\\
i_{t-1}\ar[r] & \cdots \ar[r]& j_t \ar[ul]\ar[ur] & \cdots \ar[l] & i_{t+1}\ar[l]
}
\]
and so the paths arriving at $i_t, j_t$ from $i_h, j_h$ with $h<t$ do not interact with those from $i_h, j_h$ with $h>t$.  In particular, we can decompose $x=x'+x_1$ using the same algorithm as before at row $t$ treating $t$ as the last row of the forward block $\begin{bsmallmatrix}A_t^- & \vdots \\ 0 & {\color{blue} \star }\end{bsmallmatrix}$ and the first 
row of the non-forward block  $\begin{bsmallmatrix}{\color{blue} \star} & 0 \\ \vdots & A_t^+ \end{bsmallmatrix}$.
The case when a non-forward step from $s-1$ to $s$ is followed by a forward step from $s$ to $s+1$ is dual.  This way we reduce to the case when $x$ is zero in all rows where a forward step meets a non-forward step, and we can continue as before.  This completes the proof of the lemma. 
\end{proof}

We are now able to prove Proposition~\ref{big-lemma}.
\begin{prop}\label{Abig-lemma}
Let $\zg, R(\zg)$ be a pair of compatible arcs in $\cals$.  Then $\textup{ker}\,f_{\zg} = \textup{Im}\,\bar{f}_{R(\zg)}$. 
\end{prop}

\begin{proof}
This follows directly from Lemma \ref{lem:ker-im} and Lemma~\ref{lem:kernel}.
\end{proof}

\subsection{Indecomposibility}\label{Asect 5.5}

In this section we prove Lemma~\ref{lem:ab}, see  Lemma~\ref{Alem:ab} at the end of this subsection, required to establish that the cokernel of $f_{\zg}$ is indecomposable.  First we need a few preparatory lemmas. 

\begin{lemma}\label{lem:forward-arrow}
Suppose $i_s, j_t$ appear in the crossing sequence for $\zg$ with $s\not=t$. If there is an arrow $i_s\to j_t$ in $Q$ then the steps from $s$ to $t$ are forward, and if there exists an arrow $j_t\to i_s$ then the steps from $t$ to $s$ are forward. 
\end{lemma}

\begin{proof}
Suppose there exists an arrow $i_s\to j_t$ with $s\not=t$, and $\zg$ crosses both $i_s$ and $j_t$.   The other case follows similarly.  Moreover, we can suppose that $\zg$ crosse $i_s$ first and then $j_t$.   Then we are in the situation of Figure~\ref{fig:60}.  

\begin{figure}
\centerline{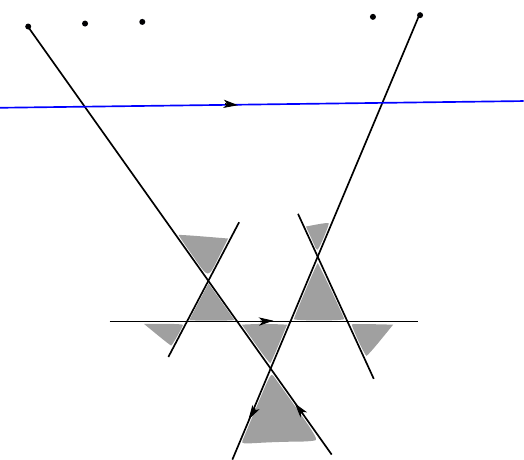}
\caption{Proof of Lemma~\ref{lem:forward-arrow}.}
\label{fig:60}
\end{figure}

Then $t\geq s$, and let $y,z$ denote the endpoints of $i_s, j_t$ respectively that lie to the left of $\zg$.  The arc $\zg$ crosses $i_s$, then traverses some white regions $W_s, \dots, W_{t-1}$ and some shaded regions, and then crosses $j_t$.  This gives a subsequence $(i_s), (i_{s+1}, j_{s+1}), \dots, (i_{t-1}, j_{t-1}), (j_t)$ of the crossing sequence for $\zg$ such that $\zg$ crosses $i_k, j_k$, traverses $W_k$, and crosses $i_{k+1}, j_{k+1}$ as it exits $W_k$ for $k\in\{s, \dots, t-1\}$.   Each white region $W_k$ has a vertex $x_k$ on the boundary of $\cals$.  These vertices must lie between $y$ and $z$ in order as in the picture.  In particular, we can start with an edge $i_k$ in $W_k$ and move counterclockwise around the edges of $W_k$ taking every other edge until we reach $j_{k+1}$.  In this way we never pass through a vertex of $W_k$ that lies on the boundary of $\cals$.  Therefore,  this sequence of edges gives a valid path $i_k\leadsto j_{k+1}$, see Lemma~\ref{lem 33}.  This means that the step from $k$ to $k+1$ is forward for all $k\in\{s, \dots, t-1\}$ and this completes the proof. The case when there is an arrow $j_t\to i_s$ follows in the same way.
\end{proof}

Suppose $\zg$ crosses $i$, and let $Q(\zg, [i])$ denote 
the full subquiver of $Q(\zg)$ consisting of all  3-cycles such that each 3-cycle contains the vertex $i$ and each of the  other two vertices is an $i$ or a $j$ vertex of the crossing sequence of $\zg$. 

\begin{lemma}\label{lem:circ}
Suppose $(i_s, j_s)$ is a crossing pair for $\zg$ such that there is an arrow $j_s\to i_s$.  If the path $i_s\leadsto j_s$ factors through some other $i$, or some other $j$ respectively, in the crossing sequence for $\zg$, then there exist $l, l'$ and $p,p'$ such that their quivers $Q(\zg, [i_s])$, or  $Q(\zg, [j_s])$ respectively, are as follows. 

\[
\xymatrix@C=10pt@R=10pt{
j_{s-l}\ar[drr]&i_p\ar[r]\ar[l]&j_{s-l+1}\ar[d]&i_{s-2}\ar[r]\ar@{.}[l]&j_{s-1}\ar[dll] && i_{s-l}  \ar[r]&j_p  \ar[dr]&i_{s-l+1}  \ar[l]&j_{s-2}  \ar[dl]\ar@{.}[l]&i_{s-1}  \ar[d]\ar[l]  \\
&&i_s\ar[rr]\ar[dr]\ar[dl]\ar[ur]\ar[ul]&&i_{s-1} \ar[u]\ar[d]&& &&j_s  \ar[ull]\ar[u]\ar[urr]\ar[drr]\ar[d]\ar[dll]&&j_{s-1} \ar[ll]\\
j_{s+l'}\ar[urr]&i_{p'}\ar[r]\ar[l]&j_{s+l'-1}\ar[u]\ar@{.}[r]&i_{s+1}\ar[r]&j_s\ar[ull] && i_{s+l'}\ar[r]&j_{p'}  \ar[ur]&i_{s+l'-1} \ar[l]\ar@{.}[r]&j_{s+1}\ar[ul]&i_s\ar[u]\ar[l]\\
}
\]

In particular, the quivers $Q(\zg, [i_s]), Q(\zg, [j_s])$ have connected dual graphs and consist of an even number of 3-cycles with $i$'s and $j$'s alternating around the boundary of the quivers.   Moreover, $p\leq s-l$ is the least positive integer such that there is an arrow $i_s\to i_p$ or $j_p\to j_s$ and  $p'\geq s+l'$ is the largest integer such that there is an arrow $i_s\to i_{p'}$ or $j_{p'}\to j_s$ respectively.
\end{lemma}

\begin{proof}
Suppose $(i_s, j_s)$ is a crossing pair for $\zg$ such that there is an arrow $j_s\to i_s$, and the path $i_s\leadsto j_s$ factors through some $i_r$.   Then there is a 3-cycle  $\xymatrix@R=10pt@C=8pt{i_s\ar[r] & i_r \ar[r] & j_s \ar@/_10pt/[ll] }$
in $Q$. Without loss of generality we may assume $r<s$.  This situation is illustrated in the left picture in Figure~\ref{fig:51}, where by assumption $\zg$ crosses $i_s, j_s, i_r$. 

\begin{figure}
\centerline{\scalebox{.8}{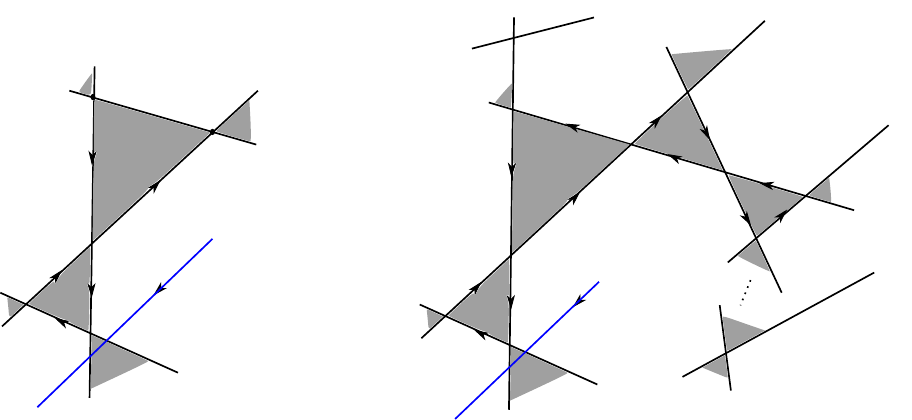}}
\caption{Proof of Lemma~\ref{lem:circ}.}
\label{fig:51}
\end{figure}

As we move in the direction of $\zg$, if the crossing of $\zg$ with $i_r$ occurs after the crossing of $\zg$ with $i_s, j_s$ then $i_r$ crosses $\zg$ from left to right.  This means that $i_r$ is a degree 1 crossing for $\zg$, which is a contradiction.  Therefore, the crossing of $\zg$ and $i_r$ occurs prior to the crossing of $\zg$ with $i_s, j_s$.  

Now, we claim that the arrow $i_s\to i_r$ is not a boundary arrow.   If $i_s \to i_r$ were a boundary arrow, then the segment with endpoints $x_1, x_2$, as in Figure~\ref{fig:51}, would be a boundary segment of $\cals$.  Since $\zg$ crosses $i_r$ before crossing $i_s, j_s$, it would follow that $\zg$ starts in $x_1$ or some other point counterclockwise from $x_1$.  If $\zg$ starts in $x_1$ then it has the same starting point as $i_s$, so $i_s$ cannot be in the crossing sequence for $\zg$.  If $\zg$ starts in another point counterclockwise from $x_1$ then $\zg$ crosses $i_s$ twice, which is also a contradiction.  This shows that $i_s\to i_r$ is not a boundary arrow. 

Then we have the following subquiver of $Q$. 

\[\xymatrix@C=10pt@R=10pt{&k\ar[dl]\\i_s\ar[rr] && i_p\ar[ul] \ar[dl]\\& j_s\ar[ul]}\]

We claim that $k=j_{s-1}$.  This case is illustrated in Figure~\ref{fig:51} on the right.


Let $W$ denote the white region containing $j_s, i_s, i_r, k$ on the boundary.  If $\zg$ crosses $k$ as it enters $W$ then we obtain $k=j_{s-1}$ as claimed.  Otherwise, $\zg$ crosses some other pair $i_{s-1}, j_{s-1}$ that lie clockwise from $k$ along the boundary of $W$. For example, $i_{s-1}, j_{s-1}$ may equal $k_1, k_2$ respectively.  In the figure we label the arcs for convenience so that there is an arrow $i_{s-1}\to j_{s-1}$, but this will not be important for the proof if we interchange the labels of these arcs.   Let $W'$ denote the white region containing $k_2, k, k_1, i_r$.  As we move along $j_{s-1}$ in the direction of its orientation we note that if $j_{s-1}$ crosses $i_r$ then we obtain that $W'$ is an interior white region, which is not possible.  Therefore, $j_{s-1}$ and $i_r$ do not cross.  Let $z, y$ denote the endpoints of $i_r, j_{s-1}$ respectively.  Then $y$ is to the right of $i_r$.  Moreover, since $\zg$ crosses $i_{s-1}, j_{s-1}$, then $\zg$ is to the right of $i_r$ so it cannot cross $i_r$, which is a contradiction.  This shows the claim that $k=j_{s-1}$.   

If there are no other arrows $i_s \to i_{s'}, i_s\to j_{s'}$ with $s'<s$ and $i_{s'}, j_{s'}$ in the crossing sequence of $\zg$, then $l=1, r=p$ and the lemma holds.   Note that the steps from $r$ to $s$ are forward by Lemma~\ref{lem:forward-arrow} as there is an arrow $i_r\to j_s$.  

Now, suppose that there is some other arrow $i_s\to q$ with $q\in\{i_{s'}, j_{s'}\}$ for $s'<s$.  Then $\zg$ crosses $q$ prior to crossing $i_{r}$ and $j_{s-1}$.  We claim that in this case $i_r=i_{s-1}$.  Suppose to the contrary $i_r\not=i_{s-1}$.  We have $j_{s-1}=k$ and $\zg$ crosses $(k_1, k)$ as it enters the white region $W$, see the picture above.   Moreover, we observe that $k_1$ and $q$ cannot cross, otherwise the white region $W''$ with edges $k_1, i_r, k, i_s$ would be interior.  But this means that $\zg$ would have to cross $k_1$ again before it can cross $q$, a contradiction.  This shows that $i_r=i_{s-1}$. 

Now, given that there is some other arrow $i_s\to q$ with $q\in\{i_{s'}, j_{s'}\}$ for $s'<s$, we claim that there exists $i_{r_1}$ in the crossing sequence for $\zg$ with $r_1<s-1$ and such that $i_{r_1}, i_s, j_{s-1}$ are in the same 3-cycle.  That is, we want to show that we have the following quiver.

\[\xymatrix@C=10pt@R=10pt{i_{r_1}\ar[r]&k=j_{s-1}\ar[dl]\\i_s\ar[rr] \ar[u]&& i_r=i_{s-1}\ar[ul] \ar[dl]\\& j_s\ar[ul]}\]

Consider Figure~\ref{fig:55}.  We see that in order for $\zg$ to cross $q$, it would have to cross $h$ as well, otherwise the white region $W'''$ would have to be interior.   Moreover, $h$ is of degree 0 in the crossing sequence for $\zg$  so $h=i_{r_1}$ for some $r_1<s-1$.   This shows the claim that $i_{r_1}, i_s, j_{s-1}$ are in the crossing sequence for $\zg$ and form a 3-cycle.

\begin{figure}
\centerline{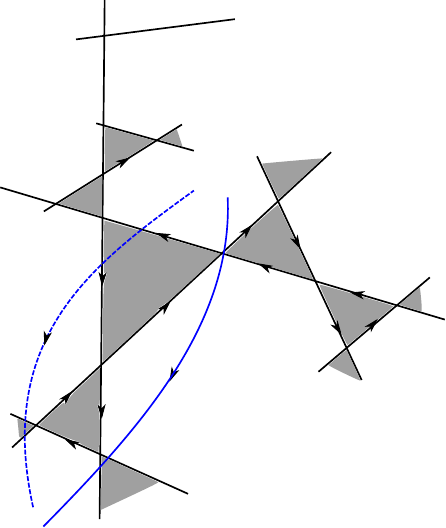}
\caption{Proof of Lemma~\ref{lem:circ}.}
\label{fig:55}
\end{figure}

Finally, we can homotope $\zg$ to obtain a new arc $\zg'$ that has the same crossing sequence as $\zg$ except we replace the crossing pairs $(i_s, j_s), (i_{s-1}, j_{s-1})$ in $\zg$ by $(i_{s-1}, j_s), (i_s, j_{s-1})$ in $\zg'$.  Now we can continue the same argument as before replacing $\zg$ with $\zg'$ and $i_s, j_s$ with $i_s, j_{s-1}$.  In this way we obtain the desired part of the quiver $Q(\zg, [i_s])$ coming from the crossings of $\zg$ before the pair $(i_s, j_s)$.  Moreover, if $i_s\leadsto j_s$ factors through another $i_{r'}$ with $i_{r'}\not=i_r$ and $r'>s$, then we can proceed in the same way as above analyzing the crossings of $\zg$ after it crosses $(i_s, j_s)$.  This shows the lemma. 
\end{proof}

\begin{lemma}\label{Alem:ab}
Let $\zg$ be a representative of a 2-diagonal in $\cals$, and let $f_{\zg}: \bigoplus_{h=1}^{n} P(j_h) \to  \bigoplus_{h=1}^m P(i_h)$ be the associated morphism.
Suppose we have a commutative diagram 
\[\xymatrix{
 \bigoplus_{h=1}^n P(j_h) \ar[r]^{f_{\zg}} \ar[d]^{g_j}&  \bigoplus_{h=1}^m  P(i_h) \ar[d]^{g_i} \\
 \bigoplus_{h=1} ^n P(j_h) \ar[r]^{f_{\zg}} &  \bigoplus_{h=1}^ m P(i_h)
}\]
for some matrices $g_j =(a_{h,h'})$ and $g_i = (b_{h,h'})$. Then $f_{\zg_{s,t}}=0$ or $b_{s,s}=a_{t,t}$ for all $s,t$.
\end{lemma}

\begin{proof}
By commutativity we have $g_i f_{\zg} = f_{\zg} g_j$ as matrices.  The diagonal entries $a_{t,t}, b_{s,s}$ in $g_j, g_i$ are elements in the field $\kb$ and off diagonal entries are non-constant paths $j_h\leadsto j_{h'}$ or  $i_h\leadsto i_{h'}$.  Note that the statement of the lemma is about the diagonal entries only.  Denote $f_{\zg}$ simply by $f$, and consider position $(s,t)$ in this composition  

\[ \sum_h b_{s,h} f_{h,t} = \sum_h  f_{s,h} a_{h,t}\]
where we index the position $(s,t)$ in $f$ so that it corresponds to a path from $i_s$ to $j_t$.  Then we get  

\begin{equation}\label{eq:11}
 b_{s,s}f_{s,t} + \sum_{h\not=s} b_{s,h} f_{h,t} = f_{s,t} a_{t,t} + \sum_{h\not=t} f_{s,h} a_{h,t}.
\end{equation}

If $f_{s,t}=0$ then the conclusion holds, thus suppose $f_{s,t}\not=0$.   Then by definition of the map $f=f_{\zg}$ all steps from $s$ to $t$ are forward and there is a valid path $i_s\leadsto j_t$, or $s=t$ and there is a nonzero path $i_s\leadsto j_s$.  Without loss of generality we can assume $t\geq s$.   

If the path $i_s\leadsto j_t$ does not factor through any other $i_h, j_{h'}$ then equation~(\ref{eq:11}) becomes $b_{s,s}f_{s,t}=f_{s,t}a_{t,t}$ with $f_{s,t}\not=0$ and the lemma holds.  In particular, if the path $i_s\leadsto j_t$ is just a single arrow then it is valid, so $f_{s,t}\not=0$, and it does not factor nontrivially through any other $i_h, j_{h'}$ and the lemma holds.   Therefore, we can assume that the path $i_s\leadsto j_t$ is a composition of two or more arrows. 

Now, suppose $f_{s,t}\not=0$ and the path  $i_s\leadsto j_t$ factors through some $i_h$ or $j_{h'}$ or both with $h\not=s, h'\not=t$.  Then we claim that $i_s\to i_h$ or $j_{h'}\to j_t$ or both of these paths consist of a single arrow.  If $s=t$ then the claim clearly holds as $f_{s,s}$ consists of at most two arrows.  If $s\not=t$, then $f_{s,t}$ is a valid path.    If $h, h' \in \{s+1, \dots, t-1\}$ then $i_s\leadsto i_h \leadsto j_t$ or $i_s\leadsto j_{h'}\leadsto j_t$ is not valid because $i_s\leadsto i_h, j_{h'}\leadsto j_t$ is not valid by Lemma~\ref{lem:ii-paths}.   Therefore, $h,h'\not\in \{s, \dots, t\}$.    Let $Q(\zg, [p,q])$ denote the subquiver of $Q(\zg)$ determined by the subsequence $(i_p, j_p), \dots, (i_q, j_q)$ of the crossing sequence for $\zg$. 
We know that step from $s$ to $s+1$ is forward, and if the step from $s-1$ to $s$ is not forward, then the two subquivers $Q(\zg, [1, s]), Q(\zg, [s, t])$  of $Q(\zg)$ only have two vertices $i_s, j_s$ connected by a single arrow in common.  In particular, in this case the path $i_s\leadsto j_t$ cannot factor through any $i_h, j_{h'}$ for $h, h' <s$.  This is a contradiction, so the step from $s-1$ to $s$ must be forward.  A similar argument implies that the step from $t$ to $t+1$ must also be forward.  Then we have the following possibilities for $Q(\zg)$,
where in the first picture there are arrows $i_s\to j_s, i_t \to j_t$ and in the second one the arrows go in the opposite direction.  

\[
\xymatrix@C=10pt@R=10pt{
i_s\ar[r] \ar[d]& k_s\ar[r] \ar[d]& \cdots \ar[r] & k_t\ar[r]  &j_t \ar[dl] &&& j_{s-1} \ar[d] & i_{s-1}\ar[d]&& j_{t+1}&i_{t+1}\\
j_s\ar[r] & \bullet \ar[r]\ar[ul] & \cdots \ar[r] & \bullet \ar[r] \ar[u]\ar[d]& i_t \ar[d]\ar[u] &&& \raisebox{0pt}[0.9\height][0.3\height]{ $\vdots$ }\ar[d] & \raisebox{0pt}[0.9\height][0.3\height]{ $\vdots$ } \ar[d]&&\raisebox{0pt}[0.9\height][0.3\height]{ $\vdots$ }\ar[u] &\raisebox{0pt}[0.9\height][0.3\height]{ $\vdots$ }\ar[u]\\
\raisebox{0pt}[0.9\height][0.3\height]{ $\vdots$ }\ar[u] & \raisebox{0pt}[0.9\height][0.3\height]{ $\vdots$ } \ar[u]&&\raisebox{0pt}[0.9\height][0.3\height]{ $\vdots$ }\ar[d] &\raisebox{0pt}[0.9\height][0.3\height]{ $\vdots$ }\ar[d]&&& i_s\ar[r] & k_s\ar[r]  \ar[dl]& \cdots \ar[r] & k_t\ar[r] \ar[u] \ar[d]&j_t \ar[d] \ar[u]\\
i_{s-1} \ar[u] & j_{s-1}\ar[u]&& i_{t+1}&j_{t+1}&&& j_s\ar[r] \ar[u]& \bullet \ar[r] \ar[u]& \cdots \ar[r] & \bullet \ar[r]  & i_t  \ar[ul]\\
}
\]

In the first case, we see that $i_s\leadsto j_t$ does not pass through any vertex outside of $Q(\zg, [s,t])$.  Therefore, this path cannot factor though any other $i_h, j_{h'}$ with $h,h'<s$ or $h,h'>t$, a contradiction.  Then we must be in the situation as in the quiver on the right.   Here, the path $i_s\leadsto j_t$ factors through $k_s, k_t$ which are the only vertices of $Q(\zg, [s,t])$  that can also be part of $Q(\zg, [1, s])$ or $Q(\zg, [t, \text{max}(m,n)])$ apart from $i_s, j_s$ or $i_t, j_t$ respectively, and where $m,n$ are as in the statement of the lemma.  

The mixed cases when there are arrows $i_s\to j_s, j_t\to i_t$ or $j_s\to i_s, i_t\to j_t$ follow in the same way, and here the path $i_s\leadsto j_t$ factors through one of $k_t$ or $k_s$ respectively. 
This proves the claim that if $i_s\leadsto j_t$ factors through $i_h$ or $j_{h'}$ or both, then $i_s\to i_h$ or $j_{h'}\to j_t$ or both of these paths respectively consist of a single arrow.

We are now ready to prove the lemma.  
Suppose first $s\not=t$ and $i_s\leadsto j_t$ is a valid path that factors through both $k_s, k_t$ such that $\zg$ crosses $k_s, k_t$.   In addition, we can take $k_s\not=k_t$.  Thus, here we consider the most general situation, and when $\zg$ crosses only one of $k_s, k_t$ then it becomes a special case of the one we have.  By the above we also know that all steps from $s-1$ to $t+1$ must be forward.  Note that if $k_s$ is a $j_{h'}$ for $h'<s$ in the crossing sequence for $\zg$, then the arrow $i_s\to j_{h'}$ implies that the steps from $s$ to $h'$ are forward by Lemma~\ref{lem:forward-arrow}.  Since $k_s\not=k_t$ this is a contradiction, so the crossing between $k_s$ and $\zg$ is in degree zero. 
We consider the quiver $Q(\zg,[i_s])$. By Lemma~\ref{lem:circ}, we have the following situation, where $l\geq 1$ is such that the path $i_s\leadsto j_{s-l}$ factors through only one other vertex $i_p$ in the crossing sequence of $\zg$,  $p\leq s-l$, and all steps from $p$ to $s$ are forward.  Moreover, $k_s=i_{s-1}$ if $l>2$ and $k_s=i_p$ for some $p\leq s-1$ if $l=1$, and $k_t=j_{h'}$ for some $h'>t$.  

\[
\xymatrix@C=10pt@R=10pt{
&j_{s-2}\ar@{.}[d]\ar[dr]&i_{s-2}\ar[r]\ar[l]&j_{s-1}\ar[dl]\\
&j_{s-l+1}\ar[r]&i_s\ar[r] \ar[dl]\ar[u]&i_{s-1}\ar[u] \ar[r] \ar[dl]& \cdots \ar[r] & j_{h'} \ar[r] & j_t\ar[d] \\
&i_p\ar[u]\ar[d]&j_s\ar[u] \ar[r] & \bullet \ar[r] & \cdots \ar[r] &\ar[r] \bullet & i_t\ar[ul]\\
&j_{s-l}\ar[uur]_{\beta}
 }
\]

The position $(s,t)$ in the matrix equation $g_i f_{\zg} = f_{\zg} g_j$ yields the following equation.

\medskip
$ \begin{array}{lllll}
(s,t) &&&& b_{s,s}f_{s,t}+b_{s,s-1}f_{s-1,t}=f_{s,t}a_{t,t}+f_{s,h'}a_{h',t}
\end{array}$
\medskip

In order to prove the lemma it suffices to show that $b_{s,s-1}=a_{h',t}=0$.  We will show that $b_{s,s-1}=0$ and the remaining part $a_{h',t}=0$ follows dually.  

We also have the following sets of equations coming from the other positions in the matrix equation~(\ref{eq:11}). They correspond to paths in the quiver $Q(\zg,[i_s])$ starting at $i_s$ and ending at $j_s, j_{s-1}, \dots, j_{s-l}$ respectively.  Except for the first and the last path here all of them factor through exactly two other $i$'s in the crossing sequence for $\zg$.  For example, there are exactly two different paths in $Q$ starting at $i_s$ and ending at $j_{s-1}$, which are $i_s\to i_{s-1}\to j_{s-1}$ and $i_s\to i_{s-2}\to j_{s-1}$.  This means that in the equation $(s, s-1)$ there are four terms because a path $i_s\leadsto j_{s-1}$ factors through four vertices that are in the crossing sequence for $\zg$, namely $i_s, i_{s-1}, i_{s-2}, j_{s-1}$. 

\[\begin{array}{ll}
(s,s) & b_{s,s}f_{s,s}+b_{s,s-1}f_{s-1,s}=f_{s,s}a_{s,s}\\
(s, s-1) & b_{s,s} f_{s,s-1} + b_{s,s-1}f_{s-1,s-1} + b_{s,s-2}f_{s-2, s-1}=f_{s,s-1}a_{s-1,s-1}\\
\vdots & \vdots\\
(s,s-l+1) & b_{s,s}f_{s,s-l+1} + b_{s,s-l+1}f_{s-l+1, s-l+1} + b_{s,p}f_{p,s-l+1}=f_{s,s-l+1}a_{s-l+1,s-l+1}\\
(s,s-l) & b_{s,s}f_{s,s-l} + b_{s, p} f_{p, s-l} = f_{s,s-l}a_{s-l, s-l}
\end{array}\]

In equation $(s,s-l)$ we have $f_{s,s-l}=0$ because $p\leq s-l<s$ and the steps from $p$ to $s$ are forward and $f_{p,s-l}\not=0$ because there exists an arrow $i_p\to j_{s-l}$.  If $\beta$ is not a boundary arrow, then the path $i_s\to i_p\to j_{s-l}$ is nonzero in the algebra $B$, and we conclude that $b_{s,p}=0$.   Now, consider equation $(s,s-l+1)$.  We have that $f_{s,s-l+1}=0$ because the steps from $s-l+1$ to $s$ are forward.  Since $b_{s,p}=0$ we conclude that $b_{s,s-l+1}=0$.  Continuing in this way we obtain that $b_{s,s-2}=0$ from equation $(s,s-2)$.  Then equation $(s,s-1)$ implies the desired claim that $b_{s,s-1}=0$.  This completes the proof of the lemma in the case when $\beta$ is not a boundary arrow. 

Now, suppose $\beta$ is a boundary arrow. First, we want to show that $b_{s,s-l}=0$.   Since the step from $s-l$ to $s-l+1$ is forward  and $\beta$ is boundary we must have that $i_p=i_{s-l}$.  We want to show that $\zg$ crosses $j_{s-l-1}$.  Suppose not.  Then $(i_{s-l}, j_{s-l})$ is the first crossing pair for $\zg$ and we are in the situation of Figure~\ref{fig:65}.    

\begin{figure}
\centerline{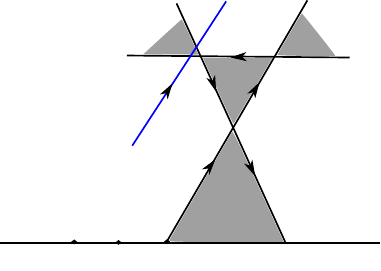}
\caption{Proof of Lemma~\ref{lem:ab}.}
\label{fig:65}
\end{figure}

Let $W$ be the white region that $\zg$ traverses before crossing $(i_{s-l}, j_{s-l})$.  Let $x$ be a boundary point of the polygon $\cals$ in $W$ where $i_s$ starts, and let $y,z$ be neighboring boundary points of $\cals$ clockwise from $x$ as in the picture.  Since $\zg$ also crosse $i_s$, it cannot start at $x$.   Because $j_{s-l}$ crosses $\zg$ from left to right, it follows that $\zg$ cannot start at the point $y$.  Therefore, $\zg$ starts in $z$ or at some point clockwise from $z$.   If $\zg$ starts at some other point clockwise from $z$ then it crosses another pair of arcs in $W$ before crossing $(i_{s-l}, j_{s-l})$, which is a contradiction to $(i_{s-l}, j_{s-l})$ being the first crossing pair for $\zg$.   Then $\zg$ starts in $z$.  If the white region $W$ has no edges on the boundary of $\cals$, then the boundary edge with endpoints $y,z$ lies in a shaded region.  In particular, $\zg$ crosses another pair of arcs before crossing $are $, which is again a contradiction.   Then $W$ has an edge on the boundary with endpoints $x,y$, and the edge with endpoints $z,y$ bounds a triangular shaded region.  Moreover,  the arc attached to $y$ crosses $\zg$ from left to right, which means that this arc is $j_{s-l-1}$ in the crossing sequence for $\zg$.   This shows the claim.  

Then we have the following quiver, which is an extension of the quiver $Q(\zg,[i_s])$ given earlier.  Note that it may happen that the path $i_{s-l}\leadsto j_{s-l-1}$ factors through some other $j_q$ with $q<s-l-1$, and here we depict the most general situation.

\[
\xymatrix@C=10pt@R=10pt{
&&&&&i_s\ar[dll]\\
j_{s-l-1}\ar[d]&\ar[l]j_q&\cdots\ar[l]&i_p=i_{s-l}\ar[d]\ar[l]\\
\bullet\ar[ur]&\ar[l]\cdots&\cdots\ar[l]&j_{s-l}\ar[uurr]_{\beta}\ar[l]\\
}
\]

The equation coming from position $(s, s-l-1)$ corresponds to paths in the quiver starting in $i_s$ and ending in $j_{s-l-1}$.  Because there is a valid path from $i_s$ to $j_{s-l-1}$ then it does not pass through any relations, and so it is the unique  nonzero path in $B$ between these vertices.  The path $i_s\leadsto j_{s-l-1}$ factors through exactly four vertices $i_s, i_{s-l}, j_q, j_{s-l-1}$ that are in the crossing sequence for $\zg$.  This means that the equation $(s, s-l-1)$ contains four terms as shown below. 

\[\begin{array}{ll}
(s,s-l-1) & b_{s,s}f_{s,s-l-1}+b_{s,s-l}f_{s-l,s-l-1}=f_{s,s-l-1}a_{s-l-1,s-l-1}+f_{s,q}a_{q,s-l-1}
\end{array}\]

Here $f_{s,s-l-1}=f_{s,q}=0$ because the steps from $s$ to $s-l$ are not forward. Then we obtain $b_{s,s-l}f_{s-l,s-l-1}=0$, where $f_{s-l,s-l-1}\not=0$ because it is a valid path and the step from $s-l$ to $s-l-1$ is forward.  Therefore, we still obtain that $b_{s,s-l}=0$ when $\beta$ is a boundary arrow.  Then we can continue as before to conclude that $b_{s,s-1}=0$ from equations $(s,s-l), \dots, (s, s-1)$.  This proves the lemma in the case $s\not=t$.    

Now, suppose $s=t$, $f_{s,s}\not=0$ and the path $i_s\leadsto j_s$ factors through some other $i$ or $j$.  If the path $i_s\leadsto j_s$ factors through one or two $i$'s, say $i_{p}, i_{p'}$ the we obtain the following equation. 

\[\begin{array}{ll}
(s,s) & b_{s,s}f_{s,s}+b_{s,p}f_{p,s}+b_{s,p'}f_{p',s}=f_{s,s}a_{s,s}
\end{array}\]

We then conclude that $b_{s,p}, b_{s,p'}$ are zero in the same manner as in the case $s\not=t$ above. First, we construct the quiver $Q(\zg, [i_s])$, see below.  Then we use equations $(s, s-1) \dots, (s,s-l)$ and equations $(s, s+1), \dots, (s, s+l')$ to show that $b_{s,p}, b_{s,p'}$ are zero respectively. 

\[
\xymatrix@C=10pt@R=10pt{
&&j_{s-l}\ar[d]&i_{s-2}\ar[r]\ar@{.}[l]&j_{s-1}\ar[dll]\\
j_{s-l'}\ar[rr]\ar@{.}[d]&&i_s\ar[rr]\ar[dll]\ar[ur]&&i_p\ar[dll]\ar[u]\\
i_{p'}\ar[rr]&&j_s\ar[u]
}
\]

If instead the path $i_s\leadsto j_s$ factors through one or two $j$'s then again we obtain the desired result by considering the quiver $Q(\zg, [j_s])$ and applying the dual argument.  

Finally, if $i_s\leadsto j_s$ factors through $i_p$ and $j_q$ then to show that $b_{s, p}=0$ we consider the subquiver $Q(\zg, [i_s])$.   Then to show that $a_{q,s}=0$ we consider the subquiver $Q(\zg, [j_s])$.  This completes the proof of the lemma in all cases. 
\end{proof}

The following statement is a direct consequence of the proof of Lemma~\ref{lem:ab} above, which we will need later in Section~\ref{sec:blocks}.  Note that here we assume that the diagonal entries of the matrices $g_i, g_j$ are zero.    

\begin{corollary}\label{cor:534}
Let $\zg$ be a representative of a 2-diagonal in $\cals$, and let $f_{\zg}: \bigoplus_{h=1}^{n} P(j_h) \to  \bigoplus_{h=1}^m P(i_h)$ be the associated morphism.
Suppose we have a commutative diagram 
\[\xymatrix{
 \bigoplus_{h=1}^n P(j_h) \ar[r]^{f_{\zg}} \ar[d]^{g_j}&  \bigoplus_{h=1}^m  P(i_h) \ar[d]^{g_i} \\
 \bigoplus_{h=1} ^n P(j_h) \ar[r]^{f_{\zg}} &  \bigoplus_{h=1}^ m P(i_h)
}\]
for some nilpotent morphisms $g_j$ and $g_i$. Then $(g_i f_{\zg})_{s,t}=0$ for all $s,t$ such that $(f_{\zg})_{s,t}\not=0$.  
\end{corollary}

\section{Proofs of section \ref{sect 6}}\label{A2pivots}
Here we prove Theorem \ref{lem 63}
and extend the definition of pivot morphisms to complexes of projectives in Theorem~\ref{prop:67} which will be used in later sections.

\begin{thm}[Theorem \ref{lem 63}]
 \label{Alem 63} Let $\zg,\zg'$ be compatible 2-diagonals in $\cals$ such that $\zg'$ is obtained from $\zg$ by a 2-pivot and let $(g_0^r,g_1^c)$ be the morphism defined above. Then we have the following commutative diagram with exact rows
 \begin{equation}\label{eq 62}\xymatrix@C50pt{ P_1\ar[r]^{f_\zg} \ar[d]_{g_1^c}&P_0\ar[d]^{g_0^r}
\ar[r]^{\pi_\zg}&M_\zg\ar[r]\ar[d]^g&0 \\
P_1'\ar[r]^{f_{\zg'}}&P_0'\ar[r]^{\pi_{\zg'}}&M_{\zg'}\ar[r]&0 \\
}
 \end{equation}
where $g$ is the induced morphism on the cokernels.
%
 In particular, $g$ is a morphism of syzygies in $\cmp\,B$.
\end{thm}

\begin{proof}
 It suffices to show the commutativity on a pair of indecomposable summands $P(j)$ of $P_1$ and $P(i)$ of $P_0'$. Thus without loss of generality, we may assume that $P_1=P(j)$ and $P_0'=P(i)$. By definition, the pivot morphism $g^c_1$ is a column vector $g^c_1\colon P(j)\to P_1'$ whose components are either the identity morphism $1_{P(j)}$, multiplication by arrows $h\to j$ ending in $j$, or equal to zero. Moreover, if a component is $1_{P(j)}$ then all other  components  are zero, and if no  component is $1_{P(j)}$ then $P(j)$ is not a direct summand of $P_1'$, so by Lemma~\ref{lem:M}, there may be up to two arrows $h_1\to j$ and $-h_2 \to j$ in $g^c_1$ and all other  components  are zero.
 
 Similarly, the map $g_0^r$ is a row vector, $g_0^r\colon P_0\to P(i)$ whose entries are $0,1_{P(i)} $ or arrows $i\to \ell $. Moreover, if an entry is $1_{P(i)}$ then all other entries are zero, and if no entry is $1_{P(i)}$ then $P(i)$ is not a summand of $P_0$, there may
be up to two arrows $i\to\ell_1$, $-i\to \ell_2$ in $g^r_0$
and all other entries are zero.

We use the notation and orientation of Figure \ref{figlem 62}. Thus $a$ is the common endpoint of $\zg$ and $\zg'$ and the 2-pivot move is from vertex $x$ to vertex $z$, skipping vertex $y$. Recall that the vertex $y$, as any boundary vertex, is the endpoint of exactly one or two radical lines, by Lemma~\ref{lem 32}(c).

Also note that $P(j)$ is not a summand of $P_0'$, because even if the radical line $\rho(j)$ were to cross both $\zg$ and $\zg'$ then these crossings would have the same degree, and thus the projective $P(j)$ would have the same degree in either of the two-term complexes of $f_\zg$ and $f_{\zg'}$, and, in that case, $P(j)$ would be a summand of $P_1'$ but not of $P_0'$. 
A similar argument shows that $P(i)$ is not a summand of $P_1$.

We consider several cases determined by the geometric configuration in the pivot area of the polygon $\cals$. 

 (1) Suppose first that $P(j)$ is not a summand of $P_1'$ and $P(i)$ is not a summand of $P_0$. Thus the radical line $\rho(j)$ crosses $\zg$ from left to right but does not cross $\zg'$. This implies that $\rho(j)$ must end at vertex $z$. 
Similarly, the radical line $\rho(i)$ crosses $\zg'$ from right to left but does not cross $\zg$, and therefore $\rho(i)$ must end at vertex $x$.

(1a)  Suppose there are two radical lines $\rho(h_1)$ and $\rho(h_2)$ at $y$ and both of them cross  $\zg$ or both of them cross $\zg'$. We may assume without loss of generality that both cross $\zg$, see  the left picture in Figure \ref{figlem 63}. 
\begin{figure}
\begin{center}
\scalebox{0.83}{\small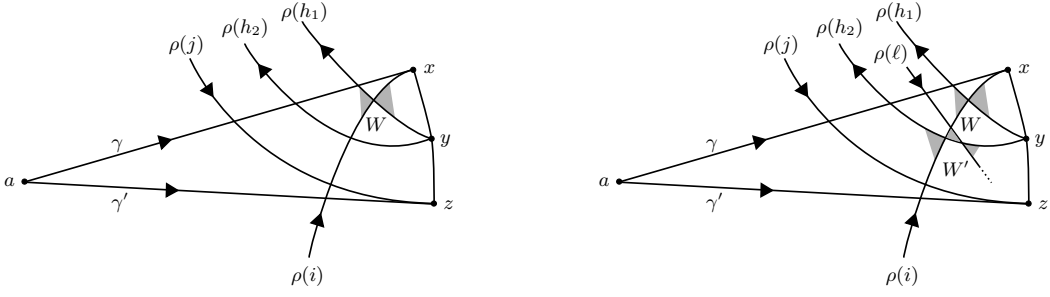}
\caption{Proof of Theorem~\ref{lem 63}.}
\label{figlem 63}
\end{center}
\end{figure}
In this case, the map $g^c_1$ is zero and $g^r_0$ is the row vector
\[ g^r_0= [0 \quad \cdots \quad 0 \quad i\to h_1\quad -i\to h_2].\]
We need to show $g^r_0f_\zg=0$.
Consider the white region $W$ incident to the crossing point of $\rho(i)$ and $\rho(h_1)$ and to the right of $\rho(i)$. Since every white region has an even number of sides, there must be some radical lines missing in the left picture in Figure \ref{figlem 63}. Moreover, an additional radical line cannot eliminate the vertex $y$ from the region $W,$ because every white region has at least one boundary vertex, by Lemma~\ref{lem 32}(a).
We therefore must have a radical line $\rho(\ell)$ that eliminates the crossing point between $\rho(i)$ and $\rho(h_2)$ from $W,$ see the right hand picture in Figure \ref{figlem 63}. Note that the orientation of $\rho(\ell)$ must be as indicated in the figure, since the orientation of the edges of $W$ must alternate around the white region, by Lemma \ref{lem orientation}. 
Furthermore, $\rho(\ell)$ must end at vertex $z$, because otherwise, it would cross $\rho(j)$ and $\zg'$ and thereby create a white region $W'$ without boundary vertex, which is impossible.
Thus $\rho(\ell)$ ends at $z$. 

(i) Suppose first that $\rho(\ell)$ is equal to $\rho(j)$. Then the white region $W$ has four sides, $\rho(h_2), \rho(j),\rho(i)$ and $\rho(h_1)$ and its cycle path is $\mathfrak{c}(W) = h_2\to j\to i\to h_1$. Therefore the quiver $Q$ contains the following subquiver 

\begin{equation}
 \label{eq quiver}
 \xymatrix{h_2\ar[r]^{\zd_2}&j\ar[ld]\\ i\ar[u]^\za \ar[r]_{\zd_1}& h_1\ar[u]_\zb}
 \end{equation}
with $\zd_1, \zd_2$ boundary arrows.
Now consider the map $f_\zg\colon P(j)\to P_0$. The subquiver above implies that the arrows  $\zd_2$ and $\zb$ are valid paths and hence nonzero components of $f_\zg$ by Lemma~\ref{lem:forward-arrow}. Recall that, on the other hand, the nonzero components of $g^r_0$ are the arrows $\za$ and $\zd_1$. Therefore, we have $ g_0^r f_\zg =\zd_1\zb-\za\zd_2 =0$. Since $g^c_1$ is zero, this completes the commutativity in this case.

(ii) Now suppose that $\rho(\ell)\ne\rho(j)$. Since the white region $W'$ must have a boundary vertex, this vertex must be $z$ and the sides of $W'$ in clockwise order are $\rho(j),\rho(i),\rho(h_2)$ and $\rho(\ell)$. Therefore its cycle path in $Q$ is 
$\mathfrak{c}(W')=\xymatrix{j\ar[r]^{\zd_3}& i\ar[r]^\za& h_2\ar[r]^{\zd_2} &\ell}$. In particular, the subpath $\zd_3\za$ lies in an oriented cycle whose third arrow we denote by $\ze$.  Thus the regions $W,W'$ give rise to the following subquiver of $Q$ 
with $\zd_1,\zd_2$ and $\zd_3$ boundary arrows.

\[\xymatrix{&h_2\ar[ld]_\ze\ar[r]^{\zd_2}&\ell\ar[ld]\\ 
j\ar[r]_{\zd_3}&i\ar[u]^\za \ar[r]_{\zd_1}& h_1\ar[u]_\zb}\]
It follows that the arrow $\ze$ is a valid path and that there is no valid path from $h_1 $ to $j$. Therefore, we have $g_0^r f_\zg= \za\ze$ which is zero since $\zd_3 $ is a boundary arrow.
This completes the proof in case (1a). 

(1b) Suppose there are two radical lines $\rho(h_1)$ and $\rho(h_2)$ at vertex $y$, one crossing $\zg$ and the other crossing $\zg'$, see Figure \ref{figlem 63b}. 
\begin{figure}
\begin{center}
\scalebox{0.75}{\small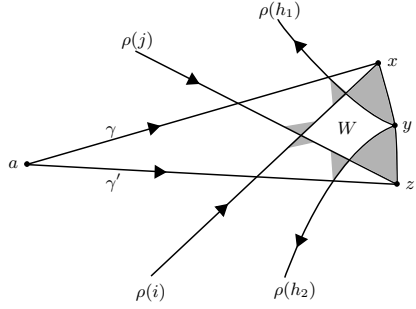}
\caption{Proof of Theorem~\ref{lem 63} case (1b).}
\label{figlem 63b}
\end{center}
\end{figure}
Then by the same argument as in case (1a), the white region $W$ has four sides $\rho(h_2),\rho(j),\rho(i)$ and $\rho(h_1)$ in clockwise order. Thus again we  have the subquiver (\ref{eq quiver})  
In this situation however, both $g^r_0$ and $g^c_1$ are nonzero.
The only nonzero component of  $g^r_0$ is the arrow $\zd_1$ and the only nonzero component of $g^c_1$ is the arrow $\zd_2$. On the other hand, the arrow $\za$ is a valid path and hence a component of $f_{\zg'}$, and the arrow $\zb$ is a component of $f_\zg$. 
Thus we have
$f_{\zg'}g^c_1=\za\zd_2$ and 
$g_0^r f_{\zg}=\zd_1\zb$. These paths are equal by Lemma \ref{lem 37}(b), and this completes the proof in case (1b).

(1c) Suppose there is exactly one radical line $\rho(h)$ at $y$. Then $\rho(h)$ must cross either $\zg$ or $\zg'$. These two cases are similar, and we only prove one of them. Suppose therefore that $\rho(h)$ crosses $\zg$, see Figure \ref{figlem 63c}. In this situation the map $g_1\colon P(j) \to P_1'$ is zero, by definition. Thus we must show $g_0f_\zg=0$.
\begin{figure}
\begin{center}
\scalebox{0.75}{\small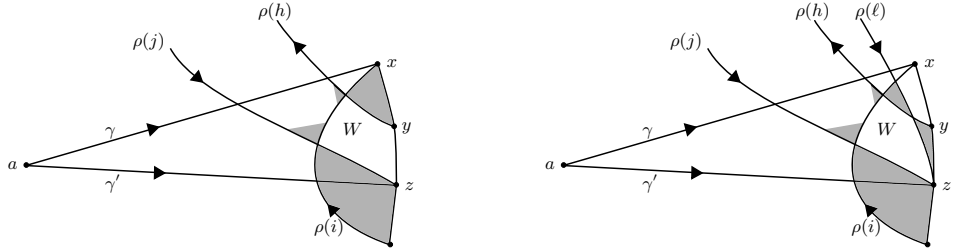}
\caption{Proof of Theorem~\ref{lem 63} case (1c).}
\label{figlem 63c}
\end{center}
\end{figure}
Let $W$ be the white region incident to the crossing point of $\rho(i)$ and $\rho(h)$ and to the right of $\rho(i)$. Since there is only one radical line at vertex $y$, the region $W$ must contain the vertex $z$, and the arrow $\zd_1\colon j\to i$ is a boundary arrow in $Q$. Thus the radical line $\rho(i)$ must start at the clockwise neighbor of $z$ on the boundary of $\cals$.

(i) Suppose first that the arrow $i\to h$ is also a boundary arrow and denote it by $\zd_2$, see the left picture in Figure \ref{figlem 63c}. Then the white region $W$ has 4 sides and hence its cycle path is $\mathfrak{c}(W)=\xymatrix{j\ar[r]^{\zd_1} &i\ar[r]^{\zd_2}&h}$, with $\zd_1,\zd_2$ boundary arrows in $Q$.
  Moreover, 
since $\mathfrak{c}(W)$ is a subpath of a chordless cycle in $Q$, there must be an arrow $\ze\colon h\to j$. Thus the radical lines $\rho(h)$ and $\rho(j)$ must cross; this crossing is not shown in Figure \ref{figlem 63c}.
Thus $Q$ contains the following subquiver
\[\xymatrix@R10pt{&i\ar[rd]^{\zd_2}\\ j\ar[ru]^{\zd_1} &&h\ar[ll]_\ze}\]

In this situation, the map $f_\zg$ is given by the arrow $\ze$ and the map $g^r_0$ by the arrow $\zd_2$. Therefore we have $g_0^r f_\zg=\zd_2\ze$, which is zero, because $\zd_1$ is a boundary arrow.

(ii) Now suppose that the arrow $i\to h$ is not a boundary arrow, see the right picture in Figure \ref{figlem 63c}.  Then there is a radical line $\rho (\ell)$ that crosses $\rho(h)$ and $\rho(i)$ forming a triangular shaded region that corresponds to a chordless cycle 
in $Q$. The white region $W$ has 4 sides and its cycle  path is $\mathfrak{c}(W)=\xymatrix{j\ar[r]^{\zd_1}& i \ar[r]^\za&h\ar[r]^{\zd_3} &\ell}$, with boundary arrows $\zd_1$ and $\zd_3$. Moreover, the paths $\zd_1\za$ and $\za\zd_3$ lie in two different chordless cycles in $Q$. Hence $Q$ has the following subquiver
\[\xymatrix{
j\ar[r]^{\zd_1}& i \ar[ld]_\za\\h\ar[u]^\ze\ar[r]_{\zd_2} &\ell\ar[u]_\zb
}
\]
In this situation, we still have  $g_0^r f_\zg=\za\ze$ which is zero because $\zd_1$ is a boundary arrow. 
This completes the proof in case (1).

(2) Suppose that $P(j)$ is a summand of $P_1'$. 
Then $\rho(j)$ crosses both $\zg$ and $\zg'$ and the map $g_1^c$ is the identity on the component $P(j)\to P(j)$ and zero on all components $P(j)\to P(\ell)$ with $\ell\ne j$. 

(2a) Suppose $P(i)$ is a summand of $P_0$. Then $g_0^r$ is the identity on the component $P(i)\to P(i)$ and zero on all components $P(\ell)\to P(i)$. Therefore
$g_0^rf_\zg\colon P(j)\to P(i)$ is equal to $f_\zg$ and 
$f_{\zg'}g_1^c\colon P(j)\to P(i)$ is equal to $f_{\zg'}$. These two maps are equal since $\zg$ and $\zg'$ are compatible and each cross both $\rho(i)$ and $\rho(j)$. 

(2b)  
Suppose $P(i)$ is not a summand of $P_0$. Then the radical line $\rho(i) $ crosses  $\zg'$ but not $\zg$, see Figure \ref{figlem 632b}.
\begin{figure}
\begin{center}
\scalebox{0.75}{\small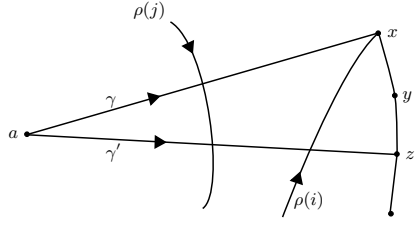}
\caption{Proof of Theorem~\ref{lem 63} case (2b).}
\label{figlem 632b}
\end{center}
\end{figure}

(i) Suppose first that there is no radical line that starts at vertex $y$ and crosses $\zg$.
Then $g_0^r\colon P_0\to P(i)$ is zero and thus $g_0^rf_\zg\colon P(j)\to P(i)$ is zero. On the other hand, $g_1^c\colon P(j)\to P_1'$ is the identity on $P(j)$ and thus  $f_{\zg'}g_1^c$ is given by the component $f_{\zg'}\colon P(j)\to P(i)$.  
So it suffices to show that there is no valid path from $i$ to $j$ in $Q$. 

 Using Lemma \ref{lem 32}(b), we let $W$ denote the unique white region containing the vertex $x$. 
 If $W $ lies to the left of $\rho(i)$, then, since there is no radical line that crosses $\zg$ and ends at $y$, 
 $\zg'$ crosses $\rho(j)$ first and runs through $W$ before crossing $\rho(i)$. Moreover, the cycle path $\mathfrak{c}(W)$ starts at $i$ and one of the maximal valid paths $\mathfrak{v}(W)$ ends at $i$, since the cycle path goes clockwise around $W$ and the maximal valid paths go counterclockwise. In particular, there
is no valid path starting at $i$ and going along the sides of $W$, because a valid path is the composition of valid paths along white regions.  Hence there is no valid path from $i$ to $j$ and we are done.

On the other hand, if $W$ lies to the right of $\rho(i)$, then there is a shaded region $S$  containing $x$ and lying to the left of $\rho(i).$ Thus $S$ is a boundary region, since it contains $x$, and therefore it must also contain the counterclockwise neighbor of $x$ on the boundary. Hence there exists a radical line $\rho(\ell)$ starting at the counterclockwise neighbor of $x$ such that $\rho(i) $ crosses $\rho(\ell)$. Let  $W'$ be the white region incident to this crossing point and to the left of $\rho(i)$. Now the result follows from the argument above replacing $W$ by $W'$. 

(ii) Now suppose there is a radical line $\rho(h)$ that starts at vertex $y$ and crosses $\zg$.  In this case, $g_0^r$ is given by the arrow $\za\colon i\to h$,  and $g_0^rf_\zg$ is a the composition of $\za$ with a valid path from $h$ to $j$. Recall that $g_1^c$ is the identity on $P(j)$, and $f_{\zg'}g_1^c$ is given by a valid path from $i $ to $j$. So we must show that every valid path from $i$ to $j$ factors through the arrow $\za$.

The radical lines $\rho(h) $ and $\rho(i)$ cross, and, from their orientations, we see that there exists a shaded region $S$ adjacent to the crossing point that lies to the left of both $\rho(h)$ and $\rho(i)$. Let $\rho(\ell)$ be the  third side of $S$. Let $W$ be the white region adjacent to $S$ via the edge $\rho(\ell)$. Then its cycle path is of the form $\mathfrak{c}(W) = \cdots\to h\to\ell\to i \to \cdots$, and the arrow $\za\colon i\to h$  is a valid path.

Now, let $u$ be a valid path from $h$ to $j$ and write $u=\zb u'$, where $\zb$ is the initial arrow of $u$ and $u'$ is the remaining subpath. Then either $\za\zb$ is a subpath of a maximal valid path of the white region $W$, or $\zg$ lies in a different white region. In both cases,  $\za\zb$ is not a subpath of a chordless cycle and hence $\za\zb u'$ is a valid path from $i$ to $j$.
By Proposition \ref{prop 39}, every valid path from $i$ to $j$ is of this form.
This completes the proof in case (2).
\end{proof}
We will also need the following result, which provides the commutative diagram analogous to the one in Theorem \ref{lem 63} when replacing the maps $f_\zg, f_{\zg'}$ by their modifications $\bar f_\zg, \bar f_{\zg'}$ introduced in Section~\ref{sect 5.3}.

\begin{corollary} \label{lem: pivot-bar} 
 Let $\zg,\zg'$ be compatible 2-diagonals in $\cals$ such that $\zg'$ is obtained from $\zg$ by a 2-pivot and let $(g_0^c,g_1^r)$ be the morphism defined above. Then the following diagram commutes.
 \[\xymatrix@C50pt{ P_1\ar[r]^{\bar{f}_\zg} \ar[d]_{g_1^r}&P_0\ar[d]^{g_0^c}\\P_1'\ar[r]^{\bar{f}_{\zg'}}&P_0'}
 \]
\end{corollary}

\begin{proof}
Recall that,  given a map $f: P_1\to P_0$, the corresponding map $\bar{f}$ is defined as $\bar{f}=J_{|P_0|}fJ_{|P_1|}$, where $J_k$ is an isomorphism that squares to the identity.  Theorem~\ref{lem 63} shows that $g_0^rf_{\zg} = f_{\zg'}g_1^c$ and we obtain,  
 \[\overline{g_0^r}\bar{f}_{\zg}=J g_0^r JJ f_{\zg} J =\overline{g_0^rf_{\zg}} = \overline{f_{\zg'}g_1^c} = J  {f_{\zg'}JJ g_1^c}J=\bar{f}_{\zg'}\overline{g_1^c}.\]
 
It therefore suffices to show that $\overline{g_1^c}=g_1^r$ and $\overline{g_0^r}=g_0^c$.   The maps $\overline{g_1^c}, \overline{g_0^r}$ are obtained from $g_1^c, g_0^r$ respectively, by introducing alternating signs in their matrices, see Remark~\ref{rem:bar}(a).  By Lemma~\ref{lem:M} each of $g_1^c, g_0^r$ has at most one nonzero off-diagonal entry, which implies  the desired result. 
\end{proof}

\subsection{The induced morphism of complexes on the projective resolutions}\label{sect 6.2new} Our next goal is to give an explicit description of the complete projective resolution and the morphism of complexes induced by the pivot maps. Recall that if two representatives $\zg_1,\zg_2$ of the same 2-diagonal $\zg$ are not compatible then they give rise to two different morphisms $f_{\zg_1},f_{\zg_2}$ in the projective presentation. Moreover, the sequence 
\begin{equation*}
 \label{eq:pr0}
\xymatrix{P_2(\zg)\ar[r]^{\overline{f}_{R\zg}} &P_1(\zg)\ar[r]^{{f}_{\zg}} &P_0(\zg)}\end{equation*}
 is exact  if the arcs $R\zg$ and $\zg$ are compatible, see Proposition~\ref{big-lemma}.
On the other hand, the commutativity of the pivot diagram in Corollary~\ref{lem: pivot-bar} relies on the compatibility of the arcs $\zg$ and its pivot $\zg'$. 
In this subsection, we will give an  explicit construction that will control this double compatibility condition. 
\smallskip

Since the syzygies in the projective resolution of $M_\zg$  correspond to the elements  $R^i\zg$ in the rotation orbit of $\zg$, we shall need the following result about the distribution of radical lines in the rotation orbit.
\begin{lemma}
\label{lem:three radical lines} Let $\za,R\za,R^2\za$ be three consecutive 2-diagonals under rotation. If $\za$ and $R^2\za$  are radical lines then $R\za$ is a radical line as well.
\end{lemma}
\begin{proof}
Say $\za=\rho(x)$ and $R^2\za=\rho(y)$ and
 let $p$ denote their crossing point in $\cals$. Recall that there are two shaded regions $S_1,S_2$ and two white regions $W_1,W_2$ at $p$. Also recall that the orientation of the radical lines is alternating along the white regions and directed along the shaded regions, see Lemma~\ref{lem orientation}. Now, since $\za$ and $R^2\za$ have the same orientation, we see that one white region $W_1$ at $p$ lies between the two incoming segments of $\za$ and $R^2\za$ and that the other white regions $W_2$ lies between the two outgoing segments of $\za$ and $R^2\za$, see Figure~\ref{fig second rotation}.
 \begin{figure}
\begin{center}
\scalebox{.8}{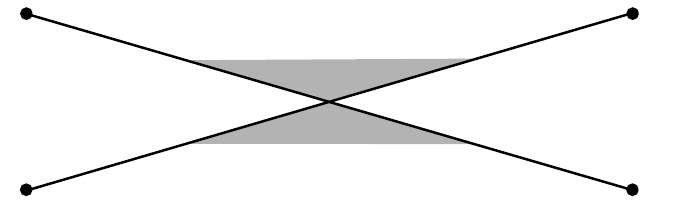}
\caption{Proof of Lemma \ref{lem:three radical lines}.
\label{fig second rotation}
}
\end{center}
\end{figure}

 Because of our assumption that every chordless cycle in $Q$ is of length 3,  the shaded region $S_1$ is a triangle. Then either the third side of $S_1$ is a boundary edge of the polygon connecting the points $c$ and $a'$, or there exists a radical line $\rho(z)$ that comes from  $W_1$, crosses $R^2\za$, then forms a boundary edge of $S_1$ and crosses $\za$, and continues in $W_2$. In particular, such a $\rho(z)$ must end at the points $b$ and $b'$, because it cannot cross $\za$ or $R^2\za$ twice. Thus $\rho(z) $ is homotopic to $R\za$ and we are done. A similar analysis for the shaded region $S_2$ shows that the only remaining case is when both $S_1$ and $S_2$ have the third side on the boundary. In that case the polygon has exactly six vertices $a,b,c,a',b',c'$, and therefore it must be one of the two shown in the first row of Figure~\ref{fig ex size}. In particular $R\za$ is a radical line.
\end{proof}

Next, we need to characterize the compatibility of the arcs in the rotation orbit. This is done in the following two lemmata.
The first lemma specifies how to choose compatible representatives of a 2-diagonal and its rotation if one of the two is a radical line. 
To illustrate the problem, consider the red arc $\zg$ in Figure~\ref{fig ex 11.1}. It is not compatible with the radical line $\rho(3)$, because the crossing sequence of $\zg$ is $(2,1),(3,4),(6,5),(8,7)$ and the one of $\rho(3)$ is either $(9),(1,2)(4,6),(5,8),(7)$ or $(1,9),(4,2),(5,6),(7,8)$, depending on whether its representative runs on the left or the right of the radical line. Neither of the two crossing sequences of $\rho(3)$ is compatible with the crossing sequence of $\zg$. 
 In order to obtain a compatible representative, we need to deform $\zg$ homotopically so that the crossing sequence becomes either $(3,1),(2,4),(6,5),(8,7)$, where the crossing with $3$ occurs at the very beginning, or $(2,1),(6,4),(8,5),(3,7)$, where the crossing with 3 occurs at the very end. We can describe the representative of the first crossing sequence by saying that $\zg$ starts at the boundary vertex labeled 18, then runs clockwise along the boundary of its first white region, crosses $\rho(3)$, then runs parallel to $\rho(3)$ until the white region that contains the other endpoint of $\zg$ and then ends by following the boundary of that white region until the boundary vertex labeled 9.
 
 \begin{lemma}
 \label{lem:homotopy radical}
 Let $\zg$ be a representative of  a 2-diagonal whose rotation $R\zg$ is a radical line  $\rho(x)$. Then $\zg$ and $R\zg$ are compatible if 
 $\zg$ starts by following the boundary of its first white region in clockwise direction, crosses $\rho(x)$ and then follows $R\zg$ up to its last white region at the other end and then ends by following the boundary of that white region in clockwise direction.
\end{lemma}
\begin{proof}
 Suppose $(y,z)$ is a crossing pair for $\zg$ and that $R\zg$ crosses both $\rho(y) $ and $\rho(z)$. Since $\zg$ runs parallel to a subpath of $R\zg$, the pair $(y,z)$ also is in the crossing sequence of $R\zg$. Thus $\zg$ and $R\zg$ are compatible.
\end{proof}

The next lemma will allow us to  fix suitable  representatives of the arcs $R^i\zg$.
\begin{lemma}
\label{lem:rotation orbit}  We can choose representatives of the elements of the rotation orbit $\{R^i \zg\}$  of $\zg$ such that for all $i$
\begin{itemize}
\item[(a)] $R^{2i+1}\zg$ is compatible with $R^{2i+2}\zg$ 
\item[(b)] There exists a homotopic deformation $(R^{2i+1}\zg)'$ of $R^{2i+1}\zg$ that is compatible with $R^{2i}\zg$.
\end{itemize}
\end{lemma}
\begin{proof}
First note that if $R^j\zg,R^{j+1}\zg$ are two consecutive arcs in the rotation orbit and none of them is a radical line then we can always choose compatible representatives. Indeed, the only type of problem that we need to avoid can come from a radical line $\rho(x)$ that crosses one but not the other; say $\rho(x)$ crosses $R^{j+1}\zg$ but not $R^j\zg$.  Then $\rho(x)$ shares an endpoint $a$ with $R^j\zg$ on the boundary of $\cals$. To achieve compatibility, we only need to ensure that $R^{j+1}\zg$ crosses $\rho(x)$ as close as possible to this point $a$ and then follows $R^j\zg$.

Suppose now that we have  three consecutive arcs $R^j\zg,R^{j+1}\zg,R^{j+2}\zg$, none of which is a radical line. By the argument above, it is possible to choose  representatives such that $R^j\zg$ and $R^{j+1}\zg$ are compatible, and it is possible to choose representatives such that $ R^{j+1}\zg$ and $R^{j+2}\zg$ are compatible. If $j$ is even, this proves the statement. If $j$ is odd however, we need to make sure that we can choose the two compatible pairs simultaneously. Therefore suppose that it is possible to apply an elementary homotopy to  $R^{j+1}\zg$. Thus $R^{j+1}\zg$ crosses all four radical lines of an hourglass shape as in Figure~\ref{fig:47}. If $R^j\zg$ and $R^{j+2}\zg$ also cross all four sides of the hourglass then there is no problem. So   the only obstruction would be if there is a radical line $\rho(x)$ that crosses $R^{j+1}\zg$ but not $R^j\zg$, so that we obtain a condition on the shape of $R^{j+1}\zg$, and there is another radical line $\rho(y)$ that crosses $R^{j+1}\zg$ but not $R^{j+2}\zg$, so that we obtain a second condition on the shape of $R^{j+1}\zg$, see Figure~\ref{fig compatible pairs}. In this situation, the arc $R^{j+1}\zg$ would pass below the crossing point of $\rho(x)$ and $\rho(y)$ for compatibility with $R^{j}\zg$, and it would  pass above the crossing point for compatibility with $R^{j+2}\zg$. Moreover, this difference would be a non-trivial homotopy if and only if there are two further radical lines $\rho(w),\rho(z)$ that cross all arcs involved and make an hourglass shape as shown in the figure. This, however, is impossible, because the radical lines $\rho(x)$ and $\rho(y)$ are  oriented in the same direction, since the endpoint $c$ is the second neighbor of the endpoint $a$ on the boundary, and thus the orientation  of the boundary segments of the shaded regions in the figure is not cyclic, contradicting Lemma~\ref{lem orientation}. 
This completes the proof in the absence of radical lines in the rotation orbit.

\begin{figure}
\begin{center}
\scalebox{0.8}{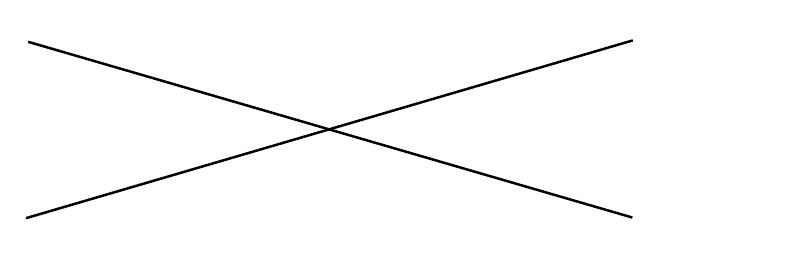}
\caption{Proof of Lemma \ref{lem:rotation orbit}}
\label{fig compatible pairs}
\end{center}
\end{figure}

Suppose now that a set of three consecutive arcs $\za,R\za,R^2\za$ contains a radical line.  If there is exactly one radical line, there are three possibilities. 
\begin{itemize}
\item[(i)] If $\za$ is the radical, we can deform $R\za$ to be compatible with $\za$ as explained in Lemma~\ref{lem:homotopy radical}. On the other hand, we can deform $R\za$ to be compatible with $R^2\za$, since none of the two is a radical line.

\item[(ii)]If $R\za$ is the radical line, we can deform both $\za$ and $R^2\za$ to be compatible with $R\za$ as explained in Lemma~\ref{lem:homotopy radical}. 

\item[(iii)]The case where $R^2\za$ is the radical line is symmetric to  case (i).

\end{itemize}
This proves the statement if the set of three arcs $R^{2i}\zg,R^{2i+1}\zg,R^{2i+2}\zg$ contains exactly one radical line.
Now suppose that there are exactly two radical lines among $\za,R\za,R^2\za$. Again there are three cases.
\begin{itemize}
\item [(iv)] If $\za$ and $R\za$ are the radical lines, then they are automatically compatible, by Definition~\ref{def:compatible}. Moreover $R^2\za$ can be deformed to be compatible with $R\za$, by Lemma~\ref{lem:homotopy radical}. 
\item[(v)] If $\za$ and $R^2\za$ are the radical lines, then Lemma~\ref{lem:homotopy radical} implies that $R\za$ is too. In this case, all three are automatically compatible.
\item[(vi)] The case where  $R\za$ and $R^2\za$ are the radical lines is symmetric to case (iv).
\end{itemize}
This proves the statement if the set of three arcs $R^{2i}\zg,R^{2i+1}\zg,R^{2i+2}\zg$ contains exactly two radical lines.
If all three arcs are radical lines, then they are automatically compatible, so there is nothing to show. 
\end{proof}


We now construct a projective resolution of $M_\zg$. We assume that the representatives $R^i\zg$ satisfy the conditions in  Lemma~\ref{lem:rotation orbit}.
 Because of condition (a), Proposition~\ref{big-lemma} implies
 \begin{equation}
\label{eq:pr1}
\im f_{R^{2i+2}\zg} =\ker \overline{f}_{R^{2i+1}\zg}, \textup{ for all $i$}.
\end{equation}
We use condition (b) in order to obtain a similar identity between the degrees $2i$ and $2i+1$ as follows. Let $(R^{2i+1}\zg)'$ be as in condition (b). From Corollary~\ref{cor:comp} we know that the homotopy gives rise to an automorphism 
\[(\varphi_{2i+1,\,1},\varphi_{2i+1,\,0})\in \Aut P_{2i+2}(\zg)\oplus \Aut P_{2i+1}(\zg)\]
such that 
\begin{equation}
\label{eq:pr1a}
 f_{(R^{2i+1}\zg)'}  =  \varphi_{2i+1,\,0}\, f_{R^{2i+1}\zg}\,\varphi_{2i+1,\,1}^{-1}
\end{equation}
is the conjugation of $ f_{R^{2i+1}\zg}$ by this automorphism.  Because ${(R^{2i+1}\zg)'}$ and 
$R^{2i}\zg$ are compatible, Proposition~\ref{big-lemma} yields
\begin{equation}
\label{eq:pr2}
\im \, \overline{ \varphi}_{2i+1,\,0}\, \overline{f}_{\!R^{2i+1}\zg}\,\overline{\varphi}_{2i+1,\,1}^{-1}
= \ker f_{\!R^{2i}\zg}, \textup{ for all $i$.}
\end{equation}
 We are now ready to write the projective resolution of $M_\zg$. 
 
\begin{prop}
 \label{prop:projective resolution}
 With the above notation, the following sequence is a projective resolution. 
 \[\begin{array}{c}\cdots\xymatrix@C45pt{ P_3(\zg)\ar[r]^{\overline{\varphi}_{1,1} f_{\!R^2\zg}} 
 &P_2(\zg)\ar[r]^{\overline{\varphi}_{1,0} \overline{f}_{\!R\zg} \overline{\varphi}_{1,1}^{-1}}
 &P_1(\zg)\ar[r]^{f_\zg}
 &P_0(\zg)\ar[r]&M_\zg\ar[r]&0}
 \\[5pt]
 \cdots\xymatrix@C85pt{
 P_{2i+2}(\zg)\ar[r]^{\overline{\varphi}_{2i+1,0}\, \overline{f}_{\!R^{2i+1}\zg} \,\overline{\varphi}_{2i+1,1}^{-1}}
&P_{2i+1}(\zg)\ar[r]^{\overline{\varphi}_{2i-1,1}\, f_{\!R^{2i}\zg}}
&P_{2i}(\zg) 
  } \cdots
  \end{array}
 \]
\end{prop}
 
\begin{proof}
 The sequence is exact at $P_1(\zg)$ by equation (\ref{eq:pr2}) with $i=0$. To show exactness at $P_{2i}(\zg)$, we note that on the one hand,
 $\im (\overline{\varphi}_{2i-1,1}\, f_{\!R^{2i}\zg}) 
 =\overline{\varphi}_{2i-1,1}(\im\, f_{\!R^{2i}})$,
 because $\overline{\varphi}_{2i-1,1}$ is an isomorphism, and on the other hand, 
 $\ker ({\overline{\varphi}_{2i-1,0}\, \overline{f}_{\!R^{2i-1}\zg} \,\overline{\varphi}_{2i-1,1}^{-1}})
 =
 \overline{\varphi}_{2i-1,1} (\ker \overline{f}_{\!R^{2i-1}\zg} )$,
 because $\overline{\varphi}_{2i-1,0}$ and $ \overline{\varphi}_{2i-1,1}^{-1}$ are isomorphisms. Now exactness follows from equation (\ref{eq:pr1}).
 
 It remains to show exactness at $P_{2i+1}(\zg)$. We have
 $\ker( \overline{\varphi}_{2i-1,1}\, f_{\!R^{2i}\zg} )=\ker f_{\!R^{2i}\zg} $, since  $\overline{\varphi}_{2i-1,1}$ is an isomorphism. Now the result follows from equation (\ref{eq:pr2}).
\end{proof}
 \smallskip
 
 We are now ready for the main result of this section. It describes the morphism of complexes induced by a 2-pivot on the projective resolutions.
 \newcommand{\newgamma}{\zd}
 \newcommand{\newphi}{\chi}
\begin{thm}\label{prop:67}
Let $\zg$ be a representative of a 2-diagonal and $\newgamma$ a representative of the 2-diagonal obtained from $\zg$ by a 2-pivot such that $\zg$ and $\newgamma$ are compatible. Choose representatives of the rotation orbits of $\zg$ and $\newgamma$ that satisfy  conditions (a) and (b) in Lemma~\ref{lem:rotation orbit}. Then the pivot morphism $g\colon M_{\zg}\to M_{\newgamma}$ induces a morphism of complexes on the projective resolutions so that the following diagram commutes.    
\[
\xymatrix@C=35pt{\  \ar@{.>}[r] 
&P_3(\zg) \ar[r]^{\overline{\varphi}_{1,1} f_{\!R^2\zg}} \ar[d]^{}
& P_2(\zg) \ar[r]^{\overline{\varphi}_{1,0} \overline{f}_{\!R\zg} \overline{\varphi}_{1,1}^{-1}} \ar[d]^{g_2^r}
& P_1 (\zg)\ar[r]^{f_{\zg}} \ar[d]^{g_1^c}& P_0(\zg) \ar[r]\ar[d]^{g_0^r} & M_{\zg} \ar[r]\ar[d]^{g} & 0
\\
\ \ar@{.>}[r]
& P_3(\newgamma) \ar[r]_{\overline{\newphi}_{1,1} f_{\!R^2\newgamma}} 
& P_2(\newgamma) \ar[r]_{\overline{\newphi}_{1,0} \overline{f}_{\!R\newgamma} \overline{\newphi}_{1,1}^{-1}} 
& P_1 (\newgamma)\ar[r]_{f_{\newgamma}} 
& P_0(\newgamma) \ar[r] & M_{\newgamma} \ar[r] & 0}
\]
\end{thm}

\begin{proof}
To fix notation let $a'$ be the common endpoint of $\zg$ and $\newgamma$ and let $b'$ be its clockwise neighbor on the boundary of $\cals.$ Let $a,b,c$ denote consecutive vertices on the boundary in clockwise order, such that $a$ is the second endpoint of $\zg$ and $c$ is the second endpoint of $\newgamma$. Then $R\zg$ has endpoints $b$ and $b'$, see Figure~\ref{fig pivotcompatible}.
 
First let us check that we can simultaneously choose the representatives of the arcs in the two rotation orbits such that they satisfy both  Lemma~\ref{lem:rotation orbit} as well as the condition that $\zg$ and $\newgamma$ are compatible. 
If $R\zg$ is not radical line then the lemma does not impose any restriction on $\zg$ and there is nothing to show. If $R\zg$ is a radical line $\rho(x)$, then we will choose $\zg$ and $R\zg$ according to the compatibility condition stated in Lemma~\ref{lem:homotopy radical}.
There are exactly two possibilities for $R\zg$ since it can run along either side of the radical line $\rho(x)$. We   choose $R\zg$ to run from $b$ to $b'$ on the left of $\rho(x)$ so that it is closer to $\newgamma$. Then we can choose $\zg$ to  start at $a'$, then follow the boundary of its first white region clockwise, then follow $R\zg$ until its last shaded region at $b$, then cross $R\zg=\zd(x)$ and  end by following the boundary of the white region in the counterclockwise direction until $a$. In the case where $R\newgamma$ is a radical we perform the `same' deformation to $\newgamma $ as to $\zg$ so that both are crossing $R\zg$ close to the endpoint $a'$. This choice assures that $\zg$ and $\newgamma$ are compatible. 

Therefore Theorem~\ref{lem 63} implies that 
 $g_0^r \,f_\zg = f_\newgamma\,g_1^c$ and thus the degree 0,1 square of our diagram commutes.
We also know that the rows of our diagram are exact, thanks to Proposition~\ref{prop:projective resolution}. 
Moreover, because of equation (\ref{eq:pr1a}) with $i=0$, we have 
\begin{equation}\label{eq 611} f_{(R\zg)'}  =  \varphi_{1,0}\, f_{R\zg}\,\varphi_{1,1}^{-1},\end{equation} 
thus
 the map from degree 2 to degree 1 in the first row of our diagram is equal to $\overline{f}_{(R\zg)'}$, where $(R\zg)'$ is a representative of $R\zg$ that is compatible with $\zg$. Similarly, the corresponding map in the second row is given as $\overline{f}_{(R\newgamma)'}$, where $(R\newgamma)'$ is an arc that is compatible with $\newgamma$. 
Thus we have the following chain of compatibility 
$ {(R\newgamma)'}\sim \newgamma\sim \zg \sim {(R\zg)'}$.

 We will now show that this implies that ${(R\newgamma)'}$ and ${(R\zg)'}$ are compatible as well. 
 \begin{figure}
 \begin{center}
\scalebox{0.8}{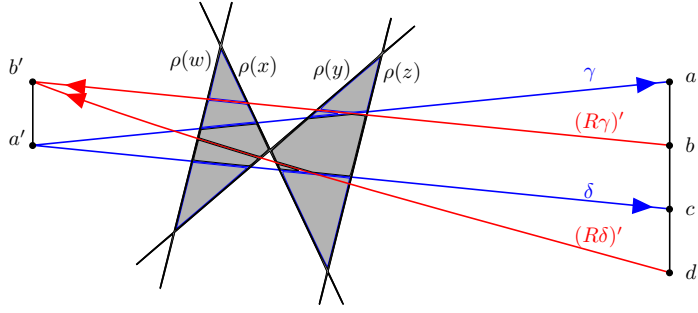}
\caption{Proof of Proposition \ref{prop:67}. In this picture , for the sake of argument, the arcs $\zg$ and $\newgamma$ are not compatible.}
\label{fig pivotcompatible}
\end{center}
\end{figure}
The relative position  of the arcs involved is illustrated in Figure~\ref{fig pivotcompatible}, however the arcs $\zg$ and $\newgamma$ are not compatible in the picture.
Suppose that $(w,x)$ is a  pair in the crossing sequence of $(R\zg)'$ such that $(R\newgamma)'$ crosses  $\rho(w) $ and $\rho(x)$. If $(w,x)$ also is a pair in the crossing sequence of $(R\newgamma)'$ we are done.
Otherwise, there exist two additional radical lines $\rho(y),\rho(z)$ that together with $\rho(w) $ and $\rho(x)$ form two triangular shaded regions meeting at the crossing point of $\rho(x)$ and $\rho(y)$, such that
none of the vertices of the shaded triangles lies on the boundary of $\cals$ and
the arcs $(R\zg)'$ and $(R\newgamma)'$
also cross $\rho(y)$ and $\rho(z)$, as shown in the figure. Then \[R\zg\sim R\zd \ssi (x,w) \textup{ is a crossing pair in both} \ssi  (y,z) \textup{ is a crossing pair in both.} \]
 This implies that both $\zg$ and $\newgamma$ cross at least three of the four radical lines and so both cross both arcs of one of the crossing pairs $(x,w)$ or $(y,z)$. Without loss of generality, we assume that they cross the pair $(x,w)$. 
Now the compatibility of $(R\zg)'$ and $\zg$ implies that $(x,w)$ is a pair in the crossing sequence of $\zg$. Therefore  we also have the pair $(x,w)$ in the crossing sequence of $\newgamma$, because $\zg$ and $\newgamma$ are compatible. But this implies that $(w,x)$ is a pair in the crossing sequence of $(R\newgamma)'$, because of the compatibility of $(R\newgamma)'$ and $\newgamma$. 
This shows that $(R\zg)'$ and $(R\newgamma)'$ are compatible.

Because of equation (\ref{eq 611}), this means that the two horizontal maps from degree 2 to degree 1 in the proposition are given by two compatible arcs $(R\zg)'$ and $(R\newgamma)'$. Therefore Corollary~\ref{lem: pivot-bar} implies that the degree 1,2 square of the diagram commutes as well, and the proof is complete.
\end{proof}
\begin{remark}
 We did not specify the vertical map in degree 3 in Proposition~\ref{prop:67}. If the  horizontal maps from degree 3 to degree 2 are also induced from a representative of a 2-diagonal, then we can iterate the argument of the proof and conclude that the vertical map in degree 3 is $g_3^c$. 
\end{remark}

\section{Proofs of section \ref{sect 7}} This section contains the proofs of several results from section~\ref{sect 7}. 
We  start by studying the matrix shape of our maps $f_\zg$ in greater detail in subsection~\ref{sect 7.1}.

\medskip
\subsection{The staircase shape of $f_\zg$}\label{sect 7.1} 
Throughout this subsection, $\zg$ is a representative of a 2-diagonal in the checkerboard polygon $\cals$. Recall that the morphism $f_\zg\colon P_1(\zg)\to P_0(\zg)$ is defined in terms of the crossing sequence of $\zg$ in Definition~\ref{def:map}.  We  assume without loss of generality that the crossing sequence is of the form $(i_1,j_1), \ldots, (i_n,j_n)$, since the cases where there is an $i_0$ (or $j_0$) that is not part of a pair, follow by the same arguments. The new key ingredient in this section is to consider the steps in the crossing sequence. Also recall that the subquiver $Q(\zg)$ of $Q$ was introduced in Definition \ref{def:5.8}.

Let $\zg$ be a 2-diagonal with crossing subsequence
 \[\xymatrix{(i_s, j_s)\ar[r]^-{\zS_s} & (i_{s+1}, j_{s+1})\ar[r]^-{\zS_{s+1}} &  \cdots\ar[r]^-{\zS_{t-1}} &  (i_t, j_t)}\]
 and suppose that every step $\zS_l$ is forward. Recall that the crossings between $\zg$ and $i_l$ are of degree zero and the crossings between $\zg$ and $j_l$ are of degree one, for $l=s, s+1, \ldots, t$.
   
 Let $\za_l$ be the arrow that connects $i_l$ and $j_l$ in the quiver $Q$. Then, 
  by condition ($Q_3$) of subsection \ref{sect 3.1.1},
  $\za_l$ lies in at most two 3-cycles in $Q$ and exactly one of them lies in the subquiver $Q(\zg)$. We denote the third vertex of that 3-cycle by $k_l$, so the cycle is $$\xymatrix@R10pt@C10pt{i_l\ar@{-}[rr]^{\za_l}&&j_l\ar@{-}[ld]\\&k_l\ar@{-}[ul]}.$$
 
 When we remove the vertex $k_l$ and the arrow $\za_l$ from the quiver $Q(\zg)$ we obtain two connected components. We say that the vertex $i_l$ is the \emph{entry} to the $l$-th pair $(i_l,j_l)$ if it belongs to the same connected component as the previous pair $(i_{l-1},j_{l-1})$, and the \emph{exit} from the $l$-th pair otherwise.

In Figure 9, the vertex $i_{s+1}$ is the entry in cases (i) and (iv) and $j_{s+1}$ is the entry in cases (ii) and (iii).

We say that the step $\zS_l$ from $l$ to $l+1$ in the crossing sequence has \emph{degree zero} if the entry to the $l$-th pair is $i_l$, and we say it has \emph{degree one} if the entry is $j_l$.

The following lemma is a reformulation of Lemma \ref{rect-trap-lemma1}.
\begin{lemma}(Rectangle-Trapezoid Lemma) \label{rect-trap-lemma}
With the notation above,   the following properties hold for all $l\ge s$.
\begin{itemize}

\item[(1a)] Suppose the step $\zS_{l+1}$ from $l+1$ to $l+2$ is rectangular. 
Then 
 there exists a pair of valid paths 
\[
\left\{
\begin{array}{ll} 
(i_{l} \leadsto j_{l+2},\  j_{l+1}\leadsto i_{l+3})
&\textup{\it if the step is of degree 0;}\\
(j_{l} \leadsto i_{l+2},\  i_{l+1}\leadsto j_{l+3})
&\textup{\it if the step is of degree 1.}
\end{array}
\right.\]
Moreover there is no valid path from $i_l$ to $j_{l+3}$ or from $j_l$ to $i_{l+3}$.

\item[(1b)] Suppose the step  $\zS_{l+1}$  from $l+1$ to $l+2$ is trapezoidal.  Then there exist a valid path 
\[\left\{\begin{array}{ll} 
i_l\leadsto j_{l+3}
&\textup{\it if the step is of degree 0;}\\
 j_l\leadsto i_{l+3}
&\textup{\it if the step is of degree 1.}\\
\end{array}\right.\]

\item[(2)] There is a valid path $i_s\leadsto j_t$ (respectively $j_s\leadsto i_t$) if and only if all the steps from $s+1$ to $t-1$ are trapezoidal and the step from $s+1$ to $s+2$ has degree 0 (respectively degree 1).  

\item[(3)] If there is a valid path $i_s\leadsto j_t$ then there are valid paths $i_{l} \leadsto j_{l'}$ for all $l, l' = s, s+1, \dots, t$ with $l<l'$. 

\item[(4)] If there exist valid paths $i_s\leadsto j_u$ and $i_t\leadsto j_w$ with $s<t<u<w$ then there exists a valid path $i_s\leadsto j_w$.  

\end{itemize}
\end{lemma}

\begin{proof}
 The proof is the same as for Lemma \ref{rect-trap-lemma1}. 
\end{proof}

 We  shall use the notation $T_0$ and $T_1$ for a trapezoidal step of degree $0$ or $1$, respectively. Similarly, $R_0$ and $R_1$ will denote a rectangular step of degree $0$ or $1$, respectively. 

\begin{corollary}
 \label{cor RTL}
 Let $\zS_{l+1}$ be a forward step in the crossing sequence of $
 \zg$ and denote the matrix of $f_\zg\colon P_1(\zg)\to P_0(\zg)$ by $f_\zg=(f_{s t})_{s,t=1,\ldots,n}$. Then 
 \begin{itemize}
\item [$(R_0$)] If the step is rectangular of degree $0$ then $f_{l(l+2)}\ne 0$ and $f_{(l+1)(l+3)}= 0$;
\item [$(R_1$)] If the step is rectangular of degree $1$ then $f_{l(l+2)}= 0$ and $f_{(l+1)(l+3)}\ne 0$;
\item [$(T_0$)] If the step is triangular of degree $0$ then  $f_{l(l+2)}\ne 0$, $f_{l(l+3)}\ne 0$, $f_{(l+1)(l+3)}\ne 0$;
\item [$(T_1$)] If the step is triangular of degree $1$ then $f_{l(l+2)} =f_{l(l+3)}=f_{(l+1)(l+3)}= 0$.
\end{itemize}
 \end{corollary}
\begin{proof}
 ($R_0$) Part (1a) of Lemma \ref{rect-trap-lemma} implies that there exists a valid path $i_l\leadsto j_{l+2}$, thus  $f_{l(l+2)}\ne 0$. The lemma also implies that there is a valid path $j_{l+1}\leadsto i_{l+3}$,  thus there is no valid path $i_{l+1}\leadsto j_{l+3}$, and hence  $f_{(l+1)(l+3)}= 0$. This shows ($R_0$) and the proof for ($R_1$) is similar.
 
 ($T_0$) Part (1b) of the lemma implies there exists a valid path $i_l\leadsto j_{l+3}$, and therefore part (3) of the lemma guarantees the existence of valid paths $i_l\leadsto j_{l+2}$ and $i_{l+1}\leadsto j_{l+3}$. Thus all three entries  $f_{l(l+3)}, f_{l(l+2)}$ and $f_{(l+1)(l+3)}$ are nonzero.
  
 ($T_1$) In this case, part (1b) of the lemma yields a valid path $j_l\leadsto i_{l+3}$, hence there are no valid paths
 $i_l\leadsto j_{l+3}$,
  $i_l\leadsto j_{l+2}$ and $i_{l+1}\leadsto j_{l+3}$. Hence all three entries $f_{l(l+3)}, f_{l(l+2)}$ and $f_{(l+1)(l+3)}$ are zero.
\end{proof}

Next we show that rectangular steps change the degree  and trapezoidal steps preserve the degree.
\begin{lemma}
 \label{lem degree}
 Let $\xymatrix{(i_s, j_s)\ar[r]^-{\zS_l}& (i_{s+1}, j_{s+1})\ar[r]^-{\zS_{l+1}}&(i_{s+2}, j_{s+2})}$ be two consecutive forward steps in the crossing sequence for $\zg$.
 
\begin{itemize}
\item[(a)] If $\zS_l$ is rectangular then $\zS_l$ and $\zS_{l+1}$ have opposite degrees.

\item[(b)] If $\zS_l$ is trapezoidal then $\zS_l$ and $\zS_{l+1}$ have the same  degree.

\item[(c)] $f_{(s+1)(s+1)}$ is an arrow if and only if $f_{s(s+2)}=0$.
\end{itemize}

\end{lemma}
\begin{proof}
 This follows immediately from Figure \ref{fig:valid-paths} in Section \ref{sect 5.1}.
 \end{proof}

\begin{example}\label{ex 74}
 Suppose the crossing sequence of $\zg$ consists of the following eight steps $T_0T_0T_0R_0T_1R_1R_0R_1$.  Then the matrix $f_\zg$ is of the form
 \[
\left[\begin{matrix}
 f_{11}&f_{12} &f_{13} &f_{14} &f_{15}&0&0&0&0\\
 0&f_{22} &f_{23} &f_{24} &f_{25}&0&0&0&0\\
  0&0 &f_{33} &f_{34} &f_{35}&0&0&0&0\\
  0&0 &0 &f_{44} &f_{45}&0&0&0&0\\
  0&0 &0 &0 &f_{55}&f_{56}&0&0&0\\
  0&0 &0 &0 &0&f_{66}&f_{67}&f_{68}&0\\
  0&0 &0 &0 &0&0&f_{77}&f_{78}&0\\
  0&0 &0 &0 &0&0&0&f_{88}&f_{89}\\
  0&0 &0 &0 &0&0&0&0&f_{99}\\
\end{matrix}\right]
\]
Indeed, the two consecutive $T_0$ in steps 2 and 3 imply the nonzero entries in positions 13, 14, 15, 24, 25 and 35. The first $R_0$ yields a nonzero entry in position 35 and a zero in position 36, which then gives zeros in positions 16 and 26. The fifth step $T_1$ yields zeros in positions 46, 47 and 57. The $R_1$ in step 6 gives a zero in position 57 and a nonzero entry in position 68, and the $R_0$ in step 7 gives a nonzero entry in position 68 and a zero in position 79, which also implies that position 69 is zero. The last step $R_1$ gives a zero in position 79.

The matrix $f_\zg$ thus has staircase shape with blocks limited by the columns 5,6,8 and 9.
\end{example}

We now describe  in general the staircase shape  of $f_\zg$ observed in the example. 
Suppose the crossing sequence of $\zg$ is $\xymatrix{(i_1, j_1)\ar[r]^-{\zS_1}& (i_{2}, j_{2})\ar[r]^-{\zS_2}& \cdots \ar[r]^-{\zS_{n-1}}&(i_n, j_n)}$ and that every step is forward. We define the \emph{source} $s(\zS_l)$ and the \emph{target} $t(\zS_l)$ of the step $\zS_l$ to be the integers $s(\zS_l)=l$ and $ t(\zS_l)=l+1$.
\begin{definition}\label{def ti}
Define a sequence of integers $1=t_0<t_1<\ldots<t_p=n$ recursively by setting 
\[t_1=\min
\left\{
\begin{array}{ll}
\textup{target of the first rectangular step of degree 0 after step 1,}\\ \textup{target of the first trapezoidal step of degree 1 after step 1,}
\\
\textup{source of the first rectangular step of degree 1 after step 1}
\end{array}
\right\}
\]
and recursively, 

\[t_{l+1}=\min
\left\{
\begin{array}{ll}
\textup{target of the first rectangular step of degree 0 after step $t_l-1$,}\\ \textup{target of the first trapezoidal step of degree 1 after step $t_l-1$},
\\
\textup{source of the first rectangular step of degree 1 after step  $t_l$}
\end{array}
\right\}
\]
if it exists, and $t_{l+1}=n$ and $p=l+1$, otherwise.
This sequence is called the \emph{staircase sequence of $f_\zg$}.
\end{definition}

Note that the definition of the staircase sequence is independent of the first step in the crossing sequence.

In Example \ref{ex 74}, we have $t_1=5$, because step four is $R_0$, $t_2=6$, because step five is $T_1$, and $t_3=8$, because step seven is $R_0$. Thus the staircase sequence in this example is $1<5<6<8<9$, as observed in the matrix of Example~\ref{ex 74}.

\begin{prop}\label{prop 71}
 The matrix $f_\zg$ has an upper triangular staircase shape whose nonzero blocks are limited by the columns $t_1,t_2,\ldots, t_n$ as illustrated in the left picture in Figure \ref{fig staircase}.
 
 In other words, $f_\zg$ is nonzero on the main diagonal $(f_{ss})_{s=1,\ldots n}$ (and on the diagonal  $(f_{s(s+1)})_{s=1,\ldots n-1}$) as well as on $p$ triangular regions $J_1,J_2,\ldots,J_{p}$, where the first row of the region $J_1$ is the row $t_{l-1}$, the last column of $J_l$ is the column $t_l$ and the hypotenuse of $J_l$ is given by the second diagonal $(f_{s(s+1)})_{ t_{l-1}\le s \le t_l-1}$.
\end{prop}
\begin{figure}
\begin{center}
\newcommand{\sto}{ \scalebox{0.5}{>}}
\scriptsize \scalebox{0.9}{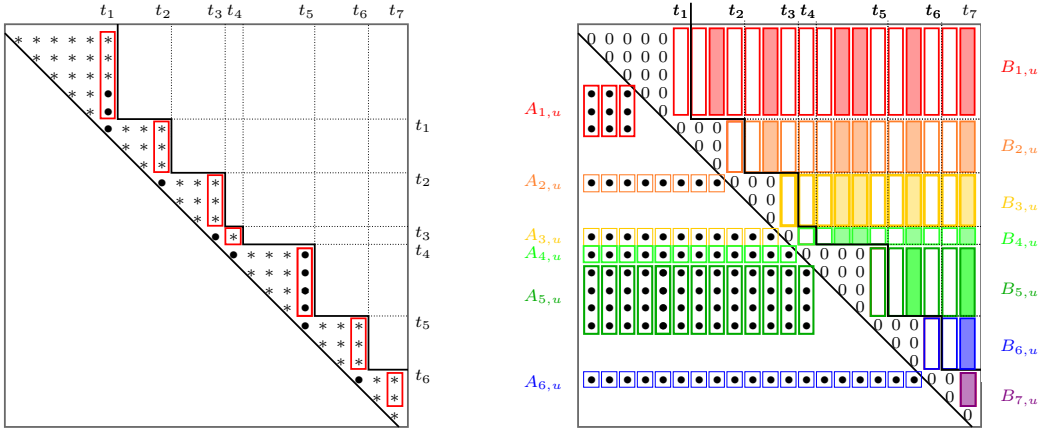}
\caption{On the left, the staircase shape of the matrix of $f_\zg$. Entries in blank positions are zero and entries in positions marked with a bullet point are given by arrows. The regions $B_l$ are framed in red. The picture on the right indicates the blocks  of the matrix $g_0$. The blocks $B_{l,u}$ are the framed columns above the diagonal; the entries in the shaded blocks are zero. The blocks $A_{l,u}$ are the framed columns with bullet points. These blocks project horizontally onto the positions of the arrows marked by bullet points in the matrix of $f_\zg$.}
\label{fig staircase}
\end{center}
\end{figure}


We now define two types of blocks $A_l$ and $B_l$, both being subsets of column $t_l$ of the matrix of $f_\zg$. 

\begin{definition} \label{def blocks of f}
\begin{itemize}
\item[(a)] Let $a_l$ be the least positive integer such that the entry in $f_\zg$  at position $(a_l,t_l)$ is an arrow. We denote by $A_l$ the following block in $f_\zg$ 
\[A_l=[f_{a_l t_l},f_{(a_l +1)t_l},\ldots,f_{t_l,t_l}]^T \] 
Then each of the positions in $A_l$ is given by an arrow in the quiver. In the example in the left picture of Figure~\ref{fig staircase}, the positions of $A_l$ are marked by the symbol $\bullet$ .

 \item[(b)] We denote by $B_l$ the last column  of the triangular region $J_l$ in the proposition. Thus 
\[B_l= 
\begin{bmatrix}
 f_{t_{l-1} t_l} & f_{(t_{l-1}+1) t_l}& \cdots & f_{(t_l-1) t_l}
\end{bmatrix}^T\]
 In the left picture of Figure~\ref{fig staircase}, these blocks are the  framed in red color.
\end{itemize}
\end{definition}

 Let $a_l$ be the least positive integer such that the entry in $f_\zg$ at position $(a_l,t_l)$ is an arrow. Then $a_l\le t_l$ and each of the positions $(a_l,t_l),(a_l+1,t_l),\ldots,(t_l,t_l)$ in $f_\zg$ is given by an arrow in the quiver. These positions are marked by a $\bullet$ in the left picture of Figure~\ref{fig staircase}.

\begin{corollary}
 \label{cor difference}
 With the notation above, 
the sequence of steps $\zS_{t_{l-1}},\zS_{t_{l-1}+1},\ldots,\zS_{t_{l}-1}$ is equal to 
 \[\left\{\begin{array}{ll} T_1  &\textup{if $t_{l}-t_{l-1}=1$};\\
 R_1\underbrace{T_0\cdots T_0}_{t_{l}-t_{l-1}-2} R_0  &\textup{if $t_{l}-t_{l-1}\ge 2$}.\\
 \end{array}\right.\]
%
In particular, $\zS_{t_{l-1}} =R_1$ unless $t_{l}-t_{l-1}=1$.
\end{corollary}
\begin{proof}
Let $m=t_{l}-t_{l-1}$. 
 Since $t_{l-1}$ is in the staircase sequence  of Definition \ref{def ti}, we must have $\zS_{t_{l-1}-1}=R_0$ or $T_1$, or $\zS_{t_{l-1}}=R_1$. 
 
Suppose first that $\zS_{t_{l-1}-1}=R_0$. Then $\zS_{t_{l-1}}$ is of degree 1, by Lemma \ref{lem degree}. If $\zS_{t_{l-1}}=T_1$ then $t_{l}=t(\zS_{t_{l-1}}) =t_{l-1}+1$, which is the  case $m=1$ of the statement.
If $\zS_{t_{l-1}}=R_1$ then its terminal is not in the staircase sequence, and thus $m>1$. The following step $\zS_{t_{l-1}+1}$ is of degree 0. If it is $R_0$ then its terminal is $t_{l}$, and we have $m=2$. Otherwise,  $\zS_{t_{l-1}+1}=T_0$, and the following step  $\zS_{t_{l-1}+2}$ is of degree 0. Continuing this way, we construct the sequence in the corollary. 

If $\zS_{t_{l-1}-1}=T_1$, then again $\zS_{t_{l-1}}$ is of degree 1, by Lemma \ref{lem degree}, and the proof is the same as above. 

Suppose now that $\zS_{t_{l-1}}=R_1$. Then the degree of $\zS_{t_{l-1}+1}$ is 0. If $\zS_{t_{l-1}+1}=R_0$ then $m=2$ and the sequence is $R_1R_0$. Otherwise 
 $\zS_{t_{l-1}+1}=T_0$ and the following step $\zS_{t_{l-1}+2}$ is of degree 0. Again we obtain the sequence in the corollary by repeating this argument.
\end{proof}

 The next lemma discusses the position of arrows in $f$.  
 
 \begin{lemma}\label{dots}
 Suppose $f_{st}$ is an arrow such that $t-s\geq 2$.  Then exactly one of the following statements holds.  
 \begin{itemize}
 \item[(i)] $f_{st}, f_{(s+1)t}, \dots, f_{tt}$ are arrows and $f_{s't'}$ are not arrows for all $s'\geq s, t'<t, s'\leq t'$ and $(s',t')\not=(s,s)$.
   \item[(ii)] $f_{ss}, f_{s(s+1)}, \dots, f_{st}$ are arrows and $f_{s't'}$ are not arrows for all $s'> s, t'\leq t, s'\leq t'$ and $(s',t')\not=(t,t)$.
\end{itemize}
\end{lemma}

\begin{proof}
Since $f_{st}$ is an arrow, it follows that the entries $f_{s't'}$ in the triangular region with $s'\geq s, t'\leq t, s'\leq t'$ are given by valid paths.  Then Lemma~\ref{lem degree}(c) implies that $f_{s's'}$ with $s<s'<t$ are not arrows.   In particular, since $t-s\geq 2$ we have that $f_{(s+1)(s+1)}$ is not an arrow. 

First, suppose that $f_{s(s+1)}$ is not an arrow, then we want to show that statement (i) of the lemma holds.   Then there is an arrow $j_t\to j_{s+1}$, because if this path were of length two or more then the path $i_{s+1}$ to $j_t$ would not be valid, hence $f_{(s+1)t}=0$, a contradiction.  Then since $f_{st}$ is an arrow, there is a path of length two $i_s\to j_t\to j_{s+1}$, see Figure~\ref{fig:lem_dots} on the left.   Note that the path from $j_s$ to $i_{s+1}$ is either a single arrow or a composition of two arrows.  If $t=s+2$ then we identify the two vertices labeled $j_t$ and $j_{s+2}$ given in the figure.   Then both $f_{(s+1)t}, f_{(s+2),t}$ are arrows and statement (i) holds.  Otherwise, if $t>s+2$ then the quiver is as in the figure.  In particular, 
$f_{(s+1)(s+2)}$ is not an arrow, so there is a path of length two $i_{s+1}\to j_t\to j_{s+2}$.  If $t=s+3$ then condition (i) holds by the same reasoning as above, and otherwise if $t>s+3$ we can continue the argument in the same way.  This shows the lemma in the case when $f_{s(s+1)}$ is not an arrow.

Now, suppose that $f_{s(s+1)}$ is an arrow, then we want to show that statement (ii) of the lemma holds.  We are in the situation of Figure~\ref{fig:lem_dots} on the right.  In particular, since $f_{s(s+1)}$ is an arrow then $f_{(s+1)(s+1)}$ is not an arrow, because otherwise there would be no valid path from $i_s$ to $j_t$ and thus $f_{st}=0$.  Then $f_{ss}$ is an arrow.  Moreover $f_{s(s+1)}$ is an arrow, because otherwise again the path from $i_s$ to $j_t$ would not be valid.  If $t=s+2$, then $f_{ss}, f_{s(s+1)}, f_{s(s+2)}$ are arrows and (ii) holds.  If $t>s+2$ then $f_{s(s+3)}$ is an arrow and we can continue the argument in the same way to conclude that condition (ii) holds. This shows the lemma in the case when $f_{s(s+1)}$ is an arrow and completes the proof. 
\end{proof}

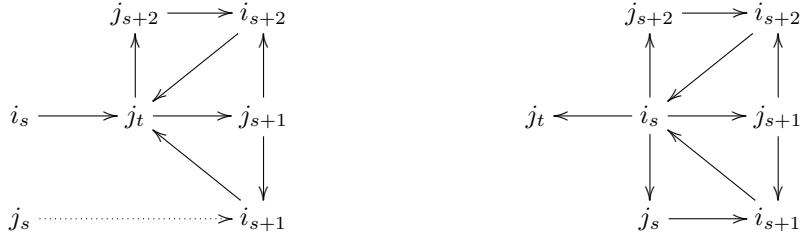
\begin{figure}
\[\xymatrix{&j_{s+2}\ar[r] & i_{s+2}\ar[dl] &&&&j_{s+2}\ar[r] & i_{s+2}\ar[dl]\\
i_s\ar[r] & j_t \ar[r] \ar[u]& j_{s+1}\ar[d] \ar[u] &&& j_t& i_s\ar[r] \ar[l] \ar[d] \ar[u]& j_{s+1}\ar[d]\ar[u]\\
j_s \ar@{..>}[rr] && i_{s+1}\ar[ul] &&&& j_s\ar[r] & i_{s+1}\ar[ul]}\]
\caption{Proof of Lemma~\ref{dots}.}
\label{fig:lem_dots}
\end{figure}

In the example of Figure~\ref{fig staircase} we have $f_{t_4t_5}$ is an arrow and $f_{t_4(t_5-1)}$ is not, so the positions $f_{(t_4+1)t_5}, \dots, f_{t_5t_5}$ are arrows.

\subsection{Nilpotent endomorphisms of $M_\zg$}\label{Asect 7.2} 
 In this subsection, we provide the proofs for section \ref{sect 7.2}.
\subsubsection{A preparatory lemma}
We need the following lemma from homological algebra.

\begin{lemma}\label{lem:cmp-morphism}
Let $g: M\to M'$ be a morphism in $\textup{mod}\,B$ that induces maps $g_0, g_1$ on the projective presentations of $M$ and $M'$ such that the following diagram commutes. 

\[\xymatrix{
P_1 \ar[r]^{f} \ar[d]_{g_1} & P_0\ar[d]^{g_0}\ar[r]^{\pi} \ar@{-->}[dl]^{h_1}& M \ar[r] \ar[d]^g \ar@{-->}[dl]^{h}& 0 \\
P_1' \ar[r]_{f'} & P_0' \ar[r]_{\pi'} & M' \ar[r] & 0
}
\]

Then $g$ is zero in $\underline{\textup{mod}}\,B$ if and only if $g_0=f' h_1+ g'$ for some maps $h_1, g'$ such that $g'f=0$.  
\end{lemma}

\begin{proof}
The following sequence of arguments yields the desired result.  
\begin{align*}
g \text{ is zero in } \underline{\textup{mod}}\,B && \text{ iff } &&& g \text{ factors through the projective cover of } M'\\
&& \text{ iff } &&& g=\pi'h \text{ for some } h: M\to P_0'\\
&& \text{ iff } &&& g\pi = \pi' h \pi \text{ because } \pi \text{ is surjective}\\
&& \text{ iff } &&& \pi'g_0 = \pi' h \pi \text{ because the diagram commutes}\\
&& \text{ iff } &&& \pi'(g_0-h\pi)=0\\
&& \text{ iff } &&& g_0-h\pi \text{ factors through } f'\\
&& \text{ iff } &&& g_0-h\pi = f' h_1 \text{ for some } h_1: P_0\to P_1'\\
&& \text{ iff } &&& g_0=f'h_1+h\pi\\
&& \text{ iff } &&& g_0=f'h_1+g' \text{ such that } g'f=0. \qedhere
\end{align*}
\end{proof}

\subsubsection{The compositions $g_0f_{\zg}$ and $f_\zg g_1$}\label{sec:comp}
We study the composions of these morphisms as products of matrices.  We denote the entries of the matrix of $f_\zg$ by $f_{uv}$ as before, and the entries of $g_0$ by $i_{uv}$ and those of $g_1$ by $j_{uv}$. Hence  paths from $i_v$ to $i_u$ will contribute to the entry $i_{uv}$, and similarly, paths from $j_v$ to $j_u$ contribute to $j_{uv}$.

The following lemmata describe the matrix entries in the composition $g_0f_\zg$ column by column. There are four cases to consider depending on the direction of the two steps in the crossing sequence immediately before and after the crossing pair at position $t$. The first lemma deals with the case where both steps are forward. Part (a) describes the entries above the diagonal and part (b) the entries below the diagonal. 
\begin{lemma}  \label{lem 7.1}
{\rm (The case $\to (i_t,j_t) \to$)}.
 Let $
\begin{bmatrix}
 0&\ldots& 0& f_{ut} &f_{(u+1)t}& \ldots& f_{tt}& 0&\ldots& 0
\end{bmatrix}^T$ be the column $t$ of $f_\zg$ with $u<t$.
 \begin{enumerate}
\item [(a)]
 For all $s=1,2,\ldots, u-1$
 \begin{equation}\label{eqlem 7.1a} 
 (g_0\,f_\zg)_{st}=i_{su}f_{ut}+i_{s(u+1)}f_{(u+1)t} 
 \end{equation}
 Moreover, if $f_{uu}\colon i_u\to j_u$ is a single arrow and 
 the steps before and after $(i_u,j_u)$ are of the form
 $\xymatrix{(i_{u-1},j_{u-1}) &(i_u,j_u)\ar[l]_(0.4){\zS_{u}}\ar[r]^(0.4){\zS_{u+1}}& (i_{u+1},j_{u+1})
 }$
 then $i_{ s(u+1)} f_{(u+1)t }=0$, so $(g_0f_\zg)_{st}= i_{su}f_{ut}$.

\item [(b)]
For all $s=t+1,t+2,\ldots, n$
 \begin{equation}\label{eqlem 7.1b} 
 (g_0\,f_\zg)_{st}=i_{st}f_{tt}+i_{st'}f_{t't} 
 \end{equation}
 where $t'<t$ is the largest integer such that $f_{t't}$ is an arrow
 and 
 
 $\left\{\begin{array}{ll}
  i_{st}f_{tt}=0 & \textup{if $f_{tt}$ is not an arrow;}   \\
  i_{st'}f_{t't}=0 &\textup{if no such $t'$ exists.}   \\
 \end{array}
\right.$
\end{enumerate}
 \end{lemma}
\begin{proof}
 Since the column $t$ of $f_\zg$ is nonzero above the diagonal and zero below the diagonal, we know that the crossing sequence at position $t$ is of the form 
 \[\xymatrix{(i_{t-1},j_{t-1}) &(i_t,j_t)\ar@{<-}[l]_(0.4){\zS_{t}}\ar[r]^(0.4){\zS_{t+1}}& (i_{t+1},j_{t+1})
 }\]
 with both steps going forward.  Therefore the staircase shape of $f_\zg$ implies that there exists a unique $l$ such that $t_l+1<t\le t_{(l+1)}$ and $u=t_l$.
 
 (a) Let $s<u$. 
 Since $f_{rt}=0$ for $r<u$ and $r>t$, we have
\begin{equation}
\label{eq gfst} 
(g_0f_\zg)_{st}=\sum_{r=u}^t i_{sr}f_{rt}.
\end{equation}
 So we must show that $i_{sr}f_{rt}=0$, for $r=u+2,u+3,\ldots,t$. For such an $r$, we would have 
$t_l=u<u+1<r\le t\le t_{(l+1)}$, and in particular $t_{(l+1)}-t_l\ge 2.$ In this situation, Corollary~\ref{cor difference} implies that the sequence of steps $(\zS_{t_l} ,\zS_{t_l+1},\ldots,\zS_{t_{(l+1)}-1})$ is equal to $(R_1, T_0,\ldots,T_0,R_0)$ with $t_{(l+1)}-t_l-2$ steps $T_0$ in the middle. 

If $t_{(l+1)}-t_l=2$, we have $r>u+1=t_l+1$, hence $r\ge t_l+2=t_{(l+1)}$. Thus $r=t=t_{(l+1)}$, and we want to show $i_{st}f_{tt}=0$, where $s<u$. The quiver is shown in the left picture of Figure~\ref{figlem 7.1}. The path $i_{st}$ runs from $i_s$ to $i_t$ and must factor through $j_t$. Therefore $i_{st}f_{tt}=0$ since it contains a cyclic subpath at $j_t$.

\begin{figure}
\begin{center}
\[\xymatrix@!@R6pt@C3pt{ 
&&j_u\ar[d]&i_u\ar[l]\ar[d]&&j_t\ar[d]&i_t\ar[l]\\
i_s\ar@{.}[r]\ar@/^10pt/@{.}[rru]&\cdot\ar[ru] &\cdot\ar[l]\ar[ru]\ar@{.}[d]&\cdot\ar[l] \ar@{.}[d]
&\cdot\ar[ru]\ar@{.}[dl]&\cdot\ar[ru]\ar@{.}[dl]\ar[l]
\\
&&\cdot\ar[d]&\cdot\ar[d]\ar[l]&\cdot\ar[l]\\
&&i_{u+1}\ar[ru]&j_{u+1}\ar[l]\ar[ru]
}
\qquad
\xymatrix@!@R6pt@C3pt{
j_{r-1}\ar[r]\ar[d] &\cdot\ar[d]\ar[r]&\cdot\ar@{.}[r] &\cdot\ar[r]&i_r\ar[d] \\
i_{r-1}\ar[r]&\cdot\ar[lu]\ar[d]\ar@{<-}[rd]\ar[r]&\cdot\ar@{.}[r]&\cdot\ar[r]&x\ar[r]\ar[d]&j_r\ar[lu]\ar[d]\\
&\cdot\ar[lu]\ar[r]&\cdot\ar@{.}[dr]&&\cdot\ar@{.}[d]\ar[r]&\cdot\ar[lu]\\
&&&i_s&j_t
}
\]
\caption{Proof of Lemma \ref{lem 7.1}(a). On the left, $s<u=t-2, t=t_{(l+1)}$. On the right, $s<u<t_{(l+1)}$ and the step $(i_{r-1},j_{r-1})\to(i_r,j_r)$ is trapezoidal of degree 0.}
\label{figlem 7.1}
\end{center}
\end{figure}
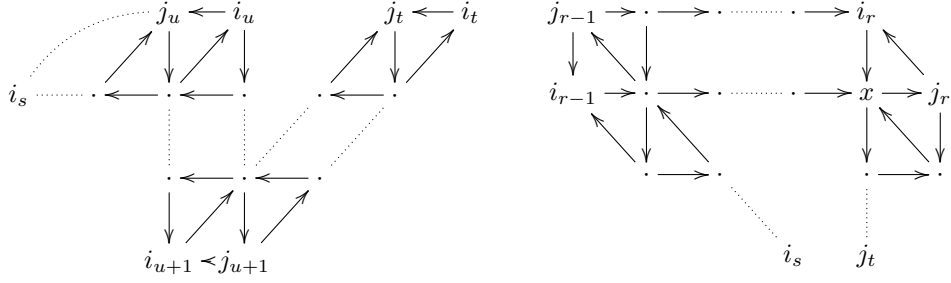

If $t_{(l+1)}-t_l>2$, then our sequence of steps contains  $t_{l+1}-t_l-2$ steps $T_0$ in the middle. The case where $r=t=t_{(l+1)}$ uses the same argument as above. In all other cases, the crossing pair $(i_r,j_r)$ is  the end of a trapezoidal step of degree 0 and at the start of some step of degree 0.  
The quiver is shown in the right picture of Figure~\ref{figlem 7.1}. 

In this case,  $i_{sr}$ is a path from $i_s$ to $i_r$ with $s<r-1$ and it must therefore factor through the vertex labeled $x$ that forms a 3-cycle with $i_{r} $ and $j_{r}$. On the other hand, $f_{rt}$ is a path from $i_r$ to $j_t$ with $t\ge r$ and this path also factors through $x$. Consequently $i_{sr}f_{rt}=0$, because it contains a cyclic subpath at $x$. This proves equation (\ref{eq gfst}).

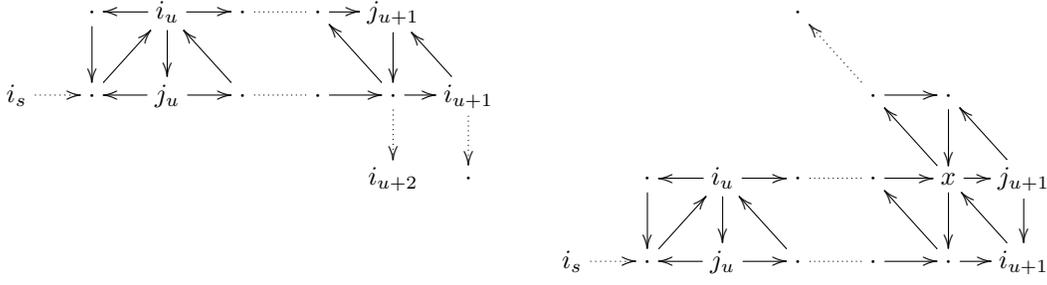
\begin{figure}
\begin{center}
\[\xymatrix@!@R6pt@C3pt{ 
 &\cdot \ar[d]&i_u\ar[l]\ar[d]\ar[r] &\cdot\ar@{.}[r] &\cdot\ar[r]& j_{u+1}\ar[d]\\
i_s \ar@{.>}[r] &\cdot \ar[ru]&j_u\ar[l]\ar[r]        &\cdot\ar[lu]\ar@{.}[r] &\cdot\ar[r]& \cdot\ar[lu]\ar[r] \ar@{.>}[d]&i_{u+1}\ar[lu]\ar@{.>}[d]\\
&                     &&&&i_{u+2} &\cdot
}
\qquad
\xymatrix@!@R6pt@C3pt{
&&&\cdot\\
&&&&\cdot\ar@{.>}[lu]\ar[r]&\cdot\ar[d]\\
&\cdot \ar[d]&i_u\ar[l]\ar[d]\ar[r] &\cdot\ar@{.}[r] &\cdot\ar[r]&x\ar[r]\ar[d]\ar[lu]& j_{u+1}\ar[d]\ar[lu]\\
i_s \ar@{.>}[r] &\cdot \ar[ru]&j_u\ar[l]\ar[r]        &\cdot\ar[lu]\ar@{.}[r] &\cdot\ar[r]& \cdot\ar[lu]\ar[r] &i_{u+1}\ar[lu]\\
}
\]
\caption{Proof of the second statement of Lemma \ref{lem 7.1} (a). On the left, the step $(i_u,j_u)\to(i_{u+1}, j_{u+1})$ is trapezoidal, and on the right it is rectangular.}
\label{figlem 7.1bis}
\end{center}
\end{figure}
To prove the moreover statement of the lemma, suppose now that $f_{uu}:i_u\to j_u$ is a single arrow and that the step $\zS_{u}$ is backward and the step $\zS_{u+1}$.  We want to show that $f_{s(u+1)}f_{(u+1)t}=0$.  The step $\zS_{u+1}$ is either trapezoidal or rectangular and both cases are illustrated in Figure~\ref{figlem 7.1bis}. If it is trapezoidal then $i_{s(u+1)} $ factors through $j_{u+1}$, as shown in the left picture of the figure. Thus if $t=u+1$ then $i_{s(u+1)}f_{(u+1) t}=0$, because it contains a cyclic subpath at $j_{u+1}$. If $t>u+1$ then $f_{ut}=0$, because every path from $i_u$ to any vertex past $j_{u+1}$ will contain two arrows of the same 3-cycle and thus cannot be valid.  This gives a contradiction to the assumption of the lemma that $f_{ut}\not=0$.  

Suppose now that $\zS_{n+1}$ is rectangular, see the right picture in Figure~\ref{figlem 7.1bis}. Then the path $i_{s(u+1)}$ factors through the vertex $x$ in the figure. The path $f_{(u+1)t}$ also factors through $x$, which implies that $i_{s(u+1)}f_{(u+1)t}=0$, since it contains a cyclic subpath at $x$. This completes the proof of part (a).

(b) Let $s>t$. 
Again we have equation (\ref{eq gfst}) and we may thus assume that $u\le r \le t$. 
Then $i_{sr}$ is a path from $i_s $ to $i_r$ that runs against the direction of the steps in the crossing sequence, and $f_{rt}$ is a path from $i_r $ to $j_t$ that follows the direction of the steps, see Figure~\ref{figlem 7.1b}.

\begin{figure}
\begin{center}
\[\xymatrix@!@R10pt@C10pt{ 
i_t\ar[d]\ar[r]&\cdot\ar[d]\ar[r] &\cdot\ar[d]
& i_s\ar@{.}[l]\\
j_t\ar[r]&y\ar[lu]\ar[d]\ar[r]&\cdot \ar[lu]
\\
&x\ar[lu]\ar[r]&\cdot \ar[lu]\\
&&i_r\ar@{.}[ul]
}
\qquad
\xymatrix@!@R10pt@C10pt{ 
j_t\ar[d]\ar[r]&\cdot\ar[d]\ar[r] &\cdot\ar[d]
& i_s\ar@{.}[l]\\
i_t\ar[r]&y\ar[lu]\ar[d]\ar[r]&\cdot \ar[lu]
\\
&\ar[lu]\ar[r]&\cdot \ar[lu]\\
&&i_r\ar@{.}[ul]
}
\]
\caption{Proof of Lemma \ref{lem 7.1} (b),  $r\le t<s$. On the left, $f_{tt}\colon i_t\to j_t$ is an arrow and on the right it is a path $i_t\to x\to j_t$.}
\label{figlem 7.1b}
\end{center}
\end{figure}
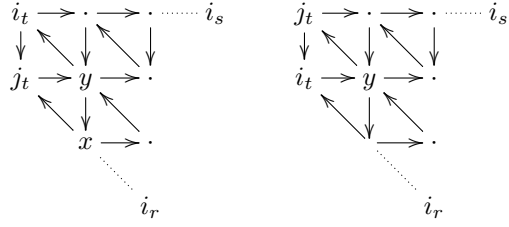

First suppose that $f_{tt}\colon i_t\to j_t$   is an arrow. This situation is shown in the left picture of the figure.

If the vertex $x\not=i_{r}$ then $i_{sr}$ factors through the vertex $x$. On the other hand, the path $f_{rt}$ from $i_r$ to $j_t$ also factors through $x$. Therefore $i_{sr}f_{rt}=0$, since it contains a cyclic subpath at $x$.

If $x= i_{r}$ then, because of the position of $x$ in the quiver, $r=t'$ is the largest integer such that $f_{t't}$ is an arrow. 
In this situation, $i_{st'}f_{t't}$ is nonzero.
Moreover, if $y=j_{t-1}$ then $t'=t-1$; otherwise $t'$ can be strictly smaller than $t-1$.

If $r=t$ then we are considering $i_{st}f_{tt}$ which is nonzero since $f_{tt}\colon i_t\to j_t$ is a single arrow.

Now suppose that $f_{tt}$ is not an arrow. Then there is an arrow $j_t\to i_t$.  This situation is shown in the right picture of Figure~\ref{figlem 7.1b}.  If the vertex $y$ is equal to some $i_{t'}$ in the crossing sequence, the argument is similar to the one above. Otherwise, the path $i_{sr}$ factors through the vertex $y$ that forms a 3-cycle with $i_t$ and $j_t$. On the other hand, $f_{rt}$ is a path from $i_r$ to $j_t$ and it also must factor through $y$. Thus if the vertex $y$ is different from $i_r$ then $i_{sr}f_{rt}=0$, since it contains a cyclic subpath at $y$. Suppose now that $y=i_r$. Then $r=t'$, where $t'$ is the largest integer such that $f_{t't}$ is an arrow. Thus, we have an arrow $i_{t'}\to j_t$ and the composition $i_{st'}f_{t't}$ is nonzero.
 \end{proof}

\begin{lemma}
 \label{lem 7.2}
{\rm (The case $\ot (i_t,j_t) \ot$)}.
 Let $
\begin{bmatrix}
 0&\ldots& 0& f_{tt}& \ldots& f_{vt}& 0&\ldots& 0
\end{bmatrix}^T$ be the column $t$ of $f_\zg$ with $t<v$.
 \begin{enumerate}
\item [(a)]
For all $s=1,2, \ldots, t-1$
 \begin{equation}\label{eqlem 7.2a} 
 (g_0\,f_\zg)_{st}=i_{st}f_{tt}+i_{st'}f_{t't} 
 \end{equation}
 where $t'>t$ is the smallest integer such that $f_{t't}$ is an arrow
 and 
 
 $\left\{\begin{array}{ll}
  i_{st}f_{tt}=0 & \textup{if $f_{tt}$ is not an arrow;}   \\
  i_{st'}f_{t't} &\textup{if no such $t'$ exists.}   \\
 \end{array}
\right.$\\

\item [(b)]
 For all $s=v+1,v+2,\ldots, n$
 \begin{equation}\label{eqlem 7.2b} 
 (g_0\,f_\zg)_{st}=i_{sv}f_{vt}+i_{s(v-1)}f_{(v-1)t} 
 \end{equation}
 Moreover, if $f_{vv}\colon i_v\to j_v$ is a single arrow and 
 the steps before and after $(i_v,j_v)$ are of the form
 $\xymatrix{(i_{v-1},j_{v-1}) &(i_v,j_v)\ar@{->}[l]_(0.4){\zS_{v}}\ar@{->}[r]^(0.4){\zS_{v+1}}& (i_{v+1},j_{v+1})
 }$
 then $i_{ s(v-1)} f_{(v-1)t }=0$, so $(g_0f_\zg)_{st}= i_{sv}f_{vt}$.
\end{enumerate} 
\end{lemma}

\begin{proof}
 This is the dual statement to Lemma \ref{lem 7.1}.
\end{proof}

\setcounter{MaxMatrixCols}{20}

\begin{lemma}
 \label{lem 7.3}
{\rm (The case $\to (i_t,j_t) \ot$)}.
 Let $
\begin{bmatrix}
 0&\ldots& 0&f_{ut}&\ldots& f_{tt} & \ldots& f_{vt} & 0&\ldots& 0
\end{bmatrix}^T$ be the column $t$ of $f_\zg$ with $u<t<v$. Then the entry $(g_0f_\zg)_{st}$ is computed according to formula (\ref{eqlem 7.1a}), if $s<u$, and according to formula  
(\ref{eqlem 7.2a}), if $s>v$,  unless we are  in one of the following two exceptional cases.

If $u=t-1$,  $f_{tt}$ is not an arrow, the steps before and after $(i_{t-1}, j_{t-1})$ are of the form 
\[\xymatrix{(i_{t-2},j_{t-2}) &(i_{t-1},j_{t-1})\ar@{<-}[l]_-(0.4){\zS_{t-1}}\ar@{->}[r]^-(0.4){\zS_{t}}& (i_{t},j_{t})},\] and there exists $t'>t$ such that $f_{t't}\colon i_{t'}\to j_{t}$ is an arrow, then $(g_0\,f_\zg)_{st} $ has an additional term $i_{st'}f_{t't}$, for all $s<u$.  

If $v=t+1$, $f_{tt}$ is not an arrow,  the steps before and after $(i_{t+1}, j_{t+1})$ are of the form 
\[\xymatrix{(i_{t},j_{t}) &(i_{t+1},j_{t+1})\ar@{->}[l]_-(0.4){\zS_{t+1}}\ar@{<-}[r]^-(0.4){\zS_{t+2}}& (i_{t+2},j_{t+2})},\] and there exists $t'<t$ such that $f_{t't}\colon i_{t'}\to j_{t}$ is an arrow, then $(g_0\,f_\zg)_{st} $ has an additional term $i_{st'}f_{t't}$, for all $s>v$.

\end{lemma}

\begin{proof}
 The proof of the general case is analogous to to the relevant parts of the two previous lemmata.
The exceptional case is illustrated in Figure~\ref{figlem73}. Clearly 
 $(g_0\,f_\zg)_{st}=i_{s(t-1)}f_{(t-1)t}+i_{st}f_{tt}+i_{st'}f_{t't}$. Since $t-1=u$ and $t=u+1$, the first two terms are precisely those given by equation (\ref{eqlem 7.1a}), while the third term is the additional term given in the statement.
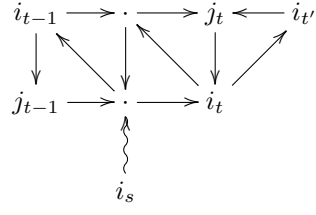
\begin{figure}
\begin{center}
\[\xymatrix@!@R10pt@C10pt{ i_{t-1}\ar[r]\ar[d]&\cdot\ar[r]\ar[d]&j_t\ar[d]&i_{t'}\ar[l]\\
j_{t-1}\ar[r]&\cdot \ar[lu]\ar[r]&i_t\ar[lu]\ar[ru]\\
&i_s\ar@{~>}[u]
}\]
\caption{The exceptional case in Lemma \ref{lem 7.3}}
\label{figlem73}
\end{center}
\end{figure}
\end{proof}
\begin{lemma}
 \label{lem 7.4}
{\rm (The case $\ot (i_t,j_t) \to$)}.
 Let $
\begin{bmatrix}
 0&\ldots& 0& f_{tt} & 0&\ldots& 0
\end{bmatrix}^T$ be the column $t$ of $f_\zg$. Then for all $s\not=t$
\begin{equation}
\label{eqlem 7.4} (g_0f_\zg)_{st} = \left\{
\begin{array}
{ll} i_{st}f_{tt} &\textup{if $f_{tt}\colon i_t\to j_t$ is an arrow};\\ 
0 &\textup{if $f_{tt}\colon i_t\to\cdot\to j_t$ is not an arrow}.
\end{array}
\right.
\end{equation}
  
\end{lemma}

\begin{proof}
 We have $(g_0f_\zg)_{st}=\sum_u i_{su}f_{ut}=i_{st}f_{tt}$, where the last identity follows from the assumption that $f_{tt}$ is the only nonzero entry in the $t$-th column of $f_\zg$. It only remains to show that $i_{st}f_{tt}=0$ when $f_{tt}$ is not an arrow. In this situation, we have the following subquiver
 \[\xymatrix@!@R10pt@C10pt
 {\cdot&j_t\ar[l]\ar[d]\ar[r]&\cdot\\
 k_1 \ar[ru]&i_t\ar[r]\ar[l]&k_2\ar[lu]
 }\]
 Thus there is an arrow $j_t\to i_t$ that lies in two 3-cycles whose third vertex is labeled $k_1$ or $k_2$ in the figure, and $f_{tt}$ is equal to either one of the two paths $i_t\to k_1\to j_t$ and $i_t\to k_2\to j_t$. On the other hand, $i_{st}$, being a path ending at  $i_t$, must end in one of the two subpaths $k_1\to j_t\to i_t$ and $k_2\to j_t\to i_t$.  In both cases, the composition $i_{st}f_{tt}$ is zero, because it contains a cyclic subpath at $k_1$ or $k_2.$
\end{proof}

Lemmata \ref{lem 7.1}-\ref{lem 7.4} also imply the following dual statements for the composition $f_\zg g_1$, obtained by reformulation using the rows of $f_\zg$ instead of the columns. 
The proof follows directly from the property of the transpose $(f_\zg g_1)^T=g_1^T f_\zg^T$.

\begin{lemma}\label{lem 7.5}
 The entries of the matrix of $f_\zg g_1$ satisfy the following equations. 
 \begin{enumerate}
\item 
{\rm  (The case $\to (i_s,j_s) \to$)}. Let $
\begin{bmatrix}
 0&\ldots&0&f_{ss}&\ldots&f_{sr}&0&\ldots&0
\end{bmatrix}
$ be the row $s$ of $f_\zg$, with $r>s$.
 \begin{enumerate}
\item 
For all $t=r+1,\ldots,n$
\begin{equation}
 \label{eqlem 7.5(1)a}
 (f_\zg g_1)_{st}= f_{sr}j_{rt}+f_{s(r-1)}j_{(r-1)t} 
\end{equation}
Moreover, if $f_{rr}\colon i_r\to j_r$ is a single arrow and the steps before and after $(i_r,j_r)$ are of the form
 $\xymatrix{(i_{r-1},j_{r-1}) &(i_r,j_r)\ar@{<-}[l]_(0.4){\zS_{r}}\ar@{<-}[r]^(0.4){\zS_{r+1}}& (i_{r+1},j_{r+1})
 }$
 then 
 $f_{s(r-1)}j_{(r-1)t}=0$, so $(f_\zg g_1)_{st}=f_{sr}j_{rt}.$

\item 
For all $t=1,2,\ldots, s-1$
 \begin{equation}\label{eqlem 7.5(1)b} 
 (f_\zg g_1)_{st}=f_{ss}j_{st}+f_{ss'}j_{s't} 
 \end{equation}
 where $s'>s$ is the smallest integer such that $f_{ss'}$ is an arrow,
 
 and 
 $\left\{\begin{array}{ll}
  f_{ss}j_{st}=0 & \textup{if $f_{ss}$ is not an arrow;}   \\
  f_{ss'}j_{s't}=0 &\textup{if no such $s'$ exists.}   \\
 \end{array}
\right.$

\end{enumerate}

\item
{\rm (The case $\ot (i_s,j_s) \ot$)}.
Let $
\begin{bmatrix}
 0&\ldots&0&f_{sr}&\ldots&f_{ss}&0&\ldots&0
\end{bmatrix}
$ be the row $s$ of $f_\zg$, with $r<s$.
 \begin{enumerate}
\item 
For all $t=s+1,\ldots,n$
 \begin{equation}\label{eqlem 7.5(2)a} 
 (f_\zg g_1)_{st}=f_{ss}j_{st}+f_{ss'}j_{s't} 
 \end{equation}
 where $s'<s$ is the largest integer such that $f_{s's}$ is an arrow,
 
 and 
 $\left\{\begin{array}{ll}
  f_{ss}j_{st}=0 & \textup{if $f_{ss}$ is not an arrow;}   \\
  f_{ss'}j_{s't}=0 &\textup{if no such $s'$ exists.}   \\
 \end{array}
\right.$

\item 
For all $t=1,2,\ldots, r-1$
\begin{equation}
 \label{eqlem 7.5(2)b}
 (f_\zg g_1)_{st}= f_{sr}j_{rt}+f_{s(r+1)}j_{(r+1)t}.
\end{equation}
Moreover, if $f_{rr}\colon i_r\to j_r$ is a single arrow and the steps before and after $(i_r,j_r)$ are of the form
 $\xymatrix{(i_{r-1},j_{r-1}) &(i_r,j_r)\ar@{<-}[l]_(0.4){\zS_{r}}\ar@{<-}[r]^(0.4){\zS_{r+1}}& (i_{r+1},j_{r+1})
 }$
 then 
 $f_{s(r+1)}j_{(r+1)t}=0$, so $(f_\zg g_1)_{st}=f_{sr}j_{rt}.$
\end{enumerate}

\item {\rm (The case $\ot (i_s,j_s) \to$).}  Let $
\begin{bmatrix}
 0&\ldots&0&f_{sr}&\ldots&f_{ss}&\ldots&f_{sr'}&0&\ldots&0
\end{bmatrix}
$ be the row $s$ of $f_\zg$, with $r<s<r'$.  Then the entry $(f_\zg g_1)_{st}$ is computed according to formula (\ref{eqlem 7.5(1)a}) if $t>r'$ and according to formula  
(\ref{eqlem 7.5(2)b}) if $t<r$, unless we are  in one of the following two exceptional cases.

If $r'=s+1$, $f_{ss}$ is not an arrow,  the steps before and after $(i_{s+1}, j_{s+1})$ are of the form 
{$\xymatrix{(i_{s},j_{s}) &(i_{s+1},j_{s+1})\ar@{<-}[l]_-(0.4){\zS_{s+1}}\ar@{->}[r]^-(0.4){\zS_{s+2}}& (i_{s+2},j_{s+2})}$}, and there exists $s'<s$ such that $f_{ss'}\colon i_{s}\to j_{s'}$ is an arrow, then $(f_\zg g_1)_{st} $ has an additional term $f_{ss'}j_{s't}$, for all $t>r'$. 

If $r=s-1$, $f_{ss}$ is not an arrow,  the steps before and after $(i_{s-1},\, j_{s-1})$ are of the form 
{$\xymatrix{(i_{s-2},j_{s-2}) &(i_{s-1},j_{s-1})\ar@{->}[l]_-(0.4){\zS_{s-1}}\ar@{<-}[r]^-(0.4){\zS_{s}}& (i_{s},j_{s})}$}, and there exists $r''>s$ such that $f_{sr''}\colon i_{s}\to j_{r''}$ is an arrow, then $(f_\zg g_1)_{st} $ has an additional term $f_{sr''}j_{r''t}$, for all $t<r$.

\item {\rm (The case $\to (i_s,j_s) \ot$).}
Let $
\begin{bmatrix}
 0&&\ldots&f_{ss}&0&\ldots&0
\end{bmatrix}
$ be the row $s$ of $f_\zg$.

 Then for all $s\not=t$
\begin{equation}
\label{eqlem 7.5(4)}
(f_\zg g_1)_{st} = \left\{
\begin{array}
{ll} f_{ss} j_{st} &\textup{if $f_{ss}\colon i_s\to j_t$ is an arrow};\\ 
0 &\textup{if $f_{ss}\colon i_s\to\cdot\to j_s$ is not an arrow}.
\end{array}
\right.
\end{equation}

\end{enumerate}

\end{lemma}

\subsubsection{Blocks of $g_0$}\label{sec:blocks}

We will think of the result in Section~\ref{sec:comp} as relations in the matrix $g_0$. To make this idea precise, we introduce the following block structure in the matrix $g_0$. 
Recall that both maps $f_\zg$ and $g_0$ are given by square matrices of the same size, and that we have defined the blocks $A_l$ and $B_l$ in $f_\zg$ as the subcolumns of the column $t_l$.
Each of these blocks $A_l$ and $ B_l$ induces a family of column blocks $A_{l,u}$ and $B_{l,u}$ in $g_0$ as follows.
%

\begin{definition}[Column Blocks]
 \label{def blocks of g0}
 Let $g$ be an endomorphism of $M_\zg$ and let $g_0$ be the induced endomorphism on the projective cover $P_0(\zg)=\oplus_{s=1}^n P(i_s)$. We treat the change of direction in parts (c) and (d).  For parts (a) and (b) below, we assume that the crossing sequence is forward without change of direction, and we use the notation of Definition~\ref{def blocks of f}, thus  $ 1=t_0<t_1<\
 \cdots <t_p=n$ is the staircase sequence of $f_\zg$ and  $a_l$ is the least positive integer such that the entry in $f_\zg$  at position $(a_l,t_l)$ is an arrow.  We define the following \emph{column blocks} in the matrix of $g_0$,
 \begin{itemize}
\item [(a)] \begin{itemize}
\item [(i)]
If  $a_l\ne t_{l-1}$ or $ f_{t_{(l-1)}t_{(l-1)}}$ is not an arrow, define 
\[A_{l,u}= 
[i_{a_lu} \ i_{(a_l+1)u}\ \ldots \ i_{t_lu}]^T,
\textup{ where $u\le a_{l}-1$}.\] 

\item[(ii)] If $a_l= t_{l-1}$ and $ f_{t_{(l-1)}t_{(l-1)}}$ is an arrow, define 
\[ A_{l,u} = [ i_{(a_l+1)u} \  i_{(a_l+2)u}\ \ldots \ i_{t_lu}]^T , \textup{
where $u\le a_{l}$.}\]
\end{itemize}
In the right picture of Figure~\ref{fig staircase}, the $A$-blocks are the colored framed regions with  bullet points below the diagonal.  In the example of Figure~\ref{fig staircase}, the block $A_{5,u}$ is of type (a)(ii). 
For a crossing sequence that is non-forward the definition is symmetric.
\item[(b)]  For all $1\le l\le p$ and $t_l\le u$ let
\[B_{l,u}= 
\begin{bmatrix}
 i_{t_{l-1} u} & i_{(t_{l-1}+1) u}& \cdots & i_{(t_l-1) u}
\end{bmatrix}^T.\]
In the right picture of Figure~\ref{fig staircase}, the $B$-blocks are the colored framed regions above the diagonal. The shading in some of these blocks will be explained in Example~\ref{ex shading}. 
For a crossing sequence that is non-forward the definition is symmetric.

\item[(c)] 
Assume the crossing sequence has a change of direction at $t_l$ of the form $\to (i_{t_l},j_{t_l})\ot$ and $f_{t_lt_l}$ is an arrow.  In addition to the $A$ and $B$ blocks defined above we also have a gluing of two of these as follows.

\begin{itemize}
\item[(i)] If $a_l\ne t_{l-1}$ or $f_{t_{(l-1)} t_{(l-1)}}$ is not an arrow, define 
 
 \[AB_{l,u}=[i_{a_lu} \ i_{(a_l+1)u}\ \ldots \ i_{t_lu} \  i_{(t_l+1)u} \ \dots \ i_{(t_{l+1})u}]^T, \textup{where $u \leq a_l-1$.}\]

\item[(ii)] If $a_l= t_{l-1}$ and $f_{t_{(l-1)} t_{(l-1)}}$ is  an arrow, define 
 
 \[AB_{l,u}=[i_{(a_l+1)u}\ \ldots \ i_{t_lu} \  i_{(t_l+1)u} \ \dots \ i_{(t_{l+1})u}]^T, \textup{where $u \leq a_l$.}\]
\end{itemize}

In a similar way we would also have blocks of type $BA_{l,u}$ with $\bar a_l$ defined as the largest integer such that $f_{\bar a_lt_l}$ is an arrow and $u\geq \bar{a_l}+1$ in type (i) or $u\geq \bar{a_l}$ in type (ii). 


%
%
%

\item[(d)] 
Assume the crossing sequence has a change of direction at $t_l$ of the form $\ot (i_{t_l},j_{t_l})\to$. Then the $A_{l,u}, B_{l,u}$ blocks are empty.  The $A_{l+1,u}, A_{l-1,u}$ blocks are as in part (a)(ii) and the $B_{l+1,u}, B_{l-1,u}$ are as in part (b). 

\end{itemize}

\end{definition}

The following lemma examines the entries of the composition $f_\zg g_1$ in positions corresponding to the column blocks of $g_0$. 

\begin{lemma}\label{column constant}
Let $\begin{bmatrix}i_{s'u} & i_{(s'+1)u} & \dots & i_{s''u} \end{bmatrix}^T$ be a column block in $g_0$, and let $s'\leq s \leq s''$.  Assume that if the column block is $A_{l,u}$ of Definition~\ref{def blocks of g0} and there exists $t'>s''$ such that $f_{s''t}$ is an arrow, then $s'\leq s<s''$. Then the coefficients of $(f_\zg g_1)_{su}$ are equal for all $s$. 
\end{lemma}

\begin{proof} 
First, suppose that the column block is of type $A_{l,u}$ and it lies below the main diagonal as in Figure~\ref{fig staircase}.   Then $f_{s's''}, f_{(s'+1)s''}, \dots, f_{s''s''}$ are arrows.  Moreover, there are no arrows in $f_{\zg}$ to the left of these entries by Definition~\ref{def blocks of g0}(a) of the column blocks.  By Lemma~\ref{lem 7.5}(1)(b)
$(f_\zg g_1)_{su}= f_{ss''}j_{s''u}$ for all $s'\leq s < s''$.   Then the coefficient of $(f_\zg g_1)_{su}$ is independent of $s$, so the lemma follows.  If $s=s''$ and no $t'$ as in the statement of the lemma exists then we still have the same equation for $(f_\zg g_1)_{s''u}$, and the lemma holds. 

If the column block is of type $AB_{l,u}$ or $BA_{l,u}$ then the argument is similar to the one above and also uses the moreover statement of part (1) and part (2) of Lemma~\ref{lem 7.5}. 

If the column block is of type $B_{l,u}$, and without loss of generality we may assume that $s''<u$.  Hence, the column block is to the right of the main diagonal in $g_0$ as in Figure~\ref{fig staircase}, and using notation of Definition~\ref{def blocks of g0} we have that $s'=t_{l-1}$ and $s''=t_l -1$ for some $l$.  Also, note that $t'\geq t_{l}$.  Then the row $s$ of $f_\zg $ is 
\[ \begin{bmatrix} 0 & \dots & 0 & f_{ss} & \dots & f_{s (t_l-1)}&f_{s (t_{l})} & 0 & \dots & 0\end{bmatrix}\]
and Lemma~\ref{lem 7.5} implies that 
\[(f_\zg g_1)_{su}=f_{s (t_l-1)}j_{(t_l-1)u}+f_{s (t_l)}j_{(t_l)u}\]
for all  $s'\leq s < s''$ and $u> t_{l}$.  In particular, the coefficient of $(f_\zg g_1)_{su}$ is the sum of the coefficients of $j_{(t_l-1)u}$ and $j_{(t_l)u}$, which is independent of $s$.  
If $u=t_l$ then by Corollary~\ref{cor:534} we have that $(f_\zg g_1)_{st_l}=0$ because $f_{st_l}\not=0$.  This is again independent of $s$, so the lemma holds. 
\end{proof}

Next we will introduce row blocks in the matrix of $g_0$. We assume that the steps in the crossing sequence are forward. The non-forward case is symmetric. The changes of direction will be discussed below as well.

\begin{definition}\label{def one-step}
(a)  We define \emph{maximal one-step sequence} in the matrix of $f_\zg$ to be a maximal sequence of consecutive columns $m,m+1,\dots,m'$ where the diagonal entry $i_{ll}$ and the entry right above the diagonal $i_{(l-1)l}$ are nonzero, and all other entries above the diagonal are zero. 

(b) We define \emph{maximal one-step sequence of arrows} in the matrix of $f_\zg$ to be a maximal sequence of consecutive columns $m,m+1,\dots,m'$ where the diagonal entry and the entry right above the diagonal are arrows.
\end{definition}
\begin{example}
 In the example in Figure~\ref{fig staircase}, there is one maximal one-step sequence of length two $m=t_4,m'=t_4+1$, and all other maximal one-step sequences are of length one $m=m'=t_l+1$, with $l=1,2,5,6$.
  
 The only two maximal one-step sequences of arrows
are $m=m'=t_1,$ or $t_5$. 
\end{example}
\begin{remark}
 Let  $m,m+1,\dots, m'$ be a maximal one-step sequence. If there is no change of direction at $m-1$ and $m'$, then it follows from the definition of the staircase sequence \ref{def ti} of $f_\zg$ that 
 there exists an $l$ such that $m=t_l, m+1=t_{l+1},\dots,m'-1=t_{l+m'-m-1}$, and $m'$ is not one of the $t_i$ of the staircase sequence.
Moreover, we have $m-1=t_{l-1}$.
In this situation the columns $m-1$ and $m'+1$ has at least three nonzero entries on and above the diagonal. 

If there is a change of direction at $m'$ then $m'=t_{l+m'-m}$ is also in the staircase sequence. In this case, the column $m'+1$ is zero at every position above the diagonal. 

If there is a change of direction at $m-1$ then the column $m-1$ is zero except at the diagonal position and $m-1$ is not in the staircase sequence.
\end{remark}

\begin{definition}[Row Blocks] 
\label{def rowblocks}
Let $g$ be an endomorphism of $M_\zg$ and let $g_0$ be the induced endomorphism on the projective cover $P_0(\zg)=\oplus_{s=1}^n P(i_s)$. We treat the change of direction in parts (a)(ii), (b)(ii) and (b)(iii).  Otherwise, we assume that the crossing sequence is forward without change of direction.   We define the following \emph{row blocks} in the matrix of $g_0$.

 \begin{itemize}
\item [(a)] 
For each maximal one-step sequence $m,m+1,\dots,m'$, we define 

\begin{itemize}
\item [(i)] Unless we are in the situation of case (ii) below, let 
\[F_{v,m}= [i_{v,m-1}\ i_{vm} \ \dots\ i_{vm'}], \textup{ where $v< m-1$}.\]
\item[(ii)] 
If there is a change of direction at $m-1$ whose diagonal entry $f_{(m-1)(m-1)}$ is an arrow then let

\[F_{v,m}= [i_{vm} \ \dots\ i_{vm'}], \textup{ where $v\le m-1$}.\]
\end{itemize}
\item[(b)]
For each maximal one-step sequence of arrows $m,m+1,\dots,m'$ we have the following row  blocks below the diagonal.
\begin{itemize}
\item [(i)] 
 Unless we are in the situation of case (ii) or (iii) below, let
\[E_{v,m}= [i_{v(m-1)}\ i_{vm}\ \dots\ i_{vm'}], \textup{ where $v\ge m'+1$}.\]
\item [(ii)] 
If  there is a change of direction at $m'+1$ such that  $f_{m'(m'+1)}$ is an arrow, $f_{(m'+1)(m'+1)}$ is not an arrow, and $f_{(m'+3)(m'+1)}=0$ then let

\[E_{v,m}= [i_{v(m-1)}\ i_{vm}\ \dots\ i_{vm'}\  i_{v(m'+1)}], \textup{ where $v\ge m'+3$}.\]

\item[(iii)] If  there is a change of direction at $m-1$ whose diagonal entry is not an arrow  then the block $E_{vm}$ extends to the left  into the adjacent $F$-block as follows,
\[FE_{v,m}=[i_{v(m'')}\dots i_{v(m-1)} \ i_{vm}\ \dots i_{vm'} ], \textup{ where $v\ge m'+1$}\]
and $m''$ is the first column of the (backward) maximal one-step sequence ending at $m'-1$.

In a similar way we would also have blocks of type $EF_{v,m}$ above the diagonal.  

\item[(iv)] If in a case (i), (ii) or (iii) above we have $f_{(m-1)(m-1)}$ or $f_{m'(m'+1)}$ is an arrow then the row block will behave somewhat differently in what follows, and we refer to it as a row block of type (iv).
\end{itemize}
\end{itemize}
\end{definition}
\begin{example}
\label{ex shading}
In the example of Figure \ref{fig staircase} the row blocks $F_{v,m}$ are formed by those colored blocks $B_{l,u}$ that are non-shaded. For example row one of the matrix has the following five row blocks.
\[ [i_{1t_1} \ i_{1(t_1+1)}] \ ,\ 
[i_{1t_2} \ i_{1(t_2+1)}]  \ , \
[i_{1t_3} \ i_{1t_4} \ i_{1(t_4+1)}]  \ , \
[i_{1t_5} \ i_{1(t_5+1)}]  \ , \
[i_{1t_6} \ i_{1(t_6+1)}]  
\]

\end{example}

\subsubsection{The cases where $g_0$ is constant on the column blocks or alternating on the row blocks} 

Each entry of the matrix of $g_0$ is given by a product of a scalar coefficient and path. If $C$ is any submatrix, we say that $g_0$ is \emph{constant on $C$}  if each entry of $g_0$ in $C$ has the same scalar coefficient. 
If the submatrix $C$ has only one row, we say that $g_0$ is \emph{alternating on $C$} if the scalar coefficient is constant up to sign and the sign alternates along the row.

\begin{lemma}
 \label{lem 713}
Let  $g$ be a nilpotent endomorphism of $M_\zg$. Then  $g=0$ in $\scmp\,B$ in each of the following cases. \begin{itemize}
\item[(i)]    $g_0$ is constant on a block $B_{l,u}$ and zero elsewhere, 
\item[(ii)]  $g_0$ is constant on a block $A_{l,u}$ and zero elsewhere,
\item[(iii)] $g_0$ is constant on $A_{l,u}\cup \{i_{a_lu}\}$ and zero elsewhere,
where $A_{l,u}= [ i_{(a_l+1)u} \  i_{(a_l+2)u}\ \ldots \ i_{t_lu}]^T$ is a column block as in part (a)(ii) of Definition~\ref{def blocks of g0},

\item[(iv)]  $g_0$ is constant on a block $AB_{l,u}$ or a block $BA_{l,u}$ and zero elsewhere.

\end{itemize}
\end{lemma}
\begin{proof}
 We are going to construct a morphism $h\colon P_0(\zg)\to P_1(\zg)$ such that $g_0=f_\zg h$. Then the result will follow from Lemma~\ref{lem:cmp-morphism}.
 
 (i) 
 In order to construct $h$, we first must take a closer look at $g_0$. Let $(s,u)$ be a nonzero position in the matrix of $g_0$ that lies in the block $B_{l,u}$. Thus $t_{l-1}\le s \le t_l-1$ and $t_l\le u$. 
 Then $(g_0)_{su}$ is a scalar multiple of a path from $i_s$  to $i_u$. We shall show that this path factors through $j_{t_l}$.

 Indeed, if $t_l-t_{l-1} =1$ then $s=t_l-1$ and the block $B_{l,u}$ consists of a single position 
 $(g_0)_{(t_l-1)u}$,  $ u\ge t_l$, and the step $\zS_{t_l-1}$ in the crossing sequence of $\zg$ is trapezoidal of degree 1, by Corollary \ref{cor difference}. It is not hard to see that in this case the path $(g_0)_{(t_l-1)u}$ from $i_{t_l-1}$ to $i_u$ factors through $j_{t_l-1}$.
 
 On the other hand, if $t_l-t_{l-1}\ge 2 $ then $\zS_{t_l-1}$ is rectangular of degree zero, again by Corollary \ref{cor difference}.
 This case is illustrated in Figure~\ref{fig 71}. Thus  the path in $(g_0)_{su}$ from $i_s$ to $i_u$ with $s\le t_l-1$, $u\ge t_l$ factors through $j_{t_l-1}$.
 
\begin{figure}
\[\xymatrix@!R=8pt@!C=8pt{&&&&&j_{t_l+1} \\
&&&&\cdot\ar[r]&\cdot\ar@{.>}[u]\ar[ld]\\
&j_{t_l-1}\ar[r]\ar[ld] &\cdot\ar@{.>}[rr]\ar[ld]&&\cdot\ar[u]\ar[ld]\ar[r] &i_{t_l}\ar[u]\ar[ld]\\
i_{t_l-1}\ar[r]&\cdot\ar[u]\ar[ld]\ar@{.>}[rr]&&\cdot\ar[r]&j_{t_l}\ar[u]\\
\cdot\ar[u] \ar[r]&\cdot\ar[u]&&i_{t_l-3}\ar[ld]\\
&&\cdot\ar@{.>}[ld]\ar[r]&j_{t_l-3}\ar[u]\ar[ld]\\
\cdot\ar@{.>}[uu]\ar[r]&\cdot\ar@{.>}[uu]\ar[ld]&\cdot\ar@{.>}[ld]\ar[u]\\
j_{t_l-2}\ar[u]\ar[r]&i_{t_l-2}\ar[u]
}\]
\caption{Proof of Lemma \ref{lem 713}. The steps $\zS_{t_l-3}, \zS_{t_l-2}$ are both $T_0$ and the step $\zS_{t_l-1}$ is $R_0$.}
\label{fig 71}
\end{figure}
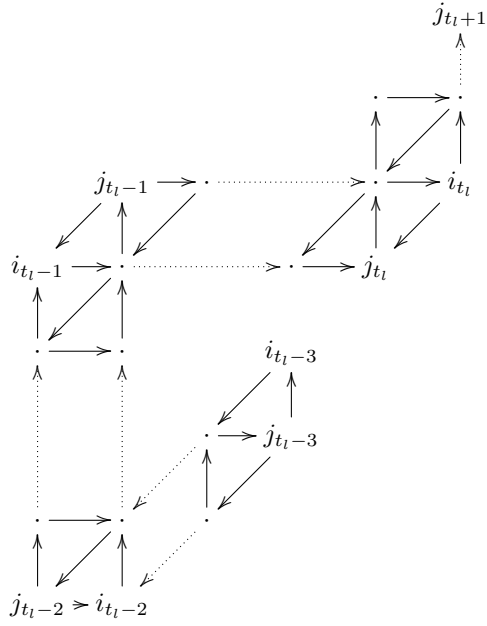 

Therefore, up to a scalar coefficient, $(g_0)_{su}$ is a path $w\colon i_s\leadsto j_{t_l-1}\leadsto i_u$. 
 Moreover, since $t_{l-1}\le s\le t_l-1$, the staircase shape of $f_\zg$ implies that $f_{s(t_l-1)}\ne 0$, and thus there exists a valid path $f_{s(t_l-1)}\colon i_s\leadsto j_{t_l-1}$.
 Let $h_{(t_l-1)u}$ denote the subpath of $w$ from $j_{t_l-1}$ to $i_u$, multiplied by the coefficient of $(g_0) _{su}$. Here we use the fact that this coefficient does not depend on $s$, since $g_0$ is constant on the block $B_{l,u}$. Then $h_{(t_l-1)u}$ is a morphism from $P(i_u)$ to $P(j_{t_l})$ and we have
 \[(g_0)_{su}=f_{s(t_l-1)}h_{(t_l-1)u}.\] 
 Note that the right hand side of this equation is the product of a single entry of $f_\zg$ with $h_{(t_l-1)u}$. We now shall consider the product of $f_\zg$ with $h_{(t_l-1)u}$.
 
 If the block $B_{l,u}$ is of size at least 2, that is $t_l-t_{l-1}\ge 2$ then 
 \[f_\zg h_{(t_l-1)u} = (f_{s'(t_l-1)}h_{(t_l-1)u})_{t_{(l-1)}\le s'\le t_l-1} =  g_0|_{B_{l,u}}.\]
  In this case, we define $h=h_{(t_l-1)u}$ at position ${(t_l-1)u}$ and zero elsewhere.

If the block $B_{l,u}$ is of size 1, that is $t_l-t_{l-1}=1$, then  
\[f_\zg h_{(t_l-1)u} = (f_{s'(t_l-1)}h_{(t_l-1)u})_{t_{l-2}\le s'\le t_l-1} \] is bigger that $B_{l,u}$. In this case, we defined $h$ to be $h_{(t_{l-1})u}$ at position ${(t_{l-1})u}$, and $-h_{(t_{l-1}-1)u}$ at position ${(t_{l-1}-1)u}$. Then if the previous block $B_{l-1,u}$ has size greater than 1 then 
$(f_\zg h)=g_0$. And if $B_{l-1,u}$ has size 1 as well, then we add $h_{(t_{l-2} -1)u}$ to $h$ in order to compensate again. Continuing this way will produce the desired map $h$.

(ii) and (iii) 
Recall that $a_l$ is the least integer such that in column $t_l$ of  the matrix of $f_\zg$ the positions $a_l,a_l+1,\ldots, t_l$ are given by arrows $i_{a_l}\to j_{t_l},i_{a_l+1}\to j_{t_l}, \ldots, i_{t_l}\to j_{t_l}$. Suppose first that the arrow between $i_{a_l}$ and $j_{a_l}$ is in the direction $j_{a_l}\to i_{a_l}$. The quiver in this case is illustrated in the left picture of Figure \ref{figlem 713}.
\begin{figure}
\begin{center}
\[\xymatrix{&j_{t_l-1}\ar[r]\ar[d] & i_{t_l-1}\ar[d] & j_{a_l+1}\ar@{.>}[l]\ar[d] \\
&i_{t_l}\ar[r] & j_{t_l}\ar[ldd]\ar[lu]\ar[ru]\ar[rd] & i_{a_l+1}\ar@{<-}[d]\ar[l]\\
&\cdot \ar[ru]& i_{a_l}\ar[u]&j_{a_l}\ar[l]
\\
i_u&\cdot\ar[u]\ar[ru]\ar@{~>}[l]
\\
} \qquad
\xymatrix{j_{t_l-1}\ar[r]\ar[d] & i_{t_l-1}\ar[d] & j_{a_l+1}\ar@{.>}[l]\ar[d] \\
i_{t_l}\ar[r] & j_{t_l}\ar[lu]\ar[ru]\ar[rd] & i_{a_l+1}\ar@{<-}[d]\ar[l]\\
& i_{a_l}\ar[u]\ar[rd]&x\ar[l]\ar@{~>}[rd]
\\ 
&&j_{a_l}\ar[u]
&
i_{u}
}
\]
\caption{Local configuration in the situation of block $A_{l,u}$. On the left, the arrow of the $a_l$-th crossing pair is  $j_{a_l}\to i_{a_l}$; on the right, the arrow is $i_{a_l}\to j_{a_l}$.}
\label{figlem 713}
\end{center}
\end{figure}
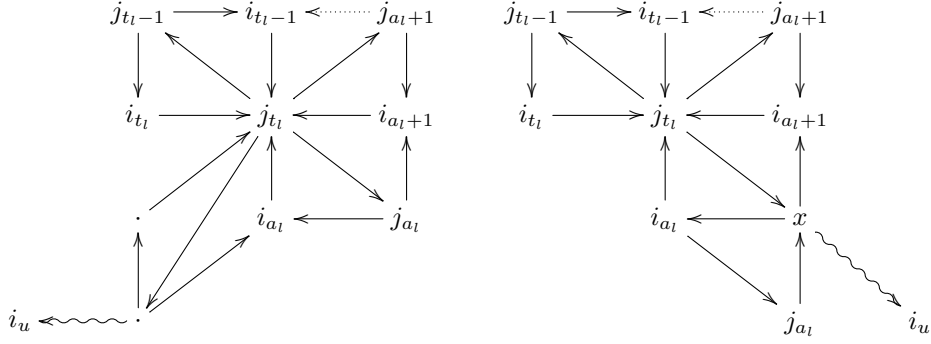
We call the configuration given by the full subquiver on the vertices $i_{a_l},j_{a_l},\ldots,i_{t_l},j_{t_l}$ a flower centered at $j_{t_l}$.  
In this situation the position $(a_l,a_l)$ in the matrix $f_\zg$ is not an arrow. Thus the block $A_{l,u}$ is of type (a)(i) in Definition~\ref{def blocks of g0}.  Hence, it consists of the positions $a_l, a_{l}+1,\ldots, {t_l} $ in column $u$ of the matrix $g_0$, where $u\le a_l-1$. The entries in this block are paths starting at one of the points $i_{a_l},\ldots,i_{t_l}$ and ending at $i_u$. From the figure it is clear that each of these paths must factor through $j_{t_l}$.   
Define $h_1$ to be the path $j_{t_l}\leadsto i_u$ at position $(t_l,u)$ and zero elsewhere. Then $f_\zg h_1=g_0|_{A_l,u},$ and we are done.

Now suppose that the arrow is $i_{a_l}\to j_{a_l}$. The quiver in this case is illustrated in the right picture of Figure \ref{figlem 713}.
In this situation, the flower centered at $j_{t_l}$ does not contain the vertex $j_{a_l}$, and the position $(a_l,a_l)$ in the matrix $f_\zg$ is an arrow. Thus the block $A_{l,u}$ consists of the positions $a_{l}+1, a_l+2,\ldots,t_l $ in column $u$ of the matrix $g_0$, where again $u\le a_l$. Again each of the entries in this block is a path that factor through $j_{t_l}$. In case (iii) this completes the proof. However, in case (ii), the map $h_1$ defined above will not work, because its composition with $f_\zg$ would contain the nonzero path $i_{a_l}\to j_{t_l}\leadsto i_u$ which is not part of the block $A_{l,u}$. In order to compensate, we note that this path factors through $j_{a_l}$ as well. So we can define $h_2=h_1-(j_{a_l}\leadsto i_u)$, so $h_2$ has two nonzero positions $j_{t_l}\leadsto i_u$ and $-(j_{a_l}\leadsto i_u)$, and then $f_\zg h_2$ is equal to $g$ on the block $A_{l,u}$ and is zero at position $(a_l,u)$. Thus if there is no other arrow from a vertex $i_s$ to the  vertex $j_{a_l}$ with $s<a_l$ then this $h_2$ has the desired property.

Now suppose there is  an arrow $i_s\to j_{a_l}$ with $s<a_l$. Here we need to consider two subcases, depending on whether or not the vertex $x$ in the right picture in Figure~\ref{figlem 713} is the vertex $j_{a_l-1}$. If it is, then  the path  $i_{a_l}\to j_{t_l}\leadsto i_u$ would be 
\[i_{a_l}\to j_{t_l}\to j_{a_l-1}\to i_{a_l-1}\to j_{a_l}\leadsto i_u,\] 
and, since the initial piece 
$i_{a_l}\to j_{t_l}\to j_{a_l-1}$ factors through $j_{a_l}$, this path is zero unless $u= a_{l}-1$. Hence our first morphism $h_1$ has the desired property if $u<a_{l}-1$, and our second morphism $h_2$ has the desired property if $u=a_l-1$. 
If, on the other hand, the vertex $x$ is not $j_{a_l-1}$ then either there is no arrow $i_{a_l-1}\to j_{a_l}$ and in this case we can use the map $h_2$, or there is an arrow  $i_{a_l-1}\to j_{a_l}$, see Figure~\ref{figlem 713b}.

\begin{figure}
\begin{center}
\[
\xymatrix{j_{t_l-1}\ar[r]\ar[d] & i_{t_l-1}\ar[d] & j_{a_l+1}\ar@{.>}[l]\ar[d] \\
i_{t_l}\ar[r] & j_{t_l}\ar[lu]\ar[ru]\ar[rd] & i_{a_l+1}\ar@{<-}[d]\ar[l]&j_{a_l-2}\ar[dl] &i_{a_l-2}\ar[l]\ar[dl] \\
& i_{a_l}\ar[u]\ar[rd]&x\ar[l]\ar[rd]\ar[rru]&j_{a_l-1}\ar[l]
\\ 
&&j_{a_l}\ar[u]
&
i_{a_l-1} \ar[l]\ar[u]
}
\]
\caption{Local configuration in the situation of block $A_{l,u}$ in the case where $i_{a_l}\to j_{a_l}$, $i_{a_l-1}\to j_{a_l}$  and $i_{a_l-1}\to j_{a_l-1}$ .}
\label{figlem 713b}
\end{center}
\end{figure}
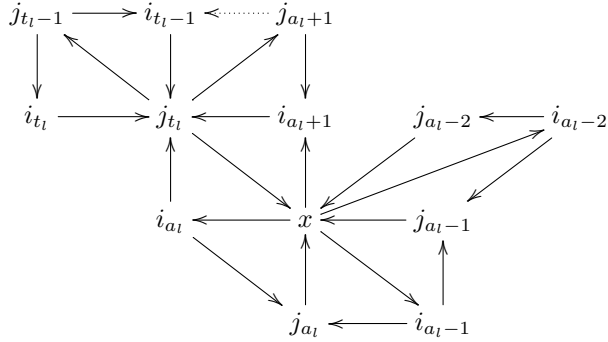
In that figure, we have an arrow from $j_{a_l-1}$ to $x$ which implies that the quiver contains a flower at $x$ and thus there exists a path from $x$ to some $i_u$ with $u<a_l-1$ (for example $u=a_l-2$) that does not factor through $i_{a_l-1}$. In this situation,
the composition $f_\zg h_2$ contains the term $-(i_{a_l-1}\to j_{a_l}\leadsto i_u)$ which is not in the block $A_{l,u}$, and again we need to compensate. In this case, the map $h_3=h_2+ (j_{a_l-1}\leadsto i_u)$ has the desired property.

This process will continue if there is another flower at a vertex $y$ that is not one of the vertices $i,j $ of the crossing sequence at that has an arrow from $x$. After a finite number of steps, it will produce a map $h$ such that $f_\zg h=g_0$ and we are done. 

Let us also remark that it is possible that  the crossing sequence has a change of direction at the pair $(i_{a_l},j_{a_l})$. In that situation, the map $h=\left\{\begin{array}{lll} 
h_1&  \textup{if $i_{a_l}\to j_{a_l}$ is an arrow;}\\
h_2&  \textup{if $i_{a_l}\leadsto j_{a_l}$ is not an arrow}\end{array}\right.$
will have the desired property. 

(iv) This case  is similar to case (ii).
 \end{proof}

Next we are going to study the case where $g_0$ is alternating on row blocks. 

\begin{lemma}
 \label{lem 714}  
 Let $g$ be a nilpotent endomorphism of $M_\zg$ such that $g_0$ is alternating on a row block and zero elsewhere else.
\begin{itemize}
\item [(a)] If the row block is not of type (iv) then $g_0f_\zg=0$. In particular $g=0$ in $\scmp\,B$.
\item [(b)] If the row block is of type (iv) then $g_0f_\zg$ is zero everywhere except possibly at two positions.
\end{itemize}

  \end{lemma}
\begin{proof} 
We may assume without loss of generality that the crossing sequence is forward. Suppose first that the row block is $F_{v,m}$ as in part (a) of Definition~\ref{def rowblocks}. Let  $(v,t)$ be a  position  in $F_{v,m}$. Suppose first that $t\ge m$.
 Since we have a maximal one-step sequence, the parameter $u$ in part (a) of Lemma \ref{lem 7.1} is equal to $t-1$. Therefore equation (\ref{eqlem 7.1a}) implies that $(g_0f_\zg)_{vt}=i_{v(t-1)}f_{(t-1)t}+i_{vt}f_{tt}$. Both terms in this expressions are equal to scalar multiples of the unique path  from $i_v$ to $j_t$ and the scalar is equal to the coefficient of $i_{v(t-1)}$ and $i_{vt}$, respectively. Thus the expression is zero because $g_0$ is alternating on the block. 
 
  Now suppose $t=m-1$. Then the column $t$ is not part of the maximal one-step sequence. Suppose first that there is not change of direction at $m-1$. Then Lemma~\ref{lem 7.1}(a) implies 
  $(g_0f_\zg)_{v(m-1)}=i_{vu}f_{u(m-1)}+i_{v(u+1)}f_{(u+1)(m-1)}$, with $u<m-2$. Thus the positions $(v,u),(v,u+1)$ are not part of the block $F_{v,m}$, hence $i_{vu}=i_{v(u+1)} =0 $, and we are done.
  Now suppose there is a change of direction at $m-1$. Then the moreover statement of Lemma~\ref{lem 7.1}(a) implies that $i_{vm}f_{mm}=0$ and thus 
  $(g_0f_\zg)_{vm}=i_{v(m-1)}f_{(m-1)m}$. But $i_{v(m-1)}=0$, since $(v,m-1)$ is not part of the block, because of Definition~\ref{def rowblocks}(a)(ii). This shows that $(g_0f_\zg)_{vm}=0$. Moreover,  $(v,m-1)$ not being part of the block also implies  
  $(g_0f_\zg)_{v(m-1)}=0$.  This completes the proof in the case when the row block is of type $F$.   
  
   Now suppose that  the row block is $E_{v,m}$ as in part (b)(i) of Definition~\ref{def rowblocks}. Let  $(v,t)$ be a  position  in $E_{v,m}$. Suppose first that $m\le t\le m'$.
 In this case, we have a maximal one-step sequence of arrows, and thus  the parameter $t'$ in part (b) of Lemma \ref{lem 7.1} is equal to $t-1$. Therefore equation (\ref{eqlem 7.1b}) implies that $(g_0f_\zg)_{vt}=i_{vt}f_{tt}+i_{v(t-1)}f_{(t-1)t}$ and both terms are nonzero. By the same argument as in the previous case, we see that the expression is zero because $g_0$ is alternating on the block.

Now suppose that $t=m-1$ and
assume first that $E_{v,m}$ is not of type (iv).
 In particular $f_{(m-1)( m-1)}$ is not an arrow.   Then in the composition $g_0f_\zg$ in column $m-1$, only the term $f_{(m-1)( m-1)}$ may contribute since $g_0$ is zero in columns $t$ with $t< m-1$.  
 Now, since $f_{(m-1)( m-1)}$ is not an arrow then by Lemma~\ref{lem 7.1} (b), $ g_0f_\zg$ is zero in column $m-1$.  Finally suppose that $t=m'+1$. Similarly to the previous case,  $f_{m'(m'+1)}$ is not an arrow.   Note, that since $f_{m'm'}$ is an arrow then $f_{(m'-1)(m'+1)}=0$, by Lemma~\ref{lem degree} (c). Thus in column $m'+1$ of  $f_\zg$, we will have two terms $f_{(m'+1)(m'+1)},f_{m'(m'+1)}$, and then $(g_0f_\zg)_{v(m'+1)}=i_{v(m'+1)}f_{(m'+1)(m'+1)}+i_{vm'}f_{m'(m'+1)}$.  Now,  $i_{vm'}f_{m'(m'+1)}=0$, because of Lemma~\ref{lem 7.1}(b) using the fact that $f_{m'(m'+1)}$ is not an arrow, and $i_{v(m'+1)}f_{(m'+1)(m'+1)}=0$, because the position $(v,m'+1)$ is not in the block.

If  the row block $E_{v,m}$ is as in part (b)(ii) of Definition~\ref{def rowblocks} then Lemma~\ref{lem 7.3} with $t'=m', t=m'+1$ implies that $(g_0f_\zg)_{v(m'+1)}$ has an additional term $i_{vm'} f_{m'(m'+1)}$. Thus
\[(g_0f_\zg)_{v(m'+1)}=i_{vm'} f_{m'(m'+1)}+i_{v(m'+1)} f_{(m'+1)(m'+1)}+i_{v(m'+2)} f_{(m'+2)(m'+1)}.\]
However, the position $(v,(m'+2))$ is not in $E_{v,m}$ so $i_{v,(m'+2)}=0$ and thus the above expression has only two non-zero terms. Again since $g_0$ is alternating on the block, the sum is zero.

If the row block is $FE_{v,m}$ or $EF_{v,m}$ as in part (b)(iii) of Definition~\ref{def rowblocks} then it consists of an $E$-block and and $F$-block that are joined. This case follows by combining the previous ones. Moreover, Lemma~\ref{lem 7.4} guarantees that the $g_0f_\zg$ is also zero at position $(v,m-1)$, because $f_{(m-1)(m-1)}$ is not an arrow.

 Thus in all cases we have $g_0f_\zg=0$. 
  In particular, Lemma~\ref{lem:cmp-morphism} implies $g=0$ in $\scmp\,B$.

(b) Now assume that $E_{v,m}$ is
of type (iv). In this case the entries $(g_0f_\zg)_{vt}$, with $t$ the first or last entry of the block may be nonzero.
  \end{proof}

\subsubsection{Proof of Theorem \ref{thm nilpotent endo} }
%

\begin{thm} 
 \label{Athm nilpotent endo}
 Let $\zg$ be a 2-diagonal and $M_\zg$ the associated indecomposable syzygy over $B$. Then $M_\zg$ does not admit any nonzero nilpotent endomorphisms in $\scmp\,B$.
\end{thm}

\begin{proof}
Let $g$ be a nilpotent endomorphism of $M_{\zg}$ with the corresponding maps $g_0, g_1$ on the projective presentation of $M_{\zg}$ as in the diagram~\eqref{diagram 72}.  Let $g_0=(i_{st}), g_1=(j_{st})$ denote the entries of these matrices.   Recall that $i_{st}$ is a path $i_s\leadsto j_t$ scaled by some coefficient from the field $\kb$, and since $g$ is nilpotent $i_{st}=0$ if $s=t$.  The strategy of the proof is as follows.  Given a nonzero entry $i_{st}$ of $g_0$ we construct a map $g_0': P_0(\zg)\to P_0(\zg)$ that induces a nilpotent endomorphism $g'$ of $M_{\zg}$ such that $g'=0$ in  $\scmp\,B$.  Moreover, $(g_0')_{st}= i_{st}$ and $g_0-g_0'$ has fewer nonzero entries than $g_0$.   Then replacing $g_0$ with $g_0-g_0'$ and repeating the argument we conclude that $g_0=0$ and, hence, $g=0$ in $\scmp\,B$.

Let $i_{st}$ be a nonzero entry of $g_0$.  Next we consider several cases based on whether $i_{st}$ belongs to a column block and/or a row block of $g_0$.  

(1) Suppose that $i_{st}$ belongs to a column block that consists of a single term.  Let $g_0'$ be the map that contains the entry $i_{st}$ in position $(s,t)$ and zero everywhere else. Then by Lemma~\ref{lem 713} $g_0'$ induces an endomorphism of $M_{\zg}$ that is zero in $\scmp\,B$.  By construction, $g_0-g_0'$ has fewer nonzero entries than $g_0$.  This shows the desired conclusion in this case. 

(2) Suppose that $i_{st}$ does not belong a column block and $f_{st}\not=0$.  Without loss of generality we may assume that $s<t$, then $t$ is not one of the $t_l$ in the staircase sequence, see Figure~\ref{fig staircase}.  In particular, $t_{l-1}\leq s<t<t_l$ for some $l$, and $f_{tu}\not=0$ if and only if $u=t, t+1, \dots, t_l$.   Moreover, note that $f_{su}\not=0$ for $u=t, \dots, t_l$.  Let $g_0'$ be the map that contains the entry $i_{st}$ in position $(s,t)$ and zero everywhere else.  Then the composition $g_0'f$ is zero everywhere except in positions $(s,u)$ where $(g_0'f_\zg)_{su}=i_{st}f_{tu}$ and $u=t, t+1, \dots, t_l$.   However, since $f_{su}\not=0$, Corollary~\ref{cor:534} implies that $(g_0'f_\zg)_{su}=0$.   Hence, $g_0'f_\zg=0$.   In particular, $g_0'$ induces an endomorphism $g'$ of $M_{\zg}$ which is zero in $\scmp\,B$.  By construction, $g_0-g_0'$ has fewer nonzero entries than $g_0$, which yields the desired conclusion in this case. 

(3) Suppose that $i_{st}$ does not belong to a column block and $f_{st}=0$.  Without loss of generality we may assume that $f_{s(s-1)}=0$, so in particular $f_{su}=0$ for all $u<s$.  
Then since $i_{st}$ does not belong to a column block and $f_{st}=0$, it follows that $s>t$ and $f_{su}$ is not an arrow for any $u\geq s$,  see Figure~\ref{fig staircase}.   In this case, Lemma~\ref{lem 7.5} parts (1b) and (4) imply the following equation.

\begin{equation}\label{eq:s}
(fg_1)_{su}=0 \text{ for } u\leq s
\end{equation}

Now we consider two subcases depending on whether $i_{st}$ belongs to a row block or not. 

(3a) Suppose that $i_{st}$ belongs to a row block 
\[ \begin{bmatrix} i_{st'} & i_{s(t'+1)} & \dots & i_{st} & \dots & i_{st''} \end{bmatrix}. \]
Then by definition of a row block there is a maximal one-step sequence in $f_\zg$ and thus column $u$ of $f_\zg$ has exactly two nonzero entries.  Then the composition 

\begin{equation}\label{eq:s2}
(g_0f)_{su}=i_{sv}f_{vu}+i_{s(v+1)}f_{(v+1)u} \text{ for } t'<u,v<t''
\end{equation}
where $v=u$ if the row block is an $F$ block and $v+1=u$ if the row block is an $E$ block.   The case when the row block is an $EF$ or an $FE$ block follows similarly, so we omit it from the discussion.   If the row block is not of type (iv) then equation~\eqref{eq:s2} also holds when $u=t''$ and $u=t'$.  In this case \eqref{eq:s} implies that two consecutive entries in $g_0$ in the row block have opposite signs.  Let $g_0'$ be the map that agrees with $g_0$ on the entries of the row block and is zero everywhere else.  By Lemma~\ref{lem 714}(a) the map $g_0'$ induces an endomorphism of $M_{\zg}$ that is zero in $\scmp\,B$. By construction, $g_0-g_0'$ has fewer nonzero entries than $g_0$, which yields the desired conclusion in the case when the row block is not of type (iv). 

Otherwise,  suppose that the row block is of type (iv).  Then $f_{t't'}$ is an arrow, and $i_{st'}f_{t't'}$ is a summand of $(g_0f)_{st'}$ or $f_{t''(t''+1)}$ is an arrow and $i_{st''}f_{t''(t''+1)}$ is a summand of $(g_0f)_{s(t''+1)}$ or both of these cases hold.   If  $(g_0f)_{st'}=i_{st'}f_{t't'}$ then equations \eqref{eq:s} and \eqref{eq:s2} imply that $g_0$ is constantly zero on the row block.  In particular, $i_{st}=0$, contrary to our assumption in the beginning of the proof, or  $(g_0f)_{st'}=0$ and here we can proceed in the same way as if the row block is not of type (iv).  If $f_{t''(t''+1)}$ is an arrow, then $f_{(t''+1)(t''+1)}$ is not an arrow because the maximal one-step sequence of arrows ends at $t''$, so $(g_0f)_{s(t''+1)}=i_{st''}f_{t''(t''+1)}$ by Lemma~\ref{lem 7.1}(b).  Then similarly to the above we conclude that $i_{st}=0$.  It remains to consider the case when $(g_0f)_{st'}=i_{st'}f_{t't'}+i_{st'''}f_{t'''t'}$ where $f_{t'''t'}$ is an arrow and $t'''<t'$ is maximal.  We claim that $i_{st'''}=0$, so again we can proceed in the same way as before. Note that $t'''<t'-1$ because the maximal one-step sequence of arrows starts at $t'+1$.  Hence, $f_{(t'-1)t'}$ is not an arrow and then Lemma~\ref{dots} implies that $f_{t'''(t'-1)}$ is an arrow and $t'''$ is maximal such that $f_{t'''(t'-1)}$ is an arrow.   In particular, Lemma~\ref{lem 7.1}(b) implies that $(g_0f)_{s(t'-1)}=i_{st'''}f_{t'''(t'-1)}$, which equals zero by \eqref{eq:s}. This shows the claim that $i_{st'''}=0$, and completes the proof in case (3a). 

(3b) Suppose that $i_{st}$ does not belong to a row block and $f_{st}=0$.    Then let $g_0'$ be the map that contains the entry $i_{st}$ in position $(s,t)$ and zero everywhere else.  First, assume that $f_{(t+1)t}\not=0$.  Then $f_{(t+2)t}\not=0$ since $t$ is not part of a maximal one-step sequence.  In particular, if $f_{tu}\not=0$ then $f_{(t+1)u}, f_{(t+2)u}$ are  also nonzero. Then the composition $g_0'f_\zg$ is zero everywhere except possibly in position $(s,u)$, where it equals $i_{st}f_{tu}$ for $u\leq s$.  However, the parameter $v$ in Lemma~\ref{lem 7.2}(b) is at least $t+2$ and equation~\eqref{eqlem 7.2b} implies that $f_{tu}$ does not contribute to $(g_0'f_\zg)_{su}$.  Therefore, $(g_0'f_\zg)_{su}=0$.   This completes the proof in case (3b) whenever $f_{(t+1)t}\not=0$.  Note that in this argument we did not need the stronger assumption of (3) that $i_{st}$ does not belong to a column block.

Now suppose that $f_{(t+1)t}=0$.  Then $(g_0'f)_{st'}$ equal zero or it equals $i_{st}f_{tt'}$ if $f_{tt'}$ is an arrow.  However, equation \eqref{eq:s} implies that $i_{st}f_{tt'}=0$ and so $g_0'f=0$.  Hence, $g_0'f=0$.   In particular, $g_0'$ induces an endomorphism $g'$ of $M_{\zg}$ which is zero in $\scmp\,B$.  By construction, $g_0-g_0'$ has fewer nonzero entries than $g_0$, which yields the desired conclusion.  This completes the proof of case (3).

(4) Suppose that $i_{st}$ belongs to a column block 
\[ \begin{bmatrix} i_{s't} & i_{(s'+1)t} & \dots & i_{st} & \dots & i_{s''t} \end{bmatrix}^T\]
of size greater than one.  Assume without loss of generality that $s>t$.

(4a) Suppose that $i_{st}$ also belongs to a row block 
\[ \begin{bmatrix} i_{st'} & i_{s(t'+1)} & \dots & i_{st} & \dots & i_{st''} \end{bmatrix}. \]
Let $C$ be the submatrix of $g_0$ with entries $i_{uv}$ where $s'\leq u\leq s''$ and $t'\leq v\leq t''$.  Hence, every column of $C$ is a column block of $g_0$ and every row of $C$ is a row block of $g_0$.  
Let $g_0'$ be map that agrees with $g_0$ on the entries in $C$ and zero everywhere else.  We will show that $g_0'$ induces a nilpotent endomorphism $g'$ of $M_{\zg}$. Recall that every $i_{uv}$ of $C$ is a path $i_u\leadsto i_v$ multiplied by a  coefficient $a_{uv} \in \kb$.
Denote by $A=(a_{uv})$  the matrix of the coefficients in $C$.

First, suppose that the row blocks in $C$ are not of type (iv). 
Assume further that if the column block is $A_{l,u}$ of Definition~\ref{def blocks of g0} then there does not exist a $v>s''$ such that $f_{s''v}$ is an arrow. In this situation, we may use 
Lemma~\ref{column constant} to conclude that for all $s'\leq u\leq s''$ and $t'\leq v\leq t''$ we have relations $
 a_{uv}+a_{u(v+1)} =a_{(u+1)v}+a_{(u+1)(v+1)} $ which imply
\begin{equation}\label{eq coeff}
 a_{(u+1)v} -a_{uv}   = - ( a_{(u+1)(v+1)}-a_{u(v+1)} )
\end{equation}

Define a matrix $A_{col}$ of the same size as $A$ such that every column of $A_{col}$ is constant and equal to the first entry in the corresponding column of $A$. Then 
  \[A-A_{col}=\left[\begin{matrix} 
 0&0&0&\cdots\\
 a_{(s'+1)s'}-a_{s's'} & a_{(s'+1)(s'+1)}-a_{s'(s'+1)}& a_{(s'+1)(s'+2)}-a_{s'(s'+2)} &\cdots &\\
 a_{(s'+2)s'}-a_{s's'} & a_{(s'+2)(s'+1)}-a_{s'(s'+1)}& a_{(s'+2)(s'+2)}-a_{s'(s'+2)} &\cdots & \\
 \vdots&\vdots&\vdots&\vdots&
 \end{matrix}\right]\] 
 We will show that this matrix is alternating on the rows. Indeed, for the first row this is trivial, and for the second row it follows directly form the relation (\ref{eq coeff}) with $u=s'$.
 For the row indexed by $u>s'+1$ it follows from the following telescoping argument
 \[a_{uv}-a_{1v}=\sum_{l=1}^{u-1} a_{(l+1)v}-a_{lv}
 \stackrel{(\ref{eq coeff})}{=} 
\sum_{l=1}^{u-1} -( a_{(l+1)(v+1)}-a_{l(v+1)})
 = 
 -(a_{u(v+1)}-a_{1(v+1)}) .\]
 Thus $g_0'$ can be written as a sum $g_0'=g_{col}+g_{row}$, where $g_{col}$ is constant on each block and $g_{row}$ is alternating on the rows. Now Lemmata \ref{lem 713}(i) and \ref{lem 714}(a) imply that $g'=0$ in $\scmp\,B$.
 
 Assume now that  the column block is $A_{l,u}$ of Definition~\ref{def blocks of g0} and there exists $v>s''$ such that $f_{s''v}$ is an arrow. Since $s>t$ this implies that the crossing sequence is forward. 
Since Lemma~\ref{column constant} also applies in this case except for the last row $s''$ of the column blocks, we can use the above argument for all rows except for $s''$. Hence we may assume that $g_0$ is zero on all positions $(s,t)$  in these blocks with $s<s''$. 
We want to find a $g_0'$ that is equal to $g_0$ on row $s''$ and zero on all rows above such that there is a morphism $h$ such that $g_0'=f_\zg h$. 

 In the very special case where $v=s''+1$ and $f_{(s''+1)(s''+1)}$ is not an arrow, we have $f_{(s''-1)(s''-1)}$ is not an arrow and $f_{(s''-1)s''}$ is an arrow (because the column block is not of size one), $f_{(s''-1)(s''+1)
}$ is zero (by Lemma~\ref{lem degree}(c)) and therefore it is easy to see that the map $h$ that is equal to the path $j_{(s''+1)}\leadsto i_t$ and zero elsewhere realizes a factorization $g_0'=f_\zg h$ whenever $g_0'$ is zero everywhere except for a path $i_{s''}\leadsto i_t$. So we are done. 
 Otherwise, $f_{s''v},f_{(s''+1)v},\dots,f_{vv}$ are arrows and we get a second set of column blocks $A_{l+1,u}$ with the property that the top entry of $A_{l+1,u}$ lies directly below the bottom entry of $A_{l,u}$. Then the block $A_{l+1,u}$ is of type (a)(ii) in Definition~\ref{def blocks of g0} and therefore  satisfies the  conditions  of Lemma~\ref{lem 713}(iii). Define $g_0'$ to be equal to $(g_0)_{s''u}$ on each entry of $A_{l+1,u}\cup\{i_{s''u}\}$ and zero elsewhere. Then Lemma~\ref{lem 713}(iii) implies $g'$ is zero in $\scmp\,B$. Subtracting $g_0'$ from $g_0$, we obtain a zero in each position on row $s''$ and we are done. This completes the case (4a) when the row blocks are not of type (iv).
 
Now, suppose that the row blocks are of type (iv).  Then at least one of $f_{t't'}, f_{t''(t''+1)}$ is an arrow.  We consider the case when $f_{t't'}$ is an arrow and the remaining cases follow similarly.   Then we still obtain the same relations in \eqref{eq coeff}, however the map $g_{row}$ which is alternating on the row blocks may no longer induce an endomorphism of $M_{\zg}$, see  Lemma \ref{lem 714}(b).  Since $f_{t't'}$ is an arrow and $f_{(t'-1)t'}$ is not an arrow, because the maximal one-step sequence of arrows starts with $t'+1$.  Then by Lemma~\ref{lem 7.5}(2a).
\begin{equation}\label{eq:t}
(g_0f)_{ut'}=i_{ut'}f_{t't'}+i_{ut'''}f_{t'''t'} \text{ for } s'\leq u\leq s''
\end{equation}
where $t'''<t'+1$ is maximal such that $f_{t'''t'}$ is an arrow, and if no such $t'''$ exists then this term does not appear in the equation.   Since $f_{t'''t'}$ is an arrow and $f_{(t'-1)t'}$ is not an arrow then Lemma~\ref{dots} implies that $f_{t'''(t'-1)}$ is also an arrow and no entry in $f$ below this one is an arrow.  Lemma~\ref{lem 7.5}(2a) implies that 
\[
(g_0f)_{u(t'-1)}=i_{ut'''}f_{t'''(t'-1)} \text{ for } s'\leq u\leq s'' 
\]
Then Lemma~\ref{column constant} implies that the coefficients $a_{ut'''}$ for $s'\leq u\leq s''$ are all the same (except maybe for $s''$, and this case is similar to the previous case). The same lemma together with \eqref{eq:t} imply that $a_{ut'}$ for all $s'\leq u\leq s''$ are also the same.   Similarly, \eqref{eq coeff} then shows that for a fixed $v$ the coefficients $a_{uv}$ are the  same for all $s'\leq u\leq s''$.  In particular, this means that the matrix $A=A_{col}$ is constant on the columns.  Hence, we conclude that $g'=0$  in $\scmp\,B$.  This completes the proof in the case (4a).

(4b) Suppose that $i_{st}$ does not belong to a row block.  If $f_{(t+1)t}\not=0$, then we proceed in the same way as in the first part of (3b).  Note that here $f_{st}=0$, since  $i_{st}$ lies in a column block but not in row block.  Therefore, we can assume that  $f_{(t+1)t}=0$.  Since $i_{st}$ is not in a row block then it is not the case that $f_{tt}, f_{(t-1)t}$ are both arrows.  Let $g_0'$ be the map that agrees with $g_0$ on the column block and is zero everywhere else.  We will show that $g_0'$ gives a nilpotent endomorphism $g'$ of $M_{\zg}$, so $g_0-g_0'$ has fewer nonzero entries then $g_0$, which yields the desired conclusion.

From Lemma~\ref{lem 7.1}(b) it follows that if the row $t$ of $f_\zg$ does not contain any arrows then the composition $g_0f_\zg$ does not contain any term that involves entries of $g_0$ coming from the column block.  Then $g_0'f_\zg=0$ and thus $g_0'$ gives a nilpotent endomorphism $g'$ of $M_{\zg}$ and the conclusion holds. 

If $f_{tt}$ is an arrow then $f_{(t-1)t}$ is not an arrow and 
\begin{equation}\label{eq:s4}
(g_0f)_{vt}=i_{vt}f_{tt}+i_{vt'}f_{t't} \text{ for } v=s', \dots, s''
\end{equation}
where $t'<t-1$ is maximal such that $f_{t't}$ is an arrow.   If no such $t'$ exists then the second summand does not appear in equation \eqref{eq:s4}.  We suppose that $t'$ exists, and the other case is a special situation of this one.  
Since $f_{(t-1)t}$ is not an arrow, by Lemma~\ref{dots} it follows that $f_{t'(t-1)}$ is an arrow and $f_{u(t-1)}$ with $u>t'$ is not an arrow.  Then
 \begin{equation}\label{eq:s5}
(g_0f)_{v(t-1)}=i_{vt'}f_{t'(t-1)} \text{ for } v=s', \dots, s''
\end{equation}
The entries of $(fg_1)_{v(t-1)}$ and $(fg_1)_{vt}$ do not depend on $v$ by Lemma~\ref{column constant} (except maybe for the last row $s''$ which is similar as in case (4a)), so equations \eqref{eq:s4} and \eqref{eq:s5} imply that either $i_{vt}f_{tt}=0$ for all $v = s', \dots s''$ and in that case $g_0'f_\zg=0$, or the coefficients of $i_{vt'}$ are nonzero and do not depend on $v$, and hence  the coefficients of $i_{vt}$ are nonzero and do not depend on $v$, thus $g_0$ is constant on the column block.  Either way we conclude that $g_0'$ gives a nilpotent endomorphism $g'$ of $M_{\zg}$ and the conclusion holds in the case when $f_{tt}$ is an arrow. 

Finally suppose  that $f_{tt}$ is not an arrow and $f_{tt'}$ is an arrow for some $t'>t$ that is minimal.  First suppose that $t'=t+1$. Since $i_{st}$ does not belong to a row block, then it is not the case that $f_{t(t+1)},  f_{(t+1)(t+1)}$ are both arrows.  Hence $f_{(t+1)(t+1)}$ is not an arrow.  Then 
\[(g_0f_\zg)_{v(t+1)}=i_{vt}f_{t(t+1)} \text{ for } v=s', \dots, s''\]
so by the same argument as above we reach the desired conclusion.   Otherwise, if $t'>t+1$ then $f_{t(t+1)}$ is not an arrow and Lemma~\ref{dots} implies that $f_{(t+1)t'}$ is an arrow.  Then the entries in the column block of $g_0$ do not appear in the composition $g_0f_\zg$ by Lemma~\ref{lem 7.1}(b), and thus $g_0'f_\zg=0$.  
This completes the proof in case (4), and finishes the proof of the theorem.
\end{proof}

\subsection{Proof of Proposition \ref{prop mesh}}\label{Asect 6.2}

\begin{prop}
 \label{Aprop mesh}  Let $\zg$ be a 2-diagonal in $\cals$  such that $R\zg$ is not a radical line. 
With the notation of diagram \textup{(\ref{eqmesh})}, we have the following identities in $\scmp\,B$.
\begin{itemize}
\item [(a)]  If $\zg$ is not short  then $h'g'=h''g''$. 
\item[(b)] if $\zg $ is short then $h'g'=0$. 
\end{itemize}
\end{prop}

\begin{proof}
(a)  
We consider two cases based on whether or not $\zg$ is a radical line. 

(1) First, suppose that  neither $\zg$ nor $R^2\zg$ is a radical line.  Since none of $\zg, R\zg, R^2\zg$ is a radical line,
we can choose representatives $\zg,\zg',\zg''$ and $R^2\zg$, as shown in  Figure~\ref{fig mesh} on the right, such that $\zg, R^2\zg$ are compatible with both $\zg', \zg''$.      

 Each of the modules $M_\za$ in the diagrams (\ref{eqmesh}) is defined via its projective presentation 
\[\xymatrix{P_1(\za)\ar[r]^{f_\za} &P_0(\za)\ar[r]& M_\za \ar[r] &0}\] 
associated to the 2-diagonal $\za\in\diag$. Each of the pivot morphisms $b\in\{g',g'',h',h''\}$ is defined on these presentations as in diagram (\ref{eq 62}) as b triple $b=(b_1,b_0,b)$. 
Let $a=h'g'-h''g'' $. We want to show that $a=0$ in $\scmp\,B$. 
Consider the following commutative diagram.
\begin{equation}
 \label{eqpivot}
 \xymatrix@C50pt{ P_1(\zg)\ar[r]^{f_\zg} \ar[d]_{a_1}&P_0(\zg)\ar[d]^{a_0}
\ar[r]
&M_\zg\ar[r]\ar[d]^a&0 \\
P_1(R^2\zg)\ar[r]^{f_{ R^2 \zg}}&P_0(R^2\zg)\ar[r]
&M_{R^2\zg}\ar[r]&0 \\
} 
\end{equation}
It suffices to show that $a_0=0$, and for that it suffices to show that, for each indecomposable summand $P(i)$ of $P_0(\zg)$ and each indecomposable summand $P(i')$ of $P_0(R^2\zg)$, the component $a_0(i,i')$  of $a_0$ that maps $P(i)$ to $P(i')$ is zero. 

If $P(i)$ is also a summand of $P_0(R^2\zg)$ then $a_0(i,i')=0$, by Definition~\ref{def:pivot-maps} of the pivot morphisms $g',g'',h'$ and $h''$. For the same reason, we have $a_0(i,i')=0$ if $P(i')$ is also a summand of $P_0(\zg)$. 

Suppose therefore that $P(i)$ is not a summand of $P_0(R^2\zg)$ and $P(i')$ is not a summand of $P_0(\zg)$. Thus the radical line $\rho(i)$ crosses $\zg$ in degree 0 but it does not cross $R^2\zg$, and the radical line $\rho(i')$ crosses $R^2\zg$ in degree 0 but it does not cross $\zg$. 

By assumption $\zg\ne \rho(i')$. This case is illustrated in the left picture in Figure \ref{fig mesh2}.
\begin{figure}
\begin{center}
\small\scalebox{0.8}{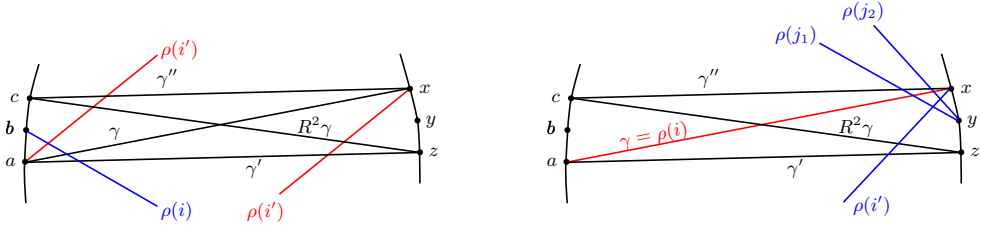}
\caption{Proof of Proposition \ref{prop mesh}. The radical line $\rho(i')$ is different from $\zg$ in the picture on the left and $\zg=\rho(i)$ in the picture on the right.}
\label{fig mesh2}
\end{center}
\end{figure}
Then the radical line $\rho(i)$ is incident to one of the points $b$ or $y$ and it crosses $\zg'$ or $\zg''$ respectively. Without loss of generality, we may assume that $\rho(i) $ is incident to $b$. Thus $\zg'$ crosses $\rho(i)$ and $\zg''$ doesn't. Consequently, the map $g'$ acts as the identity on $P(i)$.
On the other hand, the radical line $\rho(i')$ crosses $R^2\zg$ and is incident to either $a$ or $x$. Both cases are illustrated in Figure \ref{fig mesh2} in red.  If $\rho(i')$ is incident to $a$ then $\rho(i)$ and $\rho(i')$ cross, and thus there exists an arrow $\za\colon i'\to i$.
Then the maps $g'',h'\colon P(i)\to P(i')$ are given by the arrow $\za$, and the map $h'':P(i')\to P(i')$ is the identity. It follows that $a_0(i,i')=\za-\za=0$.

If $\rho(i')$ is incident to $x$ then it does not cross $\zg''$ and hence the map $h''g''\colon P(i)\to P(i')$ is zero.  Indeed, if $h''g''$ factors through some $P(k)$, a summand of $P_0(\zg'')$ then $k=i, i'$, because no radical line $\rho(k)$ can cross $\zg''$ without also crossing $\zg$ or $R^2\zg$. 
On the other hand, the map $g'$ is the identity on $P(i)$ and hence zero on $P(i)\to P(i')$, whereas the map $h'$ is the identity on $P(i')$ and hence also zero on $P(i)\to P(i')$. It follows that $a_0(i,i')=0-0=0$.  This proves the proposition in the case when $\zg$ is not short and neither $\zg$ nor $R^2\zg$ is a radical line.  

(2) Now suppose that $\zg=\rho(i)$ is a radical line. See the right picture in Figure \ref{fig mesh2}.  If $R^2\zg$ is also a radical line, then Lemma~\ref{lem:three radical lines} implies that $R\zg$ is a radical line.  This contradicts the assumption of the proposition, hence we are in the situation $\zg=\rho(i)$ and neither $R\zg$ nor $R^2\zg$ is a radical line. 
In this case, it is not possible to find all four compatible representatives $\zg, \zg' , \zg'', R^2\zg$.  Therefore, we choose particular representatives and apply automorphisms of the homotopy as follows.  
In the case of $\rho(i)$ there are exactly two possible representatives $\zg_1, \zg_2$ of $\rho(i)$ following $\rho(i)$ infinitesimally close but always staying to the right or the left of $\rho(i)$ respectively, see Figure~\ref{fig:phi}.  Moreover, these two representatives give rise to the same map $f_{\zg_1}=f_{\zg_2}=f_{\rho(i)}$, see Proposition~\ref{homotopy}.  Then the representatives for $\zg', \zg''$ are obtained from $\zg_1, \zg_2$ respectively by a pivot move so that the pairs $\zg', \zg_1$ and $\zg'', \zg_2$ are compatible. 

\begin{figure}  
\centerline{\scalebox{.9}{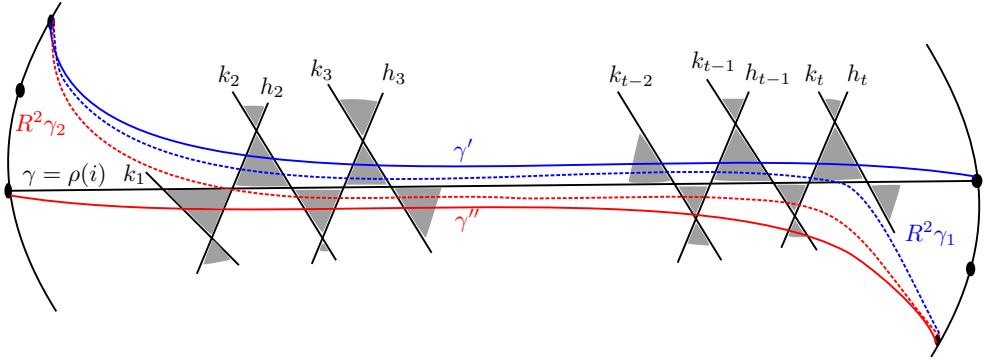}}
\caption{The proof of Proposition~\ref{prop mesh} when $\zg=\rho(i)$ is not short.}
\label{fig:phi}
\end{figure}

Now, we define $f_{R^2\zg}$ to be the map coming from a representative of $R^2\zg$ after applying a pivot move to this $\zg'$.  In this way, we see that the two chosen representatives $R^2\zg$ and $\zg'$ are compatible, and we conclude that $f_{R^2\zg}h_1' = h_0'f_{\zg'}$.   However, applying the pivot move to our chosen representative of $\zg''$ we generally obtain a different representative of the 2-diagonal $R^2\zg$, so that we may no longer guarantee that the equation $f_{R^2\zg}h_1'' = h_0''f_{\zg''}$ is satisfied for the same map $f_{R^2\zg}$.  However, according to Corollary~\ref{cor:comp} there exists an isomorphism $(\varphi_1, \varphi_0)  \in \text{Aut}(P_1(R^2\zg)) \oplus \text{Aut}(P_0(R^2\zg))$ such that 
\begin{equation}
\label{eq 613} 
 f_{R^2\zg}\varphi_1h_1'' = \varphi_0 h_0''f_{\zg''}.  
\end{equation}

Next, we find an explicit description of $\varphi_1, \varphi_0$.  Label the two representatives of $R^2\zg$ coming from $\zg_1, \zg_2$ by $R^2\zg_1, R^2\zg_2$ respectively.  Then $R^2\zg_1, R^2\zg_2$ have potentially different crossing sequences.  In Figure~\ref{fig:phi} if we move from the left to the right along $R^2\zg_1, R^2\zg_2$, then the two representatives agree until they reach a crossing pair $(k_2, h_2), (i, h_2)$ respectively, afterwards they have different crossing sequences, and then the two arcs agree again after the crossing pair $(i, h_t), (k_{t-1}, h_t)$ respectively.  Thus, we have the following crossing sequences in these cases.

\[R^2\zg_1: \hspace{1cm} \dots, (k_2, h_2), (k_3, h_3), \dots, (k_{t-1}, h_{t-1}), (i, h_t), \dots \]
\[R^2\zg_2:  \hspace{1cm}  \dots, (i, h_2), (k_2, h_3), \dots, (k_{t-2}, h_{t-1}), (k_{t-1}, h_t), \dots\]

Moreover, we note that the difference in the crossing sequences comes from 2-diagonals $k_2, \dots, k_{t-1}$ and $h_2, \dots, h_t$ that are common crossings for all of the following $\gamma,  \gamma', \gamma'', R^2\gamma$. 

Then $R^2\zg_2$ can be homotopically deformed to $R^2\zg_1$ by a sequence of moves that first replaces $(i, h_2), (k_2, h_3)$ by $(k_2, h_2), (i, h_3)$ in its crossing sequence, and then replaces $(i, h_3), (k_3, h_4)$ by $(k_3, h_3), (i, h_4)$, and continues in this way until reaching the crossing sequence for $R^2\zg_1$.  Each step $j=1, \dots, t-2$ in this process yields an isomorphism $\varphi^{(j)}$ of $P_0(R^2\zg)$ where $\varphi^{(j)}$ is obtained from $1_{P_0(R^2\zg)}$ by replacing a $2\times 2$ block $\begin{bsmallmatrix}  0 & 1_{P(i)}\\ 1_{P(k_{j+1})} & 0 \end{bsmallmatrix}$ in the matrix for $1_{P_0(R^2\zg)}$ with $\begin{bsmallmatrix}  i\to k_{j+1} & -1_{P(i)}\\ 1_{P(k_{j+1})} & 0 \end{bsmallmatrix}$ and leaving the remaining entries unchanged, see the proof of Proposition~\ref{homotopy}.  Then $\varphi_0= \varphi^{({t-2})}\dots \varphi^{(2)}\varphi^{(1)}$ and $\varphi_1=1$, so we conclude $\varphi_0 f_{R^2\zg_2}=f_{R^2\zg_1}$.  

Thus, equation (\ref{eq 613}) becomes $f_{R^2\zg} h_1'' = \varphi_0 h_0''f_{\zg''}$, and we obtain a commutative diagram as in \eqref{eqpivot}, where now we let $a_0=h_0'g_0'-\varphi_0 h_0''g_0''$ and $a_1= h_0'g_0'-h_0''g_0''$ as before.  
We want to compute $a_0(k,i'):P(k)\to P(i')$ for all $k$ and $i'$, but first, we consider what happens to the entries of $h''_0$ when we compose it with $\varphi_0$.  

Since $P(k_{2})$ is a summand of both $P_0(R^2\zg)$ and $P_0(\zg'')$, it follows form the definition of the pivot morphism that in the matrix of $h''_0$ the entry $1_{P(k_2)}$ is the only nonzero entry in its row and column.  Therefore, $\varphi^{(1)} h''_0$ replaces column $\begin{bsmallmatrix} 0 & \dots & 0 & 1_{P(k_2)} & 0 & \dots & 0\end{bsmallmatrix}^T$ with $\begin{bsmallmatrix} 0 & \dots & 0 & 1_{P(k_2)} & 0 & \dots & 0 & i\to k_2 \end{bsmallmatrix}^T$ and changes the sign of all entries in the (last) row of $h''_0$ that correspond to paths starting in $i$.   Similarly, $\varphi^{({j+1})}$ affects $\varphi^{(j)} h_0''$ by replacing a zero in the first row with an arrow $i\to k_{j+2}$ and changing the signs of all entires in this row.  This yields the desired description of $\varphi_0 h''_0$. 

To compute $a_0(k,i'):P(k)\to P(i')$, $a_0=h_0'g_0'-\varphi_0h_0''g_0''$, we consider several cases. 

(2.1) Suppose $i'\not=i$. Then we claim that $a_0(k,i')=0$. Since $i'$ crosses $R^2\zg$ in degree zero, then it crosses $\zg'$ or $\zg''$ or both in the same degree.  If $i'$ crosses both of these then it also crosses $\zg$ and $\varphi_0h_0''g_0'', h_0'g_0': P(k)\to P(i')$ is the identity map if $k=i'$ and otherwise is the zero map.  In both of these cases, we conclude that  $a_0(k,i')=0$.  

Now, suppose that $i'\not=i$ crosses $R^2\zg, \zg'$ but not $\zg''$.  Then $i'$ ends in $x$ and does not cross $\zg$, see Figure~\ref{fig mesh2} on the right.  In particular $k\not=i'$.   Then the restriction of $(h_0', -\varphi_0h''_0)$ to $P_0(\zg')\oplus P_0(\zg'')\to P(i')$ equals $P(i')\oplus P(j_1)\oplus P(j_2) \to P(i')$ with the map $[1, -(i'\to j_1), i'\to j_2]$, where $P(j_1)\oplus P(j_2)$ is a summand of $P(\zg'')$ such that the arcs $j_1, j_2$ cross $\zg''$ but not $R^2\zg$, as in the figure.  
Note that it may happen that either $j_2$ or both $j_1, j_2$ are not present, but here we depict the most general situation.    Moreover, the signs appearing here follow by  Lemma~\ref{lem:M}.
Then the restriction of $(g_0', g_0'')^T$ to $P(k)\to P(i')\oplus P(j_1)\oplus P(j_2)$ is given by the zero map if $k\not=j_1, j_2$ and otherwise it is give by $[i'\to j_1, 1, 0 ]^T, [-i'\to j_2, 0, 1 ]^T$ if $k=j_1, j_2$ respectively.  Again we conclude that $a_0(k,i')=0$ for all $k$.   The case when $i'\not=i$ crosses $R^2\zg, \zg''$ but not $\zg'$ follows similarly.  This shows the claim that $a_0(k,i')=0$ if $i\not=i'$. 

(2.2) Now, suppose that $i=i'$.  Observe, that for every summand $P(k)$ of $P_0(\zg)$ there is an arrow $\alpha:i\to k$, because  $P_0(\zg)$ is the projective cover of $\text{rad}\,P(i)$.  We claim that   $a_0(k,i)=\pm \alpha $ where the sign is the same as the sign of the arrow $\alpha$ appearing in the map $(h'_0, -\varphi_0h''_0)$.  
Observe that the maps $h_0'g_0', \varphi_0h''_0g''_0 :P(k)\to P(i)$ are either zero maps or they are given by multiplication with the arrow $\alpha$.  Therefore, we can restrict our attention to summands $P(k)$ in $P_0(\zg'), P_0(\zg'')$.  If $P(k)$ is a summand of both $P_0(\zg'), P_0(\zg'')$ then  by the description of $\varphi_0 h''_0$ obtained earlier, the restriction of $h_0', -\varphi_0h''_0$ to $P(k)\to P(i)$ is given by zero, $\pm \alpha$ respectively.  On the other hand, $g_0', g_0'': P(k)\to P(k)$ are given by the identity map, and we conclude that $a_0(k,i)=\pm \alpha$.  If $P(k)$ is a summand of  $P_0(\zg')$ but not $P_0(\zg'')$ then it is not a summand of $P_0(R^2\zg)$.  Thus, $\varphi_0h_0''g_0'':P(k)\to P(i)$ is given by the zero map while $h_0'g_0':P(k)\to P(i)$ is given by multiplication with the arrow $\pm \alpha$.  The case when $P(k)$ is a summand of $P_0(\zg'')$ but not $P_0(\zg')$ follows similarly.  This shows the desired claim that $a_0(k,i)=\pm \alpha $.


We will show that the induced morphism $a \colon M_\zg\to M_{R^2 \zg}$ on the cokernels factors through the projective $P(i)$. This will then imply that $a$ is the zero morphism in $\scmp\,B$.
To show the factorization, let $u\colon M_\zg=\rad P(i)\to P(i)$ be the inclusion morphism. As a map of projective presentations, $u$ is given by the zero map in degree 1 and by the arrows in degree 0 as follows.  For every indecomposable summand $P(k)$ of $P_0(\zg)$, in degree 0 we define $u_0: P_0(\zg)\to P(i)$ on each component $P(k)\to P(i)$ to be the multiplication with the arrow $\za: i\to k$ or $-\za$.  We can easily choose the signs in $u_0$ so that $u_0f_{\zg}=0$, because every row of the matrix of $f_{\zg}$ consists of at most two nonzero entries in adjacent columns, since $\zg$ is a radical line.  An explicit description of the map $f_{\zg}$ is given in the last paragraph of the proof of Proposition~\ref{homotopy}.  Then we obtain the following diagram.

\begin{equation}
 \label{eq 65}
 \xymatrix@C50pt{ P_1(\zg)\ar[r]^{f_\zg} \ar[d]_{0}&P(k)\oplus \overline{P_0}(\zg)\ar[d]^{u_0=[\pm \za\ \ast]}
\ar[r]
&M_\zg=\rad P(i)\ar[r]\ar[d]^u&0 \\
0\ar[r]^{0}&P(i)\ar[r]
&P(i)\ar[r]&0 \\
} 
\end{equation}
Define a map $v\colon P(i) \to M_{R^2\zg}$ as the cokernel map of the following diagram.
\begin{equation}
 \label{eq 66}
 \xymatrix@C50pt{ 0 \ar[r]^0 \ar[d]_{0}&P(i)\ar[d]^{v_0=
\left[\begin{smallmatrix}
 1\\ 0
\end{smallmatrix}\right]
 }
\ar[r]
&P(i)\ar[r]\ar[d]^v&0 \\
P_1(R^2\zg)\ar[r]^{f_{ R^2\zg}}&P(i)\oplus \overline{P_0}(R^2\zg)\ar[r]
&M_{R^2\zg}\ar[r]&0 \\
} 
\end{equation}
Thus in degree 0, the component $v_0u_0: P(k)\to P(i')$ of the composition $vu$ is given by $\pm \za$ if $i=i'$ and otherwise it equals zero.  Furthermore, we can arrange the signs, by possibly replacing $u$ with $-u$ so that the sign of every arrow $\za$ appearing in $u_0$ is opposite to the sign of $\za$ appearing in $(h_0', -\varphi_0h_0'')^T$. 
Using our computation of $a_0(k,i')$ above, we see that the difference $a(k,i')-vu$ is zero in degree 0. 
This shows that $a=vu$, and hence $a$ factors through the projective $P(i)$ and hence it is zero in the stable category $\scmp\,B$.
This completes the proof of (a). 



(b) Now, suppose that $\zg$ is short. First assume that $\zg=\rho(i) $ is a radical line.  Again using Lemma  \ref{lem:three radical lines} and our hypothesis that $R\zg$ is not a radical line, we conclude that $R^2\zg$ is also not a radical line.   Next, we consider two cases depending on whether the boundary segment between vertices $c$ and $b$ belongs to a white or a shaded region. 

Figure~\ref{fig:b} illustrates the situation when the boundary segment between vertices $c$ and $b$ belongs to a white region. Here the picture on the left shows the case when the  white region $W$ contains a boundary edge and the picture on the right shows the case when $W$ contains only a single vertex on the boundary of $\mathcal{S}$.   In the right picture, the 2-diagonal $j''$ may or may not be present depending on whether the white region $W'$ has only a vertex or an entire edge on the boundary of $\mathcal{S}$.  Thus, here we depict the most general situation.

\begin{figure}  
\centerline{\scalebox{.9}{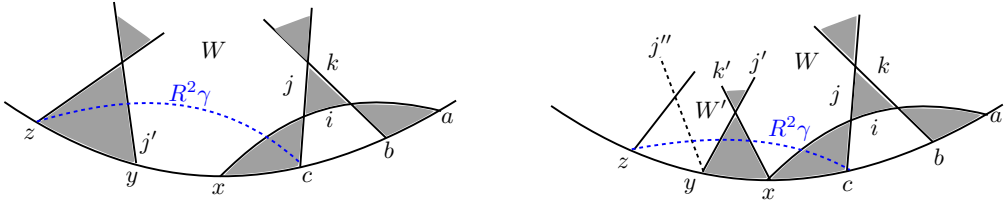}}
\caption{The proof of Proposition~\ref{prop mesh} when $\zg$ is short, and the boundary segment between $c$ and $b$ belongs to a white region.}
\label{fig:b}
\end{figure}

Then we obtain a commutative diagram with exact rows as follows coming from the situation of Figure~\ref{fig:b} on the right.  The diagram for the other case can be obtained from the given one by removing the extra summands in the projective presentations that do not cross $R\zg', R^2\zg$ according to Figure~\ref{fig:b} on the left. 

\begin{equation}
 \label{eq 655}
\xymatrix@C=24pt@R=30pt{
P(j) \ar[rrr]^{\begin{bsmallmatrix}k\to j\end{bsmallmatrix}} \ar[d]^{g_1}&&& P(k) \ar[r] \ar[d]^{g_0}& M_{\zg} \ar[r] \ar[d]^g& 0 \\
P(j'')\oplus P(j')\oplus P(j) \ar[rrr]^{\begin{bsmallmatrix}k'\leadsto j'' & k'\to j' & 0 \\ 0 & k\leadsto j' & k\to j\\ 0 & 0 & 0\end{bsmallmatrix}}\ar[d]^{h_1}&&& P(k')\oplus P(k)\oplus P(i) \ar[r] \ar[d]^{h_0}& M_{\zg'}\oplus P(i) \ar[r] \ar[d]^h& 0\\
P(j'')\oplus P(j')\ar[rrr]^{\begin{bsmallmatrix} k'\leadsto j'' & k'\to j' \\ 0 & i\to j'\end{bsmallmatrix}} &&& P(k')\oplus P(i) \ar[r] & M_{R^2\zg}\ar[r] & 0
}
\end{equation}
The vertical maps are given as follows.

\[g_1={\begin{bsmallmatrix} 0 \\ 0\\ 1_{P(j)} \end{bsmallmatrix}} 
\hspace{.5cm} g_0={\begin{bsmallmatrix} 0 \\ 1_{P(k)}\\ i\to k\end{bsmallmatrix}}
\hspace{.5cm} h_1={\begin{bsmallmatrix}1_{P(j'')} & 0 & 0 \\ 0 & 1_{P(j')} & 0\end{bsmallmatrix}}
\hspace{.5cm} h_0= {\begin{bsmallmatrix} 1_{P(k')} & 0 & 0 \\ 0 & -i\to k & 1_{P(i)}\end{bsmallmatrix}}\]
Note that here the pivot map $g_0' = [0, 1_{P(k)}]^T$ appears in the matrix $g_0$, and the pivot map $h_0' = {\begin{bsmallmatrix} 1_{P(k')} & 0   \\ 0 & -i\to k \end{bsmallmatrix}}$ appears in the matrix $h_0$.  

It is easy to see that in degree zero the composition $h_0g_0$ equals zero.  By the same reasoning as in case 2.2 of part (a), we conclude that the pivot map $M_{\zg}\to M_{\zg'}\to M_{R^2\zg}$ factors through $P(i)$, so it is zero in $\scmp\,B$.  This completes the proof in the case when the boundary  segment between vertices $c$ and $b$ belongs to a white region.  

\begin{figure}  
\centerline{\scalebox{.9}{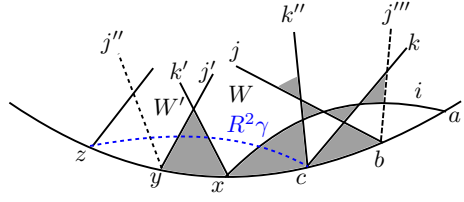}}
\caption{The proof of Proposition~\ref{prop mesh} when $\zg$ is short, and the boundary segment between $c$ and $b$ belongs to a shaded region.}
\label{fig:bb}
\end{figure}

Now, suppose that the boundary segment between $c$ and $b$ belongs to shaded triangular region.  This situation is illustrated in Figure~\ref{fig:bb}.
Note that the shaded triangle with edges $k''$ and $j$ must not be boundary, because otherwise $j=R\zg$, contrary to our assumption.  Furthermore, here we depict the most general situation when the white region $W$ has only a single vertex on the boundary.   As in the previous case, the dashed diagonals $j'', j'''$ may or may not be present.   Then we obtain the following commutative diagram with exact rows. 


\begin{equation}
\label{eq 656}
\xymatrix@C=30pt@R=30pt{
P(k)\oplus P(k'') \ar[rrr]^{\begin{bsmallmatrix} j\to k & j\to k'' \\ j'''\to k & 0\end{bsmallmatrix}} \ar[d]^{g_1}
&&& P(j)\oplus P(j''') \ar[r] \ar[d]^{g_0}
& M_{\zg} \ar[r] \ar[d]^g& 0 
\\
\genfrac{}{}{0pt}{0}{P(j'')\oplus P(j') }{\oplus P(k'')\oplus P(k)} 
\ar[rrr]^{\begin{bsmallmatrix} 0 & j\to j' & j\to k'' & j\to k \\ k' \leadsto j'' & k' \to j' & 0 & 0 \\ 0 & 0 & 0 & j'''\to k \\ 0 & 0 & 0 & 0 \end{bsmallmatrix}}\ar[d]^{h_1}
&&& 
\genfrac{}{}{0pt}{0}{P(j)\oplus P(k')}{\oplus P(j''')\oplus P(i)}
 \ar[r] \ar[d]^{h_0}& M_{\zg'}\oplus P(i) \ar[r] \ar[d]^h& 0
\\
P(j')\oplus P(j'')\ar[rrr]^{\begin{bsmallmatrix} k'\leadsto j' & k'\leadsto j'' \\ i\leadsto j' & 0 \end{bsmallmatrix}} 
&&&
 P(k')\oplus P(i) \ar[r] & M_{R^2\zg}\ar[r] & 0
}
\end{equation}
The vertical maps are given as follows.


\[\begin{array}{lll} 
g_1={\begin{bsmallmatrix} 0 & 0 \\ 0 & 0 \\ 0 & 1_{P(k'')} \\ 1_{P(k)} & 0 \end{bsmallmatrix}} 
&\hspace{.5cm} &g_0={\begin{bsmallmatrix} 1_{P(j)} & 0 \\ 0 & 0 \\ 0 & 1_{P(j''')} \\ -i\to j & i\to j'''\end{bsmallmatrix}}
\\ \\
 h_1={\begin{bsmallmatrix}0 & 1_{P(j')} & 0 & 0 \\ 1_{P(j'')} & 0 & 0 & 0\end{bsmallmatrix}}
&& h_0= {\begin{bsmallmatrix} 0 & 1_{P(k')} & 0 & 0 \\ i\to j  & 0 & -i\to j''' & 1_{P(i)}\end{bsmallmatrix}}
\end{array}\]

By the same reasoning as before, we conclude that the pivot map $M_{\zg}\to M_{\zg'}\to M_{R^2\zg}$ factors through $P(i)$, so it is zero in $\scmp\,B$. 
This completes the proof in the case when $\zg=\rho(i)$ is short.  

The case when $\zg$ is not a radical line but $R^2\zg$ is a radical line is similar to the case when $\zg$ is a radical line discussed above, and we omit the detailed discussion. 
\end{proof}

\subsection{Proof of Proposition \ref{prop:AR}}\label{Asect 7.4}

We need the following result from homological algebra. 
\begin{lemma}\label{lem:720}
Consider a commutative diagram with exact rows. 

\[\xymatrix{P_1(A) \ar[r]^{f_A} \ar[d]^{g_1}& P_0(A) \ar[r]^{\pi_A} \ar[d]^{g_0}& A \ar[r] \ar[d]^{g}& 0 \\
P_1(B) \ar[r]^{f_B} \ar[d]^{h_1}& P_0(B) \ar[r]^{\pi_B} \ar[d]^{h_0}& B \ar[r] \ar[d]^h& 0 \\
P_1(C) \ar[r]^{f_C} & P_0(C) \ar[r]^{\pi_C} & C \ar[r] & 0 }
\]
If $h_1$ is surjective and $\textup{ker}\,h_0\subset \textup{im} \,g_0$ then $\textup{ker}\,h\subset \textup{im} \,g$.
\end{lemma}

\begin{proof}
Let $b\in B$ such that $h(b)=0$.  Since the rows are exact, it follows that $\pi_B$ is surjective and there exists some $b_0\in P_0(B)$ such that $\pi_B(b_0)=b$.  By commutativity $\pi_Ch_0(b_0)=h\pi_B(b_0)=0$.  In particular, $h_0(b_0)\in \text{ker}\, \pi_C$.   Because the rows in the diagram are exact, it follows that there exists some $c_1\in P_1(C)$ such that $f_c(c_1)=h_0(b_0)$.  By assumption $h_1$ is surjective, so there exists some $b_1\in P_1(B)$ such that $h_1(b_1)=c_1$.   Then we have
\[h_0 f_B (b_1) = f_Ch_1(b_1)=f_C(c_1) = h_0(b_0)\]
and therefore $b_0-f_B(b_1)\in \text{ker}\, h_0$.  By assumption $\textup{ker}\,h_0\subset \textup{im} \,g_0$, so there exists some $a_0\in P_0(A)$ such that $g_0(a_0)=b_0-f_B(b_1)$.  Finally, we obtain 
\[g\pi_A(a_0) = \pi_B g_0(a_0)= \pi_B( b_0-f_B(b_1))= \pi_B(b_0) = b\]
where the last step follows because the rows in the diagram are exact so $\pi_Bf_B=0$.  This shows that $b\in \text{im}\,g$, thus   $\textup{ker}\,h\subset \textup{im} \,g$ as desired. 
\end{proof}

\begin{prop}\label{Aprop:AR}
Suppose $\zg$ is a 2-diagonal in $\mathcal{S}$ such that $R\zg$ is not a radical line.  
\begin{itemize}
\item[(a)] If $\zg$ 
is a radical line
$\rho(i)$ is not short then there exists a short exact sequence in $\textup{mod}\,B$ 
\[0\to M_{\zg}\to M_{\zg'}\oplus M_{\zg''} \oplus P(i) \to M_{R^2\zg}\to 0.\]
\item[(b)] If $\zg$
is a radical line 
$\rho(i)$ is short then there exists a short exact sequence in $\textup{mod}\,B$ 
\[0\to M_{\zg}\to M_{\zg'} \oplus P(i) \to M_{R^2\zg}\to 0.\]
\item[(c)] If $\{R^i \zg\}$ does not contain any radical lines then $\zg$ is not short and there exists a short exact sequence in $\textup{mod}\,B$ 
\[0\to M_{\zg}\to M_{\zg'}\oplus M_{\zg''} \to M_{R^2\zg}\to 0.\]
\end{itemize}
\end{prop}

\begin{proof}
We construct these short exact sequences from the projective presentations of these modules.   To prove part (a) consider the following diagram with exact rows coming from the projective presentations of the modules in the right column.   Moreover, we choose representatives for $\zg, \zg', \zg'', R^2\zg=R^2\zg_1$ as in Figure~\ref{fig:phi}. 

\[\xymatrix@R=30pt@C=40pt{P_1(\zg) \ar[r]^{f_{\zg}} \ar[d]^{g_1} & P_0(\zg)\ar[r]^{\pi_{\zg}} \ar[d]^{g_0} & M_{\zg}\cong \text{rad}\,P(i) \ar[r] \ar[d]^g& 0 \\
   P_1(\zg')\oplus P_1(\zg'') \ar[r]^-{\begin{bsmallmatrix} f_{\zg'} & 0 \\0 & f_{\zg''} \\ 0 & 0 \end{bsmallmatrix}}  \ar[d]^{h_1}  & P_0(\zg')\oplus P_0(\zg'')\oplus P(i)\ar[r]^{\begin{bsmallmatrix} \pi_{\zg'} & 0 & 0 \\ 0 &  \pi_{\zg''} & 0 \\ 0 & 0 & 1_{P(i)} \end{bsmallmatrix}}  \ar[d]^{h_0} & M_{\zg'}\oplus M_{\zg''}\oplus P(i) \ar[r] \ar[d]^h& 0\\
 P_1(R^2\zg) \ar[r]^{f_{R^2\zg}} & P_0(R^2\zg)\ar[r]^{\pi_{R^2\zg}} & M_{R^2\zg} \ar[r] & 0
}
\]

The vertical maps we define as follows 
\[g_1 = {\begin{bsmallmatrix} g_1'\\g_1'' \end{bsmallmatrix}} \hspace{.5cm} g_0 = \begin{bsmallmatrix} g_0'\\g_0''\\ u \end{bsmallmatrix} \hspace{.5cm} h_0=\begin{bsmallmatrix} h_0' \,& \,-\varphi_0 h_0'' \,& \,v \end{bsmallmatrix} \hspace{.5cm} h_1=\begin{bsmallmatrix} h_1' \,&\, -h_1'' \end{bsmallmatrix}.\]
The maps $g_0', g_1'$ and $g_0'', g_1''$ are the pivot morphisms coming from $\zg\mapsto\zg'$ and $\zg\mapsto\zg''$ respectively.  The maps $h_0', h_1'$ and $h_0'', h_1''$ are the pivot morphisms coming from $\zg'\mapsto R^2\zg$ and $\zg''\mapsto R^2\zg$ respectively.  Here we omit superscripts $r$ and $c$ for the pivot morphisms between the projective modules to simplify the notation.  The map $\varphi_0$  is an automorphism of $P_0(R^2\zg)$ needed to make the lower left square commute.  The same construction appears in the proof of Proposition~\ref{prop mesh}, which we refer to for further details regarding the commutativity of the diagram and the precise construction of $\varphi_0$. 

Observe that $R^2\zg$ crosses $\zg=\rho(i)$ so $P(i) \in \text{add}\, P_0(R^2\zg)$, and we define the map $v: P(i)\to P_0(R^2\zg)$ to be the inclusion that maps $P(i)\to P(i)$ via identity.  Finally, for every indecomposable summand $P(k)$ of $P_0(\zg)$ there exits an arrow $\alpha: i\to k$ in the quiver, because $P_0(\zg)$ is the projective cover of $\text{rad}\,P(i)$, so we define $u: P_0(\zg)\to P(i)$ on each component $P(k)\to P(i)$ to be the multiplication with the arrow $\alpha$ or $-\alpha$.  We can easily choose the signs in $u$ so that $uf_{\zg}=0$ because every row of the matrix of $f_{\zg}$ consists of at most two nonzero entries in adjacent columns, since $\zg$ is a radical line.  An explicit description of the map $f_{\zg}$ is given in the last paragraph of the proof of Proposition~\ref{homotopy}.  This defines the vertical maps between the projective presentations.

The maps between the projective presentations induce the corresponding maps $g,h$ between their cokernels as in the diagram, making the two squares on the right also commute.  To prove part (a) it suffices to show that $g$ is injective, $h$ is surjective, and $\text{im}\,g=\text{ker}\,h$.  

Observe that $g$ is injective, because its component $\text{rad}\,P(i)\to P(i)$ induced from the map $u$ is the inclusion map.  The module $P_0(R^2\zg)$ is a direct summand of $P_0(\zg')\oplus P_0(\zg'')\oplus P(i)$ so by the definition of the pivot maps and the map $v$ it follows that $h_0$ is surjective.  Therefore, $h$ is surjective, because both $\pi_{R^2\zg}$ and $h_0$ are surjective. Furthermore, $\text{im}\,g\subset \text{ker}\,h$ because of the mesh relations established in Proposition~\ref{prop mesh}.  Therefore, it remains to show that  $\text{ker}\,h\subset \text{im}\,g$.

Since $R\zg$ is not a radical line, it follows that $h_1$ is surjective because $P_1(R^2\zg)$ is a direct summand of $P_1(\zg')\oplus P_1(\zg'')$.   Then Lemma~\ref{lem:720} implies that $\text{ker}\,h\subset \text{im}\,g$ provided that $\text{ker}\,h_0 \subset \text{im}\,g_0$.     Therefore, it suffices to show that $\text{ker}\,h_0 \subset \text{im}\,g_0$, so we discuss these maps in more detail below. 

We already have a good understanding of the structure of the pivot maps, and we have an explicit description of $u$ and $v$. Now, we recall an explicit description of $\varphi_0 h_0''$ obtained in the proof of Proposition~\ref{prop mesh}.   In particular, the entries of the matrix $\varphi_0h_0''$ are obtained from those in $h_0''$ by replacing a zero with an arrow $\pm \alpha$, where $\alpha:i\to k$, if $P(k)$ is a summand of both $P_0(\zg')$ and $P_0(\zg'')$.   Moreover, the sign of the arrows in the matrix of $h_0$ satisfy the following important observation: 
\begin{equation}
\label{observation} 
\textup{by definition of the maps there is an arrow $x\to y$ in the matrix of $h_0$}\atop \textup{ if and only if the same arrow appears in the matrix of $g_0$.}\end{equation}
  Furthermore, by possibly replacing $u$ with $-u$ we can choose the signs so that $\pm(i\to k)$ is in $u$ if and only if $\mp(i\to k)$ is in $h_0$.

Now, we claim that $\text{ker}\,h_0 \subset \text{im}\,g_0$.
Let $m \in \text{ker}\,h_0$. We can write $m = \sum_{c\in Q_0} me_c$ where $e_c$ is the idempotent at vertex $c$ of the quiver $Q$, and it suffices to show that $me_c \in \text{im}\, g_0$ for every $c$.  Therefore, without loss of generality we may assume that every entry of $m$ consists of paths ending in some fixed vertex $c$ of $Q$.  
We think of $m$ as being a column vector, and we denote its entries corresponding to paths starting at a vertex $k$ by $m_k$.  
In particular, because there is at most one path in $B$ between any two vertices, then every entry $m_k$ of $m$ is of the form $\lambda_k (k\leadsto c)$,  a scalar multiple of a single path $k\leadsto c$ for some element $\lambda_k$ of the underlying field.  

Our strategy to show that $\text{ker}\,h_0 \subset \text{im}\,g_0$ is as follows.  We consider several cases, where we write $m=(m-m')+m'$ for some $m' \in \text{ker}\, h_0\cap \text{im}\,g_0$.  Then $m-m'\in \text{ker}\, h_0$ and it suffices to show that this element is also in $\text{im}\,g_0$.  

(a1) Suppose that $m_k\not=0$ and $P(k)$ is a summand of both $P_0(\zg')$ and $P_0(\zg'')$.  Then the map 
\[h_0=\left[ \begin{smallmatrix} \\
&&&0&&&&0&&&&0\\
&^\ast&&\vdots&&^\ast&&\vdots&&^\ast&&\vdots\\
&&&0&&&&0&&&&0\\
0&\cdots&0&1_{P(k)}&0&\cdots&0&-1_{P(k)}&0&\cdots&&0\\
&&&0&&&&0&&&&0\\
&^\ast&&\vdots&&^\ast&&\vdots&&^\ast&&\vdots\\
&&&0&&&&\pm \za&&&&1_{P(i)}\\
\end{smallmatrix}\right]\] 

Since $m$ is in the kernel of $h_0$ it follows that $m=\left[ 
\begin{smallmatrix}
 *&\cdots&*&m_k&*&\cdots&*&m_k&*&\cdots&*
\end{smallmatrix}\right]^T$, where the positions of the $m_k$ align with the columns of $h_0$ that contain $1_{P(k)}$. Define 
$m'=\left[ 
\begin{smallmatrix}
 0&\cdots&0&m_k&0&\cdots&0&m_k&0&\cdots&\mp \za m_k
\end{smallmatrix}\right]^T$.  Then it is easy to see that $h_0(m')=0$. 
On the other hand, $m'\in \im g_0$, because $g_0$ contains a column $\left[ 
\begin{smallmatrix}
 0&\cdots&0&1_{P(k)}&0&\cdots&0&1_{P(k)}&0&\cdots&\mp \za
\end{smallmatrix}\right]^T$. Thus $m'\in \ker h_0 \cap \im g_0$ and $m-m'\in \ker h_0$. Hence, in order to show that $m=(m-m')+m'$ is in the image of $g_0$ it suffices to show that $m-m'$ is in the image of $g_0$.
  

(a2) Suppose that $m_k\not=0$ and $P(k)$ is a summand of $P_0(\zg')$ but not $P_0(\zg'')$.  Then $1_{P(k)}$ is an entry in $h_0'$.  Since $m\in\text{ker}\, h_0$ then the row of $h_0$ containing $1_{P(k)}$ also contains at least one other nonzero entry, an arrow $-(k\to k'')$, and such that the path $m_k$ starts with the arrow $k\to k''$.  In particular, we can write $m_k=(k\to k'')\circ m_k''$ where $m_k''$ is a path $\lambda_k(k''\leadsto c)$.  Because of the observation (\ref{observation}),  we conclude that $P(k'')$ is a summand of $P_0(\zg)$.  Therefore, there is an arrow $\pm(i\to k'')$ in $u$.  But then by the same observation there should also be an arrow $\mp(i\to k'')$  in $h_0$.  In particular,  $h_0', -\varphi h_0''$ contain columns $\begin{bsmallmatrix} 0 & \dots & 0 & 1_{P(k)} & 0 & \dots & 0 \end{bsmallmatrix}^T$, $\begin{bsmallmatrix} 0 & \dots & 0 & -k\to k'' & 0 & \dots & \mp i\to k'' \end{bsmallmatrix}^T$ respectively, and $v$ is as in case (a1).  Then define 
\[m' = \begin{bsmallmatrix}0 & \dots & 0 & m_k & 0 & \dots & m_k'' & 0 & \dots & 0 & \pm (i\to k'')\circ m_k''\end{bsmallmatrix}^T.\]
Then $m'\in \text{ker}\,h_0$, because the matrix of $g_0$ contains a column 
\[\begin{bsmallmatrix}0 & \dots & 0 & k\to k'' & 0 & \dots & 0 & 1_{P(k'')} & 0 & \dots & 0 & \pm(i\to k'') \end{bsmallmatrix}^T\]
and again we can see that $m'$ is the image under $g_0$ of the vector that contains a single nonzero entry $m_k''$.   Therefore, applying the process in (a2) recursively we can construct $m'$ such that $m-m'$ is a vector whose every entry $(m-m')_k$ is zero whenever $P(k)$ is a summand of $P_0(\zg')$ but not $P_0(\zg'')$.

(a3) The case when $m_k\not=0$ and $P(k)$ is a summand of $P_0(\zg'')$ but not $P_0(\zg')$ follows in the same way as (a2) so we omit the detailed discussion.  

Therefore, by applying cases (a1)-(a3) and replacing $m$ with $m-m'$, it suffices to show that an element $m$ whose every entry $m_k$ is zero whenever $P(k)$ is a summand of $P_0(\zg')\oplus P_0(\zg'')$ belongs to $\text{im}\,g_0$.

(a4) Now suppose that $h_0$ contains arrows $k\to k', k\to k''$ in the same row of its matrix with $i\not=k$.  Then these arrows would also appear in $g_0$ by our earlier observation, so in $u$ we would have arrows $i\to k', i\to k''$ up to sign.  Since $i\not=k$, then there is a subquiver of $Q$ of the form 
\[\xymatrix@R=7pt@C=7pt{k\ar[r] \ar[d]& k'\\ k''&i\ar[u]\ar[l]}\]
which is not possible because there must be a sequence of 3-cycles at $i$ connecting $k'$ and $k''$ and another sequence of 3-cycles at $k$ connecting $k'$ and $k''$.  In particular, $Q$ would have interior vertices, which is a contradiction.  This shows that no row of $h_0$ contains two arrows except possibly for the last one corresponding to paths starting in $i$. 

(a5) 
Now suppose that $m_i\not=0$, recall that $M_\zg=\text{rad}\,P(i)$.  Since $m\in \text{ker}\,h_0$ there exists $m_{k''}\not=0$ such that the path $m_i$ starts with the arrow $i\to k''$.  We can write $m_i = (i\to k'')\circ m'_{k''}$ for some path $m'_{k''}$.  By steps (a1)-(a3) we have that $1_{P(k'')}$ is not in $h_0$, because otherwise $m_{k''}$ would be zero.  Then either $\pm i\to k''$ is the only nonzero entry in its column of $h_0$ or there exists another arrow $-k\to k''$ in this column.  If $-(k\to k'')\circ m'_{k''}\not=0$ then, since $m\in \text{ker}\,h_0$, it follows that there is another arrow $k\to k'''$ or $1_{P(k)}$ in $h_0$ in the same row as $-(k\to k'')$ such that $m_{k'''}$ or $m_{k}$ are nonzero.  This contradicts (a4) and (a3) respectively.  This shows that $\pm i\to k''$ is either the only nonzero entry in its column or there exists another nonzero entry $-(k\to k'')$ such that $-(k\to k'')\circ m'_{k''}=0$.  Then define
\[m' = \begin{bsmallmatrix} 0 & \dots 0 & \mp m'_{k''} & 0 & \dots & 0 & m_{i} \end{bsmallmatrix}^T\]
which belongs to the kernel of $h_0$.  The map $g_0$ contains a column 
\[\begin{bsmallmatrix} 0 & \dots & 0 & k\to k'' &0 & \dots 0 & 1_{P(k'')} & 0 & \dots & 0 & \mp i\to k'' \end{bsmallmatrix}^T\]
where the first arrow is possibly zero.  We see that $g_0$ applied to the vector that has entry $\mp m'_{k''}$ and all other entries being zero yields $m'$.  Therefore, by the same reasoning as before we can reduce $m$ to $m-m'$, the element whose every entry $(m-m')_k$ is zero whenever $P(k)$ is a summand of $P_0(\zg')\oplus P_0(\zg'')\oplus P(i)$.

(a6)  Suppose that $m_{k}\not=0$ and there is an arrow $\pm i\to k$ in $h_0$. By above we know that $1_{P(k)}$ does not appear in $h_0$.  If $(i\to k)\circ m_k=0$ then by our assumptions on $m$ if there exists another arrow $k'\to k$ in $h_0$ then $(k'\to k)\circ m_k=0$.  Thus, whether or not there exists this second arrow ending in $k$, define $m' = \begin{bsmallmatrix} 0 & \dots 0 & m_k & 0 & \dots & 0\end{bsmallmatrix}^T$ belongs to the kernel of $h_0$.   There is a column 
\[\begin{bsmallmatrix} 0 & \dots & 0 & k'\to k &0 & \dots 0 & 1_{P(k)} & 0 & \dots & 0 & \mp i\to k \end{bsmallmatrix}^T\]
in $g_0$ where $k'\to k$ may or may not be there.  Again we can conclude that this $m'\in \text{im}\,g_0$ if $(i\to k)\circ m_k=0$.   

Now, suppose that $(i\to k)\circ m_k\not=0$.  By our assumptions on $m$ we have $m_i=0$, so since $m\in \text{ker}\, h_0$ we conclude that there exists some other $m_{k'}\not=0$ such that $(i\to k)\circ m_k - (i\to k')\circ m'_{k'}$ for some path $m'_{k'}$.  It follows that the two arrows $i\to k, i\to k'$ must form a square as follows. 
\[\xymatrix@R=7pt@C=7pt{i\ar[r] \ar[d] & k \ar[d]\\ k' \ar[r] & \bullet \ar[ul]}\]
By construction of $h_0$ the two arrows $i\to k, i\to k'$ would come with opposite signs.   By our assumptions on $m$ entries $1_{P(k)}, 1_{P(k')}$ do not appear in $h_0$.  Hence, the two relevant columns of $h_0$ are as follows, where $k''\to k, k'''\to k'$ may or may not be there but we depict the most general situation.  Moreover, these arrows may also have some sings, but it will not be important for what follows.  
\[ \begin{bsmallmatrix} 0 & \dots & 0 & k''\to k &0 & \dots 0 & \pm i\to k \end{bsmallmatrix}^T  \hspace{1cm} \begin{bsmallmatrix} 0 & \dots & 0 & k'''\to k' &0 & \dots 0 & \mp i\to k' \end{bsmallmatrix}^T\]
Note that $k''\not=k'''$ by (a4), and again by our assumptions on $m$ we must have that $(k''\to k)\circ m_k=0$ and $(k'''\to k')\circ m_{k'}=0$.   The latter also implies that $(k'''\to k')\circ m'_{k'}=0$.  In particular, we conclude that $m'$ which is zero except for entries $m_k, m'_{k'}$ is in the kernel of $h_0$.  Furthermore, by similar computations as before $m'$ is then in the image of $g_0$, where $g_0$ contains two relevant columns given below. 
\[ \begin{bsmallmatrix} 0 & \dots & 0 & k''\to k &0 & \dots 0 & 1_{P(k)} &0& \dots &0& \mp i\to k \end{bsmallmatrix}^T  \hspace{1cm} \begin{bsmallmatrix} 0 & \dots & 0 & k'''\to k' &0 & \dots & 0 & 1_{P(k')} &0 & \dots &0 \pm i\to k' \end{bsmallmatrix}^T\]

Finally $m-m'$ satisfying all the earlier assumptions in cases (a1)-(a6) reduces $m$ to the zero element.  This completes the proof that $\text{ker}\,h_0 \subset \text{im}\,g_0$ and shows part (a) of the proposition.  

To show part (b) we have $\zg=\rho(i)$ is short and $R\zg$ is not a radical line.  In this case $\rho(i)$ is given as in Figure~\ref{fig:b} or Figure~\ref{fig:bb}.  In Figure~\ref{fig:b}, the picture on the left shows the case when the white region $W$ contains a boundary edge and the picture on the right shows the case when $W$ contains only a single vertex on the boundary of $\mathcal{S}$.   In the right picture, the 2-diagonal $j''$ may or may not be present depending on whether the white region $W'$ has only a vertex or an entire edge on the boundary of $\mathcal{S}$.  Thus, here we depict the most general situation.  Analogous cases may occur in the situation of Figure~\ref{fig:bb}, but here we illustrate the most general scenario.


Then we obtain a commutative diagram \eqref{eq 655} with exact rows as follows coming from the situation of Figure~\ref{fig:b} on the right.  The precise description of the vertical maps are given below the diagram.   The diagram for the situation of  Figure~\ref{fig:b} on the left can be obtained from the given one by removing the extra summands in the projective presentations that do not cross $R\zg'', R^2\zg$. 




By the same reasoning as in part (a) we conclude that $h$ is surjective, $g$ is injective, and $\text{im}\,g \subset \text{ker}\,h$.  Finally, we can easily see that $\text{ker}\,h_0\subset \text{im}\,g_0$, so Lemma~\ref{lem:720} again implies that $\text{im}\,g = \text{ker}\,h$.  This shows part (b) of the proposition if $\rho(i)$ is as in Figure~\ref{fig:b}.

If $\rho(i)$ is as in Figure~\ref{fig:bb},  then we obtain a commutative diagram \eqref{eq 656} with exact rows.  Similar reasoning shows that part (b) of the proposition holds in this case as well.   This completes the proof of part (b). 

To prove part (c) of the proposition we first observe that if $\zg$ is short then $\{R^i \zg\}$ contains a radical line $\rho(i)$ where vertex $i$ is a leaf in the dual graph of the quiver $Q$.   Since by assumption $\{R^i \zg\}$ does not contain any radical lines we conclude that $\zg$ is not short.  Then we construct a commutative diagram with exact rows, which is analogous to the one in case (a) except that we remove the direct summand $P(i)$ from the exact sequence in the middle row together with the appropriate maps.  Then we obtain maps $h, g$ as before.  Because $\zg$ is not a radical line it follows that $h_0$ is surjective, hence $h$ is surjective.  Again we obtain that $\text{im}\,g\subset \text{ker}\,h$ by Proposition~\ref{prop mesh}. 
We can also show that $\text{ker}\,g_0\subset \text{im}\,h_0$ by a similar argument as in case (a), so we omit the detailed discussion.  Note that the argument in this case would be simpler as the maps $g_0, h_0$ would not contain $u$ and $v$.  Since $R(\zg)$ is not a radical line then $h_1$ is surjective, so together with $\text{ker}\,g_0\subset \text{im}\,h_0$ we conclude that $\text{ker}\,h\subset \text{im}\,g$ by Lemma~\ref{lem:720}.  It remains to show that $g$ is injective.


Consider the following commutative diagram, where we now show the superscripts $c$ and $r$ on the vertical map as they will be important.   We know that the square on the left commutes.  In the square on the right we use the notation $\tilde{f}$ to denote a map $\varphi_0 f \varphi_1$ for some automorphisms $\varphi_1, \varphi_0$, and it is constructed as follows.   Recall that Proposition~\ref{prop:projective resolution} gives a projective resolution of $M_{\zg}$ starting with he map $f_{\zg}$ and choosing the appropriate representatives for $f_{R^i\zg}$ to produce the maps $\tilde{f}_{R^i\zg}$ for $i\geq 0$ as in the statement.   Then Proposition~\ref{prop:67} describes the first three vertical maps $g_0^r, g_1^c, g_2^r$ induced by the pivot morphism from $\zg$ to $\delta$.   A dual construction allows us to produce a projective co-resolution by starting with $f_{\zg}$ and then moving to the right with $f_{R^{-1}\zg}$ so that we obtain a commuting square on the right as in the diagram below.  Note that the vertical maps are now $g_1^c, g_0^r, g_{-1}^c$ moving from left to right.   Then $\tilde{f}_{R^{-1}\zg}$ factors through $M_{\zg}$, and $\tilde{f}_{R^{-1}\zg'}\oplus\tilde{f}_{R^{-1}\zg''}$ factors through $M_{\zg'}\oplus M_{\zg''}$ as shown below.

\[
\xymatrix@R=15pt@C=15pt{
P_1(\zg)\ar[rr]^{f_{\zg}} \ar[ddd]^{g_1^c}&& P_0(\zg) \ar[ddd]^{g_0^r} \ar[rr]^{ \tilde{f}_{R^{-1} \zg} } \ar@{->>}[dr] && P_{0}(R^{-1}\zg)\ar[ddd]^{g_{-1}^c}\\
&&&M_{\zg}\ar[d]^{g}\ar@{^{(}->}[ur] \\
&&&M_{\zg'}\oplus M_{\zg''} \ar@{^{(}->}[dr]\\
P_1(\zg')\oplus P_1(\zg'')\ar[rr]^{f_\zg'\oplus f_\zg''}&&P_0(\zg')\oplus P_0(\zg'') \ar@{->>}[ur]\ar[rr]^{\tilde{f}_{R^{-1}\zg'}\oplus\tilde{f}_{R^{-1}\zg''}}&& P_{0}(R^{-1}\zg')\oplus P_{0}(R^{-1}\zg'') 
 }
\]

Because $R^{-1}\zg$ is not a radical line then $P_0(R^{-1}\zg)$ is a direct summand of $P_0(R^{-1}\zg')\oplus P_0(R^{-1}\zg'')$, and we conclude that $g_{-1}^c$ is injective. Since the square on the right commutes and $g_{-1}^c$ is injective it follows that ${g}$ is also injective.  This completes the proof of part (c). 
\end{proof}


\begin{thebibliography}{SimSko3}
\bibitem[AIR]{AIR} T. Adachi, O. Iyama and I. Reiten,  $\tau$-tilting theory, \emph{Compos. Math.\/} {\bf 150} (2014), no. 3, 415--452. 

    \bibitem[Am]{A} C. Amiot, Cluster categories for algebras of global dimension 2 and quivers with potential, \emph{Ann. Inst. Fourier} {\bf 59} no 6, (2009), 2525--2590. 

\bibitem[As]{Assem}
 I. Assem, A course on cluster tilted algebras. Homological methods, representation theory, and cluster algebras, 127--176, CRM Short Courses, Springer, Cham, 2018. 
 



%
\bibitem[ASS]{ASS} I. Assem, D. Simson and A. Skowro\`nski,\emph{ Elements of the Representation Theory of Associative Algebras, 1: Techniques of Representation Theory}, London Mathematical Society Student Texts 65, Cambridge University Press, 2006.

\bibitem[AB]{AB} M. Auslander, M. Bridger, Stable module theory, Memoirs of the American Mathematical Society, No. 94, American Mathematical Society, Providence, R.I. 1969.

\bibitem[BHL]{BHL} J. Bastian, T. Holm and S. Ladkani,
 Towards derived equivalence classification of the cluster-tilted algebras of Dynkin type $D$. {\em J. Algebra\/} {\bf 410} (2014), 277--332.
 
\bibitem[BHL2]{BHL2} J. Bastian, T. Holm and S. Ladkani,
Derived equivalence classification of the cluster-tilted algebras of Dynkin type $E$. {\em Algebr. Represent. Theory\/} {\bf 16} (2013), no. 2, 527--551. 

\bibitem[BKM]{BKM} K. Baur, A. King and B. Marsh,  Dimer models and cluster categories of Grassmannians.  
{\em Proc. Lond. Math. Soc.\/} (3) {\em 113} (2016), no. 2, 213--260.  

\bibitem[BM]{BM} K. Baur and  B. Marsh, A geometric description of m-cluster categories, {\em Trans. Amer. Math. Soc\/} {\bf 360} (2008), no. 11, 5789--5803.


\bibitem[BM2]{BM2} K. Baur and  B. Marsh, A geometric description of the m-cluster categories of type $D_n$, {\em Int. Math. Res. Not. IMRN \/} (2007), no. 4, Art. ID rnm011, 19 pp.

\bibitem[Bo]{Bocklandt} R. Bocklandt,  A dimer ABC. {\em Bull. Lond. Math. Soc.\/} {\bf 48} (2016), no. 3, 387--451.

\bibitem[BuMa]{BuMa} A. B. Buan, and B. Marsh, Cluster-tilting theory. Trends in representation theory of algebras and related topics, 1--30, {\em Contemp. Math.\/}, {\bf 406}, Amer. Math. Soc., Providence, RI, 2006.
 \bibitem[BMRRT]{BMRRT} A. B. Buan, B. Marsh, M. Reineke, I. Reiten and G. Todorov, {Tilting theory and cluster combinatorics}, \emph{Adv. Math.\/} {\bf{204}} (2006), no. 2, 572--618.

\bibitem[BMR]{BMR}  { A. B. Buan, B. Marsh and I. Reiten},
  Cluster-tilted algebras, \emph{Trans. Amer. Math. Soc.} {\bf 359}
  (2007),  no. 1, 323--332 (electronic). 
  
  \bibitem[BRT]{BRT} A. B.  Buan,  I. Reiten, and H. Thomas, From $m$-clusters to $m$-noncrossing partitions via exceptional sequences. {\em Math. Z.\/} {\bf 271} (2012), no. 3-4, 1117--1139.
%

%
\bibitem[BR]{BR} A. Beligiannis and I. Reiten, Homological and homotopical aspects of torsion theories, {\em  Mem. Amer. Math. Soc. \/} {\bf 188} (2007).



\bibitem[Bu]{Bu} R. {Buchweitz}, Maximal Cohen-Macaulay modules and Tate-cohomology over Gorenstein rings, Mathematical Surveys and Monographs, 262. American Mathematical Society, Providence, RI, 2021.      

\bibitem[BD]{BD} I. Burban and Y.~A. Drozd, Maximal Cohen-Macaulay modules over non-isolated surface singularities and matrix problems, {\em Mem. Amer. Math. Soc.\/} {\bf 248} (2017), no.~1178, vii+114 pp.

%

\bibitem[{CCS}]{CCS}  { P. Caldero, F. Chapoton and
R. Schiffler}, Quivers with relations arising from clusters ($A_n$
case), \emph{Trans. Amer. Math. Soc.} {\bf 358} (2006), no. 3, 1347--1364. 



\bibitem[C]{C} X. Chen, 
The singularity category of a quadratic monomial algebra.  
{\em Q. J. Math.\/} {\bf 69} (2018), no. 3, 1015--1033. 
\bibitem[C2]{C2} X. Chen, 
The singularity category of an algebra with radical square zero.  
{\em Doc. Math.\/} {\bf 16} (2011), 921--936. 

\bibitem[CL]{CL} X. Chen and M. Lu, Singularity categories of skewed-gentle algebras.  
{\em Colloq. Math.\/} {\bf 141} (2015), no. 2, 183--198.

\bibitem[CGL]{CGL} 
X. Chen, S. Geng and M. Lu,
The singularity categories of the cluster-tilted algebras of Dynkin type. {\em Algebr. Represent. Theory\/} {\bf 18} (2015), no. 2, 531--554. 

\bibitem[CDZ]{CDZ} X. Chen, D. Shen and G. Zhou,
The Gorenstein-projective modules over a monomial algebra.  
{\em Proc. Roy. Soc. Edinburgh} Sect. A 148 (2018), no. 6, 1115--1134.

\bibitem[DWZ]{DWZ} H. Derksen, J. Weyman and A. Zelevinsky, Quivers with potentials and their representations. I. Mutations. {\it Selecta Math.\/} (N.S.) {\bf 14} (2008), no. 1, 59--119.

%
\bibitem[GES]{GES}
A. Garcia Elsener and R. Schiffler,
On syzygies over 2-Calabi-Yau tilted algebras.  
{\em J. Algebra\/} {\bf 470} (2017), 91--121. 

\bibitem[HK]{HK} A. Hanany and K. D. Kennaway, Dimer models and toric diagrams, (2005), arXiv:hep-th/0503149.

	\bibitem[Ha]{H}
	D.~Happel.
	\newblock {\em Triangulated categories in the representation theory of
		finite-dimensional algebras}, volume 119 of {\em London Mathematical Society
		Lecture Note Series}.
	\newblock Cambridge University Press, Cambridge, 1988.
	
\bibitem[HIMO]{HIMO} M. Herschend, O. Iyama, H. Minamoto and S. Oppermann, Representation theory of Geigle-Lenzing complete intersections, Mem. Amer. Math. Soc. {\bf 285} (2023), no.~1412, vii+141 pp.; MR4580295	
\bibitem[I]{I} 
O. Iyama, Tilting Cohen-Macaulay representations, Proceedings of the International Congress of Mathematicians–Rio de Janeiro 2018. Vol. II. Invited lectures, 125--162, World Sci. Publ., Hackensack, NJ, 2018.

\bibitem[IO]{IO} O. Iyama, S. Oppermann, Stable categories of higher preprojective algebras, \emph{Adv. Math.\/} {\bf 244} (2013), 23--68.
	
	
	\bibitem[IY]{IY} O. Iyama, and Y. Yoshino,  Mutation in triangulated categories and rigid Cohen-Macaulay modules. {\em Invent. Math.\/} {\bf 172} (2008), no. 1, 117--168.

\bibitem[JKS]{JKS} B. T. Jensen, A. D. King, and X. Su,  A categorification of Grassmannian cluster algebras. {\em Proc. Lond. Math. Soc.} (3) {\bf 113} (2016), no. 2, 185--212. 
	  
\bibitem[Ka]{Kalk}  M. Kalck,  Singularity categories of gentle algebras. {\em Bull. Lond. Math. Soc.\/} {\bf 47} (2015), no. 1, 65--74. 

\bibitem[{K}]{K}  {B. Keller}, On triangulated orbit categories,
  \emph{Documenta Math.} {\bf 10} (2005), 551--581.
  
  

 \bibitem[KR]{KR} B. Keller and I. Reiten, {Cluster-tilted algebras are Gorenstein and stably Calabi-Yau}, {\em Adv. Math.\/} {\bf 211} (2007), no. 1, 123--151.  
 
 \bibitem[KR2]{KR2} B. Keller and I. Reiten, Acyclic Calabi-Yau categories, With an appendix
by Michel Van den
Bergh. {\em Compos. Math.} {\bf 144} (2008), no. 5, 1332--1348.

\bibitem[L]{L} S. Ladkani, 2CY tilted algebras that are not Jacobian. preprint {\tt arxiv:1403.6814}.


\bibitem[LW]{LW} G. Leuschke and R. Wiegand,  Cohen-Macaulay representations. Mathematical Surveys and Monographs, 181. American Mathematical Society, Providence, RI, 2012. xviii+367 pp. 

\bibitem[Lu]{Lu} M. Lu,  Singularity Categories of some 2-CY-tilted Algebras, {\em Algebr. Represent. Theory\/} {\bf 19}  (2016), no. 6, 1257--1295.
\bibitem[LZ]{LZ} M. Lu and B. Zhu, Singularity categories of Gorenstein monomial algebras.  
{\em J. Pure Appl. Algebra} {\bf 225} (2021), no. 8, 106651, 39 pp. 

\bibitem[M]{M} M. Mastroeni, 
Matrix factorizations and singularity categories in codimension two.  
{\em Proc. Amer. Math. Soc.\/} {\bf 146} (2018), no. 11, 4605--4617. 

\bibitem[O]{O} D. Orlov, Derived categories of coherent sheaves and triangulated categories of singularities.  Algebra, arithmetic, and geometry: in honor of Yu. I. Manin. Vol. II, 503--531, 
{\em Progr. Math.\/}, {\bf 270}, Birkhäuser Boston, Boston, MA, 2009. 


\bibitem[PV]{PV} A. Polishchuk and A. Vaintrob, Matrix factorizations and singularity categories for stacks. {\em Ann. Inst. Fourier\/} (Grenoble) {\bf 61} (2011), no. 7, 2609--2642. 

\bibitem[Po]{Po} A. Postnikov, Total positivity, Grassmannians, and networks, arXiv:math/0609764.

\bibitem[Pr]{Pr} M. Pressland, Calabi-Yau properties of Postnikov diagrams, {\em Forum of Mathematics, Sigma\/}, {\bf 10}  (2022) e56.

 
 \bibitem[R]{R} I. Reiten, 2-Calabi-Yau tilted algebras, {\em S\~ao Paulo J. Math. Sci.} {\bf 4}, 3 (2010), 529--545.

	
	\bibitem[Ri]{Ri}
	C.~Riedtmann.
	\newblock Algebren, {D}arstellungsk\"{o}cher, \"{U}berlagerungen und
	zur\"{u}ck.
	\newblock {\em Comment. Math. Helv.}, {\bf 55}(2):199--224, 1980.
	



	\bibitem[S1]{S1}
	R.~Schiffler.
	\newblock A geometric model for cluster categories of type {$D_n$}.
	\newblock {\em J. Algebraic Combin.}, {\bf 27}(1):1--21, 2008.
	
 \bibitem[S2]{S2} R. Schiffler, \emph{Quiver Representations\/}, CMS 
Books in Mathematics, Springer  International Publishing, 2014.

%
 \bibitem[SS]{SS4}  R. Schiffler and K. Serhiyenko, A geometric model for syzygies over 2-Calabi-Yau tilted algebras II,  {\em Intern. Math. Res. Not.\/}, rnad078, (2023), 32 pages.
 
\bibitem[SS2]{SS5}  R. Schiffler and K. Serhiyenko,  On Gorenstein algebras of finite Cohen-Macaulay type: dimer tree algebras and their skew group algebras, {\em J. Algebra} {\bf 660}, (2024) 91--133. 
 
\bibitem[Sh]{Sh} D. Shen,  The singularity category of a Nakayama algebra.  {\em J. Algebra\/} {\bf 429} (2015), 1--18. 


\bibitem[T]{T} H. Thomas, Defining an m-cluster category.  
{\em J. Algebra\/} {\bf 318} (2007), no. 1, 37--46. 
\bibitem[Tor]{Tor} H. A. Torkildsen, A geometric realization of the $m$-cluster category of affine type $A$. \emph{Comm. Algebra\/} {\bf 43} (2015), no. 6, 2541--2567.
\bibitem[Y]{Y}Y. Yoshino, {\it Cohen-Macaulay modules over Cohen-Macaulay rings}, London Mathematical Society Lecture Note Series, 146, Cambridge Univ. Press, Cambridge, 1990.

\bibitem[Z]{PZ}
P. Zhang, Gorenstein-projective modules and symmetric recollement, {\em J. Algebra\/} {\bf 388} (2013) 65--80.
\end{thebibliography}
\end{document}